\documentclass[11pt]{amsart}

\usepackage{amsfonts}
\usepackage{color}
\usepackage{amssymb,amsthm,amsmath,hyperref,}
\usepackage{epsfig}
\usepackage{graphicx,float}
\usepackage{CJK}
\usepackage{setspace}
\usepackage{cancel}
\usepackage{stmaryrd}
\usepackage{verbatim}
\usepackage{ulem}
\usepackage{makecell}

 \newtheorem{thm}{Theorem}[section]
 \newtheorem{coro}[thm]{Corollary}
 \newtheorem{lem}[thm]{Lemma}
 \newtheorem{prop}[thm]{Proposition}
 \theoremstyle{definition}
 
 \newtheorem{rem}[thm]{Remark}
 \numberwithin{equation}{section}

\def\dl{\delta}
\def\tl{\tilde}

\def\Dl{\Delta}
\def\gt{\gtrsim}

\def\sig{\sigma}

\def\N{\mathbb{N}}

\def\nn{\nonumber}

\def\eps{\epsilon}

\def\fr{\frac}
\def\al{\alpha}

\def\la{\langle}
\def\ra{\rangle}
\def\R{\mathbb{R}}
\def\Z{\mathbb{Z}}

\def\pr{\partial}
\def\nb{\nabla}

\def\les{\lesssim}
\def\lm{\lambda}

\def\om{\omega}

\def\l|{\left\|}
\def\r|{\right\|}

\newcommand{\beq}{\begin{eqnarray}}
\newcommand{\eeq}{\end{eqnarray}}
\newcommand{\beqno}{\begin{eqnarray*}}
\newcommand{\eeqno}{\end{eqnarray*}}
\newcommand{\be}{\begin{equation}}
\newcommand{\ee}{\end{equation}}
\newcommand{\beno}{\begin{equation*}}
\newcommand{\eeno}{\end{equation*}}
\newtheorem{theorem}{Theorem}[section]

\newtheorem{Lemma A.1}{Lemma A.1}
\theoremstyle{definition}

\theoremstyle{remark}

\allowdisplaybreaks[4]

\begin{document}
\title[Landau damping for VPFP]{Landau damping, collisionless limit, and stability threshold for the Vlasov-Poisson equation with  nonlinear Fokker-Planck collisions}

\author{Jacob Bedrossian}
\address[J. Bedrossian]{Department of Mathematics, University of California, Los Angeles, CA 90095, USA}
\email{jacob@math.ucla.edu}

\author{ Weiren Zhao}
\address[W. Zhao]{Department of Mathematics, New York University Abu Dhabi, Saadiyat Island, P.O. Box 129188, Abu Dhabi, United Arab Emirates.}
\email{zjzjzwr@126.com, wz19@nyu.edu}

\author{Ruizhao Zi}
\address[R. Zi]{School of Mathematics and Statistics, and Key Laboratory of Nonlinear Analysis \& Applications (Ministry of Education), Central China Normal University, Wuhan,  430079,  P. R. China.}
\email{rzz@ccnu.edu.cn}

\date{}
\maketitle
\begin{abstract}
  In this paper, we study the Vlasov-Poisson-Fokker-Planck (VPFP) equation with a small collision frequency \(0 < \nu \ll 1\), exploring the interplay between the regularity and size of perturbations in the context of the asymptotic stability of the global Maxwellian. Our main result establishes the Landau damping and enhanced dissipation phenomena under the condition that the perturbation of the global Maxwellian falls within the Gevrey-\(\frac{1}{s}\) class and obtain that the stability threshold for the Gevrey-\(\frac{1}{s}\) class with $s>s_{\mathrm{k}}$ can not be larger than $\gamma=\frac{1-3s_{\mathrm{k}}}{3-3s_{\mathrm{k}}}$ for $s_{\mathrm{k}}\in [0,\frac{1}{3}]$. Moreover, we show that for Gevrey-$\frac{1}{s}$ with $s>3$, and for $t\ll \nu^{\frac13}$, the solution to VPFP converges to the solution to Vlasov-Poisson equation without collision. 
\end{abstract}
\setcounter{tocdepth}{1}
{\small\tableofcontents}

\section{Introduction}
In this paper, we consider the following single-species Vlasov-Poisson-(nonlinear) Fokker-Planck (VPFP) equations with a neutralizing background on $\mathbb{T}^n_x\times\mathbb{R}^n_v$ (with $\mathbb{T}_x$ normalized to length $2\pi$), which models the evolution of the distribution function of electrons which are subject to the electrostatic force coming from their Coulomb interaction and to a Brownian force which models their collisions:
\be\label{VFP}
\begin{cases}
\partial_tF+v\cdot\nabla_xF+E(t,x)\cdot\nabla_vF=\nu \varrho\big(\vartheta\Delta_vF+\nabla_v\cdot((v-u)F)\big),\\[1mm]
E(t,x)=-\nabla_x(-\Delta_x)^{-1}(\varrho-1),\\[1mm]
\displaystyle\varrho(t,x)=\int_{\mathbb{R}^n}F(t,x,v)dv,\\[1mm]
\displaystyle(\varrho u)(t,x)=\int_{\mathbb{R}^n}F(t,x,v) vdv,\\[1mm]
\displaystyle\frac{1}{2}(\varrho|u|^2)(t,x)+\frac{1}{2}n(\varrho\vartheta)(t,x)=\frac12\int_{\mathbb{R}^n}F(t,x,v) |v|^2dv,\\[1mm]
F(t=0, x, v)=F_{\mathrm{in}}(x,v).
\end{cases}
\ee
The unknown is the distribution function $F (t, x, v)$, which gives the number density of (say) electrons at location $x$ moving with velocity $v$. The electrostatic force is responsible for the self-consistent force term $E(t,x)\cdot\nb_vF$. The non-local interaction between $E(t,x)$ and the density function $\varrho(t, x)$ of electrons would be through Coulomb electrostatic interactions. The Brownian force is modeled by the Fokker-Planck term $\nu \varrho\big(\vartheta\Dl_vF+\nb_v\cdot((v-u)F)\big)$ with the small coefficient $0\leq \nu\ll 1$ which represents the collision frequency. Here we use the same nonlinear Fokker-Planck collision operator as that studied in \cite{bedrossian2017suppression}, as opposed to the linear one, namely $\varrho=\vartheta=1$, which is also commonly used in studies on weak collisions \cite{johnston1971dominant, lenard1958plasma, short2002damping}. The nonlinear Fokker-Planck collision operator reflects that the probability of collisions is positively correlated with both the density and temperature, which is natural \cite{boyd2003physics, hakim_francisquez_juno_hammett_2020, Ros-Mar-Mac-Will-Judd-Dav-PhysRev}. To understand more precisely how the energy in the electric field is converted into heat and how the plasma returns to global thermodynamic equilibrium also motivates us to study the nonlinear collision operator. Moreover, mathematically, we introduce new ideas to deal with the nonlinear effect of the collision operator in weak collision region, which might be useful in the study of other types of collision operators.  A direct calculation gives the following conservation laws:
\begin{itemize}
    \item Conservation of mass: $\displaystyle\frac{d}{dt}\int_{\mathbb{T}^n}\varrho(t,x)dx=0$.
    \item Conservation of momentum: $\displaystyle\frac{d}{dt}\int_{\mathbb{T}^n}(\varrho u)(t,x)dx=0$.
    \item Conservation of energy: $\displaystyle\frac{d}{dt}\mathcal{E}(t)=0$, where the total energy is defined by
\be
\mathcal{E}(t)=\frac12\int_{\mathbb{T}^n}\int_{\mathbb{R}^n}F(t,x,v) |v|^2dvdx+\frac12\|E(t)\|^2_{L^2}.
\ee
\end{itemize}
Assume that the initial condition $F_{\mathrm{in}}$ satisfies
\be\label{initial}
\displaystyle\int_{\mathbb{T}^n\times\mathbb{R}^n}F_{\mathrm{in}}(x,v)dxdv=(2\pi)^n,\quad
\displaystyle\int_{\mathbb{T}^n\times\mathbb{R}^n}vF_{\mathrm{in}}(x,v)dxdv=0,\quad
\mathcal{E}(0)=\frac{n}{2}(2\pi)^n.
\ee
The VPFP equation \eqref{VFP} admits a steady solution, namely the global Maxwellian, 
\be\label{gM}
\mu(v)=\frac{1}{(2\pi \theta_\infty)^{\frac{n}{2}}}e^{-\frac{|v|^2}{2\theta_\infty}},
\ee
where the final temperature $\theta_\infty$ is determined by the conservation laws via the energy per unit volume:
\be
n\theta_\infty=\frac{2}{(2\pi)^n}\mathcal{E}.
\ee
We assume $\theta_\infty=1$ for the sake of simplicity.

In this paper, we study the asymptotic stability of the global Maxwellian $\mu(v)$. It is natural to introduce the perturbations
\begin{gather}
\label{pertur-1}F(t,x,v)=\mu(v)+g(t,x,v), \quad \varrho(t,x)=1+\rho(t,x),\quad \vartheta(t,x)=1+\theta(t,x),\\
\label{pertur-2}M_1(t,x)=\int_{\mathbb{R}^n}vg(t,x,v)dv,\quad M_2(t,x)=\int_{\mathbb{R}^n}|v|^2g(t,x,v)dv.
\end{gather}
Then by using the conservation laws and the initial assumptions \eqref{initial}, one deduces that
\begin{gather}\label{2thM}
\int_{\mathbb{T}^n}\rho(t,x)dx=0,\quad
\int_{\mathbb{T}^n}M_1(t,x)dx=0,\quad
\int_{\mathbb{T}^n}M_2(t,x)dx=-\|E(t)\|_{L^2}^2.
\end{gather}
In addition,  the following relations can be derived from the fourth and fifth equations of \eqref{VFP}.
\begin{gather}
\label{1thM'}(1+\rho)u=M_1,\\
\label{2thM'}n(1+\rho)\theta=M_2-(1+\rho)|u|^2-n\rho.
\end{gather}
For notational convenience, define
\be
M_\theta=(1+\rho)\theta.
\ee

Now we rewrite the system \eqref{VFP} in terms of  $g, \rho, M_1$ and $M_\theta$:
\begin{subequations}\label{VPFP-per}
\be\label{VPFP-perturbation}
\pr_tg+v\cdot\nb_xg-\nu { L}[g]+E\cdot\nb_v\mu=-E\cdot\nb_vg+\nu\mathcal{C}_\mu+\nu\mathcal{C}[g],
\ee
with initial perturbation $g(t=0, x, v)=g_{\mathrm{in}}(x,v)$, where 
\begin{align}
    &{ L}[g]:=\Dl_v g+\nb_v\cdot(vg),\\
\label{Cg}&\mathcal{C}[g]:=\rho\big(\Dl_vg+\nb_v\cdot(vg)\big)+M_\theta\Dl_vg-M_1\cdot\nb_vg,\\
&\mathcal{C}_\mu:=(M_\theta\Dl_v\mu-M_1\cdot\nb_v\mu).
\end{align}
\end{subequations}

In many plasma physics settings, the collisions are very weak, that is $\nu\ll 1$. There are many situations when physicists neglect them entirely and consider the Vlasov-Poisson equation ($\nu=0$). However, exactly how weak collisions, and collisionless effects interact with nonlinearity, has been the subject of study and debate in the physics community for nearly six decades \cite{malmberg1964collisionless, o1968effect}. 

The most famous collisonless effect is  {\bf Landau damping (L.D.)}, namely, the rapid decay of the electric field despite the lack of dissipative mechanisms. Landau damping was discovered first by Landau in 1946 \cite{landau196561} for the linearized Vlasov-Poisson equations. Landau damping is connected to the mixing in phase-space due to the transport operator, and can be considered a variant of velocity averaging. To understand the nonlinear problem, one must study carefully the important (weakly) nonlinear effect known as the {\bf plasma echo}, observed experimentally in \cite{malmberg1968plasma}. 
These oscillations are caused by nonlinear effects exciting modes which are un-mixing in phase space, causing a transient growth of the electric field, which can then, in turn, excite further oscillations and create a cascade. The cascade leads to  derivative loss. The first mathematical justification of the Landau damping for the nonlinear Vlasov-Poisson equation was given by Mouhot and Villani \cite{MouhotVillani2011}, where nearly analytic regularity for the perturbation is required due to this regularity loss. A  simpler proof was given by Bedrossian, Masmoudi, and Mouhot \cite{BMM2016}, where the required regularity is down to the Gevrey threshold $s>\frac{1}{3}$, see also \cite{grenier2021landau}. If the perturbation is in lower regularity, we refer to \cite{lin2011small}, which shows the existence of small BGK waves, and \cite{bedrossian2020nonlinear} for a mathematical study of plasma echoes. 
We also refer to \cite{bedrossian2022linearized, han2021asymptotic,huang-NHNguyen2022sharp, ionescu2024nonlinear} for the results in the unconfined setting.  As an analogue, the {\bf inviscid damping (I.D.)} of Couette flow for 2D Euler equation was discovered  earlier by Orr \cite{Orr1907} in 1907. The nonlinear inviscid damping for Couette flow was first proved by \cite{BM2015} for the infinite periodic channel setting, see also \cite{IonescuJia2020cmp} for the finite channel case. We also refer to the recent results \cite{IJ2020, MasmoudiZhao2020} for the inviscid damping of  stable monotone shear flows. If the perturbation is less regular, we refer to \cite{castro2023traveling, lin2011inviscid, sinambela2023transition} for asymptotic instability of shear flows. Both Landau damping and inviscid damping are connected to the mixing in phase-space due to the transport operator, which is common in many other fluid models. One can refer to \cite{chen2023nonlinear, Jia2020siam, Jia2020arma,IonescuJia2021,WeiZhangZhao2020, Zillinger2017,RenZhao2017,bedrossian2023nonlinear, liu2022linear, MSZ2020,zhao2023inviscid, zhaozi2023asymptotic} and the references therein.

In the weak collision case, there is a strong interaction between the transport operator and dissipative effect. More precisely, conservative transport transfers information to high frequencies where the dissipation is increasingly dominant, which leads to the {\bf enhanced dissipation (E.D.)} effect, namely the dissipation time scale $\nu^{-\frac{1}{3}}$ is short than the heat dissipation time scale $\nu^{-1}$. The enhanced dissipation together with the small enough size of the perturbation can eliminate the plasma echoes caused by the nonlinear interaction. In \cite{bedrossian2017suppression}, the author studied the interaction between enhanced dissipation and nonlinear effect and proved the asymptotic stability of the global Maxwellian for equation \eqref{VFP}, when the perturbation is in Sobolev spaces with size $\nu^{\frac{1}{3}}$. We also refer to a recent work \cite{chaturvedi2023vlasov} where the asymptotic stability of the global Maxwellian was proved under the same assumption on the initial perturbation for the Vlasov-Poisson equation with weak Landau collisions. 

It is natural to ask  in the weak collision case, what is the size requirement if the perturbation is smoother? This leads to the stability threshold problem, namely, 

{\it 
 Given a norm $\|\cdot\|_X$, find a $\gamma=\gamma(X)$ so that
\begin{align}\label{eq:STP}
\begin{aligned}
  &\|g_{\mathrm{in}}\|_{X}\leq \nu^{\gamma} \Rightarrow \text{stability, L.D., and E. D.},\\
&\|g_{\mathrm{in}}\|_{X}\gg \nu^{\gamma} \Rightarrow \text{instability}.
\end{aligned}
\end{align}
}
In the collisionless case $(\nu=0)$ and in the weak collision case under small perturbation $(\gamma\geq\frac{1}{3})$, the key estimate is to study the interaction between one of the linear decay effects with plasma echoes. In this paper, we consider the weak collision case, with initial perturbation smoother but larger compared with \cite{bedrossian2017suppression, chaturvedi2023vlasov}. We study the plasma echoes under both Landau damping and enhanced dissipation effect and show a new transient growth mechanism that is size-regularity dependent.

\subsection{Main result} In this paper, we study the stability threshold problem with initial data in Gevrey-$1/s$ class for $s>s_{\mathrm{k}}$ with $s_{\mathrm{k}}\in [0,\frac13]$. 
Our first main result states as follows: 
\begin{theorem}\label{Thm:main1}
Let $s_{\mathrm{k}}\in [0,\frac13]$, $\lambda_{\mathrm{in}}>c_0>0$, $\sigma_0 \geq n+20$, and $m\geq n+10$. For any $1\geq s>s_{\mathrm{k}}$, we define the following weighted Gevrey norm
\begin{align}\label{eq: Weight-G-norm}
\left\|g_{\mathrm{in}}\right\|_{\mathcal{G}^{s,\lambda_{\mathrm{in}},\sigma_0}_{m+2,2\nu}}^2=\int_{\mathbb{T}^n\times\mathbb{R}^n} \left|\langle \nabla_{x,v}\rangle^{\sigma_0+1} e^{\lambda_{\mathrm{in}}\langle \nabla_{x,v}\rangle^s}\left(\langle v\rangle^{m+2}e^{2\nu |v|^2}g_{\mathrm{in}}(x, v)\right)\right|^2dxdv.
\end{align}
Then the stability threshold $\gamma(\mathcal{G}^{s,\lambda_{\mathrm{in}},\sigma}_{m+2,2\nu})$ can not be larger than $\frac{1-3s_{\mathrm{k}}}{3-3s_{\mathrm{k}}}$, namely, if
\begin{align}
\gamma=\gamma(\mathcal{G}^{s,\lambda_{\mathrm{in}},\sigma_0}_{m+2,2\nu})\geq \frac{1-3s_{\mathrm{k}}}{3-3s_{\mathrm{k}}},
\end{align}
then there exists $\epsilon_0>0$ such that for any $0<\epsilon<\epsilon_0$, it holds that
\begin{align*}
\|g_{\mathrm{in}}\|_{\mathcal{G}^{s,\lambda_{\mathrm{in}},\sigma_0}_{m+2,2\nu}}\leq \epsilon\nu^{\gamma} \Rightarrow \text{stability, L.D., and E. D.}.
\end{align*}
More precisely, if 
$\|g_{\mathrm{in}}\|_{\mathcal{G}^{s,\lambda_{\mathrm{in}},\sigma_0}_{m+2,2\nu}}\leq \epsilon \nu^{\gamma}$, then there exist $C>1$ and $\delta_1>0$, such that for any $t\in [0,\infty)$, the solution $(g, \rho, M_1, M_{\theta})$ to the system \eqref{VPFP-per} with initial data $g_{\mathrm{in}}(x, v)$ satisfies:\\
(1) Enhanced dissipation for non-zero modes
\begin{align}
    \|{P_{\neq}}g(t)\|_{L^2_{x,v}}\leq C\epsilon \nu^{\gamma} e^{-\delta_1\nu^{1/3} t}. 
\end{align}
(2) Dissipation for the zero mode
\begin{align}
    \|P_0g(t)\|_{L^2_{v}}\leq C\epsilon\nu^{\gamma}e^{-\nu t}.
\end{align}
(3) Landau damping
\begin{align}
\|\rho\|_{L^{\infty}}+\|M_1\|_{L^{\infty}}+\|M_{\theta}\|_{L^{\infty}}\leq \frac{C\epsilon \nu^{\gamma}}{\langle t \rangle^{\sigma_0-10}} e^{-\delta_1\nu^{1/3} t}.
\end{align}
Here $P_0g(t, v)=\frac{1}{(2\pi)^n}\int_{\mathbb{T}^n}g(t, x, v)dx$ and $P_{\neq}g(t, x, v)=g(t, x, v)-P_0g(t, v)$, and the constants $c_0, C, \delta_1, \epsilon$ are independent of $\nu$. 
\end{theorem}
\begin{rem}
    We refer to \eqref{equ:bootstrap} for more precise bounds of each quantity. 
\end{rem}
\begin{rem}
    This kind of `stability threshold' result is entirely analogous to the line of research on the two-dimensional Navier-Stokes equations near shear flows \cite{liMasmoudiZhao2023dynamical, MasmoudiZhao2020cpde}. Here we briefly introduce our result and compared it with the Navier-Stokes case in the following Table \ref{Table 1}. 
\begin{table}[H]
\centering
\caption{Stability threshold problem in the fluid and kinetic equations} 
\label{Table 1}
\medskip
 \begin{tabular}{|c|c|c|}
\hline
 & \makecell[c]{2D NS with small viscosity $\nu$\\
  \hline  Couette flow} & \makecell[c]{VPFP with small collision frequency $\nu$\\
   \hline global Maxwellian} \\
\hline
Sobolev & \makecell[c]{$\|u_{\mathrm{in}}\|_{H^3}\leq \epsilon\nu^{\frac{1}{3}}\ \Rightarrow$ I.D. and E.D.\\
\hline References: \cite{MasmoudiZhao2019,wei2023nonlinear}} & \makecell[c]{$\|\langle v\rangle^m g_{\mathrm{in}}\|_{H^{\sigma}}\leq \epsilon\nu^{\frac{1}{3}}\ \Rightarrow$ L.D. and E.D.\\
\hline References: \cite{bedrossian2017suppression, chaturvedi2023vlasov}}\\
\hline
Gevrey-$\frac{1}{s}$ & \makecell[c]{$\|u_{\mathrm{in}}\|_{\mathcal{G}^{s_{\mathrm{f}}}}\leq \epsilon\nu^{\gamma}\  \Rightarrow$ I.D. and E.D.  \\
\hline References: \cite{liMasmoudiZhao2022asymptotic}}&
\makecell[c]{$\|\langle v\rangle^me^{2\nu|v|^2}g_{\mathrm{in}}\|_{\mathcal{G}^{s}}\leq \epsilon\nu^{\gamma} \ \Rightarrow$ L.D. and E.D.\\
\hline this paper}\\
\hline
 \makecell[c]{Inviscid\\limit}& \makecell[c]{$\|u_{\mathrm{in}}\|_{\mathcal{G}^2}\leq \epsilon\  \Rightarrow$ I.D. and E.D.\\
\hline References: \cite{BMV2016, liMasmoudiZhao2022asymptotic}}&
\makecell[c]{$\|\langle v\rangle^me^{2\nu|v|^2}g_{\mathrm{in}}\|_{\mathcal{G}^{3-}}\leq \epsilon \ \Rightarrow$ L.D. and E.D.\\
\hline this paper}\\
\hline
\end{tabular} 
\medskip
{\scriptsize{\it \begin{itemize}
    \item Here $\gamma(\mathcal{G}^{s_\mathrm{f}})\geq \frac{1-2s_{\mathrm{f}}}{3-3s_{\mathrm{f}}}$ with $s_{\mathrm{f}}\in [0,\frac{1}{2}]$ and $\gamma(\mathcal{G}^{s})\geq \frac{1-3s_{\mathrm{k}}}{3-3s_{\mathrm{k}}}$ for $s>s_{\mathrm{k}}$ with $s_{\mathrm{k}}\in [0, \frac{1}{3}]$. 
    \item The Gevrey-$\frac{1}{s}$ norm is defined by 
    \begin{align*}
        \|f\|_{\mathcal{G}^s(\mathbb{T}^n\times\mathbb{R}^n})=\left(\sum_{k\in \mathbb{Z}^n}\int_{\mathbb{R}^n}e^{2\lambda |k,\xi|^s}|\widehat{f}(k,\xi)|^2d\xi\right)^{\frac{1}{2}},
    \end{align*}
    with some $\lambda>0$. Here we have used the short hand $|k,\xi| = |k| + |\xi|$. 
    \item The constants $ \epsilon>0$ are independent of $\nu$. The constants $\sigma,m$ are sufficiently large. 
\end{itemize} 
}}
\end{table} 
\end{rem}
\begin{rem}
Our estimates also work for the collisionless case $(\nu=0)$ by simply sending $\nu$ to 0, which gives a slightly better estimate than \cite{BMM2016} since the time growth in the top norm of the perturbation of the distribution function in \cite{BMM2016} can be avoided in this paper. This is motivated by \cite{grenier2021landau}, see also \cite{bedrossian2022brief}. 
\end{rem}
\begin{rem}
It is interesting to discuss the optimality of those indexes in Theorem \ref{Thm:main1}. There are few results about the optimality of the thresholds in both fluid and kinetic equation. We refer to \cite{deng2023long, liMasmoudiZhao2023dynamical} for  fluid equations, and refer to \cite{bedrossian2020nonlinear} for  kinetic equations. We also refer to section \ref{sec:Further discussion} for further discussions. 
\end{rem}
\begin{rem}\label{Rmk: new X}
By modifying our Fourier multipliers and some estimates, one may prove that the asymptotic stability holds for  initial data which satisfies 
\begin{align*}
\int_{\mathbb{T}^n\times\mathbb{R}^n} \left|\langle \nabla \rangle^{\sigma_0+1} \exp{\left\{\lambda_{\mathrm{in}}\min\big\{(\nu^{\gamma}|\nabla|)^{\fr{1}{L}},\,  \nu^{\frac{3\gamma-1}{2L}}\big\}\right\}}\left(\langle v\rangle^{m+2}e^{2\nu |v|^2}g_{\mathrm{in}}(x, v)\right)\right|^2dxdv\leq \epsilon^2 \nu^{2\gamma}.
\end{align*}
with $1\leq L<3$. See Section \ref{sec:Further discussion} for more discussions.
\end{rem} 
\begin{rem}
When we finished this paper, we noticed that an independent work was done by Yue Luo \cite{luo2024weak} for the Vlasov-Poisson with linear Fokker-Planck collision operator. After reading our abstract \footnote{The Workshop on the recent progress of kinetic theory and related topics in TSIMF, 15-19th January 2024.}, he improved his result.
\end{rem}

The second main result is the collisionless limit. 
\begin{theorem}\label{Thm:main2}
Let $s>\frac{1}{3}$, and let $\lambda_{\mathrm{in}}, \sigma_0, m, \epsilon_0, \nu_0$ be the same as in Theorem \ref{Thm:main1}. Suppose that the initial data $g_{\mathrm{in}}$ satisfies 
\begin{align}\label{eq: Weight-G-norm}
\left\|g_{\mathrm{in}}\right\|_{\mathcal{G}^{s,\lambda_{\mathrm{in}},\sigma_0}_{m+2,2\nu_0}}\leq \epsilon.
\end{align}
Let $g(t)$ solve \eqref{VPFP-per} for any $0<\nu\leq \nu_0$ and let $g^0(t)$ solve \eqref{VPFP-per} with $\nu=0$. There are $C>1$ and $\delta_0>0$ independent of $\nu, t$ such that for any $t\leq \delta_0\nu^{\frac13}$, 
\begin{align*}
    \|g(t)-g^0(t)\|_{L^{\infty}_{x,v}}\leq C\nu t^3.
\end{align*}
\end{theorem}

\subsection{Main ideas}
The nonlinear Fokker-Planck collision operator brings mainly two effects: the enhanced dissipation from the linear effect and a velocity localization loss in $v$ from the nonlinear effect. The main ideas in estimating and designing norms are to take advantage of the enhanced dissipation and overcome the velocity localization loss. 

\subsubsection{The time-dependent Gaussian weight}
To deal with the collision operator, the classical way is to introduce a Gaussian weight, which is commonly used to deal with the Landau collision operator and Boltzmann collision operator. The main idea is to use the Gaussian weight to absorb the bad term  and rewrite the collision operator as a non-negative self-adjoint operator. For the nonlinear Fokker-Planck collision operator, this method seems not efficient. One reason is that the kernel of the linearized Fokker-Planck collision operator is not the same as the nonlinear Fokker-Planck collision operator, namely, 
\begin{align*}
    \mathrm{Kerl}\, (L+\mathcal{C})\neq \mathrm{Kerl}\, (L). 
\end{align*}
In \cite{bedrossian2017suppression}, the author introduced a new coordinate system in frequency space by observing that after taking the Fourier transform, the term $\nu\nabla\cdot (vg)$ which creates the kernel of the linearized Fokker-Planck collision operator becomes a transport term $-\nu \xi\cdot \nabla_{\xi} \widehat{g}$ which can be absorbed by the new coordinate. Such an idea is useful to control the nonlinear terms such as $\nu\rho\nabla\cdot (vg)$, when the perturbation has enough $\nu$-dependent smallness. In our case, since the perturbation can be $\nu$-independent, we cannot apply this method directly. 

One of the  new ideas is to introduce a time-dependent Gaussian weight in $v$, namely, we define
\be\label{def-gw}
g=e^{-\lambda_1(t) |v|^2}g^w,
\ee
with
\be\label{def-lm1}
\lambda_1(t)=\nu+\frac{\nu }{\la t\ra^{a_0}},\quad 0<a_0\ll 1.
\ee
 We refer to \cite{ignatova2016almost, wang2021global, zhang2016long} where a similar time-dependent weight is used to study the global existence for the Prandtl system \footnote{W. Z. would like to thank Professor Zhifei Zhang for mentioning those references.}. 

The equation for $g^w$ can be written as (see \eqref{eq-gw})
\begin{align*}
\pr_tg^w&+v\cdot\nb_xg^w-\nu \nabla_v\cdot(vg^{w})\\
&-\nu \Delta_vg^w-\dot{\lambda}_1(t)|v|^2g^w
-2\nu\lambda_1(t)(1-2\lambda_1(t))|v|^2g^w\\
&-4\nu\lambda_1(t) v\cdot\nb_vg^w=\text{other terms}.
\end{align*}
By introducing the time-dependent Gaussian weight, we obtain two new good terms in the equation 
\begin{align}\label{eq: good-termv}
    -\dot{\lambda}_1(t)|v|^2g^w\quad \text{and}\quad -2\nu\lambda_1(t)(1-2\lambda_1(t))|v|^2g^w,
\end{align}
which can be used to control low-high interactions in the nonlinear collision terms such as $\nu(\rho+\lambda_1(\rho+M_{\theta}))\nabla\cdot (vg^w)$ and $\nu(2\lambda_1(1-2\lambda_1)\rho+4\lambda_1^2M_{\theta})|v|^2g^{w}$, where there is  velocity localization loss. We also note that since $e^{-\lambda_1(t) |v|^2}$ is not in the kernel of the linearized collision operator, there are still two bad linear terms
\begin{align*}
    -\nu\nabla\cdot (vg^w)\quad \text{and}\quad -4\nu \lambda_1(t) v\cdot\nabla_vg^w.
\end{align*}
For the first term $-\nu\nabla\cdot (vg^w)$, we treat it as a transport term in frequency space and use the idea in \cite{bedrossian2017suppression}. More precisely, after taking the Fourier transform, the problematic term $-\nu\nabla_v\cdot (vg^w)$ becomes $\nu \xi\cdot \nabla_{\xi}\widehat{g^{w}}$. We introduce the following characteristic curve to absorb the two linear transport terms $-k\nabla_{\xi}\widehat{g^w}+\nu \xi\cdot \nabla_{\xi}\widehat{g^{w}}$:
\begin{align}\label{eq-curve}
\frac{d}{dt}\bar{\eta}(t;k,\eta)=-k+\nu\bar{\eta}(t;k,\eta), \quad \bar{\eta}(0; k, \eta)=\eta,
\end{align}
which gives that 
\be\label{curve}
\bar{\eta}(t;k,\eta)=e^{\nu t}\eta-k\fr{e^{\nu t}-1}{\nu}.
\ee
Throughout the paper, to simplify the presentation, we introduce the notation
\[
t^{\rm ap} =\frac{1-e^{-\nu t}}{\nu}.
\]
Here `{\rm ap}' stands for `approximate'. It is easy to check that as $\nu\to 0_+$, $t^{\rm ap}\to t$. 
We define
\begin{align}\label{def-fw}
(\widehat{f^w})_k(t,\eta)=(\widehat{g^w})_k(t, \bar{\eta}(t; k,\eta)), \quad \mathrm{or} \quad (\widehat{g^w})_k(t, \xi)=(\widehat{f^w})_k\left(t, e^{-\nu t}\xi+kt^{\rm ap} \right).
\end{align}
We then get that $\widehat{f^w}_k(t, \eta)$ satisfies a better equation (see \eqref{eq-fw})
\begin{align*}
&\partial_t(\widehat{f^w})_k(t,\eta)+e^{-2\nu t}\dot{\lambda}_1(t)(\Delta_\eta \widehat{f^w})_k(t,\eta)+\nu|\bar{\eta}|^2(\widehat{f^w})_k(t,\eta)\\
&+\nu \left[-(2\lambda_1-4\lambda_1^2)e^{-2\nu t}(\Delta_\eta \widehat{f^w})_k\left(t, \eta \right)-4\lambda_1 e^{-\nu t}\nabla_\eta\cdot\left(\bar{\eta} \widehat{f^w}\right)_k\left(t, \eta \right)\right]\\
&=\text{other terms}. 
\end{align*}

The second term $-4\nu \lambda_1(t) v\cdot\nabla_vg^w$ can be controlled by the dissipation term and the new good terms, namely 
\begin{align*}
    -2\nu\lambda_1(t)(1-2\lambda_1(t))|v|^2g^w\quad \text{and}\quad -\nu\Delta_vg^w,
\end{align*}
when $\nu$ is small enough. 

\begin{rem}
    Since the time-dependent Gaussian weight is also $\nu$ dependent, we can not use it to get a $\nu$-uniform integrability in $v$. Thus we also use the standard polynormial weight, see more discussions in the section \ref{sec:Proof of the main theorems}. 
\end{rem}

\subsubsection{The time-dependent Fourier multiplier}\label{sec time-dependent Fourier multiplier}
Now let us consider the main nonlinear interaction and explain how the enhanced dissipation affects  the plasma echoes through the nonlinearity. We have the following formal formula for the density
\begin{align*}
\hat{\rho}(t, k)=\sum_{l\in \mathbb{Z}_*^n}\int_0^t\hat{\rho}(\tau, l)\frac{l\cdot k}{|l|^2}(t-\tau)\hat{g}_{k-l}(\tau, kt-l\tau)d\tau+\text{Easy terms}.
\end{align*}
Let us assume that the size of perturbation is $\epsilon\nu^{\gamma}$. We consider the worst-case scenario (taking the 1D example):
\begin{align}\label{eq:grow1}
\begin{aligned}
|\hat{\rho}(t, k)|
&\lesssim \sum_{l=k\pm 1}\int_{kt\approx l\tau}|\hat{\rho}(\tau, l)|\frac{\tau}{|l|}\frac{\epsilon\nu^{\gamma}e^{-\delta_1\nu^{\frac13}\tau}}{\langle kt-l\tau\rangle^{100}}d\tau+\text{Easy terms}\\
&\lesssim \frac{\epsilon t^{1-3\gamma}}{|k|^2}\left|\hat{\rho}\left(\frac{kt}{k+1}, k+1\right)\right|+\text{Easy terms}.
\end{aligned}
\end{align}
The density at time $\tau=\frac{kt}{k+1}$ and frequency $k+1$ is strongly amplifying $\hat{\rho}(t,k)$. This is a plasma echo under the enhanced dissipation effect. This kind of resonance is the synergism of the enhanced dissipation, the transient growth of the plasma echoes, and the nonlinearity. We can rewrite this plasma echo chain in terms of the distribution function $g$:
\begin{align*}
|\hat{g}(t, k, kt)|\lesssim \frac{\epsilon (|k|t)^{1-3\gamma}}{|k|^{3-3\gamma}}\left|\hat{g}\left(\frac{kt}{k+1}, k+1, kt\right)\right|+\text{Easy terms}. 
\end{align*}
By replacing $kt$ by $\eta$, we obtain that the amplification in each step is $\max\left\{\frac{\epsilon |\eta|^{1-3\gamma}}{|k|^{3-3\gamma}},1\right\}$ which together with Stirling's formula $n!\sim \sqrt{2\pi n}(\frac{n}{e})^n$ gives a total growth 
\begin{align*}
    \prod_{\frac{\epsilon |\eta|^{1-3\gamma}}{|k|^{3-3\gamma}}\geq 1}\frac{\epsilon |\eta|^{1-3\gamma}}{|k|^{3-3\gamma}}\approx \prod_{k=1}^{|\eta|^{\frac{1-3\gamma}{3-3\gamma}}} \frac{\epsilon |\eta|^{1-3\gamma}}{|k|^{3-3\gamma}}\approx e^{c|\eta|^{\frac{1-3\gamma}{3-3\gamma}}},
\end{align*}
here equality holds up to a polynomial correction. 
We refer to section \ref{Sec: Nonlinear collisionless contributions} for more precise and mathematically rigorous estimates of these plasma echoes that give the relationship  between the regularity index $s$ and the size index $\gamma$ as in Theorem \ref{Thm:main1}. 

Based on the above estimate, to capture the growth of different frequencies at different times, for $s>s_{\mathrm{k}}=\frac{1-3\gamma}{3-3\gamma}$, we introduce the time-dependent Gevrey-$\frac1s$ type Fourier multiplier 
\begin{align}
    \label{mul-A1}
 {\bf A}_k^{\sigma_0+1,\frac{2}{5}}(t,\eta)&=e^{{\mathbf 1}_{k\neq 0}\delta_1\nu^\frac25t}e^{\lambda(t)\langle k,\eta\rangle^s}\langle k,\eta\rangle^{\sigma_0+1}\mathfrak{m}_k(t, \eta)
\end{align}
with Sobolev correction $\langle k,\eta\rangle^{\sigma+1}=:(|k|^2+|\eta|^2)^{\frac{\sigma+1}{2}}$. 
Here $\lambda(t)$ is the Gevrey radius, which is defined as follows:
\be\label{def-dl}
\lambda(t)=\lambda_\infty+\fr{\tl\dl}{(1+t)^a}
\ee
for positive constants $\tl\dl$ and $a$ chosen sufficiently small, and the multiplier 
\be\label{def-m}
\mathfrak{m}_k(t,\eta)=
\begin{cases}1, \quad\mathrm{if}\quad k=0,\\
\exp\big\{-\int_0^t\fr{\nu^\fr13}{1+\nu^\fr23e^{2\nu s}\left| \eta-ks^{\mathrm{ap}}\right|^2}\fr{e^{2\nu s}|\nu\eta-k|^2}{\la e^{\nu s}(\nu\eta-k)\ra^2}ds\big\}, \quad\mathrm{if}\quad k\ne0, 
\end{cases}
\ee
is the modified ghost-type Fourier multiplier which, together with the collision operator, gives the enhanced dissipation. We refer to Appendix \ref{sec:multiplier m} for more properties of the Fourier multiplier $\mathfrak{m}(t,\nabla)$. The weight $e^{{\mathbf 1}_{k\neq 0}\delta_1\nu^\frac25t}$ gives a suboptimal enhanced dissipation rate. Therefore, 
 we also introduce the time-dependent Gevrey-$\frac1s$ type Fourier multiplier with less Sobolev correction $\langle k,\eta\rangle^{\sig_1}$
 and stronger enhanced dissipation 
\begin{align}\label{mul-A2}
 {\bf A}_k^{ \sig_1 ,\frac{1}{3}}(t,\eta)=e^{{\mathbf 1}_{k\neq 0}\delta_1\nu^\frac13t}e^{\lambda(t)\langle k,\eta\rangle^s}\langle k,\eta\rangle^{ \sig_1 }\mathfrak{m}_k(t, \eta),
\end{align}
with
\[
\sig_1+2+\fr{n+1}{2}\le \sig_0.
\]
For any $\sigma>0$ and $\mathfrak{e}\in \{\frac{1}{3},\frac{2}{5}\}$, we use the notation 
\begin{align*}
    {\bf A}^{\sigma,\mathfrak{e}}f={\bf A}^{\sigma,\mathfrak{e}}(t,\nabla)f=({\bf A}^{\sigma,\mathfrak{e}}_k(t,\xi)\hat{f}_k(t,\xi))^{\vee}
\end{align*}
to simplify the presentation.

Now let us explain briefly why we only obtain a weaker enhanced dissipation rate for the highest norm. We consider the low-high nonlinear interaction $E\cdot \nabla_v g^w$, where the electric field $E$ has rapid decay, but there is a derivative loss. The standard idea is to take advantage of the transport structure and introduce a commutator. The worst-case scenario is $(\nabla_v-t\nabla_x) \approx \nabla_v$. The idea of controlling the derivative loss and paying less smallness (in terms of $\nu$) is to use the dissipation term and the `CK' term, namely, 
\begin{align*}
\nu\|(\nabla_v-t\nabla_x){\bf A}^{\sigma,\mathfrak{e}}(v^{\al}f^w)\|_{L^2}^2\quad  \text{and}\quad \dot{\lambda}(t)\|\langle\nabla\rangle^{\frac{s}{2}}{\bf A}^{\sigma,\mathfrak{e}}(v^{\al}f^w)\|_{L^2}^2.
\end{align*}
More precisely, we have for $\sigma=\sigma_0+1$ and $\mathfrak{e}=\frac25$, it holds that
\begin{align*}
\|{\bf A}^{\sigma,\mathfrak{e}}|\nabla_v|^{1-\frac{s}{2}}(v^{\al}f^w)\|_{L^2}\leq \|{\bf A}^{\sigma,\mathfrak{e}}|\nabla_v-t\nabla_x|(v^{\alpha}f^w)\|_{L^2}^{\frac{2-2s}{2-s}}\|\langle\nabla\rangle^{\frac{s}{2}}{\bf A}^{\sigma,\mathfrak{e}}(v^{\alpha}f^w)\|_{L^2}^{\frac{s}{2-s}}.
\end{align*}
The total loss is hence,
\beno
\nu^{\gamma}\nu^{-\frac{1-s}{2-s}}< \nu^{\gamma+\frac{2}{3\gamma-5}}\leq \nu^{-\frac25},
\eeno
thus we need to gain $\nu^{\frac25}$ from the commutator, which gives the enhanced dissipation rate for the highest norm. Note that the commutator $e^{{\mathbf 1}_{k\neq 0}\delta_1\nu^\frac25t}-1$ does not absorb derivatives, but contributes smallness $\nu^{\frac25}$. For the lower norm ${\bf A}^{\sigma_1,\fr13}$, such commutator estimates are not necessary, and thus we can obtain the sharp enhanced dissipation rate $e^{-\delta_1\nu^{\frac{1}{3}}t}$.

For $\sig\ge0, \dl_1<\dl<\fr{\dl_0}{16}$, where $\dl_0$ is the constant appearing in \eqref{S-prop1},  let us denote
\be 
B_k^{\sigma}(t)=e^{\dl\nu^\fr13t}e^{\lambda(t)\left\la k, kt^{\rm ap}\right\ra^s}\left\la k, kt^{\rm ap}\right\ra^{\sigma}.
\ee
Then for the macroscopic quantities $\rho, M_1, M_{\theta}$,   the following two Fourier multipliers 
\be 
{\bf B}_k(t)=|k|^\fr12B_k^{\sigma_0}(t),\ \  {\rm and}\ \ B^{\sig_1}_k(t)
\ee 
 are involved in the higher and lower norm estimates. 

\subsubsection{Three levels of energy functionals}\label{sec:Three levels energy functionals}
We first recall that there are four main operations: 
\begin{enumerate}
    \item Applying the time-dependent Gaussian-type weight.
    \item Applying the change of coordinate.
    \item Applying the polynomial weight. 
    \item Applying the time-dependent Fourier multipliers.
\end{enumerate}
It is worth pointing out that the order of four operations is chosen to avoid the bad commutator $[{\bf A}^{\sigma}(t,\nabla),e^{\lambda_1|v|^2}]$ and derive a good working system since the inverse Gaussian weight does not commute well with the infinitely many derivatives in $A$. In \cite{chaturvedi2023vlasov}, the authors study the stability of the Vlasov-Poisson-Landau equation in Sobolev spaces. The order of operations in their paper is (1) applying the classical Gaussian weight, (2) applying the finite number of derivatives, (3) applying the polynomial weight and the exponential weight $e^{\frac{q_0|v|^{\varsigma}}{2}}$ introduced in \cite{guo2002landau, guo2012vlasov, strain2008exponential}.

It is natural to introduce the following two energy functionals for the distribution functions: $\mathcal{E}_m^{\sigma_0+1,\frac{2}{5}}$ and $\mathcal{E}^{ \sig_1 ,\frac{1}{3}}_m$, where 
\begin{align}
\|f^w(t)\|_{\mathcal{E}_m^{\sigma_0+1,\frac{2}{5}}}^2=\left(\sum_{\al\in\N^n:|\alpha|\le m} \frac{e^{-2|\alpha|\nu t}}{K_{|\alpha|}}\left\| {\bf A}^{\sigma_0+1,\frac{2}{5}}(v^\alpha f^w(t))\right\|_{L^2}^2\right)^\fr12,
\end{align}
is used to control the highest regularity, and 
\begin{align}
\|f^w(t)\|_{\mathcal{E}^{ \sig_1 ,\frac{1}{3}}_m}^2=\left(\sum_{\al\in\N^n:|\al|\le m} \frac{e^{-2|\alpha|\nu t}}{K_{|\al|}}\left\|{\bf A}^{ \sig_1 ,\frac{1}{3}}(v^\al f^w(t))\right\|_{L^2}^2\right)^\frac12,
\end{align}
is used to obtain the optimal enhanced dissipation rate. Here $\{K_{|\alpha|}\}_{|\alpha|=0}^{m+2}$ is an increasing sequence with $\frac{K_{|\alpha|}}{K_{|\alpha|+1}}\ll 1$ which is determined in the proof. As mentioned in the previous section, to control the macroscopic quantity $\rho$, we use the multiplier ${\bf B}_k$ which is of half derivative less than the highest regularity. This loss comes from the low-high interaction in 
\begin{align*}
\hat{\rho}(t, k)=\sum_{|l|\leq \frac{1}{4}|k|}\int_0^{t/2}\hat{\rho}(\tau, l)\frac{l\cdot k}{|l|^2}(t-\tau)\hat{g}_{k-l}(\tau, kt-l\tau)d\tau+\text{other terms}.
\end{align*}
Due to this derivative loss, we need to carefully treat the high-low interactions in the nonlinear collision operator, for examples
\begin{align*}
    \nu \rho \Delta_v g^{w},\quad \nu\lambda_1(1-\lambda_1) \rho |v|^2g^w.
\end{align*}
For these, we work back on the microscopic quantity and use the identity
\begin{align*}
    \hat{\rho}_k(t)=\hat{g}_k(t,  0)=\hat{f}_k\left(t, kt^{\rm ap}\right).
\end{align*}
We refer to Section \ref{sec-collision-I} for more details. We also face the velocity localization loss problem in the treatment of the high-low interaction in $\nu\lambda_1(1-\lambda_1) \rho |v|^2g^w$. Thus we introduce the third energy functional
\begin{align}
\|f^w(t)\|_{{\mathcal{E}}^{ \sig_1 ,\frac{2}{5}}_{m+2}}^2=\left(\sum_{\al\in\N^n:|\al|\le m+2} \frac{e^{-2|\alpha|\nu t}}{K_{|\al|}}\left\| {\bf A}^{ \sig_1 ,\frac{2}{5}}(v^\alpha f^w(t))\right\|_{L^2}^2\right)^\frac12,
\end{align}
which is used to control the velocity localization loss caused by the nonlinear collision operator. It is worth  pointing out that to close the energy estimate of 
${\mathcal{E}}^{ \sig_1 ,\frac{2}{5}}_{m+2}(t)$, since it is in  lower regularity, there is no derivative loss problem. In the energy estimates of ${\mathcal{E}}^{ \sig_1 ,\frac{2}{5}}_{m+2}(t)$, even the high-low interaction in $\nu\lambda_1(1-\lambda_1) \rho |v|^2g^w$ can be treated as the low-high interactions where we also take advantage of the two good terms \eqref{eq: good-termv}.  

\subsubsection{Further discussion}\label{sec:Further discussion}
There is a dual problem to the stability threshold problem \eqref{eq:STP}:

{\it 
 Given a real number $\gamma\geq 0$, find a largest function space $X_{\gamma, \nu}$ so that
\begin{align}\label{eq:STP2}
\begin{aligned}
  &\|g_{\mathrm{in}}\|_{X_{\gamma, \nu}}\leq \nu^{\gamma} \Rightarrow \text{stability, L.D., and E. D.},\\
&\|g_{\mathrm{in}}\|_{X_{\gamma, \nu}}\gg \nu^{\gamma} \Rightarrow \text{instability}.
\end{aligned}
\end{align}
}
In terms of regularity, if the function space is required to be $\nu$-independent, then the function space \eqref{eq: Weight-G-norm} in Theorem \ref{Thm:main1} seems optimal. 

Let us now discuss the case that the function space is $\nu$-dependent. We revisit \eqref{eq:grow1} and get that
\begin{align}\label{eq:grow2}
\begin{aligned}
|\hat{\rho}(t, k)|
&\lesssim \sum_{l=k\pm 1}\int_{kt\approx l\tau}|\hat{\rho}(\tau, l)|\frac{\tau}{|l|}\frac{\epsilon\nu^{\gamma}e^{-\delta_1\nu^{\frac13}\tau}}{\langle kt-l\tau\rangle^{100}}d\tau+\text{Easy terms}\\
&\lesssim \frac{\epsilon \nu^{\gamma} t e^{-\delta_1\nu^{\frac{1}{3}}t}}{|k|^2}\left|\hat{\rho}\left(\frac{kt}{k+1}, k+1\right)\right|+\text{Easy terms}.
\end{aligned}
\end{align}
By a similar argument, we write this plasma echo chain in terms of distribution function $g$:
\begin{align*}
|\hat{g}(t, k, \eta)|\lesssim  \frac{\epsilon \nu^{\gamma} |\eta| e^{-\delta_1\nu^{\frac{1}{3}}\frac{|\eta|}{|k|}}}{|k|^3}\left|\hat{g}\left(\frac{kt}{k+1}, k+1, \eta \right)\right|+\text{Easy terms}. 
\end{align*}
Thus we obtain that the amplification in each step is $\max\Big\{ \frac{\epsilon \nu^{\gamma} |\eta| e^{-\delta_1\nu^{\frac{1}{3}}\frac{|\eta|}{|k|}}}{|k|^3}, 1\Big\}$, which gives a total growth
\begin{align*}
\prod_{ \epsilon \nu^{\gamma} |\eta| e^{-\delta_1\nu^{\frac{1}{3}}{|\eta|}/{|k|}}\geq {|k|^3}} \frac{\epsilon \nu^{\gamma} |\eta| e^{-\delta_1\nu^{\frac{1}{3}}\frac{|\eta|}{|k|}}}{|k|^3}\lesssim e^{C(\nu^{\gamma}|\eta|)^\fr13}e^{-\delta_1 \nu^{\frac{1}{3}}|\eta|}
\lesssim \exp{\left\{C\min\big\{(\nu^{\gamma}|\eta|)^{\fr13},\,  \nu^{\frac{3\gamma-1}{6}}\big\}\right\}}.
\end{align*}
We also note that 
\beno
\min\Big\{(\nu^{\gamma}\langle k,\eta\rangle)^\fr{1}{L},\,  \nu^{\frac{3\gamma-1}{2L}}\Big
\}\lesssim \langle k,\eta\rangle^{s}
\eeno
with $s=\frac{3}{L}\frac{1-3\gamma}{3-3\gamma}>\frac{1-3\gamma}{3-3\gamma}$. Based on the above estimates, we can get Remark \ref{Rmk: new X} by modifying our Fourier multipliers and some estimates.  

To reduce the technical estimates, the $\nu$-independent function space \eqref{eq: Weight-G-norm} is used which is of the classical Gevrey type.

\section{The working system}
In this section, we derive a good working system as mentioned in section \ref{sec:Three levels energy functionals}. We first apply the time-dependent Gaussian weight and introduce the weighted distribution function $g^w(t, x, v)=g(t, x, v)e^{\lambda_1(t)|v|^2}$. Then we apply the change of coordinate after taking the Fourier transform to absorb the linear transport part. 
\subsection{Reformulation (I):   the distribution function}
Let $g^w$ be defined in \eqref{def-gw}.
Then we infer from \eqref{VPFP-perturbation} that the equation for $g^w$ takes the form of
\begin{align}\label{eq-gw}
\nn&\pr_tg^w-\dot{\lm}_1(t)|v|^2g^w+v\cdot\nb_xg^w-\nu \left(Lg^w
+ \tl{L}_1^{\lm_1} [g^w]\right)+E\cdot\nb_v\mu\, e^{\lm_1(t) |v|^2}\\
\nn=&-E\cdot\left(\nb_vg^w+\tl{L}_2^{\lm_1}[g^w] \right)+\nu(M_\theta\Dl_v\mu-M_1\cdot\nb_v\mu)e^{\lm_1(t) |v|^2}+\nu\rho\left({ L}g^w+\tl{L}_1^{\lm_1}[g^w]\right)\\
&+\nu M_\theta\left(\Dl_vg^w+\tl{L}_3^{\lm_1}[g^w]\right)-\nu M_1\cdot\left(\nb_vg^w+\tl{L}_2^{\lm_1}[g^w] \right),
\end{align}
where
\begin{align*}
\tl{L}_1^{\lm_1}[g^w]=&2n\lm_1(t) g^w-2\lm_1(t)(1-2\lm_1(t))|v|^2g^w-4\lm_1(t) \nb_v\cdot(vg^w),\\
\tl{L}_2^{\lm_1}[g^w]=&-2\lm_1(t)  vg^w,\\
\tl{L}_3^{\lm_1}[g^w]=&2n\lm_1(t) g^w+4(\lm_1(t))^2|v|^2g^w-4\lm_1(t) \nb_v\cdot(vg^w).
\end{align*}
Taking Fourier transform on \eqref{eq-gw} yields
\begin{align}\label{eq-gw-hat}
\nn&\pr_t\widehat{g^w}+\dot{\lm}_1(t)\Dl_\xi \widehat{g^w}-k\cdot\nb_\xi\widehat{g^w}+\nu \xi\cdot\nb_\xi\widehat{g^w}+\nu |\xi|^2\widehat{g^w}\\
\nn&+\nu\left[2n\lm_1 \widehat{g^w}-(2\lm_1-4\lm_1^2)\Dl_{\xi}\widehat{g^w}-4\lm_1 \nb_{\xi}\cdot(\xi \widehat{g^w})\right]+\hat{E}_k(t)\cdot\mathcal{F}\left[ \nb_v\mu\, e^{\lm_1 |v|^2}\right](\xi)\\
\nn=&-\mathcal{F}\left[E\cdot\nb_vg^w\right]_k(t,\xi)+\nu\mathcal{F}\left[(M_\theta\Dl_v\mu-M_1\cdot\nb_v\mu)e^{\lm_1 |v|^2}\right]_k(t,\xi)+\nu\mathcal{F}\big[\mathcal{C}[{g^w}]\big]_k(t,\xi)\\
\nn&\nn-\mathcal{F}\left[E\cdot \tl{L}_2^{\lm_1} [g^w]\right]_k(t,\xi)+\nu \mathcal{F}\left[\rho \tl{L}_2^{\lm_1}[g^w]\right]_k(t,\xi)\\
&+\nu \mathcal{F}\left[M_\theta \tl{L}_3^{\lm_1}[g^w]\right]_k(t,\xi)-\nu \mathcal{F}\left[M_1 \tl{L}_2^{\lm_1}[g^w]\right]_k(t,\xi).
\end{align}
Recalling the definition of $\widehat{f^w}$ in \eqref{def-fw}, and reorganizing the nonlinearities on the right hand side of \eqref{eq-gw-hat} in terms of moments:
\begin{align}
\nu\mathcal{C}[g^w]+\nu \rho \tl{L}_1^{\lm_1}[g^w]+\nu M_\theta \tl{L}_3^{\lm_1}[g^w]-\nu M_1 \tl{L}_2^{\lm_1}[g^w]=\mathcal{M}_0[g^w]+\mathcal{M}_1[g^w]+\mathcal{M}_2[g^w],
\end{align}
where
\begin{align}
\mathcal{M}_0[g^w]\nn=&\nu(\rho+M_\theta)\Dl_vg^w-\nu M_1\cdot\nb_vg^w+2n\nu\lm_1(\rho+M_\theta)g^w,\\
\mathcal{M}_1[g^w]\nn=&\nu\big(\rho-4\lm_1(\rho+M_\theta)\big)\nb_v\cdot(vg^w)+2\nu\lm_1 M_1\cdot vg^w,\\
\mathcal{M}_2[g^w]\nn=&-2\nu\lm_1\big((1-2\lm_1)\rho-2\lm_1M_\theta\big)|v|^2g^w,
\end{align}
one deduces  that $(\widehat{f^w})_k(t,\eta)$ solves
\begin{align}\label{eq-fw}
\nn&\pr_t(\widehat{f^w})_k(t,\eta)+e^{-2\nu t}\dot{\lm}_1(t)(\Dl_\eta \widehat{f^w})_k(t,\eta)+\nu|\bar{\eta}(t;k,\eta)|^2(\widehat{f^w})_k(t,\eta)\\
\nn&+\hat{E}_k(t)\cdot\mathcal{F}\left[\nb_v\mu e^{\lm_1 |v|^2} \right](\bar{\eta}(t; k,\eta))+\nu \left[2n\lm_1 (\widehat{f^w})_k(t,\eta)\right.\\
\nn&\left.-(2\lm_1-4\lm_1^2)e^{-2\nu t}(\Dl_\eta \widehat{f^w})_k\left(t, \eta \right)-4\lm_1 e^{-\nu t}\nb_\eta\cdot\left(\bar{\eta}(t;k,\eta) \widehat{f^w}\right)_k\left(t, \eta \right)\right]\\
\nn=&-\mathcal{F}\left[E\cdot\nb_vg^w\right]_k(t,\bar{\eta}(t;k,\eta))-\mathcal{F}\left[E\cdot \tl{L}_2^{\lm_1} [g^w]\right]_k(t,\bar{\eta}(t;k,\eta))\\
\nn&+\nu\mathcal{F}\left[(M_\theta\Dl_v\mu-M_1\cdot\nb_v\mu)e^{\lm_1 |v|^2}\right]_k(t,\bar{\eta}(t;k,\eta))\\
&+\mathcal{F}\big[\mathcal{M}_0[g^w]\big]_k(t,\bar{\eta}(t;k,\eta))+\mathcal{F}\big[\mathcal{M}_1[g^w]\big]_k(t,\bar{\eta}(t;k,\eta))+\mathcal{F}\big[\mathcal{M}_2[g^w]\big]_k(t,\bar{\eta}(t;k,\eta)),
\end{align}
where
\begin{align}
\mathcal{F}\left[E\cdot\nb_vg^w\right]_k(t,\bar{\eta}(t;k,\eta))
\nn=&\sum_{l\in\mathbb{Z}^n}\hat{E}_l(t)\cdot i\bar{\eta}(t;k,\eta)(\widehat{f^w})_{k-l}\left(t,\eta-lt^{\rm ap}\right),\\
\mathcal{F}\left[E\cdot \tl{L}_2^{\lm_1} [g^w]\right]_k(t,\bar{\eta}(t;k,\eta))
\nn=&-2\lm_1\sum_{l\in\mathbb{Z}^n}\hat{E}_l(t)\cdot i e^{-\nu t} (\nb_\eta\widehat{f^w})_{k-l}\left(t,\eta-lt^{\rm ap}\right),\\
\mathcal{F}\big[\mathcal{M}_0[g^w]\big]_k(t,\bar{\eta}(t;k,\eta))
\nn=&-\nu\sum_{l\in\mathbb{Z}^n}\left(\hat{\rho}_l(t)+(\widehat{M_\theta})_l(t)\right) |\bar{\eta}(t; k,\eta)|^2(\widehat{f^w})_{k-l}\left(t,\eta-lt^{\rm ap}\right)\\
\nn&-\nu\sum_{l\in\mathbb{Z}^n}(\widehat{M_1})_l(t)\cdot i\bar{\eta}(t; k,\eta)(\widehat{f^w})_{k-l}\left(t,\eta-lt^{\rm ap}\right)\\
\label{M0}&+2n\nu\lm_1\sum_{l\in\mathbb{Z}^n}\left(\hat{\rho}_l(t)+(\widehat{M_\theta})_l(t)\right) (\widehat{f^w})_{k-l}\left(t,\eta-lt^{\rm ap}\right),
\end{align}
and
\begin{align} \label{Merr}
\nn&\mathcal{F}\big[\mathcal{M}_1[g^w]\big]_k(t,\bar{\eta}(t;k,\eta))\\\nn=&\nu\sum_{l\in\mathbb{Z}^n}\left(\hat{\rho}_l(t)-4\lm_1(\hat{\rho}_l(t)+(\widehat{M_\theta})_l(t))\right)e^{-\nu t}\bar{\eta}(t;k,\eta)\cdot( \nb_\eta\widehat{f^w})_{k-l}\left(t,\eta-lt^{\rm ap}\right)\\
&+2\nu\lm_1\sum_{l\in\mathbb{Z}^n}(\widehat{M_1})_l(t)\cdot i e^{-\nu t}( \nb_\eta\widehat{f^w})_{k-l}\left(t,\eta-lt^{\rm ap}\right),
 \end{align}
 and
 \begin{align}\label{M2}
\nn&\mathcal{F}\big[\mathcal{M}_2[g^w]\big]_k(t,\bar{\eta}(t;k,\eta))\\
=&2\nu\lm_1\sum_{l\in\mathbb{Z}^n}\left((1-2\lm_1)\hat{\rho}_l(t)-2\lm_1(\widehat{M_\theta})_l(t)\right) e^{-2\nu t} (\Dl_\eta\widehat{f^w})_{k-l}\left(t,\eta-lt^{\rm ap}\right).
\end{align}

\subsection{Reformulation (II):  the density }
Taking Fourier transform in $(x, v)$ of the  equation 	\eqref{VPFP-perturbation},  we arrive at
\be\label{eq-g}
\pr_t\hat{g}_k(t, \xi)-k\cdot\nb_{\xi}\hat{g}_k(t,\xi)+\nu\xi\cdot\nb_\xi\hat{g}_k(t,\xi)+\nu|\xi|^2\hat{g}_k(t, \xi)+\hat{E}_k(t)\cdot i\xi\hat{\mu}(\xi)=\mathcal{N}_k(t,\xi),
\ee
where
\beno
\mathcal{N}_k(t,\xi)=-(\widehat{E\cdot\nb_vg})_k(t,\xi)+\nu(\widehat{\mathcal{C}_\mu})_k(t,\xi)+\nu(\widehat{\mathcal{C}_g})_k(t,\xi).
\eeno
Note that the equations of $\widehat{g^w}$ and $\widehat{g}$ have the same linear transport part. We then use the same change of coordinate and set
\be\label{def-f}
\hat{f}_k(t, \eta)=\hat{g}_k(t,  \bar{\eta}(t;k,\eta)),
\ee
or equivalently 
\be\label{g-f}
\hat{g}_k(t, \xi)=\hat{f}_k\left(t,  e^{-\nu t}\xi+kt^{\rm ap} \right).
\ee
Then from \eqref{eq-g}, we find that the equation of $\hat{f}_k(t,\eta)$ takes the form of
\be\label{eq-f}
\pr_t\hat{f}_k(t,\eta)+\nu|\bar{\eta}(t;k,\eta)|^2\hat{f}_k(t,\eta)+\mathcal{L}_k(t,\bar{\eta}(t;k,\eta))=\mathcal{N}_k(t,\bar{\eta}(t;k,\eta)),
\ee
where
\begin{align}\label{def-L}
\mathcal{L}_k(t,\bar{\eta}(t;k,\eta))=&\hat{E}_k(t)\cdot i\bar{\eta}(t;k,\eta)\hat{\mu}(\bar{\eta}(t;k,\eta)),\\
\label{def-N}
\mathcal{N}_k(t,\bar{\eta}(t;k,\eta))\nn=&-\sum_{l\in\mathbb{Z}^n_*}\hat{E}_l(t)\cdot i\bar{\eta}(t; k,\eta)\hat{f}_{k-l}\left(t, \eta-lt^{\rm ap}\right)\\
&+\nu (\widehat{\mathcal{C}_\mu})_k(t,\bar{\eta}(t;k,\eta))+\nu(\widehat{\mathcal{C}[g]})_k(t,k,\bar{\eta}(t;\eta)),
\end{align}
with
\begin{align}\label{Cmu-hat}
\nn(\widehat{\mathcal{C}_\mu})_k(t,\bar{\eta}(t;k,\eta))=&-(\widehat{M_\theta})_k(t)|\bar{\eta}(t; k,\eta)|^2\hat{\mu}(\bar{\eta}(t; k,\eta))\\
&-(\widehat{M_1})_k(t)\cdot i\bar{\eta}(t; k,\eta)\hat{\mu}(\bar{\eta}(t; k,\eta)),
\end{align}
and
\begin{align}
\nn(\widehat{\mathcal{C}[g]})_k(t,\bar{\eta}(t;k,\eta))=&-\sum_{l\in\mathbb{Z}^n_*}\hat{\rho}_l(t)|\bar{\eta}(t; k,\eta)|^2\hat{f}_{k-l}\left(t, \eta-lt^{\rm ap}\right)\\
\nn&-e^{-\nu t}\sum_{l\in\mathbb{Z}^n_*}\hat{\rho}_l(t)\bar{\eta}(t; k,\eta)\cdot(\nb_\eta\hat{f})_{k-l}\left(t, \eta-lt^{\rm ap}\right)\\
\nn&-\sum_{l\in\mathbb{Z}^n}(\widehat{M_\theta})_l(t)|\bar{\eta}(t; k,\eta)|^2\hat{f}_{k-l}\left(t, \eta-lt^{\rm ap}\right)\\
&-\sum_{l\in\mathbb{Z}^n_*}(\widehat{M_1})_l(t)\cdot i\bar{\eta}(t; k,\eta)\hat{f}_{k-l}\left(t, \eta-lt^{\rm ap}\right).
\end{align}

Now we are in a position to give the macroscopic quantities $\rho, M_1$ and $M_2$ in terms of $f$. In fact,  from \eqref{initial}, \eqref{pertur-1} and \eqref{pertur-2}, we have
\begin{gather}
\label{rhof}\hat{\rho}_k(t)=\hat{g}_k(t,  0)=\hat{f}_k\left(t, kt^{\rm ap}\right),\\
\label{M1f}(\widehat{M_1})_k(t)=(\widehat{vg})_k(t,  0)=ie^{-\nu t}(\nb_\eta\hat{f})_k\left(t, kt^{\rm ap}\right),\\
\label{M2f}(\widehat{M_2})_k(t)=(\widehat{|v|^2g})_k(t,  0)=-e^{-2\nu t}(\Dl_\eta\hat{f})_k\left(t, kt^{\rm ap}\right).
\end{gather}

Let
\be\label{S1}
S_k(t,\tau; \eta)=\exp\left(-\nu\int_\tau^t|\bar{\eta}(s;k,\eta)|^2ds\right)=\exp\left(-\nu\int_\tau^t\left|e^{\nu s}\eta-k\fr{e^{\nu s}-1}{\nu}\right|^2ds\right).
\ee
Take $\eta=kt^{\rm ap}$ in \eqref{S1}, and denote 
\begin{align}\label{S2}
S_k(t,\tau):=&S_k\left(t,\tau;  kt^{\rm ap}\right)
=\exp\left(-\fr{|k|^2}{\nu }\left[t-\tau-2(t-\tau)^{\rm ap}+\fr{1-e^{-2\nu(t-\tau)}}{2\nu}\right]\right).
\end{align}
In particular,
\be\label{S3}
S_k(t):=S_k\left(t,0; kt^{\rm ap}\right)=\exp\left\{-\fr{|k|^2}{\nu }\left(t-2t^{\rm ap}+\fr{1-e^{-2\nu t}}{2\nu}\right)\right\}.
\ee
It is easy to check that
$S_k(t,\tau)=S_k(t-\tau)$. 
By Duhamel's principle, the solution $\hat{f}_k(t,\eta)$ to \eqref{eq-f} can be expressed as
\begin{align}\label{ex-f}
\hat{f}_k(t,\eta)\nn=&S_k(t,0;\eta)(\widehat{f_{\mathrm{in}}})_k(\eta)-\int_0^tS_k(t,\tau; \eta)\mathcal{L}_k(\tau,\bar{\eta}(\tau;k,\eta))d\tau\\
&+\int_0^tS_k(t,\tau;\eta)\mathcal{N}_k(\tau,\bar{\eta}(\tau;k,\eta))d\tau.
\end{align}
Taking $\eta=kt^{\rm ap}$ in \eqref{ex-f}, and noting that 
\be\label{eta-bar}
\bar{\eta}(\tau;k,kt^{\rm ap})=k(t-\tau)^{\mathrm{ap}},
\ee
 we get 
\begin{align}\label{ex-rho}
\hat{\rho}_k(t)\nn=&S_k(t)(\widehat{f_{\mathrm{in}}})_k\left(kt^{\rm ap}\right)-\int_0^tS_k(t-\tau)\hat{E}_k(\tau)\cdot ik(t-\tau)^{\mathrm{ap}}\hat{\mu}\left(k(t-\tau)^{\mathrm{ap}}\right)d\tau\\
&+\int_0^tS_k(t-\tau)\mathcal{N}_k(\tau, k(t-\tau)^{\rm ap})d\tau.
\end{align}
We postpone the equations for $M_1$ and $M_2$ to Section \ref{high-m}.

For convenience, let us denote
$
\pr_v^t=e^{\nu t}\left(\nb_v-t^{\rm ap}\nb_x \right),
$
whose Fourier symbol is $i\bar{\eta}(t, k, \eta)$, namely, 
\begin{align}\label{pr-vt1}
 &\widehat{\pr_v^t}(k,\eta)=i\bar{\eta}(t, k,\eta)=\widehat{\pr_v^t}\left(k-l,\eta-lt^{\rm ap}\right),\\
\label{pr-vt2}
&i k(t-\tau)^{\mathrm{ap}}=ik(t-\tau)^{\rm ap}    =\widehat{\pr_v^\tau}\left(k-l, kt^{\rm ap}-l \tau^{\mathrm{ap}}\right).
\end{align}

\subsection{Relationship between unknowns $f$ and $f^w$}

Recalling that $g=e^{-\lm_1 |v|^2}  g^{w}$, from \eqref{def-f} and \eqref{def-fw}, one deduces the relation between $\hat{f}_k(t, \eta)$ and $(\widehat{f^{w}})_k(t,\eta)$:
\be\label{relation}
\hat{f}_k(t,\eta)=  \int_{\R^n}(\widehat{f^w})_k(t,\xi)\left(\fr{\sqrt{\pi}e^{\nu t}}{\sqrt{\lm_1}}\right)^{n} e^{-\left(\fr{e^{\nu t}}{2\sqrt{\lm_1}}|\eta-\xi|\right)^2}d\xi.
\ee
\begin{lem}[from $f^w$ to $f$]\label{lem-f-f}
Let $0<s<1, \lm>0$, and $\sigma\ge0$. Then for any multi-index $\al\in\N^n$, there holds
\be\label{f-f1}
\|v^\al f\|_{\mathcal{G}^{\lm, \sig;s}}\les  \|v^\al f^w\|_{\mathcal{G}^{\lm, \sig;s}},
\ee
and
\be\label{f-f2}
\left\|\pr_v^t (v^\al f)\right\|_{\mathcal{G}^{\lm,\sig;s}}\les \left\|\pr_v^t (v^\al f^w)\right\|_{\mathcal{G}^{\lm,\sig;s}}+\lm_1e^{-\nu t}\left\| vv^\al f^w\right\|_{\mathcal{G}^{\lm,\sig;s}}.
\ee
\end{lem}
\begin{proof}
From \eqref{relation}, we write 
\[
\pr^\al_\eta \hat{f}_k(t,\eta)=  \int_{\R^n}(\pr^\al_\eta\widehat{f^w})_k(t,\xi)\left(\fr{\sqrt{\pi}e^{\nu t}}{\sqrt{\lm_1}}\right)^{n} e^{-\left(\fr{e^{\nu t}}{2\sqrt{\lm_1}}|\eta-\xi|\right)^2}d\xi.
\]
To bound $e^{\lm\la k,\eta\ra^s}\la k,\eta\ra^s$, we use the following partition of unity
\[
1={\bf 1}_{|\eta-\xi|\le\fr12|k,\xi|}+{\bf 1}_{|k,\xi|\le 2|\eta-\xi|}.
\]
Then in view of \eqref{app3} and \eqref{app5}, we have
\[
e^{\lm\la k,\eta\ra^s}\la k,\eta\ra^\sig\les e^{\lm\la k,\xi\ra^s}\la k,\xi\ra^\sig e^{c\lm\la \eta-\xi\ra^s}+e^{\lm \la \eta-\xi\ra^s}\la \eta-\xi\ra^\sig e^{c\lm\la k,\xi\ra^s}.
\]
Consequently, noting that $\fr{e^{\nu t}}{\sqrt{\lm_1}}\gt \fr{1}{\sqrt{\nu}}$
\begin{align}\label{e-f-f1}
\nn&e^{\lm\la k,\eta\ra^s}\la k,\eta\ra^\sig|\pr^\al_\eta \hat{f}_k(t,\eta)|\\
\nn\les& \int_{\R^n} e^{\lm\la k,\xi\ra^s}\la k,\xi\ra^\sig\left|(\pr^\al_\eta\widehat{f^w})_k(t,\xi)\right| e^{c\lm\la \eta-\xi\ra^s}\left(\fr{\sqrt{\pi}e^{\nu t}}{\sqrt{\lm_1}}\right)^{n} e^{-\left(\fr{e^{\nu t}}{2\sqrt{\lm_1}}|\eta-\xi|\right)^2}d\xi\\
\nn&+\int_{\R^n} e^{c\lm\la k,\xi\ra^s}\left|(\pr^\al_\eta\widehat{f^w})_k(t,\xi)\right| e^{\lm\la \eta-\xi\ra^s}\la \eta-\xi\ra^\sig\left(\fr{\sqrt{\pi}e^{\nu t}}{\sqrt{\lm_1}}\right)^{n} e^{-\left(\fr{e^{\nu t}}{2\sqrt{\lm_1}}|\eta-\xi|\right)^2}d\xi\\
\les&\int_{\R^n} e^{\lm\la k,\xi\ra^s}\la k,\xi\ra^\sig\left|(\pr^\al_\eta\widehat{f^w})_k(t,\xi)\right| \left(\fr{\sqrt{\pi}e^{\nu t}}{\sqrt{\lm_1}}\right)^{n} e^{-\left(\fr{e^{\nu t}}{4\sqrt{\lm_1}}|\eta-\xi|\right)^2}d\xi.
\end{align}
Then \eqref{f-f1} follows immediately. To prove \eqref{f-f2}, we divide $\bar{\eta}(t;k,\eta)\pr^\al_\eta \hat{f}_k(t,\eta)$ into two parts:
\beno
\bar{\eta}(t;k,\eta)\pr^\al_\eta \hat{f}_k(t,\eta)={\rm I}+{\rm II},
\eeno
where
\begin{align}\label{e-f-f2}
{\rm I}=&\int_{\R^n}\bar{\xi}(t;k,\xi) (\pr^\al_\eta\widehat{f^w})_k(t,\xi) \left(\fr{\sqrt{\pi}e^{\nu t}}{\sqrt{\lm_1}}\right)^{n} e^{-\left(\fr{e^{\nu t}}{2\sqrt{\lm_1}}|\eta-\xi|\right)^2}d\xi,\\
\label{e-f-f3}
\nn{\rm II}=&\int_{\R^n} (\pr^\al_\eta\widehat{f^w})_k(t,\xi) e^{\nu t}(\eta-\xi)\left(\fr{\sqrt{\pi}e^{\nu t}}{\sqrt{\lm_1}}\right)^{n} e^{-\left(\fr{e^{\nu t}}{2\sqrt{\lm_1}}|\eta-\xi|\right)^2}d\xi\\
\nn=&2\lm_1e^{-\nu t}\int_{\R^n} (\pr^\al_\eta\widehat{f^w})_k(t,\xi) \left(\fr{\sqrt{\pi}e^{\nu t}}{\sqrt{\lm_1}}\right)^{n} \nb_{\xi} e^{-\left(\fr{e^{\nu t}}{2\sqrt{\lm_1}}|\eta-\xi|\right)^2}d\xi\\
=&-2\lm_1e^{-\nu t}\int_{\R^n} (\nb_\eta\pr^\al_\eta\widehat{f^w})_k(t,\xi) \left(\fr{\sqrt{\pi}e^{\nu t}}{\sqrt{\lm_1}}\right)^{n} e^{-\left(\fr{e^{\nu t}}{2\sqrt{\lm_1}}|\eta-\xi|\right)^2}d\xi
\end{align}
Then we bound \eqref{e-f-f2} and \eqref{e-f-f3} in $\mathcal{G}^{\lm,\sig;s}$ in the same manner as \eqref{e-f-f1}, we get \eqref{f-f2}. This completes the proof of Lemma \ref{lem-f-f}.
\end{proof}

\section{Proof of the main theorems}\label{sec:Proof of the main theorems}
In this section, we prove Theorem \ref{Thm:main1} and Theorem \ref{Thm:main2}, starting the primary steps as propositions which are proved in the subsequent sections.
We recall the multiplier $\mathbf{A}_k^{\sigma,\mathfrak{e}}(t,\eta)$ we used in this paper and the associated three-level energy functionals $\left\| f^w(t)\right\|_{\mathcal{E}_{m}^{\sigma_0+1,\frac{2}{5}}}$, $\left\| f^w(t)\right\|_{\mathcal{E}_{m+2}^{ \sig_1 ,\frac{2}{5}}}$, and $\left\| f^w(t)\right\|_{\mathcal{E}_{m}^{ \sig_1 ,\frac{1}{3}}}$, where for $\sigma=\{\sigma_0+1, \sig_1 \}$ and $\mathfrak{e}=\{\frac{2}{5},\frac{1}{3}\}$,
\begin{align*}
    \mathbf{A}_k^{\sigma,\mathfrak{e}}(t,\eta)=&e^{{\bf 1}_{k\ne0}\dl_1\nu^\mathfrak{e}t}e^{\lm(t)\la k,\eta\ra^s}\la k,\eta\ra^{\sigma}\frak{m}_k(t, \eta),
\end{align*}
and correspondingly
\begin{align}
\left\| f(t)\right\|_{\mathcal{E}_{m}^{\sigma,\mathfrak{e}}}&=\left(\sum_{\al\in\N^n:|\al|\le m} \fr{e^{-2|\al|\nu t}}{K_{|\al|}}\left\| \mathbf{A}^{\sigma,\mathfrak{e}}(v^\al f(t))\right\|_{L^2}^2\right)^\fr12.
\end{align}
By the definition for the Fourier multiplier $\mathbf{A}^{\sigma,\mathfrak{e}}(t,\nabla)$, when taking the time-derivative of the above functionals, we have the corresponding good terms:
\begin{align}
\|f^w(t)\|_{\mathcal{D}_m^{\sigma,\mathfrak{e}}}^2
=&\sum_{\alpha\in \mathbb{N}^n: |\alpha|\leq m}\frac{1}{K_{|\alpha|}}\bigg[-e^{-2|\al|\nu t}\dot{\lm}(t)\left\|\la \nb\ra^{\fr{s}{2}}\mathbf{A}^{\sigma,\mathfrak{e}}(v^\al f^w) \right\|_{L^2}^2\\
\nn&+e^{-2|\al|\nu t}\left\|\sqrt{-\fr{\pr_t\frak{m}}{\frak{m}}} \mathbf{A}^{\sigma,\mathfrak{e}}(v^\al f^w)\right\|_{L^2}^2+\nu e^{-2|\al|\nu t}\left\|\pr_v^t\mathbf{A}^{\sigma,\mathfrak{e}}(v^\al f^w) \right\|_{L^2}^2\\
\nn&+\nu|\al| e^{-2|\al|\nu t}\| \mathbf{A}^{\sigma,\mathfrak{e}}(v^\al f^w)\|_{L^2}^2+2n\nu\lm_1(t)e^{-2|\al|\nu t}\left\| \mathbf{A}^{\sigma,\mathfrak{e}}(v^\al f^w)\right\|_{L^2}^2\\
\nn&+\left(2\nu \lm_1(1-2\lm_1)-\dot{\lm}_1(t)\right)e^{-2(|\al|+1)\nu t}\left\|\mathbf{A}^{\sigma,\mathfrak{e}}(vv^\al f^w)\right\|_{L^2}^2\bigg],
\end{align}
where the first two terms arise from the derivative acting on the time-dependent Fourier multiplier, the remaining terms arise from the good structures of $f^w$ equation, facilitated by the introduction of the time-dependent Gaussian weight. 
Since the Fourier multiplier $\mathfrak{m}(t,\nabla)$ is bounded upper and below, see \eqref{m1}, to simplify the presentation, sometimes we use the following notation
\begin{align*}
\underline{{A}}_k^{\sigma,\mathfrak{e}}(t,\eta)=&e^{{\bf 1}_{k\ne0}\dl_1\nu^\mathfrak{e}t}e^{\lm(t)\la k,\eta\ra^s}\la k,\eta\ra^{\sigma}.
\end{align*}

\begin{rem}\label{rem-transform}
    We remark that the Fourier multipliers commutes smoothly with polynomial weights in $v$. Indeed, we use the following estimates in our proof:

For all multi-index $\al\in\N^n$,
\be
\left| \pr^\al_\eta\underline{A}_k^{\sigma,\mathfrak{e}}(t,\eta)\right|\les \fr{1}{\la k,\eta\ra^{|\al|(1-s)}}\underline{A}_k^{\sigma,\mathfrak{e}}(t,\eta).
\ee
Accordingly, together with Lemma \ref{lem-f-f}, we have
\begin{align}\label{emb}
e^{-m\nu t}\left\| (\underline{A}^{\sigma,\mathfrak{e}}\hat{f})_k(t,\eta)\right\|_{L^2_kH^m_\eta}\nn\les& e^{-m\nu t}\sum_{|\al|\le m}\sum_{\beta\le\al}\begin{pmatrix} \al\\ \beta \end{pmatrix}\left\|\pr_\eta^{\al-\beta} \underline{A}^{\sigma,\mathfrak{e}} \pr_\eta^\beta\hat{f} \right\|_{L^2_kL^2_\eta}\\\les_m&\sum_{|\al|\le m}e^{-|\al|\nu t}\left\|\underline{A}^{\sigma,\mathfrak{e}}\pr_\eta^\al \hat{f} \right\|_{L^2_\eta}\les \left\| f^w(t)\right\|_{\mathcal{E}_m^{\sigma,\mathfrak{e}}}, 
\end{align}
and
\begin{align}\label{emb'}
\nn&\nu\sum_{k\in\mathbb{Z}^n_*}\int_0^{T^*} \fr{1}{\la \tau\ra^{1+}} \left(e^{-m\nu\tau}\left\|\mathcal{F}\left[ \pr_v^\tau \underline{A}^{\sigma,\mathfrak{e}}f_{\ne}\right]_{k}(\tau)\right\|_{H^{\fr{n-1}{2}+}_\eta}\right)^2 d\tau\\
\nn\les&\nu\int_0^{T^*}\fr{1}{\la\tau \ra^{1+}}\left(e^{-m\nu\tau}\left\|\bar{\eta}(\tau; k,\eta)\underline{A}^{\sigma,\mathfrak{e}}\hat{f}_{\ne}\right\|_{L^2_kH^m_\eta}\right)^2d\tau\\
\nn\les&\nu\int_0^{T^*}\fr{1}{\la\tau \ra^{1+}}\left(\left\| f^w(\tau)\right\|_{\mathcal{E}_{m-1}^{\sigma,\mathfrak{e}}}^2+\left\| f^w(\tau)\right\|_{\mathcal{D}_{m}^{\sigma,\mathfrak{e}}}^2\right)d\tau\\
\les&\sup_t\left\| f^w(t)\right\|_{\mathcal{E}_{m-1}^{\sigma,\mathfrak{e}}}^2+\nu\int_0^{T^*}\left\| f^w(t)\right\|_{\mathcal{D}_{m}^{\sigma,\mathfrak{e}}}^2dt.
\end{align}
\end{rem}

\subsection{Bootstrap}
A standard argument gives the local wellposedness for the VPFP system in Gevrey-$\frac{1}{s}$ spaces. In this way, we may safely ignore the time interval $[0,1]$ by further restricting the size of the initial data. 

The goal is next to prove by a continuity argument that the three-level energy functionals
\begin{align}
\left\| f^w(t)\right\|_{\mathcal{E}_{m}^{\sigma_0+1,\frac{2}{5}}}^2,\quad \left\| f^w(t)\right\|_{\mathcal{E}_{m+2}^{ \sig_1 ,\frac{2}{5}}}^2, \quad \left\| f^w(t)\right\|_{\mathcal{E}_{m}^{ \sig_1 ,\frac{1}{3}}}^2
\end{align} 
(together with some related quantities) are uniformly bounded for all time if $\epsilon$ is sufficiently small. We define the following controls referred to in the sequel as the bootstrap hypotheses for $t\geq 1$:
\begin{subequations}\label{equ:bootstrap}
\begin{align}
\label{H-f1}\sup_{t\in [1,T]}\left\| f^w(t)\right\|_{\mathcal{E}_{m}^{\sigma_0+1,\frac{2}{5}}}^2+\int_1^T\left\| f^w(t)\right\|_{\mathcal{D}_{m}^{\sigma_0+1,\frac{2}{5}}}^2 dt \leq (4{\rm C}_f\epsilon\nu^{\gamma})^2,\\
\label{H-f2}\sup_{t\in [1,T]}\left\| f^w(t)\right\|_{\mathcal{E}_{m}^{ \sig_1 ,\frac{1}{3}}}^2+\int_1^T\left\| f^w(t)\right\|_{\mathcal{D}_{m}^{ \sig_1 ,\frac{1}{3}}}^2 dt \leq (4{\rm C}_f\epsilon\nu^{\gamma})^2,\\
\label{H-f3}\sup_{t\in [1,T]}\left\| f^w(t)\right\|_{\mathcal{E}_{m+2}^{ \sig_1 ,\frac{2}{5}}}^2+\int_1^T\left\| f^w(t)\right\|_{\mathcal{D}_{m+2}^{ \sig_1 ,\frac{2}{5}}}^2 dt \leq (4{\rm C}_f\epsilon\nu^{\gamma})^2,\\
\label{H-f0}
\sum_{|\al|\le3}e^{-|\al|\nu t}\big\| e^{\lm(t)\la \eta\ra^s}\la e^{\nu t}\eta\ra^{\sig_0-6}(\pr^\al_\eta\hat{f})_0(t,\eta)\big\|_{L^\infty_\eta}\le 4{\rm C}_fe^{-\nu t}\eps\nu^\gamma.
\end{align}
and
\begin{align}
\label{H-M0-H}\left\|\la t\ra^b  {\bf B}  \rho\right\|_{L^2_{t,x}}^2\leq (4{\rm C}_0\eps\nu^\gamma)^2,\\
\label{H-M1-H}\left\|\la t\ra^b  {\bf B}  M_1\right\|_{L^2_{t,x}}^2\leq (4{\rm C}_1\eps\nu^\gamma)^2,\\
\label{H-M2-H}\left\|\la t\ra^b  {\bf B}  M_2\right\|_{L^2_{t,x}}^2\leq (4{\rm C}_2\eps\nu^\gamma)^2,\\
\label{H-M0-L}\left\|   { B}^{\sig_1}  \rho\right\|_{L^\infty_{t}L^2_x}^2\leq (4{\rm \tl C}_0\eps\nu^\gamma)^2,\\
\label{H-M1-L}\left\|   { B}^{\sig_1}  M_1\right\|_{L^\infty_{t}L^2_x}^2\leq (4{\rm \tl C}_1\eps\nu^\gamma)^2,\\
\label{H-M2-L}\left\|  { B}^{\sig_1}  M_2\right\|_{L^\infty_{t}L^2_x}^2\leq (4{\rm \tl C}_2\eps\nu^\gamma)^2.
\end{align}
Here ${\rm C}_f, {\rm C}_\ell, {\rm \tl C}_\ell, \ell =0, 1, 2 $ are positive constants independent of $t, \nu$ fixed by the proof, satisfying ${\rm C}_\ell<\fr12 {\rm \tl C}_\ell$. 
\end{subequations}
\begin{prop}\label{prop: btsp}
    Let $\sigma_0, m$ be the same as in Theorem \ref{Thm:main1}. Assume that the bootstrap hypotheses \eqref{equ:bootstrap} hold on $[1, T^*]$. Then there exist $\epsilon_0>0$ and  $\nu_0>0$, such that for any $0\leq \nu<\nu_0$ and $0<\epsilon<\epsilon_0$, the same estimates in \eqref{equ:bootstrap} hold with the occurrences of 4 on the right-hand side replaced by 2. 
\end{prop}
Theorem \ref{Thm:main1} follows directly from Proposition \ref{prop: btsp}. 
The remainder of this paper is devoted to the proof of Proposition \ref{prop: btsp}. Here we highlight the main idea of the proof. For the distribution function, we prove the Gr\"onwall-type inequalities:
\begin{align*}
    \frac{1}{2}\frac{d}{dt}\|f^w(t)\|_{\mathcal{E}^{\sigma_0+1,\frac{2}{5}}_m}^2
    &+\|f^w(t)\|_{\mathcal{D}^{\sigma_0+1,\frac{2}{5}}_m}^2\\
    &\quad \leq \left(c_0+C\|f^w(t)\|_{\mathcal{E}^{ \sig_1 ,\frac{2}{5}}_{m+2}}\right)\|f^w(t)\|_{\mathcal{D}^{\sigma_0+1,\frac{2}{5}}_m}^2+\frac{c_0\epsilon^2\nu^{2\gamma}}{\langle t\rangle^2},\\
    \frac{1}{2}\frac{d}{dt}\|f^w(t)\|_{\mathcal{E}^{ \sig_1 ,\frac{1}{3}}_m}^2
    &+\|f^w(t)\|_{\mathcal{D}^{ \sig_1 ,\frac{1}{3}}_m}^2
    \leq c_0\|f^w(t)\|_{\mathcal{D}^{ \sig_1 ,\frac{1}{3}}_m}^2+\frac{c_0\epsilon^2\nu^{2\gamma}}{\langle t\rangle^2},\\
    \frac{1}{2}\frac{d}{dt}\|f^w(t)\|_{\mathcal{E}^{ \sig_1 ,\frac{2}{5}}_{m+2}}^2
    &+\|f^w(t)\|_{\mathcal{D}^{ \sig_1 ,\frac{2}{5}}_{m+2}}^2
    \leq c_0\|f^w(t)\|_{\mathcal{D}^{ \sig_1 ,\frac{2}{5}}_{m+2}}^2+\frac{c_0\epsilon^2\nu^{2\gamma}}{\langle t\rangle^2},
\end{align*}
where $c_0>0$ is small enough under the bootstrap hypotheses, by choosing $\epsilon_0$, $\nu_0$, $\frac{K_{|\alpha|}}{K_{|\alpha|+1}}$, and $a_0$ small enough. In the energy estimate of $\|f^w(t)\|_{\mathcal{E}^{\sigma_0+1,\frac{2}{5}}_m}^2$, we put $\left(c_0+C\|f^w(t)\|_{\mathcal{E}^{ \sig_1 ,\frac{2}{5}}_{m+2}}\right)$ instead of $c_0$ to emphasize that a higher moment is required, see \eqref{loss2m} for more details. We refer to section \ref{Estimates on the distribution function} for the improved estimates of the distribution function. 

For the zero mode of the distribution function, we apply  Duhamel’s principle by using the explicit expression of the solution operator $S_0(t, 0, \eta)$ corresponding to the Fokker-Planck operator. 

For the density and higher moments, we use the Penrose's identity and take advantage of the time-dependent Fourier multiplier. The improved estimates of density and higher moments is in section \ref{sec:Density estimates}.

\subsection{Proof of Theorem \ref{Thm:main2}}
The proof of Theorem \ref{Thm:main2} follows a similar energy estimate as proving Landau damping for the quasilinear Vlasov-Poisson equation with a controllable force. Let $g_{\delta}=g-g^0$ and $E_{\delta}=E-E^0$, where $(g, E)$ and $(g^0, E^0)$ solve \eqref{VPFP-per} with $\nu>0$ and $\nu=0$ respectively. Then $(g_{\delta}, E_{\delta})$ satisfies
\begin{equation}
    \label{eq:g_delta}
    \left\{
    \begin{aligned}
        &\partial_t g_{\delta}+v\cdot \nabla_xg_{\delta}
        +E_{\delta}\cdot \nabla_v\mu=-E_{\delta}\cdot \nabla_v g^0-E^0\cdot \nabla_vg_{\delta}+\nu (L[g]+\mathcal{C}[g]+\mathcal{C}_{\mu})\\
        &E_{\delta}=-\nabla_x(-\Delta_x)^{-1}\rho_{\delta},\quad \rho_{\delta}=\int_{\mathbb{R}^n}g_{\delta}(t, x, v)dv.
    \end{aligned}
    \right.
\end{equation}
Let us take the Fourier transform in $(x,v)$ and define $\hat{f}_{\delta}(t, k, \eta)=\hat{g}_{\delta}(t, k, \eta-kt)$. Let $\hat{h}(t, k, \eta)=\hat{g}(t, k, \eta-kt)$, then we get that $\widehat{\rho}_{\delta}(t, k)=\hat{f}_{\delta}(t, k, kt)$ and $\hat{f}(t, k, \eta)=\hat{g}(t, k, e^{\nu t}(\eta-kt^{\mathrm{ap}}))=\hat{h}(t, k, e^{\nu t}\eta-k(t^{\mathrm{ap}}-t))$. The following lemma follows directly. 
\begin{lem}
    Let $\widehat{f^0}(t, k, \eta)=\widehat{g^0}(t, k, \eta-kt)$. There is $\nu_0$ such that for any $0<\nu<\nu_0$, and $0\leq t\leq \nu^{\frac{1}{3}}$, it holds for $s>\frac{1}{3}$ that
    \begin{align*}
        \|e^{\frac{\lambda_{\infty}}{2}\langle k, \eta\rangle^{s}}\langle k,\eta\rangle^{\sigma_0+1}(\langle \nabla_{\eta}\rangle^{m}\hat{h}(t, k, \eta))\|_{l^2_kL^2_{\eta}}\leq C\epsilon,\\
        \|e^{\frac{\lambda_{\infty}}{2}\langle k, \eta\rangle^{s}}\langle k,\eta\rangle^{\sigma_0+1}(\langle \nabla_{\eta}\rangle^{m}\widehat{g^0}(t, k, \eta))\|_{l^2_kL^2_{\eta}}\leq C\epsilon,
    \end{align*}
    which gives that 
    \begin{align*}
        &\left\|e^{\frac{\lambda_{\infty}}{2}\langle k, \eta\rangle^{s}}\langle k,\eta\rangle^{\sigma_0+1}\big[\langle \nabla_{\eta}\rangle^{m}\big(\widehat{L[g]}+\widehat{\mathcal{C}[g]}+\widehat{\mathcal{C}_{\mu}}\big)(t, k, \eta-kt)\big]\right\|_{l^2_kL^2_{\eta}}\leq C\epsilon t^2,\\
        &\left\|e^{\frac{\lambda_{\infty}}{2}\langle k, kt \rangle^{s}}\langle k,kt\rangle^{\sigma_0+1}\widehat{\rho^0}(t,k)\right\|_{l^2_k}\leq Ce^{-c t^{s}}. 
    \end{align*}
    
\end{lem}
A direct calculation gives that $\partial_t \hat{f}_{\delta}+\mathcal{L}^{\delta}_k(\eta)
+\mathcal{R}^{\delta}_k(\eta)+\mathcal{T}^{\delta}_k(\eta)+\mathcal{F}_k(\eta)=0$,
where 
\begin{align*}
    &\mathcal{L}^{\delta}_k(t, \eta)=i\widehat{E}_{\delta}(k)(\eta-kt)\hat{\mu}(\eta-kt),\\
    &\mathcal{R}^{\delta}_k(t, \eta)=i \sum_{l} \widehat{E}_{\delta}(l)(\eta-(k-l)t)\widehat{h^0}(t, k-l, \eta),\\
    &\mathcal{T}^{\delta}_k(t, \eta)=i \sum_{l} \widehat{E^0}(k-l)(\eta-lt)\widehat{f}_{\delta}(t, l, \eta),\\
    &\mathcal{F}_k(t, \eta)=\big(\widehat{L[g]}+\widehat{\mathcal{C}[g]}+\widehat{\mathcal{C}_{\mu}}\big)(t, k, \eta-kt).
\end{align*}
We also get that
\begin{align*}
    \widehat{\rho}_{\delta}(t, k)=-\int_0^t\big(\mathcal{L}^{\delta}_k(\tau, kt)
    +\mathcal{R}^{\delta}_k(\tau, kt)+\mathcal{T}^{\delta}_k(\tau, kt)+\mathcal{F}_k(\tau, kt)\big)d\tau. 
\end{align*}
We then define the following energy functionals:
\begin{align*}
    &\|f_{\delta}\|_{E}=\|e^{\lambda_3(t)\langle\nabla_{x,v}\rangle^s}\langle\nabla_{x,v}\rangle^{\sigma_0+1}(\langle v\rangle^mf_{\delta})\|_{L^2},\\
    &\|f_{\delta}\|_{E_{low}}=\|e^{\lambda_3(t)\langle\nabla_{x,v}\rangle^s}\langle\nabla_{x,v}\rangle^{ \sig_1 }(\langle v\rangle^mf_{\delta})\|_{L^2},
\end{align*}
with $\lambda_3(t)={\lambda_{\infty}}/{3}+\frac{1}{100(1+t)^a}\lambda_{\infty}$. 
By following the same argument, we obtain that there is $\epsilon_1>0$ independent of $\nu, t$, such that for any $t\leq \epsilon_1\nu^{-\frac{1}{3}}$, it holds that
\begin{align*}
&\|f_{\delta}\|_{E}+\|f_{\delta}\|_{E_{low}}\leq C\nu t^3,\\
&\langle t\rangle^b\|e^{\lambda_3(t)\langle\nabla_{k,kt}\rangle^s}\langle k, kt\rangle^{\sigma_0}|k|^{\frac{1}{2}}\widehat{\rho_{\delta}}\|_{L_t^2l_k^2}\leq C\nu t^3. 
\end{align*}
We omit the proof and point out that the estimate of $\mathcal{R}^{\delta}_k(t, \eta)$ is similar to the treatments of ${\bf N}_{E;\neq}^{\mathrm{HL}}, {\bf N}_{E;0}^{\mathrm{HL}}$, and $\mathrm{N}^{\mathrm{HL}}_{M}$, the estimate of $\mathcal{T}^{\delta}_k(t, \eta)$ is similar to the treatments of ${\bf N}_{E;\neq}^{\mathrm{LH}}, {\bf N}_{E;0}^{\mathrm{LH}}$, and $\mathrm{N}^{\mathrm{LH}}_{M}$, and the treatment of the linear part $\mathcal{L}^{\delta}_k(t, \eta)$ is the same as $\mathrm{L}$.

\subsection{Some estimates on $M_\theta$}
The following estimates for $M_\theta$ derived from the bootstrap assumptions \eqref{equ:bootstrap} will be frequently used in the remainder of this paper.
\begin{prop} \label{prop-Mtheta}
Let $m>\fr{n}{2}+4$. Under the bootstrap hypotheses, it holds that
\begin{align}\label{AM}
&e^{-m\nu t}\left\|{\bf 1}_{k\ne0}\underline{A}_k^{\sig_0+1,\fr25}\left(t, kt^{\rm ap}\right)(\widehat{M_\theta})_k(t)\right\|_{L^2_k}\les \|f_{\ne}^w(t)\|_{\mathcal{E}_m^{\sig_0+1,\fr25}},\\
\label{BM}
&\|  {\bf B}  M_\theta\|_{L^2_x}\les \left\| {\bf B}  M_2\right\|_{L^2_x}+\left\| {\bf B}  M_1\right\|_{L^2_x}+\left\| {\bf B}  \rho\right\|_{L^2_x},\\
\label{BM'}
&\| B^{ \sig_1 }(M_\theta)_{\ne}\|_{L^2_x}\les \left\|B^{ \sig_1 }(M_2)_{\ne}\right\|_{L^2_x}+\left\|B^{ \sig_1 }M_1\right\|_{L^2_x}+\left\|B^{ \sig_1 }\rho\right\|_{L^2_x},\\
\label{e-Mtheta0}&\left|\left(M_\theta\right)_0(t)\right|\les \eps^2\nu^{2\gamma} e^{-2\dl\nu^\fr13t}\fr{1}{\la t^{\rm ap}\ra^{2\sig_1 }}.
\end{align}
\end{prop}
\begin{proof}
From \eqref{2thM'}, we get
\begin{align}\label{exp-Mtheta}
M_\theta=\fr{1}{n}M_2-\fr{1}{n}\sum_{j=0}^\infty(-1)^j\rho^j|M_1|^2-\rho.
\end{align}
Recalling \eqref{rhof}--\eqref{M2f}, one deduces that
\begin{align}
\nn(\widehat{M_\theta})_k(t)=&-\fr{1}{n}e^{-2\nu t}\Dl_\eta\hat{f}_k\left(t, kt^{\rm ap}\right)-\hat{f}_k\left(t, kt^{\rm ap}\right)-\fr{1}{n}\sum_{j=0}^\infty(-1)^j \mathcal{F}\left[\rho^j|M_1|^2\right]_k(t).
\end{align}
Then
\begin{align}\label{AMtheta}
\underline{A}_k^{\sig_0+1,\fr25}\left(t,kt^{\rm ap}\right)(\widehat{M_\theta})_k(t)\nn=&-\fr{1}{n}e^{-2\nu t} (\underline{A}^{\sig_0+1,\fr25}\Dl_\eta\hat{f})_k\left(t, kt^{\rm ap}\right)-(\underline{A}^{\sig_0+1,\fr25}\hat{f})_k\left(t, kt^{\rm ap}\right)\\
&-\fr{1}{n}\sum_{j=0}^\infty(-1)^j \underline{A}_k^{\sig_0+1,\fr25}\left(t,kt^{\rm ap}\right)\mathcal{F}\left[\rho^j|M_1|^2\right]_k(t).
\end{align}
Note that 
\[
\underline{A}_k^{\sig_0+1,\fr25}\left(t, kt^{\rm ap}\right)=e^{\dl_1\nu^\fr25t}\underline{B}_k^{\sig_0+1}(t), \quad \mathrm{for} \quad k\ne0,
\]
where $\underline{B}_k^{\sig_0+1}(t)$ is introduced in Lemma \ref{lem-alg} in Section \ref{sec-app-ineq}.
Then in view of \eqref{alg1}, and using \eqref{rhof},  \eqref{M1f} and \eqref{emb}, we find that
\begin{align}\label{Aproduct}
\nn&\sum_{j=0}^\infty\left\| \underline{A}_k^{\sig_0+1,\fr25}\left(t, kt^{\rm ap}\right)\mathcal{F}\left[\rho^j|M_1|^2\right]_k(t)\right\|_{L^2_k}\\
\nn\les&e^{\dl_1\nu^\fr25t}\left\|\underline{B}_k^{\sig_0+1}(t)\mathcal{F}\left[M_1\right]_k(t)\right\|_{L^2_k}^2\left(1+\sum_{j=1}^\infty\left\|\underline{B}_k^{\sig_0+1}(t)\mathcal{F}\left[\rho^j\right]_k(t)\right\|_{L^2_k}\right)\\
\nn\les&e^{-\dl_1\nu^{\fr25} t}\left\|\underline{A}_k^{\sig_0+1,\fr25}\left(t,kt^{\rm ap}\right)\mathcal{F}\left[M_1\right]_k(t)\right\|_{L^2_k}^2\left(1+\sum_{j=1}^\infty C_0^{j-1}\left\|\underline{B}_k^{\sig_0+1}(t)\hat{\rho}_k(t)\right\|_{L^2_k}^j\right)\\
\nn\les&e^{-\dl_1\nu^{\fr25} t}e^{-2\nu t}\left\|{\bf 1}_{k\ne0}(\underline{A}^{\sig_0+1,\fr25}\nb_\eta \hat{f})_k\left(t,kt^{\rm ap}\right)\right\|_{L^2_k}^2\\
\nn&\times\left(1+\sum_{j=1}^\infty \left(C_0e^{-\dl_1\nu^\fr25t}\left\|{\bf 1}_{k\ne0}(\underline{A}_k^{\sig_0+1,\fr25}\hat{f})_k\left(t, kt^{\rm ap}\right)\right\|_{L^2_k}\right)^j\right)\\
\nn\les&e^{-\dl_1\nu^{\fr25} t}e^{-2\nu t}\left\|{\bf 1}_{k\ne0}(\underline{A}^{\sig_0+1,\fr25}\nb_\eta \hat{f})_k(t,\eta)\right\|_{L^2_kL^\infty_\eta}^2\\
\nn&\times\left(1+\sum_{j=1}^\infty \left(C_0e^{-\dl_1\nu^\fr25t}\left\|{\bf 1}_{k\ne0}(\underline{A}_k^{\sig_0+1,\fr25}\hat{f})_k\left(t, \eta\right)\right\|_{L^2_kL^\infty_\eta}\right)^j\right)\\
\les&\|f_{\ne}^w(t)\|_{\mathcal{E}_m^{\sig_0+1,\fr25}}^2\left(1+\sum_{j=1}^\infty(C_0\|f_{\ne}^w(t)\|_{\mathcal{E}_m^{\sig_0+1,\fr25}})^j\right)\les \|f_{\ne}^w(t)\|_{\mathcal{E}_m^{\sig_0+1,\fr25}}^2,
\end{align}
where $C_0$ is the constant appearing \eqref{alg1}.
It follows from \eqref{AMtheta} and \eqref{Aproduct}  that
\begin{align}
\nn&e^{-m\nu t}\left\|{\bf 1}_{k\ne0}\underline{A}_k^{\sig_0+1,\fr25}\left(t, kt^{\rm ap}\right)(\widehat{M_\theta})_k(t)\right\|_{L^2_k}\\
\nn\les& e^{-(m+2)\nu t}\left\|{\bf 1}_{k\ne0}(\underline{A}^{\sig_0+1,\fr25}\Dl_\eta\hat{f})_k(t,\eta)\right\|_{L^2_kH^{\fr{n}{2}+}_\eta}\\
\nn&+e^{-m\nu t}\left\|{\bf 1}_{k\ne0}(\underline{A}^{\sig_0+1,\fr25}\hat{f})_k(t,\eta)\right\|_{L^2_kH^{\fr{n}{2}+}_\eta}+\|f_{\ne}^w(t)\|_{\mathcal{E}_m^{\sig_0+1,\fr25}}^2.
\end{align}
Arguing as \eqref{emb} again, we obtain \eqref{AM} immediately. 
 
Next, we turn to prove \eqref{BM}. One can see from \eqref{exp-Mtheta} that
\be\label{BMtheta}
\left\| {\bf B}  M_\theta\right\|_{L^2_x}\les \left\| {\bf B}  M_2\right\|_{L^2_x}+\left\| {\bf B}  \rho\right\|_{L^2_x}+\sum_{j=0}^\infty\left\| {\bf B}  _k(t)\mathcal{F}\left[\rho^j|M_1|^2\right]_k(t)\right\|_{L^2_k}.
\ee
It suffices to bound the last term on the right-hand side of the above inequality. Clearly, 
\[
 {\bf B}_k  (t)\le e^{\dl\nu^\fr13t}\la k\ra^\fr12\underline{B}^{\sig_0}_k(t)\quad \mathrm{and}\quad \la k\ra^\fr12\underline{B}^{\sig_0}_k(t)\le\underline{B}^{\sig_0+1}_k(t).
\]
Then using \eqref{alg2}, similar to \eqref{Aproduct}, and recalling that $ (M_1)_0=0$, we have
\begin{align}\label{Bproduct}
\nn&\sum_{j=0}^\infty\left\| {\bf B}_k  (t)\mathcal{F}\left[\rho^j|M_1|^2\right]_k(t)\right\|_{L^2_k}\\
\nn\les&e^{\dl\nu^\fr13 t}\left\|\la k\ra^\fr12\underline{B}^{\sig_0}_k(t)(\widehat{M_1})_k(t) \right\|_{L^2_k}^2\left(1+\sum_{j=1}^\infty \tl C_0^{j-1}\left\|\la k\ra^\fr12 \underline{B}_k^{\sig_0}(t)\hat{\rho}_k(t)\right\|_{L^2_k}^j \right)\\
\les& \left\| {\bf B}  M_1 \right\|_{L^2_x}\left(e^{-\dl_1\nu^\fr25t}e^{-\nu t}\left\|{\bf 1}_{k\ne0}(\underline{A}^{\sig_0+1,\fr25}\nb_\eta\hat{f})_k\left(t, \eta\right)\right\|_{L^2_kL^\infty_\eta}\right)\\
\nn&\times\left(1+\sum_{j=1}^\infty \left(\tl C_0e^{-\dl_1\nu^\fr25t}\left\|{\bf 1}_{k\ne0}(\underline{A}^{\sig_0+1,\fr25}\hat{f})_k\left(t, \eta\right)\right\|_{L^2_kL^\infty_\eta}\right)^j\right)
\les\left\| {\bf B}  M_1 \right\|_{L^2_x}.
\end{align}
Substituting \eqref{Bproduct} into \eqref{BMtheta} yields \eqref{BM}. Similarly, the inequality \eqref{BM'} can be obtained  by using \eqref{alg1} instead of \eqref{alg2}.  We omit the details for brevity.  

To prove \eqref{e-Mtheta0},  from \eqref{2thM}--\eqref{2thM'}, we find that
\begin{align}
\left|\left(M_\theta\right)_0(t)\right|\nn=&\left| \left(M_2\right)_0(t)-\fr{1}{n}\left((1+\rho)|u|^2\right)_0(t)\right|\\
\nn=&\left|\fr{\|E(t)\|_{L^2}^2}{(2\pi)^n}+\fr{1}{(2\pi)^n}\int_{\mathbb{T}^n}|M_1|^2- |M_1|^2\fr{\rho}{1+\rho} dx\right|\\
\nn\les&\left\|\fr{\hat{\rho}_k(t)}{|k|}\right\|_{L^2_k}^2+\|M_1(t)\|_{L^2_x}^2\left(1+\left\|\fr{\rho}{1+\rho}\right\|^2_{L^\infty_x}\right)\\
\les&\|\rho(t)\|_{L^2}^2+\|M_1(t)\|_{L^2}^2.
\end{align}
Then \eqref{e-Mtheta0} follows immediately.
The proof of Proposition \ref{prop-Mtheta} is completed.
\end{proof}

\section{Density estimates}\label{sec:Density estimates}
In this section, we improve \eqref{H-M0-H} and \eqref{H-M0-L}. Due to the appearance of $f$ in the equation of $\rho$, see \eqref{ex-rho}, we will use Remark \ref{rem-transform} frequently to recover the estimates of $f$ from $f^w$.

Before proceeding any further, let us consider the  following linear Volterra equation:
\begin{align}\label{Volterra}
    \hat{\rho}(t, k)=\mathcal{Q}(t, k)
  +\int_0^t\mathcal{K}^{\nu}(t-\tau,k)\hat{\rho}(\tau, k)d\tau,
\end{align}
with
$
    \mathcal{K}^{\nu}(t, k)=-S_k(t)t^{\mathrm{ap}}\hat{\mu}(k t^{\mathrm{ap}}), t\geq 0.
$
Note that when $\nu t\le1$, $\hat{\mu}(k t^{\mathrm{ap}})$ decays faster than $S_k(t)$ and when $\nu t>1$, $S_k(t)$ decays faster than $\hat{\mu}(k t^{\mathrm{ap}})$. Then by \eqref{S3} and \eqref{S-prop2},
it is not difficult to verify that there exists a sufficiently small $\dl_2>0$ independent of $\nu$, such that
\begin{align}\label{eq: K^nu}
&|\mathcal{K}^{\nu}(t, k)|\lesssim \fr{1}{|k|}e^{-\dl_2|k|t},\\
\label{eq: d-est-K^nu}&\left|\pr_t\mathcal{K}^{\nu}(t, k)\right|\lesssim e^{-\delta_2|k|t},\quad
\left|\pr_{tt}\mathcal{K}^{\nu}(t, k)\right|\lesssim |k|e^{-\delta_2|k|t}.
\end{align}
Let 
$\widetilde{\mathcal{K}^{\nu}}(z,k)$ be the Fourier-Laplace transform of $\mathcal{K}^{\nu}(t,k)$ with respect  to the $t$ variable:
\begin{align}\label{eq: Four-Lapl}
\widetilde{\mathcal K^{\nu}}(z,k)=\fr{1}{2\pi}\int_0^\infty e^{-zt}\mathcal{K}^{\nu}(t,k)dt.
\end{align}

We have the following estimate of the Fourier-Laplace transform of $\mathcal{K}^{\nu}(t, k)$ in $t$, which guarantees the Penrose stability condition. 
\begin{lem}\label{Lem: Penrose stability condition}
For all $\nu>0$ sufficiently small, there holds that
\begin{align*}
    &\inf_{z\in \mathbb{C}: \frak{Re}\, z>-c\delta_2|k|}\left|1-\widetilde{\mathcal{K}^{\nu}}(z, k)\right|\geq \kappa>0.
\end{align*}
where $\kappa>0, c\in(0,1)$ can be taken independent of $\nu>0$. 
\end{lem}
The proof of Lemma \ref{Lem: Penrose stability condition} is standard. We highlight the main idea here. On one hand, by the point-wise decay estimate \eqref{eq: K^nu}, we have
\be\label{lower-K1} 
|\widetilde{\mathcal{K}^{\nu}}(z, k)|\lesssim |k|^{-2}, \quad {\rm for \ all}\quad 
\frak{Re}\, z>-c\delta_2|k|.
\ee
Hence, we only need to focus on the small wave number case, which is finite. On the other hand, by taking the limit $\nu\to 0+$, we have $\lim_{\nu\to 0+}\mathcal{K}^{\nu}(t, k)=-t\hat{\mu}(kt)=:\mathcal{K}^0(t,k)$, whose Fourier-Laplace transform satisfies
\begin{align*}
    &\inf_{z\in \mathbb{C}: \frak{Re}\, z>-c\delta_2|k|}\left|1-\widetilde{\mathcal{K}^{0}}(z, k)\right|\geq 2\kappa>0.
\end{align*}
The lemma then follows immediately from the fact that $\widetilde{\mathcal{K}^{\nu}}(z, k)$ is a continuous function in $\nu$, which is a direct conclusion of the Lebesgue-dominated convergence theorem. 
We now have the estimate for the linearized Volterra equation in Gevrey norms:
\begin{lem}\label{lem-linear-rho}
    For $0<\delta\ll 1$ and all $\nu$ sufficiently small, we have
    \begin{align*}
        \|\langle t\rangle^b {\bf B}\rho\|_{L^2_{t, x}}\lesssim \|\langle t\rangle^b {\bf B}\mathcal{Q}\|_{L^2_{t, x}}.
    \end{align*}
\end{lem}
\begin{proof}
We use the idea in \cite{bedrossian2022brief,grenier2021landau} to treat the linearized Voltterra equation. First, by applying integration by part twice in \eqref{eq: Four-Lapl}, and applying \eqref{eq: d-est-K^nu}, together with \eqref{lower-K1}, we get that
 \begin{align}\label{eq:tildeK^nu}
 |\widetilde{\mathcal{K}^{\nu}}(z,k)|\lesssim \frac{1}{|k|^2+|\mathfrak{Im}\, z|^2}, \quad {\rm for \ all}\quad 
\frak{Re}\, z>-c\delta_2|k|.
 \end{align}
Thanks to Lemma D in \cite{bedrossian2022brief}, we get from Lemma \ref{Lem: Penrose stability condition} and \eqref{eq:tildeK^nu} that
  \begin{align}
     \widetilde{\hat \rho}(\om,k)=\widetilde{\mathcal{Q}}(\om,k)+{\mathcal{H}}(\om, k)\widetilde{\mathcal{Q}}(\om,k). 
 \end{align}
where ${\mathcal{H}}(\om, k)=\frac{\widetilde{\mathcal{K}^{\nu}}(\om, k)}{1-\widetilde{\mathcal{K}^{\nu}}(\om, k)}$. Let $\mathcal{R}(t, k)$ be the inverse Fourier-Laplace transform of ${\mathcal{H}}(w, k)$. By using \eqref{eq:tildeK^nu}, we get that $|\mathcal{R}(t, k)|\lesssim \frac{1}{|k|}e^{-\delta_3 |k|t}$ for some $\dl_3>0$ small enough, and then
\begin{align}\label{eq: rho-R}
    \hat{\rho}(t, k)=\mathcal{Q}(t, k)+\int_0^t\mathcal{R}(t-\tau, k)\mathcal{Q}(\tau, k)d\tau.
\end{align}
Applying the operator $\langle t\rangle^b{\bf B}$ to \eqref{eq: rho-R}, and taking $\nu$ sufficiently small, the exponential decay of the kernel $\mathcal{R}(t, k)$ enables us to use Young's inequality to complete the proof of this   lemma  immediately. 
\end{proof}

\subsection{$L^2_{t,x}$ estimate on $\rho$ }\label{sec-L2rho}
Taking 
$\mathcal{Q}(t,k)=S_k(t)(\widehat{f_{\mathrm{in}}})_k\left(kt^{\rm ap}\right)+\int_0^tS_k(t-\tau)\mathcal{N}_k(\tau, k(t-\tau)^{\rm ap}    )d\tau,$
in \eqref{Volterra}, by Lemma \ref{lem-linear-rho}, we have
\begin{align}
\|\la t\ra^b  {\bf B}  \rho\|_{L^2_{t,x}}^2\les {\bf I}+{\bf N}_{E}+{\bf N}_{\mu}+{\bf N}_{g},
\end{align}
where
\begin{align}
\nn{\bf I}=&\int_0^{T^*}\sum_{k\in\mathbb{Z}^n_*}\left|\langle t\rangle^{b} {\bf B}_k  (t)S_k(t)\widehat{g_{\rm in}}(k,kt^{\rm ap})\right|^2dt,\\
\nn{\bf N}_{E}=&\int_0^{T^*}\sum_{k\in\mathbb{Z}^n_*}\Bigg[\langle t\rangle^{b} {\bf B}_k  (t)\int_0^tS_k(t-\tau)\sum_{l\in\mathbb{Z}^n_*}\hat{\rho}_l(\tau)\fr{l}{|l|^2}\cdot k(t-\tau)^{\mathrm{ap}}\hat{f}_{k-l}\left(\tau, kt^{\rm ap}-l \tau^{\mathrm{ap}}\right)d\tau\Bigg]^2dt,\\
{\bf N}_{\mu}\nn=&\nu^2\int_0^{T^*}\sum_{k\in\mathbb{Z}^n_*}\left[\langle t\rangle^{b} {\bf B}_k  (t)\int_0^tS_k(t-\tau)(\widehat{\mathcal{C}_{\mu}})_k(\tau, k(t-\tau)^{\mathrm{ap}})d\tau\right]^2dt,\\
{\bf N}_{g}\nn=&\nu^2\int_0^{T^*}\sum_{k\in\mathbb{Z}^n_*}\left[\langle t\rangle^{b} {\bf B}_k  (t)\int_0^tS_k(t-\tau)(\widehat{\mathcal{C}[g]})_k(\tau, k(t-\tau)^{\mathrm{ap}})d\tau\right]^2dt.
\end{align}
Recalling \eqref{Cg}, $\widehat{\mathcal{C}[g]}$ consists of the following four parts
$
\widehat{\mathcal{C}[g]}:=\widehat{\mathcal{C}^1_g}+\widehat{\mathcal{C}^2_g}+\widehat{\mathcal{C}^3_g}+\widehat{\mathcal{C}^4_g},
$
with
\begin{subequations}\label{Cg1234}
\begin{align}
(\widehat{\mathcal{C}^1_g})_k\left(\tau,k(t-\tau)^{\mathrm{ap}}\right)=&-\sum_{l\in\mathbb{Z}^n_*}\hat\rho_l(\tau)|k(t-\tau)^{\mathrm{ap}}|^2\hat{f}_{k-l}\left(\tau,  kt^{\rm ap}-l \tau^{\mathrm{ap}}\right),\\
(\widehat{\mathcal{C}^2_g})_k\left(\tau, k(t-\tau)^{\mathrm{ap}}\right)=&-\sum_{l\in\mathbb{Z}^n_*}\hat\rho_l(\tau)k(t-\tau)^{\mathrm{ap}}\cdot e^{-\nu \tau}(\nb_\eta\hat{f})_{k-l}\left(\tau,  kt^{\rm ap}-l \tau^{\mathrm{ap}}\right),\\
(\widehat{\mathcal{C}^3_g})_k\left(\tau, k(t-\tau)^{\mathrm{ap}}\right)=&-\sum_{l\in\mathbb{Z}^n}(\widehat{M_\theta})_l(\tau)|k(t-\tau)^{\mathrm{ap}}|^2\hat{f}_{k-l}\left(\tau, kt^{\rm ap}-l \tau^{\mathrm{ap}}\right),\\
(\widehat{\mathcal{C}^4_g})_k\left(\tau, k(t-\tau)^{\mathrm{ap}}\right)=&-i\sum_{l\in\mathbb{Z}^n_*}(\widehat{M_1})_l(\tau)k(t-\tau)^{\mathrm{ap}}\hat{f}_{k-l}\left(\tau,  kt^{\rm ap}-l \tau^{\mathrm{ap}}\right).
\end{align}
\end{subequations}

\subsubsection{$\mathbf{I}$ initial contribution} Let us first treat $\mathbf{I}$, the contribution from the initial data. 
Note that
\be\label{tau-up}
t=\fr{\nu t}{1-e^{-\nu t}}t^{\rm ap}\les\la\nu t\ra t^{\rm ap}.
\ee
Combining this with \eqref{S-prop2}, taking the change of variable
 \be\label{change2}
 t'=|k|t^{\rm ap},
 \ee
 and then  using the Sobolev trace theorem (see Lemma 3.4 in \cite{BMM2016}), thanks to \eqref{emb}, we find that
\begin{align}
{\bf I}\nn\les&\sum_{k\in\mathbb{Z}^n_*}|k|\int_0^{T^*}\langle \nu t\rangle^{2b}e^{2\dl\nu^\fr13t}S_k^2(t) e^{2\lm(0)\la k,kt^{\rm ap} \ra^s}\left\la k, kt^{\rm ap} \right\ra^{2(\sig_0+b)}\left|\widehat{g_{\rm in}}(k,kt^{\rm ap})\right|^2dt\\
\nn\les&\sum_{k\in\N^n_*}\sup_{|\om|=1}\int_{-\infty}^{\infty}\left|\mathcal{F}\left[e^{\lm(0)\la \nb\ra^s}\la \nb\ra^{\sig_0+b}g_{\rm in}(k,\om t')\right]\right|^2dt'\\
\nn\les&  \left\| \mathcal{F}\left[e^{\lm(0)\la \nb\ra^s}\la \nb\ra^{\sig_0+b}g_{\rm in}\right](k,\eta)\right\|_{L^2_\eta H^{\fr{n-1}{2}+}_\eta}^2
\les \left\| e^{\lm(0)\la \nb\ra^s}\la \nb\ra^{\sig_0+b}\left(\la v\ra^mg_{\rm in}\right)\right\|_{L^2}^2\les \epsilon^2\nu^{2\gamma}.
\end{align}

\subsubsection{${\bf N}_{E}$ collisionless contributions}\label{Sec: Nonlinear collisionless contributions} 
In this section, we treat ${\bf N}_{E}$. This is the term that is discussed in section \ref{sec time-dependent Fourier multiplier} and determines the relationship between $s$ and $\gamma$.
According to whether $f$ is at zero frequency or not, we  divide ${\bf N}_{E}$ into two parts:
\begin{align}\label{e-NCL}
\nn{\bf N}_{E}={\bf N}_{E;\ne}+{\bf N}_{E;0}.
\end{align}

\noindent{\bf \em Treatment of ${\bf N}_{E;\ne}$.} We split ${\bf N}_{E;\ne}$ into high-low and low-high interactions and write
\[
{\bf N}_{E;\ne}={\bf N}_{E;\ne}^{\mathrm{HL}}+{\bf N}_{E;\ne}^{\mathrm{LH}},
\]
where the above decomposition is given according to the following partition of unity:
\be\label{pou-rho}
1={\bf 1}_{\left|k-l,kt^{\rm ap}-l \tau^{\mathrm{ap}}\right|\le2\left|l,l \tau^{\mathrm{ap}}\right|}+{\bf 1}_{ \left|l,l \tau^{\mathrm{ap}}\right|<\fr12\left|k-l,kt^{\rm ap}-l \tau^{\mathrm{ap}}\right|}.
\ee
To bound ${\bf N}_{E;\ne}^{\mathrm{HL}}$, on the one hand,
in view of \eqref{app5} and \eqref{S-prop2},  we find that for some $0<c<1$,
\begin{align}\label{BS-up1}
\nn &\la t\ra^b {\bf B}_k  (t)S_k(t-\tau){\bf 1}_{\left|k-l,kt^{\rm ap}-l \tau^{\mathrm{ap}}\right|\le2\left|l,l \tau^{\mathrm{ap}}\right|}\\
\les& S^{\fr34}_k(t-\tau)  {\bf B}_l    (\tau)\la t\ra^b e^{\left(\lm(t)-\lm(\tau)\right)\left\la k,kt^{\rm ap}\right\ra^s}  e^{c\lm(\tau)\left\la k-l,kt^{\rm ap}-l \tau^{\mathrm{ap}}\right\ra^s}\la k-l\ra^\fr12.
\end{align}
We also have for $k\ne l$ that
\begin{align}\label{regu-overflow}
|k|(t-\tau)^{\mathrm{ap}}
\les e^{\nu\tau}\left\la \tau^{\mathrm{ap}}\right\ra\left| k-l, kt^{\rm ap}-l \tau^{\mathrm{ap}} \right|.
\end{align}
Therefore,
\begin{align}\label{NHL1}
{\bf N}_{E;\ne}^{\mathrm{HL}}\nn\les&\int_0^{T^*}\sum_{k\in\mathbb{Z}^n_*}\Bigg[\sum_{l\in\mathbb{Z}^n_*,l\ne k}\int_0^tS_k^\fr34(t-\tau)| {\bf B}  \hat{\rho}_l(\tau)|\la t\ra^b\\\nn&\times\fr{e^{\left(\lm(t)-\lm(\tau)\right)\left\la k,kt^{\rm ap}\right\ra^s}}{|l|}e^{c\lm(\tau)\left\la k-l,kt^{\rm ap}-l \tau^{\mathrm{ap}}\right\ra^s}\\
&\times e^{\nu\tau}\la \tau\ra \left|\mathcal{F}\left[\la \nb\ra^\fr32 f\right]_{k-l}\left(\tau, kt^{\rm ap}-l \tau^{\mathrm{ap}}\right)\right|d\tau\Bigg]^2dt.
\end{align}
By using the Sobolev embedding \eqref{emb}, we get that
\begin{align}\label{embd}
\nn&e^{\nu\tau}\la \tau\ra \left|\mathcal{F}\left[\la \nb\ra^\fr32 f\right]_{k-l}\left(\tau, kt^{\rm ap}-l \tau^{\mathrm{ap}}\right)\right|{\bf 1}_{l\ne k}\\
\nn=&\la\tau\ra^{1-3\gamma}\left[e^{\nu\tau}\left(\nu^\fr13\la\tau\ra\right)^{3\gamma}e^{m\nu\tau}e^{-\dl_1\nu^\fr13\tau}\right]e^{-\lm(\tau)\la k-l,kt^{\rm ap}-l \tau^{\mathrm{ap}}\ra^s}\\
\nn&\times \nu^{-\gamma} e^{-m\nu\tau}e^{\dl_1\nu^\fr13\tau}\left|\mathcal{F}\left[e^{\lm(\tau)\la\nb \ra^s}\la \nb\ra^\fr32 f\right]_{k-l}\left(\tau, kt^{\rm ap}-l \tau^{\mathrm{ap}}\right)\right|{\bf 1}_{l\ne k}\\
\nn\les&e^{-\fr34\dl_1\nu^\fr13\tau}\la\tau\ra^{1-3\gamma}e^{-\lm(\tau)\la k-l,kt^{\rm ap}-l \tau^{\mathrm{ap}}\ra^s}\\
\nn&\times\nu^{-\gamma}\sup_t\left(e^{-m\nu t}\left\| \underline{A}^{ \sig_1 ,\fr13}f_{\ne}\left(t, \eta\right)\right\|_{L^2_kH^{\fr{n}{2}+}_\eta}\right)\\
\nn\les&e^{-\fr34\dl_1\nu^\fr13\tau}\la\tau\ra^{1-3\gamma}e^{-\lm(\tau)\la k-l,kt^{\rm ap}-l \tau^{\mathrm{ap}}\ra^s}\left(\nu^{-\gamma}\sup_t\left\| f^w_{\ne}\right\|_{\mathcal{E}^{ \sig_1 ,\fr13}_m}\right)\\
\les&\eps e^{-\fr34\dl_1\nu^\fr13\tau}\la\tau\ra^{1-3\gamma}e^{-\lm(\tau)\la k-l,kt^{\rm ap}-l \tau^{\mathrm{ap}}\ra^s}.
\end{align}
Substituting this into \eqref{NHL1}, we find that
\begin{align}\label{NHL2}
{\bf N}_{E;\ne}^{\mathrm{HL}}\nn\les&\eps^2\int_0^{T^*}\sum_{k\in\mathbb{Z}^n_*}\Bigg[\sum_{l\in\mathbb{Z}^n_*}\int_0^tS^\fr34_k(t-\tau)| {\bf B}  \hat{\rho}_l(\tau)|\la t\ra^b\fr{e^{\left(\lm(t)-\lm(\tau)\right)\left\la k,kt^{\rm ap}\right\ra^s}}{|l|}\\
&\times e^{-\fr34\dl_1\nu^\fr13\tau}\la\tau\ra^{1-3\gamma}e^{-\underline{c}\la k-l,kt^{\rm ap}-l \tau^{\mathrm{ap}}\ra^s}d\tau\Bigg]^2dt,
\end{align}
where $\underline{c}=(1-c)\lm_\infty$. To treat the right-hand side of \eqref{NHL2}, we consider two cases. 

{\bf Case 1: $\tau\le\fr{t}{2}$.} A direct calculation gives that
\be\label{gap1}
\lm(\tau)-\lm(t)=\fr{\tl\dl}{(1+\tau)^a}-\fr{\tl \dl}{(1+t)^a}\gt\fr{t-\tau}{(1+ \tau)^a(1+t)}, \quad \mathrm{for \ \ all}\quad \tau<t.
\ee
Combining this  with \eqref{tau-up}, \eqref{S-prop2} and the fact $t\le 2(t-\tau)$ yields, for $\Gamma\ge\fr{b+1}{s-a}$, 
\begin{align}\label{small-tau1}
\nn &S^\fr14_k(t-\tau)\la t\ra^b\la\tau\ra^{1-3\gamma}\fr{e^{\left(\lm(t)-\lm(\tau)\right)\left\la k,kt^{\rm ap}\right\ra^s}}{|l|}{\bf 1}_{k\ne0}
\les S^\fr14_k(t-\tau)\fr{\la t\ra^{b+1-3\gamma+a\Gamma}}{\la k, kt^{\rm ap}\ra^{s\Gamma}}{\bf 1}_{k\ne0}\\
\les&\la \nu (t-\tau)\ra^{b+1+a\Gamma}S^\fr14_k(t-\tau)\fr{\la t^{\rm ap}\ra^{b+1-3\gamma+a\Gamma}}{\la k, kt^{\rm ap}\ra^{s\Gamma}}{\bf 1}_{k\ne0}\les1.
\end{align}
Substituting this into \eqref{NHL2}, and using Schur's test, we are led to
\begin{align}\label{schur1}
{\bf N}_{E;\ne}^{\mathrm{HL}}\nn\les&\eps^2\int_0^{T^*}\sum_{k\in\mathbb{Z}^n_*}\Bigg[\sum_{l\in\mathbb{Z}^n_*}\int_0^tS^\fr12_k(t-\tau) e^{-\fr12\dl_1\nu^\fr13\tau}e^{-\underline{c}\la k-l,kt^{\rm ap}-l \tau^{\mathrm{ap}}\ra^s}| {\bf B}  \hat{\rho}_l(\tau)|d\tau\Bigg]^2dt\\
\les&\eps^2\| {\bf B}  \rho\|_{L^2_{t,x}}^2\left[\sup_{t}\sup_{k\ne 0}\sum_{l\in\mathbb{Z}^n_*}\int_0^t\mathcal{K}_1(t, \tau, k, l)d\tau\right] 
\left[\sup_{\tau}\sup_{l\ne0}\sum_{k\in\mathbb{Z}^n_*}\int_\tau^{T*}\mathcal{K}_1(t, \tau, k, l)dt\right],
\end{align}
where 
$
\mathcal{K}_1(t, \tau, k, l)=S^\fr12_k(t-\tau) e^{-\fr12\dl_1\nu^\fr13\tau}e^{-\underline{c}\la k-l,kt^{\rm ap}-l \tau^{\mathrm{ap}}\ra^s}.
$
Taking the change of variable 
\be\label{change1}
\tau'=|l|\tau^{\mathrm{ap}},
\ee
and using \eqref{app1}, we obtain
\begin{align}
\nn\sup_{k\ne 0}\sum_{l\in\mathbb{Z}^n_*}\int_0^t\mathcal{K}_1(t, \tau, k, l)d\tau\les&\sup_k\sum_{l\in\Z^n_*}e^{-c_s\underline{c}\la k-l\ra^s}\int_{\R}  e^{-c_s\underline{c}\la kt^{\rm ap}-\fr{l}{|l|}\tau'\ra^s}d\tau'\les1.
\end{align}
For the other kernel estimate, the only difference is that we need to use \eqref{S-prop2} to absorb the Jacobian $e^{\nu t}$ arising from the change of variable \eqref{change2}.
Thus, for $\tau\le\fr{t}{2}$, we arrive at
 \be
{\bf N}_{E;\ne}^{\mathrm{HL}}\les \eps^2\| {\bf B}  \rho\|_{L^2_{t,x}}^2,
 \ee
which is consistent with Proposition \ref{prop: btsp} provided we set $\eps$ sufficiently small.

{\bf Case 2: $\fr{t}{2}\le\tau\le t$.} The dimensional reducing process in \cite{BMM2016, MouhotVillani2011} enables us to  consider only the case $n=1$. Furthermore, by symmetry, it suffices to investigate the case $k\ge1$. If $l<0$ or $0\le l<k$, we have
\beno
|kt^{\rm ap}-l \tau^{\rm ap}|\ge \tau^{\rm ap}.
\eeno
This, together with \eqref{tau-up}, implies that
\begin{align}\label{absorb-tau}
\nn &e^{-\fr34\dl_1\nu^\fr13\tau}\la\tau\ra^{1-3\gamma}e^{-\underline{c}\la k-l,kt^{\rm ap}-l \tau^{\mathrm{ap}}\ra^s}\\
\nn\les& e^{-\fr34\dl_1\nu^\fr13\tau}\la\nu\tau\ra^{1-3\gamma}\la\tau^{\rm ap}\ra^{1-3\gamma}e^{-\fr12{\underline{c}}\la \tau^{\mathrm{ap}}\ra^s}e^{-\fr12{\underline{c}}\la k-l,kt^{\rm ap}-l \tau^{\mathrm{ap}}\ra^s}\\
\les & e^{-\fr12\dl_1\nu^\fr13\tau}e^{-\fr12{\underline{c}}\la k-l,kt^{\rm ap}-l \tau^{\mathrm{ap}}\ra^s}.
\end{align}
Then similar to \eqref{schur1}, we have
 \be
{\bf N}_{E;\ne}^{\mathrm{HL}}\les \eps^2\|\la t\ra^b {\bf B}  \rho\|_{L^2_{t,x}}^2.
 \ee
For the remaining case $l\ge k+1$, if $|kt^{\rm ap}-l\tau^{\rm ap}|\ge\fr{t^{\rm ap}}{2}$, then \eqref{absorb-tau} still holds. Therefore, it suffices to consider the case
\be
  |kt^{\rm ap}-l\tau^{\rm ap}|\le\fr{t^{\rm ap}}{2} \quad {\rm with}\quad l\ge k+1.
\ee
 This in turn shows that
\beno
t^{\rm ap}-\tau^{\rm ap}=t^{\rm ap}-\fr{kt^{\rm ap}}{l}+\fr{kt^{\rm ap}}{l}-\tau^{\rm ap}=\fr{(l-k)t^{\rm ap}}{l}+\fr{kt^{\rm ap}-l\tau^{\rm ap}}{l}\ge\fr{t^{\rm ap}}{2l}.
\eeno
We have 
$t^{\rm ap}-\tau^{\rm ap}=e^{-\nu\tau}(t-\tau)^{\rm ap}\le e^{-\nu\tau}(t-\tau)$,
which together with \eqref{gap1} gives
\be
\lm(\tau)-\lm(t)\gt\fr{e^{-\nu(t-\tau)} e^{\nu t} t^{\rm ap}}{\la \tau\ra^a(1+t)|l|}\gt \fr{e^{-\nu(t-\tau)} }{\la \tau\ra^a|l|}.
\ee
Hence, similar to \eqref{small-tau1},  for a fixed large constant $\Gamma_0$ satisfying
\be\label{kappa}
a\Gamma_0+1-3\gamma=s\Gamma_0,
\ee
 there holds
\begin{align}\label{gap2}
\nn&S^\fr14_k(t-\tau)e^{-\fr{1}{4}\dl_1\nu^\fr13\tau}\la\tau\ra^{1-3\gamma}e^{\left(\lm(t)-\lm(\tau)\right)\left\la k, kt^{\rm ap}\right\ra^s}{\bf 1}_{k\ne0}\\
\les_{\Gamma_0}&S^\fr14_k(t-\tau)e^{-\fr{1}{4}\dl_1\nu^\fr13\tau} \fr{e^{\Gamma_0\nu(t-\tau)} t^{a\Gamma_0+1-3\gamma}|l|^{\Gamma_0}}{\left\la k, kt^{\rm ap}\right\ra^{s\Gamma_0}} {\bf 1}_{k\ne0}\les|l|^{\Gamma_0-s\Gamma_0}\la k-l\ra^{s\Gamma_0}.
\end{align}
Substituting \eqref{gap2}  into \eqref{NHL2} and noting that $\la t\ra^b\approx \la \tau\ra^b$, using  Schur's test again, we have
\begin{align}\label{NHL4}
{\bf N}_{E;\ne}^{\mathrm{HL}}
\nn\les&\eps^2\int_0^{T^*}\sum_{k\in\mathbb{Z}^n_*}\Bigg[\sum_{l\in\mathbb{Z}^n_*}\int_0^t\left(\la\tau\ra^b| {\bf B}  \hat{\rho}_l(\tau)|\right) |l|^{\Gamma_0-1-s\Gamma_0}S^\fr12_k(t-\tau)\\
\nn&\times e^{-\fr{\dl_1}{2}\nu^\fr13\tau}e^{-\fr12\underline{c}\la k-l,kt^{\rm ap}-l \tau^{\mathrm{ap}}\ra^s}d\tau\Bigg]^2dt\\
\nn\les&\eps^2\|\la t\ra^b {\bf B}  \rho\|_{L^2_{t,x}}^2\left[\sup_t\sup_{k\ne0}\sum_{l\ne0}\int_0^t\mathcal{K}_2(t,\tau, k,l)d\tau\right]\\
&\times\left[\sup_{\tau}\sup_{l\ne0}\sum_{k\ne0}\int_\tau^{T^*}\mathcal{K}_2(t,\tau, k,l)dt\right],
\end{align}
where
$
\mathcal{K}_2(t,\tau, k,l)=|l|^{\Gamma_0-1-s\Gamma_0}S^\fr12_k(t-\tau)
e^{-\fr{\dl_1}{2}\nu^\fr13\tau}e^{-\fr12\underline{c}\la k-l,kt^{\rm ap}-l \tau^{\mathrm{ap}}\ra^s}.
$
By the change of variable \eqref{change1} and \eqref{app1},
\beq
\sum_{l\ne0}\int_0^t\mathcal{K}_2(t,\tau, k,l)d\tau\nn&\les&\sum_{l\ne0}|l|^{\Gamma_0-2-s\Gamma_0}\int_0^t
\left(e^{-\fr{\dl_1}{2}\nu^\fr13\tau}e^{\nu\tau}\right)e^{-\fr12\underline{c}\la k-l,kt^{\rm ap}-\fr{l}{|l|}\tau'\ra^{s}}d\tau'\\
\nn&\les&\sum_{l\ne0} e^{-\fr12c_s\underline{c}\la k-l\ra^s}\int_0^t
e^{-\fr12c_s\underline{c}\la kt^{\rm ap}-\fr{l}{|l|}\tau'\ra^s}d\tau'\les1,
\eeq
provided
\be\label{gamma1}
\Gamma_0-2-s\Gamma_0\le0.
\ee
Similarly, using the change of variable \eqref{change2}, \eqref{app1} and \eqref{S-prop2}, we obtain
\begin{align}
\sum_{k\ne0}\int_0^t\mathcal{K}_2(t,\tau, k,l)d\tau\nn\les&\sum_{k\ne0}|k|^{\Gamma_0-2-s\Gamma_0}\int_0^t
\left(S_k^\fr12(t-\tau)e^{-\fr{\dl_1}{2}\nu^\fr13\tau}e^{\nu t}\right)\\
\nn&\times\la k-l\ra^{\Gamma_0-1-s\Gamma_0}e^{-\fr12\underline{c}\la k-l,\fr{k}{|k|}t'-l \tau^{\mathrm{ap}}\ra^{s}}dt'\\
\nn\les&\sum_{k\ne0} \la k-l\ra e^{-\fr12c_s\underline{c}\la k-l\ra^s}\int_0^t
e^{-\fr12c_s\underline{c}\la \fr{k}{|k|}t'-l \tau^{\mathrm{ap}}\ra^{s}}dt'\les1.
\end{align}
Combining \eqref{kappa} with \eqref{gamma1} yields
\be\label{eq: s-gamma}
s\ge\fr{1-3\gamma}{3-3\gamma}+\fr{2a}{3-3\gamma}>\fr{1-3\gamma}{3-3\gamma}.
\ee
Now we turn to bound ${\bf N}_{E;\ne}^{\mathrm{LH}}$. If $\tau\le\fr{t}{2}$, similar to \eqref{small-tau1}, we have
\begin{align}\label{small-tau2}
 S^\fr14_k(t-\tau)\la t\ra^be^{\left(\lm(t)-\lm(\tau)\right)\left\la k,kt^{\rm ap}\right\ra^s}{\bf 1}_{k\ne0}\les1.
\end{align}
If $\fr{t}{2}\le\tau\le t$,
\be\label{large-tau}
S^\fr14_k(t-\tau)\la t\ra^be^{\left(\lm(t)-\lm(\tau)\right)\left\la k,kt^{\rm ap}\right\ra^s}{\bf 1}_{k\ne0}\les\la \tau\ra^b.
\ee
Combining these two inequalities with  \eqref{app3}, \eqref{S-prop2}  and \eqref{regu-overflow}, we have
\begin{align}\label{BS-up3}
\nn&\la t\ra^b {\bf B}_k  (t)S_k(t-\tau)|k|(t-\tau)^{\mathrm{ap}}{\bf 1}_{ \left|l,l \tau^{\mathrm{ap}}\right|<\fr12\left|k-l,kt^{\rm ap}-l \tau^{\mathrm{ap}}\right|}{\bf 1}_{k\ne l}\\
\nn\les&|k|^\fr12 S^\fr12_k(t-\tau)e^{-\fr12\dl_1\nu^\fr25\tau}\left(\la\tau\ra^{b}  \left\la \tau^{\mathrm{ap}}\right\ra e^{\dl \nu^\fr13  \tau}   e^{c\lm(\tau)\left\la l,l \tau^{\mathrm{ap}}\right\ra^s}\right)\\
&\times  \underline{A}_{k-l}^{\sig_0+1,\fr25}\left(\tau, kt^{\rm ap}-l \tau^{\mathrm{ap}}\right){\bf 1}_{k\ne l}.
\end{align}
Consequently,  using Cauchy-Schwarz yields
\begin{align}\label{NLH1}
\nn{\bf N}_{E;\ne}^{\mathrm{LH}}\les&\int_0^{T^*}\sum_{k\in\mathbb{Z}^n_*}|k|
\Bigg[\sum_{l\in\mathbb{Z}^n_*, l\ne k}\int_0^tS^\fr{1}{2}_k(t-\tau) e^{-\fr14\dl_1\nu^\fr25\tau}e^{-\underline{c}\left\la l,l \tau^{\mathrm{ap}}\right\ra^s}\la \tau\ra^{b} | {\bf B}  \hat{\rho}_l(\tau)| \\
\nn&\times\left|e^{-m\nu\tau} (\underline{A}^{\sig_0+1,\fr25}\hat{f})_{k-l}\left(\tau, kt^{\rm ap}-l \tau^{\mathrm{ap}}\right)\right|d\tau\Bigg]^2dt\\
\nn\les&\int_0^{T^*}\sum_{k\in\mathbb{Z}^n_*}|k|
\sum_{l\in\mathbb{Z}^n_*,l\ne k}\int_0^t  e^{-\fr14\dl_1\nu^\fr25\tau}   e^{-\underline{c}\left\la l,l \tau^{\mathrm{ap}}\right\ra^s}\la \tau\ra^{b} | {\bf B}  \hat{\rho}_l(\tau)|d\tau\\
\nn&\times \sum_{l\in\mathbb{Z}^n_*,l\ne k}\int_0^t S_k(t-\tau) e^{-\fr14\dl_1\nu^\fr25\tau}    e^{-\underline{c}\left\la l,l \tau^{\mathrm{ap}}\right\ra^s}\la \tau\ra^{b} | {\bf B}  \hat{\rho}_l(\tau)|\\
&\times\left|e^{-m\nu\tau}(\underline{A}^{\sig_0+1,\fr25}\hat{f})_{k-l}\left(\tau, kt^{\rm ap}-l \tau^{\mathrm{ap}}\right)\right|^2d\tau dt.
\end{align}
By \eqref{app1}, one deduces  that
\begin{align}\label{les-rho-low1}
\nn&\sum_{l\in\mathbb{Z}^n_*}\int_0^te^{-\fr14\dl_1\nu^\fr25\tau}   e^{-\underline{c}\left\la l,l \tau^{\mathrm{ap}}\right\ra^s}\la \tau\ra^{b} | {\bf B}  \hat{\rho}_l(\tau)|d\tau\\
\nn\les&\sum_{l\in\mathbb{Z}^n_*}\int_0^te^{-\fr{1}{4}\dl_1\nu^\fr25\tau}  e^{-c_s\underline{c}\la l\ra^s}  e^{-c_s\underline{c}\left\la l \tau^{\mathrm{ap}}\right\ra^s}\la \tau\ra^{b} | {\bf B}  \hat{\rho}_l(\tau)|d\tau\\
\les&\left(\sum_{l\in\mathbb{Z}^n_*}e^{-2c_s\underline{c}\la l\ra^s} \int_0^te^{-\fr{1}{2}\dl_1\nu^\fr25\tau}   e^{-2c_s\underline{c}\left\la l \tau^{\mathrm{ap}}\right\ra^s}d\tau\right)^\fr12\left\|\la t\ra^b  {\bf B}  \rho\right\|_{L^2_{t,x}}
\les \left\|\la t\ra^b  {\bf B}  \rho\right\|_{L^2_{t,x}},
\end{align}
where we have used the change of variable \eqref{change1}. Substituting this into \eqref{NLH1}, using Fubini's theorem, noting that \eqref{S-prop2} implies that $S_k(t-\tau)\les e^{-\nu(t-\tau)}$ for sufficiently small $\nu$, and using the change of variable \eqref{change2}, the standard Sobolev trace theorem, see Lemma 3.4 in \cite{BMM2016}, for instance, and \eqref{emb}, we find that
\begin{align}\label{NLH2}
\nn{\bf N}_{E;\ne}^{\mathrm{LH}}
\nn\les&\left\| \la t\ra^b {\bf B}  \rho\right\|_{L^2_{t,x}}\int_0^{T^*}
 \sum_{l\in\mathbb{Z}^n_*} e^{-\fr{\dl_1}{8}\nu^\fr25\tau}    e^{-\underline{c}\left\la l,l \tau^{\mathrm{ap}}\right\ra^s}\la \tau\ra^{b} | {\bf B}  \hat{\rho}_l(\tau)|\\
\nn&\times  \sum_{k\in\mathbb{Z}^n_*}\int_\tau^{T^*}|k| e^{-\nu t} \left|e^{-m\nu \tau}(\underline{A}^{\sig_0+1,\fr25}\hat{f})_{k-l}\left(\tau, kt^{\rm ap}-l \tau^{\mathrm{ap}}  \right)\right|^2 dtd\tau\\
\nn\les&\left\|\la t\ra^b  {\bf B}  \rho\right\|_{L^2_{t,x}}^2\sup_{0\le\tau\le T^*}\sum_{k\in\mathbb{Z}^n}\sup_{\zeta\in\mathbb{R}^n}\sup_{|\om|=1}\int_{-\infty}^{\infty}\left|e^{-m\nu\tau}(\underline{A}^{\sig_0+1,\fr25}\hat{f})_k\left(\tau, \om t'+\zeta\right)\right|^2 dt' \\
 \nn\les&\left\|\la t\ra^b  {\bf B}  \rho\right\|^2_{L^2_{t,x}}\sup_{0\le\tau\le T^*}\left(e^{-m\nu \tau}\left\|(\underline{A}^{\sig_0+1,\fr25}\hat{f})_k(\tau)\right\|_{L^2_kH^{\fr{n-1}{2}+}_\eta}\right)^2\\
  \les&\left\|\la t\ra^b  {\bf B}  \rho\right\|^2_{L^2_{t,x}}\sup_{0\le t\le T^*}\|f^w(t)\|_{\mathcal{E}^{\sig_0+1,\fr25}_m}^2.
\end{align}

\noindent\underline{\bf Treatment of ${\bf N}_{E;0}$.} 
We split ${\bf N}_{E;0}$ into high-low and low-high interactions:  
$
{\bf N}_{E;0}={\bf N}_{E;0}^{\mathrm{HL}}+{\bf N}_{E;0}^{\mathrm{LH}},
$
where
\begin{align}
\nn{\bf N}_{E;0}^{\mathrm{HL}}=&\int_0^{T^*}\sum_{k\in\mathbb{Z}^n_*}\Bigg[\int_0^t{\bf 1}_{\left|k(t^{\rm ap}-\tau^{\mathrm{ap}})\right|\le2\left|k,k\tau^{\mathrm{ap}}\right|}S_k(t-\tau)\la t\ra^b {\bf B}_k  (t)\hat{\rho}_k(\tau)\\
\nn&\times e^{\nu \tau}(t^{\rm ap}-\tau^{\mathrm{ap}})\hat{f}_0\left(\tau, k(t^{\rm ap}-\tau^{\mathrm{ap}})\right)d\tau\Bigg]^2dt,\\
\nn{\bf N}_{E;0}^{\mathrm{LH}}=&\int_0^{T^*}\sum_{k\in\mathbb{Z}^n_*}\Bigg[\int_0^t{\bf 1}_{\left|k,k\tau^{\mathrm{ap}}\right|<\fr12\left|k(t^{\rm ap}-\tau^{\mathrm{ap}})\right|}S_k(t-\tau)\la t\ra^b {\bf B}_k  (t)\hat{\rho}_k(\tau)\\
\nn&\times(t-\tau)^{\mathrm{ap}}\hat{f}_0\left(\tau, k(t^{\rm ap}-\tau^{\mathrm{ap}})\right)d\tau\Bigg]^2dt.
\end{align}

For ${\bf N}_{E;0}^{\mathrm{HL}}$, arguing as in \eqref{small-tau2} and \eqref{large-tau}, one deduces that
\begin{align}\label{BS-up2}
\nn&\la t\ra^b {\bf B}_k  (t)S_k(t-\tau){\bf 1}_{|k(t^{\rm ap}-\tau^{\rm ap})|\le 2|k,k\tau^{\mathrm{ap}}|}\\
\nn\les& S^{\fr34}_k(t-\tau)\la t\ra^b {\bf B}_k  (\tau)e^{\left(\lm(t)-\lm(\tau)\right)\left\la k,kt^{\rm ap}\right\ra^s}e^{c\lm(\tau)\left\la k(t^{\rm ap}-\tau^{\mathrm{ap}})\right\ra^s}\\
\les&\left(\la\tau\ra^b {\bf B}_k  (\tau) \right)e^{c\lm(\tau)\left\la k(t^{\rm ap}-\tau^{\mathrm{ap}})\right\ra^s}\fr{S_k^\fr12(t-\tau)}{\la k(t-\tau)^{\rm ap}\ra^2}\la k(t-\tau)^{\rm ap}\ra^2.
\end{align}
In view of \eqref{S-prop2}, we have ${\bf 1}_{k\ne0}S_k^{\fr12}(t)\les e^{-\nu t}$. Then
\be\label{borrow finite regu}
\int_{-\infty}^\infty \fr{S^{\fr12}_k(t-\tau)}{\la k(t-\tau)^{\rm ap} \ra^{2}}d\tau=\int_{-\infty}^\infty \fr{S^{\fr12}_k(t)}{\la kt^{\rm ap} \ra^{2}}dt\les1.
\ee
Thus, by  Schur's test, we obtain that
\begin{align}\label{NHL0}
{\bf N}_{E;0}^{\mathrm{HL}}
\nn\les&\left(\sup_\tau\sup_\eta e^{\lm(\tau)\left\la \eta\right\ra^s}\la e^{\nu\tau}\eta\ra^3\left|\hat{f}_0\left(\tau, \eta\right)\right|\right)^2\\
\nn&\times\int_0^{T^*}\sum_{k\in\mathbb{Z}^n_*}\Bigg[\int_0^t\fr{S^{\fr12}_k(t-\tau)}{\la k(t-\tau)^{\rm ap}\ra^2}\la\tau\ra^b\left| {\bf B}  \hat{\rho}_k(\tau)\right|d\tau\Bigg]^2dt\\
\nn\les&\sup_t \left\| e^{\lm(t)\left\la \eta\right\ra^s}\la e^{\nu t}\eta\ra^3 \hat{f}_0\left(t, \eta\right)\right\|_{L^\infty_\eta}^2\sup_{k\in\Z^n_*}\sup_t\int_0^t\fr{S^{\fr12}_k(t-\tau)}{\la k(t-\tau)^{\rm ap}\ra^2}d\tau\\
\nn&\times\sup_{k\in\Z^n_*}\sup_\tau\int_\tau^{T^*}\fr{S^{\fr12}_k(t-\tau)}{\la k(t-\tau)^{\rm ap}\ra^2}dt \left\|\la t \ra^b {\bf B}  \rho\right\|_{L^2_{t,x}}^2\\
\les&\eps^2\nu^{2\gamma}\left\|\la t \ra^b {\bf B}  \rho\right\|_{L^2_{t,x}}^2.
\end{align}

Next we consider ${\bf N}_{E;0}^{\mathrm{LH}}$. It is worth emphasizing that under the restriction
$
\left|k,k\tau^{\mathrm{ap}}\right|<\fr12\left|k(t^{\rm ap}-\tau^{\mathrm{ap}})\right|
$
it holds that
\be\label{nutau-up}
e^{\nu\tau}<\fr32.
\ee
Consequently, now the analogue of  \eqref{BS-up3} reads
\begin{align}\label{BS-up4}
\nn&\la t\ra^b  {\bf B}_k  (t)S_k(t-\tau)(t-\tau)^{\mathrm{ap}}{\bf 1}_{\left|k,k\tau^{\mathrm{ap}}\right|<\fr12\left|k(t^{\rm ap}-\tau^{\mathrm{ap}})\right|}\\
\les& |k|^\fr12 S_k^\fr12(t-\tau) \left(e^{-\nu\tau}\la \tau\ra^b e^{\dl \nu^\fr13 \tau}   e^{c\lm(\tau)\left\la k, k\tau^{\mathrm{ap}}\right\ra^s}\right) A_0^{\sig_0+1,\fr25}\left(\tau, k(t^{\rm ap}-\tau^{\mathrm{ap}}) \right).
\end{align}
Thus, similar to \eqref{les-rho-low1} and \eqref{NLH2}, we infer from the above two inequalities  that
\begin{align}\label{NLH0}
\nn{\bf N}_{E;0}^{\mathrm{LH}}\les&\int_0^{T^*}\sum_{k\in\mathbb{Z}^n_*}|k|\Bigg[\int_0^t    \la \tau\ra^b| {\bf B}  \hat{\rho}_k(\tau)|e^{-\underline{c}\la k,k\tau^{\mathrm{ap}}\ra^s}  \\
\nn&\times S^\fr12_k(t-\tau) e^{-m\nu \tau}(\underline{A}^{\sig_0+1,\fr25}\hat{f})_0\left(\tau, k(t^{\rm ap}-\tau^{\mathrm{ap}})\right)d\tau\Bigg]^2dt\\
\nn\les&\left\|\la t\ra^b  {\bf B}  \rho\right\|_{L^2_{t,x}}\int_0^{T^*}\sum_{k\in\mathbb{Z}^n_*}|k|\int_0^t  \la \tau\ra^b| {\bf B}  \hat{\rho}_k(\tau)|e^{-\nu\tau}e^{-\underline{c}\la k,k\tau^{\mathrm{ap}}\ra^s}  \\
\nn&\times S_k(t-\tau) \left|e^{-m\nu\tau}(\underline{A}^{\sig_0+1,\fr25}\hat{f})_0\left(\tau, kt^{\rm ap}-k\tau^{\mathrm{ap}}\right)\right|^2d\tau dt\\
\nn\les&\left\|\la t\ra^b  {\bf B}  \rho\right\|_{L^2_{t,x}}^2
  \sup_\tau\sup_{|\om|=1}\sup_{\zeta\in\mathbb{R}^n}\int^{\infty}_{-\infty}\left| e^{-m\nu\tau}(\underline{A}^{\sig_0+1,\fr25}\hat{f})_0\left(\tau, \om t'+\zeta\right)\right|^2 dt'\\
\les&\left\|\la t\ra^b  {\bf B}  \rho\right\|^2_{L^2_{t,x}}\sup_{0\le t\le T^*}\|f^w(t)\|_{\mathcal{E}^{\sig_0+1,\fr25}_m}^2.
\end{align}

\subsubsection{${\bf N}_\mu$ collision contributions} Recalling \eqref{Cmu-hat} and \eqref{eta-bar}, we have
\begin{align*}
{\bf N}_{\mu}\les&\nu^2\int_0^{T^*}\sum_{k\in\mathbb{Z}^n_*}\left[\langle t\rangle^{b} {\bf B}_k  (t)\int_0^tS_k(t-\tau)(\widehat{M_\theta})_k(\tau)(\widehat{\Dl_v\mu})(k(t-\tau)^{\rm ap})d\tau\right]^2dt\\
&+\nu^2\int_0^{T^*}\sum_{k\in\mathbb{Z}^n_*}\left[\langle t\rangle^{b} {\bf B}_k  (t)\int_0^tS_k(t-\tau)(\widehat{M_1})_k(\tau)\cdot(\widehat{\nb_v\mu})(k(t-\tau)^{\rm ap})d\tau\right]^2dt\\
=&{\bf N}_{\mu,1}+{\bf N}_{\mu,2}.
\end{align*}
Since both $\nabla_v\mu$ and $\Delta_v\mu$ are smooth and of fast decay, the treatments of ${\bf N}_{\mu,1}$ and ${\bf N}_{\mu,2}$ are  the same, here we only bound ${\bf N}_{\mu,1}$. To this end, using \eqref{BS-up2}, \eqref{BS-up4} and the fact $t^{\rm ap}-\tau^{\rm ap}=e^{-\nu\tau}(t-\tau)^{\rm ap}\le (t-\tau)^{\rm ap}$, we have
\beno
\la t\ra^b {\bf B}  _k(t)S_k(t-\tau)\les\left(\la\tau\ra^b {\bf B}  _k(\tau) \right) e^{\lm(0)\la k(t-\tau)^{\rm ap}\ra^s}\la k(t-\tau)^{\rm ap} \ra^{\sig_0+2}\fr{S^{\fr12}_k(t-\tau)}{\la k(t-\tau)^{\rm ap} \ra^{2}}.
\eeno
Then in view of \eqref{borrow finite regu}, similar to \eqref{NHL0}, one deduces that
\begin{align}
{\bf N}_{\mu,1}\nn\les&\sup_\eta \left(e^{\lm(0)\la \eta\ra^s}\la \eta\ra^{\sig_0+2}|(\widehat{\Dl_v\mu})(\eta)|\right)^2\\
\nn&\times\nu^2\sum_{k\in\mathbb{Z}^n_*}\int_0^{T^*}\left[\int_0^t\langle \tau\rangle^{b} {\bf B}_k  (\tau)(\widehat{M_\theta})_k(\tau)\fr{S^{\fr12}_k(t-\tau)}{\la k(t-\tau)^{\rm ap} \ra^{2}}d\tau\right]^2dt\\
\nn\les&\nu^2\sum_{k\in\mathbb{Z}^n_*}\int_0^{T^*}\int_0^t\left|\langle \tau\rangle^{b} {\bf B}_k  (\tau)(\widehat{M_\theta})_k(\tau)\right|^2\fr{S^{\fr12}_k(t-\tau)}{\la k(t-\tau)^{\rm ap} \ra^{2}}d\tau \int_0^t \fr{S^{\fr12}_k(t-\tau)}{\la k(t-\tau)^{\rm ap} \ra^{2}}d\tau dt\\
\nn\les&\nu^2\sum_{k\in\mathbb{Z}^n_*}\int_0^{T^*}\left|\langle \tau\rangle^{b} {\bf B}_k  (\tau)(\widehat{M_\theta})_k(\tau)\right|^2\int_\tau^{T^*}\fr{S^{\fr12}_k(t-\tau)}{\la k(t-\tau)^{\rm ap} \ra^{2}} dt d\tau 
\les \nu^2\left\|\la t\ra^b {\bf B}  M_{\theta}\right\|_{L^2_{t,x}}^2.
\end{align}
Thus, we get
\be
{\bf N}_{\mu}\les\nu^2\left(\left\|\la t\ra^b {\bf B}  M_{\theta}\right\|_{L^2_{t,x}}^2+\left\|\la t\ra^b {\bf B}  M_{1}\right\|_{L^2_{t,x}}^2\right),
\ee
which is consistent with Proposition \ref{prop: btsp} provided $\nu$ is chosen sufficiently small.

\subsubsection{${\bf N}_g$ collision contributions} Recalling \eqref{Cg1234}, it is natural to write
\be\label{Jg}
{\bf N}_g={\bf N}_{g,1}+{\bf N}_{g,2}+{\bf N}_{g,3}+{\bf N}_{g,4}.
\ee
For ${\bf N}_{g,1}$, we divide into two cases based on whether $f$ is at zero frequency or not as above:
$
{\bf N}_{g,1}\les {\bf N}_{g,1;\ne}+{\bf N}_{g,1;0}$,
with
\begin{align}
\nn {\bf N}_{g,1;\ne}=&\nu^2\int_0^{T^*}\sum_{k\in\mathbb{Z}^n_*}\Bigg[\la t\ra^b {\bf B}_k  (t)\int_0^tS_k(t-\tau)\sum_{l\in\mathbb{Z}^n_*,l\ne k}\hat\rho_l(\tau)|k(t-\tau)^{\mathrm{ap}}|^2\\
\nn&\times\hat{f}_{k-l}\left(\tau, kt^{\rm ap}-l \tau^{\mathrm{ap}}\right)d\tau\Bigg]^2dt
={\bf N}_{g,1;\ne}^{\mathrm{HL}}+{\bf N}_{g,1;\ne}^{\mathrm{LH}},\\
\nn {\bf N}_{g,1;0}=&\nu^2\int_0^{T^*}\sum_{k\in\mathbb{Z}^n_*}\Bigg[\la t\ra^b {\bf B}_k  (t)\int_0^tS_k(t-\tau)\hat\rho_k(\tau)|k(t-\tau)^{\mathrm{ap}}|^2\\
\nn&\times\hat{f}_{0}\left(\tau, e^{-\nu\tau}k(t-\tau)^{\mathrm{ap}}\right)d\tau\Bigg]^2dt
={\bf N}_{g,1;0}^{\mathrm{HL}}+{\bf N}_{g,1;0}^{\mathrm{LH}}.
\end{align}

 Consider ${\bf N}_{g,1;\ne}^{\mathrm{HL}}$ first. It is worth point out  that, different from ${\bf N}^{\rm HL}_{E;\ne}$, now we can use the inverse Knudsen number $\nu$ to absorb the quadratic time growth $\la \tau\ra^2$ with the aid of enhanced dissipation, namely, $\nu^\fr23 \la \tau\ra^2\les e^{\fr14\dl_1\nu^\fr13\tau}$. Thus, by \eqref{regu-overflow}, similar to \eqref{BS-up1}, \eqref{embd} and \eqref{small-tau1},   we have
\begin{align}\label{BS-up-5}
\nn&\nu \la t\ra^b  {\bf B}_k  (t) S_k(t-\tau) |k(t-\tau)^{\mathrm{ap}}|^2\left|\hat{f}_{k-l}\left(\tau, kt^{\rm ap}-l \tau^{\mathrm{ap}}\right)\right|
{\bf 1}_{\left|k-l,kt^{\rm ap}-l \tau^{\mathrm{ap}}\right|\le2\left|l,l \tau^{\mathrm{ap}}\right|}{\bf 1}_{k\ne l}\\
\nn\les&\nu e^{(m+2)\nu\tau}\la\tau\ra^2S^\fr34_k(t-\tau)\left(\la\tau\ra^b  {\bf B}_l    (\tau)\right)e^{-\underline{c}\la k-l,kt^{\rm ap}-l \tau^{\mathrm{ap}}\ra^s}\\
\nn&\times\left|e^{-m\nu\tau}\mathcal{F}\left[e^{\lm(\tau)\la\nb\ra^s}\la \nb \ra^\fr52 f\right]_{k-l}\left(\tau, kt^{\rm ap}-l \tau^{\mathrm{ap}}\right)\right|\\
\nn\les&\nu^\fr13S^\fr12_k(t-\tau)e^{-\fr12\dl_1\nu^\fr13\tau}e^{-\underline{c}\la k-l,kt^{\rm ap}-l \tau^{\mathrm{ap}}\ra^s}\left(\la\tau\ra^b  {\bf B}_l    (\tau)\right)\\
&\times \sup_\tau\left(e^{-m\nu\tau}\left\|e^{\dl_1\nu^\fr13\tau}\mathcal{F}\left[e^{\lm(\tau)\la\nb \ra^s}\la \nb\ra^\fr52 f_{\ne}\right]_{k}\left(\tau, \eta\right)\right\|_{L^2_kH^{\fr{n}{2}+}_\eta}\right).
\end{align}
Then by Schur's test, similar to \eqref{schur1}, one deduces that
\begin{align}\label{Ng1HLne}
\nn {\bf N}_{g,1;\ne}^{\mathrm{HL}}\nn\les& \nu^{\fr23+2\gamma}\eps^2\left\|\la t\ra^b  {\bf B}  \rho\right\|_{L^2_{t,x}}^2\left(\sup_{t}\sup_{k\ne 0}\sum_{l\in\mathbb{Z}^n_*}\int_0^t\mathcal{K}_1(t, \tau, k, l)d\tau\right)\\
&\times\left(\sup_{\tau}\sup_{l\ne0}\sum_{k\in\mathbb{Z}^n_*}\int_\tau^{T*}\mathcal{K}_1(t, \tau, k, l)dt\right)
 \les\nu^{\fr23+2\gamma}\eps^2\left\|\la t\ra^b  {\bf B}  \rho\right\|_{L^2_{t,x}}^2.
\end{align}

Next we turn to ${\bf N}_{g,1;\ne}^{\mathrm{LH}}$. From \eqref{pr-vt2} and \eqref{regu-overflow}, similar to \eqref{BS-up3}, for $l\ne0$, we find that
\begin{align}\label{BS-up6}
\nn&\la t\ra^b {\bf B}_k  (t)S_k(t-\tau)|k(t-\tau)^{\mathrm{ap}}|^2{\bf 1}_{ \left|l,l \tau^{\mathrm{ap}}\right|<\fr12\left|k-l,kt^{\rm ap}-l \tau^{\mathrm{ap}}\right|}{\bf 1}_{k\ne l}\\
\nn\les&|k|^\fr12 S^\fr12_k(t-\tau)e^{-\fr12\dl_1\nu^\fr25\tau}\left(\la\tau\ra^{b}  \left\la \tau^{\mathrm{ap}}\right\ra e^{\dl \nu^\fr13  \tau}   e^{c\lm(\tau)\left\la l,l \tau^{\mathrm{ap}}\right\ra^s}\right)\\
\nn&\times  (|\widehat{\pr_v^\tau}|\underline{A}^{\sig_0+1,\fr25})_{k-l}\left(\tau, kt^{\rm ap}-l \tau^{\mathrm{ap}}\right){\bf 1}_{k\ne l}\\
\nn\les&|k|^\fr12 S^\fr12_k(t-\tau)e^{-\fr14\dl_1\nu^\fr25\tau}e^{-\underline{c}\left\la l,l \tau^{\mathrm{ap}}\right\ra^s}
\left(e^{\dl \nu^\fr13  \tau}   e^{\lm(\tau)\left\la l,l \tau^{\mathrm{ap}}\right\ra^s}\left\la l,l \tau^{\mathrm{ap}}\right\ra^{b+\fr32+}\right)\\
&\times \fr{1}{\la \tau\ra^{\fr12+}} e^{-m\nu \tau}(|\widehat{\pr_v^\tau}|\underline{A}^{\sig_0+1,\fr25})_{k-l}\left(\tau, kt^{\rm ap}-l \tau^{\mathrm{ap}}\right){\bf 1}_{k\ne l}.
\end{align}
Then similar \eqref{NLH1}--\eqref{NLH2}, using the trace theorem and \eqref{emb'},  we are led to
\begin{align}\label{JLH1}
\nn {\bf N}_{g,1;\ne}^{\mathrm{LH}}\les&\nu^2\int_0^{T^*}\sum_{k\in\mathbb{Z}^n_*}|k|
\Bigg[\sum_{l\in\mathbb{Z}^n_*}\int_0^tS^\fr12_k(t-\tau)  e^{-\fr14\dl_1\nu^\fr25\tau} e^{-\underline{c}\left\la l,l \tau^{\mathrm{ap}}\right\ra^s}\\
\nn&\times e^{\dl \nu^\fr13 \tau} e^{\lm(\tau)\left\la l,l \tau^{\mathrm{ap}}\right\ra^s}\left\la l, l \tau^{\mathrm{ap}}\right\ra^{b+\fr32+} |\hat{\rho}_l(\tau)|\\
\nn&\times \fr{1}{\la\tau\ra^{\fr12+}}\left|e^{-m\nu \tau}\mathcal{F}\left[ \pr_v^\tau \underline{A}^{\sig_0+1,\fr25}  f_{\ne}\right]_{k-l}\left(\tau, kt^{\rm ap}-l \tau^{\mathrm{ap}}\right)\right| d\tau\Bigg]^2dt\\
\nn\les&\nu^2\sup_\tau\sup_{l\in\mathbb{Z}^n_*}\left(e^{\dl \nu^\fr13 \tau} e^{\lm(\tau)\left\la l,l \tau^{\mathrm{ap}}\right\ra^s}\left\la l, l \tau^{\mathrm{ap}}\right\ra^{b+\fr32+} |\hat{\rho}_l(\tau)|\right)^2\\
\nn&\times \sum_{l\in\Z^n_*}\int_0^{T^*}\fr{1}{\la\tau\ra^{1+}}\sum_{k\in\Z^n_*}|k|\int_\tau^{T^*} S_k(t-\tau)e^{-\fr14\dl_1\nu^\fr25\tau}e^{-\underline{c}\la l, l \tau^{\mathrm{ap}}\ra^s}\\
\nn&\times \left|e^{-m\nu \tau}\mathcal{F}\left[ \pr_v^\tau \underline{A}^{\sig_0+1,\fr25} f_{\ne}\right]_{k-l}\left(\tau, kt^{\rm ap}-l \tau^{\mathrm{ap}}\right)\right|^2dt d\tau\\
\nn\les&\sup_t \left\|B^{ \sig_1 }\rho(t) \right\|_{L^2_x}^2 \nu^2\sum_{k\in\mathbb{Z}^n_*}\int_0^{T^*} \fr{1}{\la \tau\ra^{1+}}   \left(e^{-m\nu\tau}\left\|\mathcal{F}\left[ \pr_v^\tau \underline{A}^{\sig_0+1,\fr25} f_{\ne}\right]_{k}(\tau)\right\|_{H^{\fr{n-1}{2}+}_\eta}\right)^2 d\tau\\
\les&\eps^2\nu^{2\gamma+1}\left(\sup_t\left\| f^w(t)\right\|_{\mathcal{E}_{m}^{\sigma_0+1,\fr25}}^2+\nu\int_0^{T^*}\left\| f^w(t)\right\|_{\mathcal{D}_{m}^{\sigma_0+1,\fr25}}^2dt\right).
\end{align}

The treatment of ${\bf N}_{g,1;0}^{\mathrm{HL}}$ is similar to \eqref{NHL0}, we just state the result:
\be\label{JLH0}
{\bf N}_{g,1;0}^{\mathrm{HL}}\les\nu^2\sup_t \left(\left\|e^{\lm(t)\left\la \eta\right\ra^s}\la e^{\nu t}\eta\ra^4\hat{f}_0\left(t, \eta\right)\right\|_{L^{\infty}_{\eta}}\right)^2\left\|\la t \ra^b {\bf B}  \rho\right\|_{L^2_{t,x}}^2.
\ee
As for the  treatment of ${\bf N}_{g,1;0}^{\mathrm{LH}}$, we will use  the restriction \eqref{nutau-up} again. In fact, using \eqref{nutau-up} and \eqref{S-prop2}, similar to \eqref{BS-up4}, one easily deduces that
\begin{align}\label{BS-up5}
\nn&\nu\la t\ra^b  {\bf B}_k  (t)S_k(t-\tau)|k(t-\tau)^{\mathrm{ap}}|^2{\bf 1}_{\left|k,k\tau^{\mathrm{ap}}\right|<\fr12\left|k(t^{\rm ap}-\tau^{\mathrm{ap}})\right|}\\
\les&\nu^\fr23 |k|^\fr12 S_k^\fr12(t-\tau) \left(\la \tau\ra^b e^{\dl \nu^\fr13 \tau}   e^{c\lm(\tau)\left\la k, k\tau^{\mathrm{ap}}\right\ra^s}|k|\right) A_0^{\sig_0+1,\fr25}\left(\tau, k(t^{\rm ap}-\tau^{\mathrm{ap}}) \right).
\end{align}
Then ${\bf N}_{g,1;0}^{\mathrm{LH}}$ can be bounded in the same manner as  \eqref{NLH0}:
\begin{align}\label{JLH0}
\nn {\bf N}_{g,1;0}^{\mathrm{LH}}\les&\nu^\fr43\int_0^{T^*}\sum_{k\in\mathbb{Z}^n_*}|k|\Bigg[\int_0^t\la \tau\ra^b| {\bf B}  \hat{\rho}_k(\tau)|e^{-\underline{c}\la k,k\tau^{\mathrm{ap}}\ra^s} \\
\nn&\times S^\fr12_k(t-\tau) e^{-m\nu \tau}(\underline{A}^{\sig_0+1,\fr25}\hat{f})_0\left(\tau, k(t^{\rm ap}-\tau^{\mathrm{ap}})\right)d\tau\Bigg]^2dt\\
\nn\les&\nu^\fr43\left\|\la t\ra^b  {\bf B}  \rho\right\|^2_{L^2_{t,x}}\sup_{0\le\tau\le T^*}\left(e^{-m\nu \tau}\left\|(\underline{A}^{\sig_0+1,\fr25}\hat{f})_0(\tau)\right\|_{H^{\fr{n-1}{2}+}_\eta}\right)^2\\
\les&\nu^\fr43\left\|\la t\ra^b  {\bf B}  \rho\right\|^2_{L^2_{t,x}}\sup_{0\le t\le T^*}\|f^w(t)\|_{\mathcal{E}^{\sig_0+1,\fr25}_m}^2.
\end{align}

Now we turn to ${\bf N}_{g,2}$. Indeed, ${\bf N}_{g,2}$ can be bounded by combining the treatments of ${\bf N}_{E}$ and ${\bf N}_{g,1}$. More precisely, now the analogue of  \eqref{BS-up-5} reads
\begin{align}\label{BS-up5'}
\nn&\nu \la t\ra^b  {\bf B}_k  (t) S_k(t-\tau) |k(t-\tau)^{\mathrm{ap}}|\left|e^{-\nu\tau}(\nb_\eta\hat{f})_{k-l}\left(\tau, kt^{\rm ap}-l \tau^{\mathrm{ap}}\right)\right|\\
\nn&\times{\bf 1}_{\left|k-l,kt^{\rm ap}-l \tau^{\mathrm{ap}}\right|\le2\left|l,l \tau^{\mathrm{ap}}\right|}{\bf 1}_{k\ne l}\\
\nn\les&\nu^\fr23S^\fr12_k(t-\tau)e^{-\fr12\dl_1\nu^\fr13\tau}e^{-\underline{c}\la k-l,kt^{\rm ap}-l \tau^{\mathrm{ap}}\ra^s}\left(\la\tau\ra^b  {\bf B}_l    (\tau)\right)\\
\nn&\times \sup_\tau\left(e^{-(m+1)\nu\tau}\left\|e^{\dl_1\nu^\fr13\tau}\mathcal{F}\left[e^{\lm(\tau)\la\nb \ra^s}\la \nb\ra^\fr32 v f_{\ne}\right]_{k}\left(\tau, \eta\right)\right\|_{L^2_kH^{\fr{n}{2}+}_\eta}\right)\\
\nn\les&\nu^\fr23S^\fr12_k(t-\tau)e^{-\fr12\dl_1\nu^\fr13\tau}e^{-\underline{c}\la k-l,kt^{\rm ap}-l \tau^{\mathrm{ap}}\ra^s}\left(\la\tau\ra^b  {\bf B}_l    (\tau)\right)\sup_t\left\|f^w(t)\right\|_{\mathcal{E}^{ \sig_1 ,\fr13}_m}\\
\les&\nu^\fr23S^\fr12_k(t-\tau)e^{-\fr12\dl_1\nu^\fr13\tau}e^{-\underline{c}\la k-l,kt^{\rm ap}-l \tau^{\mathrm{ap}}\ra^s}\left(\la\tau\ra^b  {\bf B}_l    (\tau)\right)\eps\nu^{\gamma}.
\end{align}
Then similar to \eqref{Ng1HLne}, we have
\begin{align}\label{Ng2HLne}
 {\bf N}_{g,2;\ne}^{\mathrm{HL}}
 \les\nu^{\fr43+2\gamma}\eps^2\left\|\la t\ra^b  {\bf B}  \rho\right\|_{L^2_{t,x}}^2.
\end{align}
For  ${\bf N}_{g,2;\ne}^{\mathrm{LH}}$, instead of using the strategy of treating ${\bf N}_{g,1;\ne}^{\mathrm{LH}}$, we just follow the way to treat ${\bf N}_{E;\ne}^{\mathrm{LH}}$, see \eqref{BS-up3}--\eqref{NLH2}. Thus,
\begin{align}
{\bf N}_{g,2;\ne}^{\mathrm{LH}}\nn\les&\nu^2\left\|\la t\ra^b  {\bf B}  \rho\right\|^2_{L^2_{t,x}}\sup_{0\le\tau\le T^*}\left(e^{-(m+1)\nu \tau}\left\|(\underline{A}^{\sig_0+1,\fr25}\nb_\eta\hat{f})_k(\tau)\right\|_{L^2_kH^{\fr{n-1}{2}+}_\eta}\right)^2\\
 \nn \les&\nu^2\left\|\la t\ra^b  {\bf B}  \rho\right\|^2_{L^2_{t,x}}\sup_{0\le t\le T^*}\|f^w(t)\|_{\mathcal{E}^{\sig_0+1,\fr25}_m}^2
  \les \eps^2\nu^{2\gamma+2}\left\|\la t\ra^b  {\bf B}  \rho\right\|^2_{L^2_{t,x}}.
\end{align}
The rest part ${\bf N}_{g,2;0}$ can be treated in  the same way as ${\bf N}_{E;0}$. In fact, similar to \eqref{NHL0} and \eqref{NLH0}, we have
\begin{align*}
{\bf N}_{g,2;0}^{\rm HL}\les&\nu^2\sup_t \left(e^{-\nu t} \|e^{\lm(t)\la\eta\ra^s}\la e^{\nu t}\eta\ra^3(\nb_\eta\hat{f})_0(t,\eta)\|_{L^\infty_\eta}\right)^2\left\|\la t\ra^b  {\bf B}  \rho\right\|^2_{L^2_{t,x}},\\
{\bf N}_{g,2;0}^{\rm LH}\nn\les&\nu^2\left\|\la t\ra^b  {\bf B}  \rho\right\|_{L^2_{t,x}}^2
  \sup_\tau\sup_{|\om|=1}\sup_{\zeta\in\mathbb{R}^n}\int^{\infty}_{-\infty}\left| e^{-(m+1)\nu\tau}(\underline{A}^{\sig_0+1,\fr25}\nb_\eta\hat{f})_0\left(\tau, \om t'+\zeta\right)\right|^2 dt'\\
\les& \eps^2\nu^{2\gamma+2}\left\|\la t\ra^b  {\bf B}  \rho\right\|^2_{L^2_{t,x}}.
\end{align*}

Next we consider ${\bf N}_{g,3}$.
\begin{align}
{\bf N}_{g,3}\nn\les&\int_0^{T^*}\sum_{k\in\mathbb{Z}^n_*}\Bigg[\la t\ra^b {\bf B}_k  (t)\int_0^tS_k(t-\tau)\nu\sum_{l\in\mathbb{Z}^n_*,l\ne k}(\widehat{M_\theta})_l(\tau)|k(t-\tau)^{\mathrm{ap}}|^2
\hat{f}_{k-l}\left(\tau, kt^{\rm ap}-l \tau^{\mathrm{ap}}\right)d\tau\Bigg]^2dt\\
\nn&+\int_0^{T^*}\sum_{k\in\mathbb{Z}^n_*}\Bigg[\la t\ra^b {\bf B}_k  (t)\int_0^tS_k(t-\tau)\nu(\widehat{M_\theta})_k(\tau)|k(t-\tau)^{\mathrm{ap}}|^2
\hat{f}_{0}\left(\tau, e^{-\nu\tau}k(t-\tau)^{\mathrm{ap}}\right)d\tau\Bigg]^2dt\\
\nn&+\int_0^{T^*}\sum_{k\in\mathbb{Z}^n_*}\Bigg[\la t\ra^b {\bf B}_k  (t)\int_0^tS_k(t-\tau)\nu(\widehat{M_\theta})_0(\tau)|k(t-\tau)^{\mathrm{ap}}|^2
\hat{f}_{k}\left(\tau, kt^{\rm ap}\right)d\tau\Bigg]^2dt\\
\nn=&{\bf N}_{g,3;\ne}+{\bf N}_{g,3;0}+{\bf N}_{g,3;0\ne}.
\end{align}
Clearly, ${\bf N}_{g,3;\ne}$ and  ${\bf N}_{g,3;0}$ can be treated in the same way as ${\bf N}_{g,1;\ne}$ and ${\bf N}_{g,1;0}$, respectively, with $\rho$ replaced by $(M_\theta)_{\ne}$. Here we only  estimate ${\bf N}_{g,3;0\ne}$.  If fact, now \eqref{BS-up6} reduces to
\begin{align}\label{BS-up7}
\nn&\la t\ra^b {\bf B}_k  (t)S_k(t-\tau)|k(t-\tau)^{\mathrm{ap}}|^2\\
\les&|k|^\fr12 S^\fr12_k(t-\tau)e^{-\fr14\dl_1\nu^\fr25\tau}\left(e^{\dl \nu^\fr13  \tau}   \left\la  \tau^{\mathrm{ap}}\right\ra^{b+\fr32+} \right)
\fr{e^{-m\nu \tau}}{\la \tau\ra^{\fr12+}} (|\widehat{\pr_v^\tau}|\underline{A}^{\sig_0+1,\fr25})_{k}\left(\tau, kt^{\rm ap}\right){\bf 1}_{k\ne 0}.
\end{align}
Then similar to \eqref{JLH1}, using \eqref{emb'} and \eqref{e-Mtheta0}, together with 
\begin{align*}
&\int_0^t e^{-\fr14\dl_1\nu^\fr25\tau} \left(e^{\dl \nu^\fr13  \tau}   \left\la  \tau^{\mathrm{ap}}\right\ra^{b+\fr32+} \left|(\widehat{M_\theta})_0(\tau)\right|\right)d\tau
\les \nu^{-\fr25} \sup_\tau\left(e^{\dl \nu^\fr13  \tau}   \left\la  \tau^{\mathrm{ap}}\right\ra^{b+\fr32+} \left|(\widehat{M_\theta})_0(\tau)\right|\right).
\end{align*}
we have
\begin{align}\label{JLH3}
\nn {\bf N}_{g,3;0\ne}\les&\nu^2\int_0^{T^*}\sum_{k\in\mathbb{Z}^n_*}|k|
\Bigg[\int_0^t   e^{-\fr14\dl_1\nu^\fr25\tau} \left(e^{\dl \nu^\fr13  \tau}   \left\la  \tau^{\mathrm{ap}}\right\ra^{b+\fr32+} \left|(\widehat{M_\theta})_0(\tau)\right|\right)\\
\nn&\times  S^\fr12_k(t-\tau) \fr{1}{\la \tau\ra^{\fr12+}} e^{-m\nu \tau} \left|\mathcal{F}\left[\pr_v^\tau \underline{A}^{\sig_0+1,\fr25} f_{\ne}\right]_{k}\left(\tau, kt^{\rm ap}\right)\right| d\tau\Bigg]^2dt\\
\nn\les&\nu^2\int_0^{T^*}\sum_{k\in\mathbb{Z}^n_*}|k| 
\int_0^t e^{-\fr14\dl_1\nu^\fr25\tau} \left(e^{\dl \nu^\fr13  \tau}   \left\la  \tau^{\mathrm{ap}}\right\ra^{b+\fr32+} \left|(\widehat{M_\theta})_0(\tau)\right|\right)d\tau\\
\nn&\times \int_0^t  e^{-\fr14\dl_1\nu^\fr25\tau} \left(e^{\dl \nu^\fr13  \tau}   \left\la  \tau^{\mathrm{ap}}\right\ra^{b+\fr32+} \left|(\widehat{M_\theta})_0(\tau)\right|\right)S_k(t-\tau)\\
\nn&\times \fr{1}{\la \tau\ra^{1+}}  \left|e^{-m\nu \tau}\mathcal{F}\left[\pr_v^\tau \underline{A}^{\sig_0+1,\fr25} f_{\ne}\right]_{k}\left(\tau, kt^{\rm ap}\right)\right|^2d\tau dt\\
\nn\les&\nu^{\fr85}\sup_\tau\left(e^{\dl \nu^\fr13  \tau}   \left\la  \tau^{\mathrm{ap}}\right\ra^{b+\fr32+} \left|(\widehat{M_\theta})_0(\tau)\right|\right)^2\sum_{k\in\mathbb{Z}^n_*}\int_0^{T^*} \fr{1}{\la \tau\ra^{1+}}\\
\nn&\times \int_\tau^{T^*}|k|e^{-\nu t}   \left|e^{-m\nu \tau}\mathcal{F}\left[\pr_v^\tau \underline{A}^{\sig_0+1,\fr25} f_{\ne}\right]_{k}\left(\tau, kt^{\rm ap}\right)\right|^2dt d\tau\\
\nn\les&\nu^{\fr35}\sup_t\left(e^{\dl \nu^\fr13  t}   \left\la  t^{\mathrm{ap}}\right\ra^{b+\fr32+} \left|(\widehat{M_\theta})_0(t)\right|\right)^2\\
\nn&\times\nu\sum_{k\in\mathbb{Z}^n_*}\int_0^{T^*} \fr{1}{\la \tau\ra^{1+}}   \left(e^{-m\nu\tau}\left\|\mathcal{F}\left[ \pr_v^\tau \underline{A}^{\sig_0+1,\fr25}f_{\ne}\right]_{k}(\tau)\right\|_{H^{\fr{n-1}{2}+}_\eta}\right)^2 d\tau,\\
\les& \eps^4\nu^{4\gamma+\fr35}\left(\sup_t\left\| f^w(t)\right\|_{\mathcal{E}_{m}^{\sigma_0+1,\fr25}}^2+\nu\int_0^{T^*}\left\| f^w(t)\right\|_{\mathcal{D}_{m}^{\sigma_0+1,\fr25}}^2dt\right).
\end{align}

Noting that $(M_1)_0=0$, the last term ${\bf N}_{g,4}$ can be treated in the same way as  ${\bf N}_{g,2}$. We thus omit the details to avoid repetition.

\subsection{ $L^\infty_t$ control on $\rho$}\label{sec-rho-low}
The lower norm estimate on $\rho$ is much more straightforward than the higher norm estimate in Section \ref{sec-L2rho}. From \eqref{ex-rho},  we have
\beno
\left\|B^{ \sig_1 }\rho(t)\right\|_{L^2_x}^2\les {\bf I}+{\bf L}+{\bf N}_E+{\bf N}_{\mu}+{\bf N}_{g},
\eeno
where
\begin{align*}
{\bf I}=&\sum_{k\in\Z^n_*}\left| B^{ \sig_1 }_k(t)S_k(t)\widehat{g_{\rm in}}(k,kt^{\rm ap})\right|^2,\\
{\bf L}=&\sum_{k\in\Z^n_*}\left[B^{ \sig_1 }_k(t)\int_0^tS_k(t-\tau)\hat{\rho}_k(\tau)(t-\tau)^{\rm ap} \hat{\mu}(k(t-\tau)^{\rm ap})d\tau\right]^2,\\
{\bf N}_{E}=&\sum_{k\in\mathbb{Z}^n_*}\Bigg[B_k^{ \sig_1 }(t)\int_0^tS_k(t-\tau)\sum_{l\in\mathbb{Z}^n_*}\hat{\rho}_l(\tau)\fr{l}{|l|^2}\cdot k(t-\tau)^{\mathrm{ap}}\hat{f}_{k-l}\left(\tau, kt^{\rm ap}-l \tau^{\mathrm{ap}}\right)d\tau\Bigg]^2,\\
{\bf N}_{\mu}=&\nu^2\sum_{k\in\mathbb{Z}^n_*}\left[B_k^{ \sig_1 }(t)\int_0^tS_k(t-\tau)(\widehat{\mathcal{C}_{\mu}})_k(\tau, k(t-\tau)^{\mathrm{ap}})d\tau\right]^2,\\
{\bf N}_{g}\nn=&\nu^2\sum_{k\in\mathbb{Z}^n_*}\left[B_k^{ \sig_1 }(t)\int_0^tS_k(t-\tau)(\widehat{\mathcal{C}[g]})_k(\tau, k(t-\tau)^{\mathrm{ap}})d\tau\right]^2.
\end{align*}

Clearly, by \eqref{S-prop2} and Sobolev embedding, there holds
\beno
{\bf I}\les \sum_{k\in\Z^n_*}\left|e^{\lm(0)\la k, kt^{\rm ap}\ra^s}\la k,kt^{\rm ap} \ra^{ \sig_1 } \widehat{g_{\rm in}}(k,kt^{\rm ap})\right|^2\les\left\| e^{\lm(0)\la\nb\ra^s}\la \nb\ra^{ \sig_1 }\left(\la v\ra^mg_{\rm in}\right)\right\|_{L^2}^2\les \eps^2\nu^{2\gamma}.
\eeno
To bound ${\bf L}$, using again \eqref{S-prop2}, for $k\ne0$, we have
\begin{align}
B^{ \sig_1 }_k(t)S_k(t-\tau)\nn\les& e^{\dl\nu^\fr13\tau} e^{\lm(\tau)\la k,k\tau^{\rm ap}\ra^s}\la k,k\tau^{\rm ap}\ra^{ \sig_1 } e^{c\lm(\tau)\la k(t^{\rm ap}-\tau^{\rm ap})\ra^s}S^{\fr12}_k(t-\tau)\\
\nn&+e^{\dl\nu^\fr13\tau} e^{c\lm(\tau)\la k,k\tau^{\rm ap}\ra^s}\la k(t^{\rm ap}-\tau^{\rm ap})\ra^{ \sig_1 } e^{\lm(\tau)\la k(t^{\rm ap}-\tau^{\rm ap})\ra^s}S^{\fr12}_k(t-\tau)\\
\nn\les& {\bf B}  _k(\tau) e^{\lm(0)\la k(t-\tau)^{\rm ap}\ra^s}\la k(t-\tau)^{\rm ap}\ra^{ \sig_1 }S^{\fr12}_k(t-\tau).
\end{align}
Consequently, in view of \eqref{borrow finite regu}, one deduces that
\begin{align}\label{e-L-low}
{\bf L}\nn\les&\sum_{k\in\Z^n_*}\left[ \int_0^t | {\bf B}  \hat{\rho}_k(\tau)|\fr{S_k^{\fr12}(t-\tau)}{\la k(t-\tau)^{\rm ap}\ra}e^{\lm(0)\la k(t-\tau)^{\rm ap}\ra^s}\la k(t-\tau)^{\rm ap}\ra^{\sig_1+2}|\hat{\mu}(k(t-\tau)^{\rm ap})|d\tau\right]^2\\
\nn\les&\sum_{k\in\Z^n_*}\left[ \int_0^t | {\bf B}  \hat{\rho}_k(\tau)|\fr{S_k^{\fr12}(t-\tau)}{\la k(t-\tau)^{\rm ap}\ra}d\tau\right]^2\sup_{\eta} \left(e^{\lm(0)\la \eta\ra^s} \la \eta\ra^{\sig_1+2}|\hat{\mu}(\eta)|\right)^2\\
\nn\les&\sum_{k\in\Z^n_*} \int_0^t | {\bf B}  \hat{\rho}_k(\tau)|^2d\tau \sup_{k\in\Z^n_*} \int_0^t\fr{S_k(t-\tau)}{\la k(t-\tau)^{\rm ap}\ra^2}d\tau\\
\les&\| {\bf B}  \rho(t)\|_{L^2_{t,x}}^2\les \left({\rm C}_0\eps\nu^{\gamma}\right)^2,
\end{align}
which is consistent with Proposition \ref{prop: btsp} as long as ${\rm \tl C}_0\gg {\rm C}_0$.

Next we turn to ${\bf N}_{E}$. As in Section \ref{sec-L2rho}, we split it into two parts based on $f$ is at zero frequency or not. Compared with \eqref{NHL4} and \eqref{NLH2}, the treatment of ${\bf N}_{E;\ne}$ now is much  easier since the plasma echoes will never happen even when $\rho$ is at high frequency. In fact, now one can see from  \eqref{S-prop2}, \eqref{pou-rho} and \eqref{regu-overflow} that
\begin{align}\label{BS-up8}
\nn&B^{\sig_1}_k(t)S_k(t-\tau)|k(t-\tau)^{\rm ap}|{\bf1}_{k\ne l}\\
\nn\les&B_l^{\sig_1+1}(\tau)\left(e^{-m\nu \tau}e^{\dl_1\nu^\fr25\tau}e^{\lm(\tau)\la k-l,kt^{\rm ap}-l\tau^{\rm ap}\ra^s}\la k-l,kt^{\rm ap}-l^{\rm ap}\ra^{\sig_1+1}\right)\\
\nn&\times e^{(m+1)\nu \tau}e^{-\dl_1\nu^\fr25\tau}{\bf1}_{k\ne l}\\
\les&B_l^{\sig_1+1}(\tau)\left(e^{-m\nu \tau}\underline{A}^{\sig_0+1,\fr25}_{k-l}(\tau, kt^{\rm ap}-l\tau^{\rm ap})\right) e^{-\fr12\dl_1\nu^\fr25\tau}{\bf1}_{k\ne l}.
\end{align}
Moreover, note that
\begin{align}\label{reg-gap}
\sum_{l\in\Z^n_*}\int_0^t B^{\sig_1+1}_l(\tau)|\rho_l(\tau)|d\tau\nn\les&\sum_{l\in\Z^n_*}\fr{1}{\la l\ra^{\fr{n}{2}+}}\int_0^t\la\tau\ra^{\fr12+}\la l\ra^{\fr{n}{2}+} B^{\sig_1+1}_l(\tau)|\rho_l(\tau)| \fr{1}{\la\tau\ra^{\fr12+}}d\tau\\
\les& \left\|\la t\ra^{\fr12+} {\bf B}  \rho \right\|_{L^2_{t,x}},
\end{align}
provided $ \sig_1+1+\fr{n+1}{2}\le \sig_0$. Then it follows from these two inequalities and \eqref{emb} that
\begin{align}\label{NEne-low}
{\bf N}_{E;\ne}
\nn\les&\sum_{k\in\mathbb{Z}^n_*}\Bigg[\int_0^t\sum_{l\in\mathbb{Z}^n_*,l\ne k} |B^{\sig_1+1}\hat{\rho}_l(\tau)| e^{-m\nu\tau} \left|(\underline{A}^{\sig_0+1,\fr25} \hat{f})_{k-l}\left(\tau, kt^{\rm ap}-l \tau^{\mathrm{ap}}\right)\right|d\tau\Bigg]^2\\
\nn\les&\sum_{l\in\Z^n_*}\int_0^t |B^{\sig_1+1}\hat{\rho}_l(\tau)|d\tau\int_0^t \sum_{k\in\Z^n_*}\sum_{l\in\Z^n_*, l\ne k} |B^{\sig_1+1}\hat{\rho}_l(\tau)|\\
\nn&\times \left(e^{-m\nu\tau}\left|(\underline{A}^{\sig_0+1,\fr25} \hat{f})_{k-l}\left(\tau, kt^{\rm ap}-l \tau^{\mathrm{ap}}\right)\right| \right)^2d\tau\\
\les&\left\|\la t\ra^{\fr12+} {\bf B}  \rho \right\|_{L^2_{t,x}}^2\sup_t\left(e^{-m\nu t}\left\|\underline{A}^{\sig_0+1,\fr25} \hat{f}_{\ne} \right\|_{L^2_k H^{\fr{n}{2}+}_\eta}\right)^2
\les\eps^2\nu^{2\gamma}\left\|\la t\ra^{\fr12+} {\bf B}  \rho \right\|_{L^2_{t,x}}^2.
\end{align}
The rest part  ${\bf N}_{E;0}$ can be bounded in a similar manner to the corresponding term in high norm estimate. Indeed, compared with \eqref{BS-up2} and \eqref{BS-up4}, now it is easier to get 
\begin{align}\label{BS-up9}
\nn&B^{\sig_1}_k(t)S_k(t-\tau)|k(t-\tau)^{\rm ap}|{\bf 1}_{|k(t^{\rm ap}-\tau^{\rm ap})|\le2|k,k\tau^{\rm ap}|}\\
\les& B^{\sig_1}_k(\tau)e^{c\lm(\tau)\la k(t^{\rm ap}-\tau^{\rm ap})\ra^s} |e^{\nu\tau}k(t^{\rm ap}-\tau^{\rm ap})|,
\end{align}
and
\begin{align}\label{BS-up10}
\nn&B^{\sig_1}_k(t)S_k(t-\tau)|k(t-\tau)^{\rm ap}|{\bf 1}_{|k,k\tau^{\rm ap}|<\fr12|k(t^{\rm ap}-\tau^{\rm ap})|}\\
\les&e^{-\nu\tau}e^{-\underline{c}\la k,k\tau^{\rm ap}\ra^s} e^{\dl\nu^\fr13\tau}e^{\lm(\tau)\la k,k\tau^{\rm ap}\ra^s}\left(e^{-m\nu t}\underline{A}_0^{\sig_0+1,\fr25}(\tau, k(t^{\rm ap}-\tau^{\rm ap})) \right).
\end{align}
Then by \eqref{reg-gap}  and \eqref{BS-up9},  one deduces that
\begin{align}\label{NE0-low}
{\bf N}^{\rm HL}_{E;0}\nn\les&\sup_\tau\left\|e^{\lm(\tau)\left\la \eta\right\ra^s}| e^{\nu\tau}\eta|\hat{f}_0\left(\tau, \eta\right)\right\|_{L^\infty_{\eta}}^2\sum_{k\in\Z^n_*}\left[\int_0^t B^{\sig_1}_k(\tau)|\hat{\rho}_k(\tau)|d\tau \right]^2\\
\les&\eps^2\nu^{2\gamma}\left[\sum_{k\in\Z^n_*}\int_0^t B^{\sig_1}_k(\tau)|\hat{\rho}_k(\tau)|d\tau \right]^2\les \eps^2\nu^{2\gamma}\left\|\la t\ra^{\fr12+} {\bf B}  \rho \right\|_{L^2_{t,x}}^2.
\end{align}
From \eqref{les-rho-low1} and \eqref{BS-up10}, we are led to
\begin{align}\label{e-NCL-L0}
\nn {\bf N}_{E;0}^{\rm LH}
\nn&\sum_{k\in\mathbb{Z}^n_*}\Bigg[\int_0^t e^{-\nu\tau}e^{-\underline{c}\la k,k\tau^{\rm ap}\ra^s} \left|e^{\dl\nu^\fr13\tau}e^{\lm(\tau)\la k,k\tau^{\rm ap}\ra^s}\hat{\rho}_k(\tau)\right|\\
\nn&\times  \left|e^{-m\nu t}\underline{A}_0^{\sig_0+1,\fr25}\hat{f}_{0}\left(\tau, k(t^{\rm ap}-\tau^{\mathrm{ap}})\right)\right| d\tau\Bigg]^2\\
\nn\les&\sup_t \left(e^{-m\nu t}\left\|\underline{A}_0^{\sig_0+1,\fr25}\hat{f}_0(t,\eta) \right\|_{L^\infty_\eta}\right)^2\sum_{k\in\Z^n_*}\left[\int_0^t e^{-\nu\tau}e^{-\underline{c}\la k,k\tau^{\rm ap}\ra^s}|{\bf B}\hat{\rho}_k(\tau)|d\tau \right]^2\\
\les&\eps^2\nu^{2\gamma}\left\|{\bf B}\rho \right\|_{L^2_{t,x}}^2.
\end{align}
The term ${\bf N}_{\mu}$ can be treated in the same way as ${\bf L}$. In fact, similar to \eqref{e-L-low}, we have
\begin{align}
{\bf N}_{\mu}\nn\les& \sup_\eta \left(e^{\lm(0)\la \eta\ra^s}\la \eta\ra^{\sig_1+1}\left(|(\widehat{\Dl_v\mu})(\eta)|+|(\widehat{\nb_v\mu})(\eta)|\right)\right)^2\\
\nn&\times\nu^2\sum_{k\in\Z^n_*}\left(\int_0^tB^{\sig_1}_k(\tau)(|(\widehat{M_\theta})_k(\tau)|+|(\widehat{M_1})_k(\tau)|)\fr{S^{\fr12}_k(t-\tau)}{\la k(t-\tau)^{\rm ap} \ra} d\tau\right)^2\\
\les&\nu^2\left( \left\| {\bf B}  M_\theta\right\|_{L^2_{t,x}}^2+\left\| {\bf B}  M_1\right\|_{L^2_{t,x}}^2\right).
\end{align}

Finally, we turn to ${\bf N}_{g}$, which consists of four parts as before. Here we only treat the most tricky one ${\bf N}_{g,3}$:
\begin{align}
{\bf N}_{g,3}\nn\les&\nu^2\sum_{k\in\mathbb{Z}^n_*}\left[B_k^{\sig_1}(t)\int_0^tS_k(t-\tau)\sum_{l\in\Z^n_*,l\ne k}(\widehat{M_\theta})_l(\tau)|k(t-\tau)^{\rm ap}|^2\hat{f}_{k-l}(\tau, kt^{\rm ap}-l\tau^{\mathrm{ap}})d\tau\right]^2\\
\nn&+\nu^2\sum_{k\in\mathbb{Z}^n_*}\left[B_k^{\sig_1}(t)\int_0^tS_k(t-\tau)(\widehat{M_\theta})_k(\tau)|k(t-\tau)^{\rm ap}|^2\hat{f}_{0}(\tau, k(t^{\rm ap}-\tau^{\mathrm{ap}}))d\tau\right]^2\\
\nn&+\nu^2\sum_{k\in\mathbb{Z}^n_*}\left[B_k^{\sig_1}(t)\int_0^tS_k(t-\tau)(\widehat{M_\theta})_0(\tau)|k(t-\tau)^{\rm ap}|^2\hat{f}_{k}(\tau, kt^{\mathrm{ap}})d\tau\right]^2\\
\nn=&{\bf N}_{g,3;\ne}+{\bf N}_{g,3;0}+{\bf N}_{g,3;0\ne}.
\end{align}
For ${\bf N}_{g,3;\ne}$, clearly, the analogue of \eqref{BS-up8} now reads
\begin{align}
B^{\sig_1}_k(t)S_k(t-\tau)|k(t-\tau)^{\rm ap}|^2{\bf1}_{k\ne l}
\nn\les B_l^{\sig_1+2}\left(e^{-m\nu \tau}\underline{A}^{\sig_0+1,\fr25}_{k-l}(\tau, kt^{\rm ap}-l\tau^{\rm ap})\right) e^{-\fr12\dl_1\nu^\fr25\tau}{\bf1}_{k\ne l}.
\end{align}
In addition, \eqref{reg-gap} still holds with $\rho$ replaced by $M_\theta$ and $\sig_1+1$ replaced by $\sig_1+2$, respectively, as long as 
$\sig_1+2+\fr{n+1}{2}\le\sig_0$. Then similar to \eqref{NEne-low}, we arrive at
\begin{align}
{\bf N}_{g,3;\ne}\nn\les&\eps^2\nu^{2+2\gamma}\left\|\la t\ra^{\fr12+}{\bf B}M_\theta \right\|_{L^2_{t,x}}^2.
\end{align}
As for ${\bf N}_{g,3;0}$, we just need to replace $|k(t-\tau)^{\rm ap}|$ by $|k(t-\tau)^{\rm ap}|^2$ in \eqref{BS-up9} and \eqref{BS-up10}. Then following \eqref{NE0-low} and \eqref{e-NCL-L0}, one easily deduces that
\begin{align}
{\bf N}_{g,3;0}\nn\les&\nu^2\sup_\tau\left\|e^{\lm(\tau)\left\la \eta\right\ra^s}| e^{\nu\tau}\eta|^2\hat{f}_0\left(\tau, \eta\right)\right\|_{L^\infty_{\eta}}^2\left\|\la t\ra^{\fr12+}{\bf B}\rho \right\|_{L^2_{t,x}}^2\\
\nn&+\nu^2\sup_t \left(e^{-m\nu t}\left\|\underline{A}_0^{\sig_0+1,\fr25}\hat{f}_0(t,\eta) \right\|_{L^\infty_\eta}\right)^2\left\|{\bf B}\rho \right\|_{L^2_{t,x}}^2\\
\les&\eps^2\nu^{2+2\gamma}\left\|\la t\ra^{\fr12+}{\bf B}\rho \right\|_{L^2_{t,x}}^2.
\end{align}
To bound ${\bf N}_{g,3;0\ne}$, instead of using \eqref{BS-up7}, noting that $(t-\tau)^{\rm ap}=e^{\nu\tau}(t^{\rm ap}-\tau^{\rm ap})\le e^{\nu\tau}t^{\rm ap}$, now we have
\begin{align}
\nn&B^{\sig_1}_k(t)S_k(t-\tau) |k(t-\tau)^{\rm ap}|^2{\bf 1}_{k\ne0}\\
\nn\les& e^{\dl\nu^\fr13\tau}e^{(m+2)\nu\tau}e^{-\dl_1\nu^\fr25\tau}\left(e^{-m\nu \tau}e^{\dl_1\nu^\fr25\tau}e^{\lm(\tau)\la k,kt^{\rm ap}\ra^s}\la k, kt^{\rm ap}\ra^{\sig_1+2} \right){\bf 1}_{k\ne0}\\
\nn\les&e^{\dl\nu^\fr13\tau} e^{-\fr12\dl_1\nu^\fr25\tau}\left(e^{-m\nu\tau}\underline{A}^{\sig_0+1,\fr25}_k(\tau, tk^{\rm ap})\right){\bf 1}_{k\ne0}.
\end{align}
Then by Cauchy-Schwarz inequality and \eqref{emb},
\begin{align}
{\bf N}_{g,3;0\ne}\nn\les&\nu^2\sum_{k\in\mathbb{Z}^n_*}\left[\int_0^t\left(e^{\dl\nu^\fr13\tau}(\widehat{M_\theta})_0(\tau)\right)\left\|e^{-m\nu\tau}\underline{A}^{\sig_0+1,\fr25}\hat{f}_{k}(\tau, \eta)\right\|_{L^\infty_\eta}d\tau\right]^2\\
\nn\les&\nu^2\int_0^t \left(e^{\dl\nu^\fr13\tau}(\widehat{M_\theta})_0(\tau)\right)^2       \sum_{k\in\Z^n_*}\left\|e^{-m\nu\tau}\underline{A}^{\sig_0+1,\fr25}\hat{f}_{k}(\tau, \eta)\right\|_{L^\infty_\eta}^2d\tau\\
\nn\les&\eps^2\nu^{2\gamma+2}\sup_t\left(e^{\dl\nu^\fr13t}\la t\ra^{\fr12+}(\widehat{M_\theta})_0(t)\right)^2.
\end{align}

\section{Higher order moment estimates}\label{high-m}
In this section, we improve the estimates \eqref{H-M1-H}, \eqref{H-M2-H}, \eqref{H-M1-L} and \eqref{H-M2-L}  for $M_1$ and $M_2$. As one can see from Section \ref{sec-rho-low}, the $L^\infty_t$ estimates are much more straightforward than the $L^2_t$ estimates. We will only sketch the details for the much harder $L^2_t$ control.

To begin with, we  give the expressions of $(\widehat{M_1})_k(t)$ and $(\widehat{M_2})_k(t)$. As a matter of fact, keeping in mind the relations \eqref{M1f} and \eqref{M2f}, the expressions of $(\widehat{M_1})_k(t)$ and $(\widehat{M_2})_k(t)$ can be obtained in a similar way to that of $\hat{\rho}_k(t)$, see \eqref{ex-f}--\eqref{ex-rho}. To avoid  unnecessary repetition, the expressions of $(\widehat{M_1})_k(t)$ and $(\widehat{M_2})_k(t)$ are given below directly, and we refer to \cite{bedrossian2017suppression} for more details for the derivations:
\begin{align}\label{ex-M1}
(\widehat{M_1})_k(t,\eta)\nn=&ie^{-\nu t}S_k(t)(\nb_\eta\widehat{f}_{\mathrm{in}})_k\left(kt^{\rm ap}\right)\\
\nn&-\int_0^tS_k(t-\tau)ie^{-\nu t}[\nb_\eta\mathcal{L}]_k(\tau, k(t-\tau)^{\rm ap}    )d\tau\\
\nn&+\int_0^tS_k(t-\tau)ie^{-\nu t}[\nb_\eta\mathcal{N}]_k(\tau,k(t-\tau)^{\rm ap}    )d\tau\\
\nn&-2\nu \left(\int_0^tie^{-\nu(t-\tau)}k(t-\tau)^{\rm ap}    d\tau \right)S_k(t)(\widehat{f_{\mathrm{in}}})_k\left(kt^{\rm ap}\right)\\
\nn&+2\nu \int_0^t\left(\int_{\tau'}^t ie^{-\nu(t-\tau)}k(t-\tau)^{\rm ap}    d\tau\right) S_k(t-\tau')\mathcal{L}_k(\tau', k(t-\tau')^{\rm ap}    )d\tau' \\
&-2\nu \int_0^t\left(\int_{\tau'}^t ie^{-\nu(t-\tau)}k(t-\tau)^{\rm ap}    d\tau\right) S_k(t-\tau')\mathcal{N}_k(\tau', k(t-\tau')^{\rm ap}    )d\tau',
\end{align}
and
\begin{align}\label{ex-M2}
(\widehat{M}_2)_k(t)
\nn=&-e^{-2\nu t}S_k(t)(\Dl_\eta\widehat{f_{\rm in}})_k\left(kt^{\rm ap}\right)+e^{-2\nu t}\int_0^tS_k(t-\tau)[\Dl_\eta\mathcal{L}]_k(\tau,k(t-\tau)^{\rm ap}    )d\tau\\
\nn&-e^{-2\nu t}\int_0^tS_k(t-\tau)[\Dl_\eta\mathcal{N}]_k(\tau,k(t-\tau)^{\rm ap}    )d\tau\\
\nn&+4\nu e^{-\nu t}\left(\int_0^t e^{-\nu (t-\tau)} k(t-\tau)^{\rm ap}    d\tau\right) \cdot  S_k(t) (\nb_\eta\widehat{f}_{\mathrm{in}})_k\left(kt^{\rm ap}\right)\\
\nn&+4\nu e^{-\nu t}\int_0^t\left(\int_{\tau'}^te^{-\nu (t-\tau)} k(t-\tau)^{\rm ap}d\tau\right)\cdot S_k(t-\tau') [\nb_\eta\mathcal{L}]_k(\tau', k(t-\tau')^{\rm ap}    )d\tau' \\
\nn&+4\nu e^{-\nu t}\int_0^t\left(\int_{\tau'}^te^{-\nu (t-\tau)} k(t-\tau)^{\rm ap}d\tau\right)\cdot S_k(t-\tau')[\nb_\eta\mathcal{N}]_k(\tau', k(t-\tau')^{\rm ap}    )d\tau'\\
\nn&-8\nu^2\left(e^{-2\nu t}\int_0^t\int_0^\tau e^{\nu \tau} k(t-\tau)^{\rm ap}    \cdot e^{\nu\tau'}k(t-\tau')^{\rm ap}    d\tau' d\tau\right) 
 S_k(t)(\widehat{f_{\mathrm{in}}})_k\left(kt^{\rm ap}\right)\\
\nn&+8\nu^2\int_0^t \left(e^{-2\nu t}\int^t_{\tau''} \int^\tau_{\tau''}e^{\nu \tau} k(t-\tau)^{\rm ap}    \cdot e^{\nu\tau'}k(t-\tau')^{\rm ap}    d\tau'd\tau\right)\\
\nn&\times S_k(t- \tau'')\mathcal{L}_k(\tau'', k(t-\tau'')^{\rm ap}    )d\tau''\\
\nn&-8\nu^2\int_0^t \left(e^{-2\nu t}\int^t_{\tau''} \int^\tau_{\tau''}e^{\nu \tau} k(t-\tau)^{\rm ap}     \cdot e^{\nu\tau'}k(t-\tau')^{\rm ap}    d\tau'd\tau\right)\\
\nn&\times S_k(t-\tau'')\mathcal{N}_k(\tau'', k(t-\tau'')^{\rm ap}    )d\tau'' \\
\nn&+2n\nu \left(\int_0^te^{-2\nu(t-\tau)}d\tau\right) S_k(t)(\widehat{f_{\rm in}})_k\left(kt^{\rm ap}\right)\\
\nn&-2n\nu \int_0^t \left(\int^t_{\tau'} e^{-2\nu(t-\tau)}d\tau\right) S_k(t-\tau')\mathcal{L}_k(\tau',k(t-\tau')^{\rm ap}    )d\tau'\\
&+2n\nu \int_0^t\left(\int_{\tau'}^t e^{-2\nu(t-\tau)}d\tau\right)  S_k(t-\tau')\mathcal{N}_k(\tau',k(t-\tau')^{\rm ap}    )d\tau',
\end{align}
where we have  used the notations
\begin{align}
[\nb_\eta\mathcal{L}]_k(\tau,  k(t-\tau)^{\rm ap}    )\nn=&\big(\nb_\eta[\mathcal{L}_k(\tau,\bar{\eta}(\tau; k,\eta)]\big)\big|_{\eta=kt^{\rm ap}},\\
[\nb_\eta\mathcal{N}]_k(\tau,  k(t-\tau)^{\rm ap}    )\nn=&\big(\nb_\eta[\mathcal{N}_k(\tau,\bar{\eta}(\tau; k,\eta)]\big)\big|_{\eta=kt^{\rm ap}},\\
[\Dl_\eta\mathcal{L}]_k(\tau,  k(t-\tau)^{\rm ap}    )\nn=&\big(\Dl_\eta[\mathcal{L}_k(\tau,\bar{\eta}(\tau; k,\eta)]\big)\big|_{\eta=kt^{\rm ap}},\\
[\Dl_\eta\mathcal{N}]_k(\tau,  k(t-\tau)^{\rm ap}    )\nn=&\big(\Dl_\eta[\mathcal{N}_k(\tau,\bar{\eta}(\tau; k,\eta)]\big)\big|_{\eta=kt^{\rm ap}}.
\end{align}

To improve \eqref{H-M1-H} and \eqref{H-M2-H}, now we do not have to use Lemma \ref{lem-linear-rho}. Instead, we will apply $\la t\ra^b{\bf B}(t)$ to \eqref{ex-M1} and \eqref{ex-M2} directly, and then take the $L^2_{t,x}$ norms.  Note that
\begin{align}\label{ker0}
\int^t_{\tau'} e^{-2\nu(t-\tau)}d\tau=\fr{1}{2\nu}\left(1-e^{-2\nu(t-\tau')}\right),
\end{align}

\begin{align}\label{ker1}
\int_{\tau'}^t e^{-\nu(t-\tau)}k(t-\tau)^{\rm ap}    d\tau=\fr{k}{2\nu^2}\left(1-e^{-\nu(t-\tau')} \right)^2,
\end{align}
and
\begin{align}\label{ker2}
e^{-2\nu t} \int^t_{\tau''} \int^\tau_{\tau''} e^{\nu \tau} k(t-\tau)^{\rm ap} \cdot e^{\nu\tau'}k(t-\tau')^{\rm ap}d\tau'd\tau=\fr{|k|^2}{8\nu^4}\left(1-e^{-\nu(t-\tau'')}\right)^4.
\end{align}
Combining \eqref{ker1}, \eqref{ker2} with \eqref{S-prop3} and \eqref{S-prop4}, we find that the last three term on the right hand side of \eqref{ex-M1} and the last six terms on the right hand side of \eqref{ex-M2} can be bounded by 
\begin{align}\label{0m}
\nn& S^{1-p}_k(t)\left|(\widehat{f_{\mathrm{in}}})_k\left(kt^{\rm ap}\right)\right|+ \int_0^t S_k^{1-p}(t-\tau')\left|\mathcal{L}_k(\tau', k(t-\tau')^{\rm ap}    )\right|d\tau' \\
&+ \int_0^t S_k^{1-p}(t-\tau')\left|\mathcal{N}_k(\tau', k(t-\tau')^{\rm ap}    )\right|d\tau',
\end{align}
for any fixed $p\in (0,1)$. Similarly, in view of \eqref{ker1} and \eqref{S-prop3}, the first three terms on the right hand side of \eqref{ex-M1} and the fourth to sixth terms on the right side of \eqref{ex-M2} can be bounded by
\begin{align}\label{1m}
\nn&S^{1-p}_k(t)e^{-\nu t}\left|(\nb_\eta\widehat{f}_{\mathrm{in}})_k\left(kt^{\rm ap}\right)\right|+\int_0^tS^{1-p}_k(t-\tau)e^{-\nu t}\big|[\nb_\eta\mathcal{L}]_k(\tau, k(t-\tau)^{\rm ap}    )\big|d\tau\\
&+\int_0^tS^{1-p}_k(t-\tau)e^{-\nu t}\big|[\nb_\eta\mathcal{N}]_k(\tau,k(t-\tau)^{\rm ap}    )\big|d\tau,
\end{align}
for any fixed $p\in(0,1)$. We would like to remark that, thanks to the property \eqref{S-prop2} possessed by $S_k(\cdot)$,   the index $p$ in  \eqref{0m} and \eqref{1m} does not matter and $S_k(\cdot)^{1-p}$ can serve as $S_k(\cdot)$.  In addition, we have to deal with the first three terms on the right hand side of \eqref{ex-M2}, among which and the six  terms in \eqref{0m} and \eqref{1m}, the second and third terms on the right hand side of \eqref{ex-M2} are  the most representative and complicated. Before  proceeding any further, let us write $[\Dl_\eta\mathcal{L}]_k(\tau,k(t-\tau)^{\rm ap}    )$ and $[\Dl_\eta\mathcal{N}]_k(\tau, k(t-\tau)^{\rm ap}    )$ explicitly. From \eqref{def-L}, one easily deduces that
\begin{align}\label{ex-2L}
[\Dl_\eta\mathcal{L}]_k(\tau,k(t-\tau)^{\rm ap}    )\nn=&2e^{2\nu\tau}\hat{\rho}_k(\tau)\fr{k}{|k|^2}\cdot(\nb_\eta \hat{\mu})\left(k(t-\tau)^{\mathrm{ap}} \right)\\
&+e^{2\nu\tau}\hat{\rho}_k(\tau)(t-\tau)^{\mathrm{ap}}(\Dl_\eta\hat{\mu})\left(k(t-\tau)^{\mathrm{ap}}\right).
\end{align}
Recalling the definition of $\mathcal{N}_k(\tau, \bar{\eta}(\tau;k,\eta))$ in \eqref{def-N}, after  lengthy but routine calculations, we obtain
\begin{align}\label{ex-2N}
[\Dl_\eta\mathcal{N}]_k(\tau, k(t-\tau)^{\rm ap}    )={\rm I}+{\rm II}+{\rm III},
\end{align}
where
\begin{align}\label{ex-2N-1}
{\rm I}\nn=&-2ie^{\nu\tau}\sum_{l\in\Z^n_*}\hat{E}_l(\tau)\cdot(\nb_\eta\hat{f})_{k-l}\left(\tau, kt^{\rm ap}-l \tau^{\mathrm{ap}}\right)\\
&-i\sum_{l\in\Z^n_*}\hat{E}_l(\tau)\cdot k(t-\tau)^{\mathrm{ap}}(\Dl_\eta\hat{f})_{k-l}\left(\tau, kt^{\rm ap}-l \tau^{\mathrm{ap}}\right)={\rm I}_1+{\rm I}_2,\\
\label{ex-2N-2}
{\rm II}\nn=&-2\nu e^{2\nu\tau}(\widehat{M_\theta})_k(\tau)\hat{\mu}\left(k(t-\tau)^{\mathrm{ap}}\right)\\
\nn&-4\nu e^{2\nu\tau}(\widehat{M_\theta})_k(\tau)k(t-\tau)^{\mathrm{ap}}\cdot(\nb_\eta\hat{\mu})\left(k(t-\tau)^{\mathrm{ap}}\right)\\
\nn&-\nu e^{2\nu\tau}(\widehat{M_\theta})_k(\tau)\left|k(t-\tau)^{\mathrm{ap}}\right|^2(\Dl_\eta\hat{\mu})\left(k(t-\tau)^{\mathrm{ap}}\right)\\
\nn&-2i\nu e^{2\nu\tau}(\widehat{M_1})_k(\tau)\cdot(\nb_\eta\hat{\mu})\left(k(t-\tau)^{\mathrm{ap}} \right)\\
&-i\nu e^{2\nu\tau}(\widehat{M_1})_k(\tau)\cdot k(t-\tau)^{\mathrm{ap}}(\Dl_\eta\hat{\mu})\left(k(t-\tau)^{\mathrm{ap}}\right),\\
\label{ex-2N-3}
{\rm III}\nn=&-2\nu e^{2\nu\tau}\sum_{l\in\Z^n_*}\hat{\rho}_l(\tau)\hat{f}_{k-l}\left(\tau, kt^{\rm ap}-l \tau^{\mathrm{ap}}\right)\\
\nn&-4\nu e^{\nu\tau}\sum_{l\in\Z^n_*}\hat{\rho}_l(\tau)k(t-\tau)^{\mathrm{ap}}\cdot(\nb_\eta\hat{f})_{k-l}\left(\tau, kt^{\rm ap}-l \tau^{\mathrm{ap}}\right)\\
\nn&-\nu \sum_{l\in\Z^n_*}\hat{\rho}_l(\tau)\left|k(t-\tau)^{\mathrm{ap}}\right|^2(\Dl_\eta\hat{f})_{k-l}\left(\tau, kt^{\rm ap}-l \tau^{\mathrm{ap}}\right)\\
\nn&-2 \nu \sum_{l\in\Z^n_*}\hat{\rho}_l(\tau)(\Dl_\eta\hat{f})_{k-l}\left(\tau, kt^{\rm ap}-l \tau^{\mathrm{ap}}\right)\\
\nn&-\nu \sum_{l\in\Z^n_*}\hat{\rho}_l(\tau)e^{-\nu\tau}k(t-\tau)^{\mathrm{ap}}\cdot(\nb_\eta\Dl_\eta\hat{f})_{k-l}\left(\tau, kt^{\rm ap}-l \tau^{\mathrm{ap}}\right)\\
\nn&-2\nu e^{2\nu\tau}\sum_{l\in\Z^n}(\widehat{M_\theta})_l(\tau)\hat{f}_{k-l}\left(\tau, kt^{\rm ap}-l \tau^{\mathrm{ap}}\right)\\
\nn&-4\nu e^{\nu\tau}\sum_{l\in\Z^n}(\widehat{M_\theta})_l(\tau)k(t-\tau)^{\mathrm{ap}}\cdot(\nb_\eta\hat{f})_{k-l}\left(\tau, kt^{\rm ap}-l \tau^{\mathrm{ap}}\right)\\
\nn&-\nu \sum_{l\in\Z^n}(\widehat{M_\theta})_l(\tau)\left|k(t-\tau)^{\mathrm{ap}}\right|^2(\Dl_\eta\hat{f})_{k-l}\left(\tau, kt^{\rm ap}-l \tau^{\mathrm{ap}}\right)\\
\nn&-2\nu i e^{\nu\tau}\sum_{l\in\Z^n_*}(\widehat{M_1})_l(\tau)\cdot(\nb_\eta\hat{f})_{k-l}\left(\tau, kt^{\rm ap}-l \tau^{\mathrm{ap}}\right)\\
&-\nu i \sum_{l\in\Z^n_*}(\widehat{M_1})_l(\tau)\cdot k(t-\tau)^{\mathrm{ap}}(\Dl_\eta\hat{f})_{k-l}\left(\tau, kt^{\rm ap}-l \tau^{\mathrm{ap}}\right).
\end{align}

Now we first consider 
\be\label{2L}
\sum_{k\in\Z^n_*}\int_0^{T^*}\left[\la t\ra^b {\bf B}_k(t)e^{-2\nu t}\int_0^tS_k(t-\tau)[\Dl_\eta\mathcal{L}]_k(\tau,k(t-\tau)^{\rm ap}    )d\tau\right]^2dt.
\ee
The treatment of \eqref{2L} is slightly more direct than that of ${\bf N}_{E;0}$ in Section \ref{sec-L2rho}. Indeed,   similar to \eqref{BS-up2} and \eqref{BS-up4}, it holds that
\begin{align}
\la t\ra^b {\bf B}_k(t)S_k(t-\tau)\nn\les&S^{\fr12}_k(t-\tau)\la\tau\ra^b {\bf B}_k(\tau) e^{c\lm(\tau)\left\la e^{-\nu\tau}k (t-\tau)^{\mathrm{ap}}\right\ra^s}\\
\nn&+S^{\fr12}_k(t-\tau)\left(\la \tau\ra^be^{\dl\nu^\fr13\tau}|k|^\fr12e^{c\lm(\tau)\left\la k, k\tau^{\mathrm{ap}}\right\ra^s}\right)\\
\nn&\times e^{\lm(\tau)\left\la e^{-\nu\tau}k (t-\tau)^{\mathrm{ap}}\right\ra^s}\left\la e^{-\nu \tau} k(t-\tau)^{\mathrm{ap}} \right\ra^{\sig_0}\\
\nn\les&\fr{S^{\fr12}_k(t-\tau)}{\la k(t-\tau)^{\rm ap}\ra^2}\left(\la\tau\ra^b {\bf B}_k(\tau)\right) e^{\lm(0)\left\la k (t-\tau)^{\mathrm{ap}}\right\ra^s}\left\la  k(t-\tau)^{\mathrm{ap}} \right\ra^{\sig_0+2}.
\end{align}
Then similar to \eqref{NHL0}, we arrive at
\begin{align}\label{e-2L}
\nn&\sum_{k\in\Z^n_*}\int_0^{T^*}\left[\la t\ra^b {\bf B}_k(t)e^{-2\nu t}\int_0^tS_k(t-\tau)[\Dl_\eta\mathcal{L}]_k(\tau,k(t-\tau)^{\rm ap}    )d\tau\right]^2dt\\
\nn\les&\sum_{|\al|\le2}\sum_{k\in\Z^n_*}\int_0^{T^*}\Bigg[\int_0^t\la \tau \ra^b \left|{\bf B}\hat{\rho}_k(\tau)\right|\fr{S_k^{\fr12}(t-\tau)}{\la k(t-\tau)^{\rm ap} \ra^2} e^{\lm(0)\left\la k (t-\tau)^{\mathrm{ap}}\right\ra^s}\\
\nn&\times\left\la  k(t-\tau)^{\mathrm{ap}} \right\ra^{\sig_0+3}\left|(\pr_\eta^\al \hat{\mu})\left(k(t-\tau)^{\mathrm{ap}} \right)\right|d\tau\Bigg]^2dt\\
\les&\left(\sum_{|\al|\le2} \left\|e^{\lm(0)\left\la \eta\right\ra^s}\left\la \eta\right\ra^{\sig_0+3}(\pr^\al_\eta\hat{\mu})\left(\eta\right) \right\|_{L^\infty_\eta}\right)^2\left\|\la t\ra^b{\bf B}\rho\right\|_{L^2_{t,x}}^2
\les ({\rm C}_0\eps\nu^{\gamma})^2,
\end{align}
which is consistent with Proposition \ref{prop: btsp} provided we choose ${\rm C}_2\gg {\rm C}_0$.

Next we turn to investigate
\be\label{2N}
\sum_{k\in\Z^n_*}\int_0^{T^*}\left[\la t\ra^b {\bf B}_k(t)e^{-2\nu t}\int_0^tS_k(t-\tau)[\Dl_\eta\mathcal{N}]_k(\tau,k(t-\tau)^{\rm ap}    )d\tau\right]^2dt.
\ee
On can see from \eqref{ex-2N}--\eqref{ex-2N-3} that there are many terms involved in \eqref{2N}, while these contributions can be treated similarly by using the techniques in Section \ref{sec-L2rho}. For the collisionless contributions,  two terms ${\rm I}_1$ and ${\rm I}_2$
are involved, see \eqref{ex-2N-1}. Clearly, ${\rm I}_2$ is the leading term, which leads to the contribution
\begin{align*}
&\sum_{k\in\Z^n_*}\int_0^{T^*}\Bigg[\la t\ra^b {\bf B}_k(t)e^{-2\nu t}\int_0^tS_k(t-\tau)\sum_{l\in\Z^n_*}\hat{\rho}_l(\tau)\fr{l}{|l|^2}\cdot k(t-\tau)^{\mathrm{ap}}\\
&\times (\Dl_\eta\hat{f})_{k-l}\left(\tau, kt^{\rm ap}-l \tau^{\mathrm{ap}}\right)d\tau\Bigg]^2dt.
\end{align*}
This can be treated exactly in the same as ${\bf N}_{E}$ in Section \ref{Sec: Nonlinear collisionless contributions} with $\hat{f}_{k-l}\left(\tau, kt^{\rm ap}-l \tau^{\mathrm{ap}}\right)$ replaced by $e^{-2\nu\tau}(\Dl_\eta\hat{f})_{k-l}\left(\tau, kt^{\rm ap}-l \tau^{\mathrm{ap}}\right)$. For ${\rm I}_1$, let us denote
\begin{align}
{\bf \tl N}_{E}\nn=&\sum_{k\in\Z^n_*}\int_0^{T^*}\Bigg[\la t\ra^b {\bf B}_k(t)e^{-2\nu t}\int_0^tS_k(t-\tau)\sum_{l\in\Z^n_*}\fr{\hat{\rho}_l(\tau)}{|l|^2}l\cdot  e^{\nu\tau} (\nb_\eta\hat{f})_{k-l}\left(\tau, kt^{\rm ap}-l \tau^{\mathrm{ap}}\right)d\tau\Bigg]^2dt\\
\nn=&{\bf \tl N}_{E;\ne}^{\rm HL}+{\bf \tl N}_{E;\ne}^{\rm LH}+{\bf \tl N}_{E;0}.
\end{align}
Due to the absence of $k(t-\tau)^{\rm ap}$, it is easy to see  that now ${\bf \tl N}_{E;\ne}^{\rm LH}$ and ${\bf \tl N}_{E;0}$  can be treated in the same way as ${\bf N}_{E;\ne}^{\rm LH}$ and ${\bf  N}_{E;0}$, respectively, in Section \ref{Sec: Nonlinear collisionless contributions} with  $\hat{f}_{k-l}\left(\tau, kt^{\rm ap}-l \tau^{\mathrm{ap}}\right)$ replaced by $e^{-\nu\tau}(\nb_\eta\hat{f})_{k-l}\left(\tau, kt^{\rm ap}-l \tau^{\mathrm{ap}}\right)$. Compared with ${\bf N}_{E;\ne}^{\rm HL}$ in Section \ref{Sec: Nonlinear collisionless contributions}, 
${\bf \tl N}_{E;\ne}^{\rm HL}$ is much easier to treat. Indeed, arguing as in Section \ref{Sec: Nonlinear collisionless contributions} based on $\tau\le\fr{t}{2}$ or $\fr{t}{2}\le\tau\le t$, ignoring the factor $\la \tau\ra^{1-3\gamma}$ in \eqref{small-tau1}, for $k\ne l$, we infer from \eqref{BS-up1} that
\begin{align*}
&\la t\ra^b {\bf B}_k  (t)S_k(t-\tau)\left|e^{-\nu\tau}(\nb_\eta\hat{f})_{k-l}\left(\tau, kt^{\rm ap}-l \tau^{\mathrm{ap}}\right)\right|{\bf 1}_{\left|k-l,kt^{\rm ap}-l \tau^{\mathrm{ap}}\right|\le2\left|l,l \tau^{\mathrm{ap}}\right|}\\
\les&\la \tau\ra^b{\bf B}_l(\tau)S_k^\fr12(t-\tau)e^{-\fr12\underline{c}\la k-l,kt^{\rm ap}-l\tau^{\rm ap}\ra^s}\left(e^{\lm(\tau)\la k-l,kt^{\rm ap}-l\tau^{\rm ap}\ra^s}\left|e^{-\nu\tau}(\nb_\eta\hat{f})_{k-l}\left(\tau, kt^{\rm ap}-l \tau^{\mathrm{ap}}\right)\right|\right)\\
\les&\la \tau\ra^b{\bf B}_l(\tau)S_k^\fr12(t-\tau)e^{-\fr12\underline{c}\la k-l,kt^{\rm ap}-l\tau^{\rm ap}\ra^s}e^{-\fr12\dl_1\nu^\fr13\tau}\sup_t\left\|f^w_{\ne}(t)\right\|_{\mathcal{E}^{\sig_1,\fr13}_m}.
\end{align*}
Then similar to \eqref{schur1}, we immediately have
\begin{align}
{\bf \tl N}_{E;\ne}^{\mathrm{HL}}\nn\les&\eps^2\nu^{2\gamma}\int_0^{T^*}\sum_{k\in\mathbb{Z}^n_*}\Bigg[\sum_{l\in\mathbb{Z}^n_*}\int_0^tS^\fr12_k(t-\tau) e^{-\fr12\dl_1\nu^\fr13\tau}e^{-\underline{c}\la k-l,kt^{\rm ap}-l \tau^{\mathrm{ap}}\ra^s}| {\bf B}  \hat{\rho}_l(\tau)|d\tau\Bigg]^2dt\\
\les&\eps^2\nu^{2\gamma}\| {\bf B}  \rho\|_{L^2_{t,x}}^2.
\end{align}

The collisional contributions stemming from ${\rm II}$ and ${\rm III}$ can be treated by following the way of treating the collisions in Section \ref{Sec: Nonlinear collisionless contributions} without additional complications. To explain the appearance of up to  the third derivative on $\hat{f}_0(t,\eta)$ with respect to  $\eta$ in \eqref{H-f0}, here we only sketch the estimate of the term that $(\nb_\eta\Dl_\eta\hat{f})_0\left(\tau, k(t^{\rm ap}-\tau^{\mathrm{ap}})\right)$ is involved. To this end, let us denote
\beno
{\rm III}_{5;0}=-\nu \hat{\rho}_k(\tau)e^{-\nu\tau}k(t-\tau)^{\mathrm{ap}}\cdot(\nb_\eta\Dl_\eta\hat{f})_{0}\left(\tau, k(t^{\rm ap}-\tau^{\mathrm{ap}})\right).
\eeno
Similar to \eqref{NHL0} and \eqref{NLH0}, we have
\begin{align}
\nn&\sum_{k\in\Z^n_*}\int_0^{T^*}\left[\la t\ra^b {\bf B}_k(t)e^{-2\nu t}\int_0^tS_k(t-\tau){\rm III}_{5;0}d\tau\right]^2dt\\
\nn\les&\nu^2\Bigg[\sup_t\left\| e^{-3\nu t}e^{\lm(t)\left\la \eta\right\ra^s}\la e^{\nu t}\eta\ra^3(\nb_\eta\Dl_\eta\hat{f})_0\left(t, \eta\right)\right\|_{L^\infty_\eta}^2\\
\nn&+\sup_{0\le\tau\le T^*}\left(e^{-(m+3)\nu \tau}\left\|(\underline{A}^{\sig_0+1,\fr25}\nb_\eta\Dl_\eta\hat{f})_0(\tau)\right\|_{H^{\fr{n-1}{2}+}_\eta}\right)^2\Bigg]\left\|\la t\ra^b {\bf B}\rho\right\|^2_{L^2_{t,x}}\\
\les& \eps^2\nu^{2\gamma+2}\left\|\la t\ra^b {\bf B}\rho\right\|^2_{L^2_{t,x}}.
\end{align}

\section{Estimates on the distribution function }\label{Estimates on the distribution function}
In this section, we improve \eqref{H-f1}--\eqref{H-f3}.
The following partition of unity 
\be\label{pou}
1={\bf 1}_{ \left|l, lt^{\rm ap}\right|<\fr12\left|k-l,\eta-lt^{\rm ap} \right| }+{\bf 1}_{\left|k-l,\eta-lt^{\rm ap} \right|\le 2\left|l, lt^{\rm ap}\right|}
\ee
will be used frequently.

\begin{lem}\label{lem-A}
Let $(\sig, \frak{e})\in\left\{ (\sig_0+1, \fr25),  ( \sig_1 , \fr13), ( \sig_1 , \fr25)\right\}$. Then there exists a positive constant $c\in (0, 1)$ such that
\begin{align}\label{A-LH}
 {\bf A}^{\sig,\frak{e}}_k(t,\eta){\bf 1}_{\left|l, lt^{\rm ap}\right|<\fr12\left|k-l,\eta-lt^{\rm ap} \right|}\les e^{-(\dl \nu^\fr13-\dl_1\nu^{\frak{e}})t}e^{\dl \nu^\fr13t}e^{c\lm(t)\la l,lt^{\rm ap}\ra^s}  {\bf A}^{\sig,\frak{e}}_{k-l}\left(t, \eta-lt^{\rm ap}\right),
\end{align}
and
\begin{align}\label{A-HL}
{\bf A}^{\sig,\frak{e}}_k(t,\eta){\bf 1}_{\left|k-l,\eta-lt^{\rm ap} \right|\le 2\left|l, lt^{\rm ap}\right|}
\les e^{c\lm(t)\la k-l,\eta-lt^{\rm ap}\ra^s}\left(e^{\dl_1\nu^\frak{e}t}e^{\lm(t)\left\la l, lt^{\rm ap}\right\ra^s}\left\la l, lt^{\rm ap}\right\ra^{\sig}\right).
\end{align}
In particular, if $k=l$, and $(\sig,\frak{e})=(\sig_0+1, \fr25)$, we have
\begin{align}\label{A-HL0}
{\bf A}^{\sig_0+1,\fr25}_k(t,\eta){\bf 1}_{\left|\eta-kt^{\rm ap} \right|\le 2\left|k, kt^{\rm ap}\right|}
\nn\les&e^{-s(\dl\nu^\fr13-\dl_1\nu^\fr25)t}\la t\ra^se^{c\lm(t)\la \eta-kt^{\rm ap}\ra^s} |k|^\fr{s}{2}\left({\bf B}_k(t)\right)^s\\
&\times \left(e^{\dl_1\nu^\fr25 t}e^{\lm(t)\la k, kt^{\rm ap}\ra^s}\left\la k,kt^{\rm ap}\right\ra^{\sig_0+1}\right)^{1-s}.
\end{align}
\end{lem}
\begin{proof}
Noting that $|k,\eta|\le |k-l,\eta-lt^{\rm ap}|+|l,lt^{\rm ap}|$, using \eqref{app3} and \eqref{app5}, we get \eqref{A-LH} and \eqref{A-HL} respectively. To prove \eqref{A-HL0}, using  \eqref{app5} again yields
\begin{align}
{\bf A}^{\sig_0+1,\fr25}_k(t,\eta){\bf 1}_{\left|\eta-kt^{\rm ap} \right|\le 2\left|k, kt^{\rm ap}\right|}
\nn\les&e^{c\lm(t)\la \eta-kt^{\rm ap}\ra^s}\left(e^{\dl_1\nu^\fr25t}e^{\lm(t)\la k, kt^{\rm ap}\ra^s}\left\la k,kt^{\rm ap}\right\ra^{\sig_0+1}\right)^{1-s}\\
\nn&\times \left(e^{-(\dl\nu^\fr13-\dl_1\nu^\fr25)t} e^{\dl\nu^\fr13t}e^{\lm(t)\la k, kt^{\rm ap}\ra^s}\left\la k,kt^{\rm ap}\right\ra^{\sig_0+1}\right)^{s}.
\end{align}
Combining this with the definition of ${\bf B}     _k(t)$ and the fact 
${\bf 1}_{k\ne0}\left\la k,  kt^{\rm ap}\right\ra\les |k|\left\la t^{\rm ap}\right\ra\les |k|\la t\ra$,
we find that \eqref{A-HL0} follows immediately. The proof of Lemma \ref{lem-A} is completed.
\end{proof}
Let $A(t,\nb)\in\left\{{\bf A}^{\sig_0+1, \fr25} (t,\nb), {\bf A}^{ \sig_1 , \fr13} (t,\nb), {\bf A}^{ \sig_1 , \fr25} (t,\nb)\right\}$, $\al\in\mathbb{N}^n$ be any fixed multi-index.
From  \eqref{mul-A1},  \eqref{mul-A2} and \eqref{eq-fw},   we have the following energy identity
\begin{align}\label{en-f-H}
\nn&\fr12\fr{d}{dt}\left(e^{-2|\al|\nu t}\left\| A(v^\al f^w)\right\|_{L^2}^2\right)-e^{-2|\al|\nu t}\dot{\lm}(t)\left\|\la \nb\ra^{\fr{s}{2}}A(v^\al f^w) \right\|_{L^2}^2\\
\nn&+e^{-2|\al|\nu t}\left\|\sqrt{-\fr{\pr_t\frak{m}}{\frak{m}}} A(v^\al f^w)\right\|_{L^2}^2+\nu e^{-2|\al|\nu t}\left\|\pr_v^tA(v^\al f^w) \right\|_{L^2}^2\\
\nn&+\nu|\al| e^{-2|\al|\nu t}\| A(v^\al f^w)\|_{L^2}^2+2n\nu\lm_1(t)e^{-2|\al|\nu t}\left\| A(v^\al f^w)\right\|_{L^2}^2\\
\nn&+\left(2\nu \lm_1(1-2\lm_1)-\dot{\lm}_1(t)\right)e^{-2(|\al|+1)\nu t}\left\|A(vv^\al f^w)\right\|_{L^2}^2 \\
=&\dl_1\nu^{\frak e}e^{-2|\al|\nu t}\left\|A(v^\al f^w_{\ne})\right\|_{L^2}^2+\mathsf{L}+\mathsf{CL}+\sum_{1\le j\le4}\mathsf{LE}_j+\mathsf{NE}_{E}+\mathsf{N}+\mathsf{CN}_0+\mathsf{CN}_1+\mathsf{CN}_2,
\end{align}
where $\frak e\in\{\frac{1}{3},\frac{2}{5}\}$ is determined by $A(t,\nabla)$, and
\begin{align*}
\mathsf{L}=&-e^{-2|\al|\nu t}\frak{Re}\sum_{k\in\Z^n}\int_\eta  A_k(t,\eta)\hat{E}_k(t)\cdot\pr^\al_\eta \left(\mathcal{F}\left[\nb_v\mu e^{\lm_1 |v|^2} \right](\bar{\eta}(t; k,\eta))\right)   A\pr_\eta^\al(\overline{\widehat{f^w}})_k(t,\eta)d\eta,\\
\mathsf{CL}=&\nu e^{-2|\al|\nu t}\frak{Re}\sum_{k\in\mathbb{Z}^n}\int_\eta  A\pr_\eta^\al\left(\mathcal{F}\left[(M_\theta\Dl_v\mu-M_1\cdot\nb_v\mu)e^{\lm_1 |v|^2}\right]_k(t,\bar{\eta}(t;k,\eta))\right)\\
\nn&\times A\pr_\eta^\al(\overline{\widehat{f^w}})_k(t,\eta)d\eta,\\
\mathsf{LE}_1=&2\dot{\lm}_1(t)e^{-2(|\al|+1)\nu t}\frak{Re}\sum_{k\in\Z^n}\int_\eta  A(\nb_\eta \pr_\eta^\al\widehat{f^w})_k\left(t, \eta \right)\cdot \nb_\eta  A_k(t,\eta)\pr^\al_\eta(\overline{\widehat{f^w}})_k(t,\eta) d\eta,\\
\mathsf{LE}_2=&-4\nu \lm_1(1-2\lm_1)e^{-2(|\al|+1)\nu t}\frak{Re}\sum_{k\in\Z^n}\int_\eta  A(\nb_\eta \pr_\eta^\al\widehat{f^w})_k\left(t, \eta \right)\cdot \nb_\eta  A_k(t,\eta)\pr^\al_\eta(\overline{\widehat{f^w}})_k(t,\eta) d\eta,\\
\mathsf{LE}_3=&4\nu\lm_1 e^{-\nu t}e^{-2|\al|\nu t}\frak{Re}\sum_{k\in\Z^n}\int_\eta A\nb_\eta\cdot \pr_\eta^\al\left(\bar{\eta}(t;k,\eta) \widehat{f^w}\right)_k\left(t, \eta \right) A\pr^\al_\eta(\overline{\widehat{f^w}})_k(t,\eta) d\eta, \\
\mathsf{LE}_4=&-\nu e^{-2|\al|\nu t}\frak{Re}\sum_{k\in\mathbb{Z}^n}\int_\eta  A[\pr_\eta^\al,|\bar{\eta}|^2](\widehat{f^w})_k(t,\eta) A\pr_\eta^\al(\overline{\widehat{f^w}})_k(t,\eta)d\eta,\\
\mathsf{NE}_E=&2\lm_1e^{-2|\al|\nu t}\frak{Re}\sum_{k,l\in\Z^n}\int_\eta   A_k(t,\eta)\hat{E}_l(t)\cdot i e^{-\nu t} (\nb_\eta\pr^\al_\eta\widehat{f^w})_{k-l}\left(t,\eta-lt^{\rm ap}\right)A\pr^\al_\eta (\overline{\widehat{f^w}})_k(t,\eta)d\eta,\\
\mathsf{N}=&-e^{-2|\al|\nu t}\frak{Re}\sum_{k, l\in\mathbb{Z}^n}\int_\eta  A\hat{E}_l(t)\cdot i\pr_\eta^\al\left[\bar{\eta}(t; k,\eta)(\widehat{f^w})_{k-l}\left(t, \eta-lt^{\rm ap}\right)\right] A\pr_\eta^\al(\overline{\widehat{f^w}})_k(t,\eta)d\eta,\\
\mathsf{CN}_{\ell}=& e^{-2|\al|\nu t}\frak{Re}\sum_{k\in\mathbb{Z}^n}\int_\eta  A\pr_\eta^\al\left(\mathcal{F}\big[\mathcal{M}_{\ell}[g^w]\big]_k(t,\bar{\eta}(t;k,\eta))\right) A\pr_\eta^\al(\overline{\widehat{f^w}})_k(t,\eta)d\eta,\quad \ell=0,1,2.
\end{align*}

\subsection{Linear contributions}
In this section, we study the linear contributions, namely, we treat $\mathsf{L}$, $\mathsf{CL}$, and $\mathsf{LE}_j, \, j=1,2, 3,4$. Let us consider $\mathsf{CL}$ first, which is from the linearization of the Fokker-Planck operator.
The two terms in $\mathsf{CL}$ can be treated in the same way, it suffices to consider 
\begin{align}
\mathsf{CL}_{M_\theta}
\nn=&\nu e^{-2|\al|\nu t}\frak{Re}\sum_{k\in\mathbb{Z}^n_*}\int_\eta  A_k(\eta)(\widehat{M_\theta})_k(t)\pr_\eta^\al\left(\mathcal{F}\left[\Dl_v\mu e^{\lm_1 |v|^2}\right](\bar{\eta}(t;k,\eta))\right) 
A\pr_\eta^\al(\overline{\widehat{f^w}})_k(\eta)d\eta\\
\nn&+\nu e^{-2|\al|\nu t}\frak{Re}\int_\eta  A_0(t,\eta)(\widehat{M_\theta})_0(t)\pr_\eta^\al\left(\mathcal{F}\left[\Dl_v\mu e^{\lm_1 |v|^2}\right](e^{\nu t}\eta)\right) A\pr_\eta^\al(\overline{\widehat{f^w}})_0(t,\eta)d\eta\\
=&\mathsf{CL}_{M_\theta;\ne}+\mathsf{CL}_{M_\theta;0}.
\end{align}
Note that
\begin{align}\label{Dl-mu}
\pr_\eta^\al\left(\mathcal{F}\left[\Dl_v\mu e^{\lm_1 |v|^2}\right](t,\bar{\eta}(t;k,\eta))\right)=e^{|\al|\nu t}\pr_\eta^\al\left(\mathcal{F}\left[\Dl_v\mu e^{\lm_1 |v|^2}\right]\right)(t,\bar{\eta}(t;k,\eta)).
\end{align}
Then
\begin{align}\label{e-CLMtheta0}
\mathsf{CL}_{M_\theta;0}\nn\le& \nu e^{-|\al|\nu t}\left| (\widehat{M_\theta})_0(t)\right|\int_\eta e^{\lm(0)\la e^{\nu t}\eta\ra^s}\la e^{\nu t}\eta\ra^{\sig_0+1}\left|\pr_\eta^\al\left(\mathcal{F}\left[\Dl_v\mu e^{\lm_1 |v|^2}\right]\right)(t,e^{\nu t}\eta) \right| \\
\nn&\times \left|A\pr_\eta^\al({\widehat{f^w}})_0(t,\eta)\right|d\eta\\
\les&\nu e^{-\fr{n}{2}\nu t}\left| (\widehat{M_\theta})_0(t)\right| \left(e^{-|\al|\nu t}\left\|A(v^\al f^w)\right\|_{L^2}\right),
\end{align}
where we have used the fact
\be\label{mu-L2}
\left(\int_\eta e^{2\lm(0)\la e^{\nu t}\eta\ra^s}\la e^{\nu t}\eta\ra^{2(\sig_0+1)}\left|\pr_\eta^\al\left(\mathcal{F}\left[\Dl_v\mu e^{\lm_1 |v|^2}\right]\right)(t,e^{\nu t}\eta) \right|^2d\eta\right)^\fr12\les e^{-\fr{n}{2}\nu t}.
\ee
Since $\left| (\widehat{M_\theta})_0(t)\right|$ decays very fast, see \eqref{e-Mtheta0}, the estimate \eqref{e-CLMtheta0}  is consistent with Proposition \ref{prop: btsp} for sufficiently small $\nu$.

Next turn to $\mathsf{CL}_{M_\theta;\ne}$. By using the partition of unity
\be\label{pou-0}
1={\bf 1}_{ \left|k, kt^{\rm ap}\right|<\fr12\left|\eta-kt^{\rm ap} \right| }+{\bf 1}_{\left|\eta-kt^{\rm ap} \right|\le 2\left|k, kt^{\rm ap}\right|},
\ee
we divide $\mathsf{CL}_{M_\theta;\ne}$ into two parts:
\begin{align}
\mathsf{CL}_{M_\theta;\ne}\nn\les&\nu e^{-|\al|\nu t}\sum_{k\in\mathbb{Z}^n_*}\int_\eta  A_k(t,\eta)\left|(\widehat{M_\theta})_k(t)\right| \left|\pr_\eta^\al\left(\mathcal{F}\left[\Dl_v\mu e^{\lm_1 |v|^2}\right]\right)(t,\bar{\eta}(t;k,\eta))\right|\\ 
\nn&\times\left|A\pr_\eta^\al({\widehat{f^w}})_k(t,\eta)\right|d\eta=\mathsf{CL}_{M_\theta;\ne}^{\rm LH}+\mathsf{CL}_{M_\theta;\ne}^{\rm HL}.
\end{align}
Now \eqref{A-LH} reduces to
\begin{align}\label{ALH0}
A_k(t,\eta){\bf 1}_{ \left|k, kt^{\rm ap}\right|<\fr12\left|\eta-kt^{\rm ap} \right| }
\les e^{\dl_1\nu^{\frak{e}}t}e^{c\lm(t)\la k,kt^{\rm ap}\ra^{s}}e^{\lm(t)\la \eta-kt^{\rm ap}\ra^s}\left\la \eta-kt^{\rm ap}\right\ra^{\sig},
\end{align}
for $\sig\in\{\sig_0+1,  \sig_1 \}$.
It follows this and \eqref{mu-L2} that
\begin{align}\label{CLLH}
\mathsf{CL}_{M_\theta;\ne}^{\rm LH}\nn\les&\nu e^{-|\al|\nu t} \sum_{k\in\mathbb{Z}^n_*}  e^{\dl_1\nu^{\frak{e}}t}e^{c\lm(t)\la k,kt^{\rm ap}\ra^{s}} \left|(\widehat{M_\theta})_k(t)\right|\left\|A\pr_\eta^\al({\widehat{f^w}})_k(t,\eta)\right\|_{L^2_\eta}\\
\les&\nu e^{-(\dl-\dl_1)\nu^\fr13 t}\left\| {\bf B}M_{\theta}\right\|_{L^2_x}\left( e^{-|\al|\nu t}\left\|A(v^\al f^w_{\ne})\right\|_{L^2}\right).
\end{align}
The estimate of $\mathsf{CL}_{M_\theta;\ne}^{\rm HL}$ is tricker  when $(\sig,\frak{e})=(\sig_0+1,\fr25)$. Indeed,   from \eqref{A-HL0}, \eqref{Dl-mu}, and \eqref{mu-L2},  we use a interpolation technic to obtain
\begin{align}\label{CLHL}
\mathsf{CL}_{M_\theta;\ne}^{\rm HL}\nn\les&\nu \la t\ra^s e^{-s(\dl-\dl_1)\nu^\fr13t}e^{-|\al|\nu t}\sum_{k\in\mathbb{Z}^n_*}\int_\eta |k|^\fr{s}{2}\left|A\pr_\eta^\al(\widehat{f^w})_k(t,\eta) \right| \\
\nn&\times \left|\underline{A}_k^{\sig_0+1,\fr25}\left(t, kt^{\rm ap}\right)(\widehat{M_\theta})_k(t)\right|^{1-s}\left|{\bf B}     (\widehat{M_\theta})_k(t) \right|^s \\
\nn&\times e^{c\lm(t)\la \eta-kt^{\rm ap}\ra^s}\left|\pr_\eta^\al\left(\mathcal{F}\left[\Dl_v\mu e^{\lm_1 |v|^2}\right]\right)(t,\bar{\eta}(t;k,\eta))\right|d\eta\\
\nn\les&\nu \la t\ra^s e^{-s(\dl-\dl_1)\nu^\fr13t}  e^{-|\al|\nu t}\sum_{k\in\Z^n_*}\left\||k|^\fr{s}{2}A\pr_\eta^\al(\widehat{f^w})_k(t,\eta) \right\|_{L^2_\eta}\\
\nn&\times \left|\underline{A}_k^{\sig_0+1,\fr25}\left(t, kt^{\rm ap}\right)(\widehat{M_\theta})_k(t)\right|^{1-s} \left|{\bf B}     (\widehat{M_\theta})_k(t) \right|^s\\
\nn\les&\nu \fr{\la t\ra^s}{\sqrt{-\dot{\lm}(t)}} e^{-s(\dl-\dl_1)\nu^\fr13t} e^{(1-s)m\nu t} \left(e^{-|\al|\nu t}\sqrt{-\dot{\lm}(t)}\left\|\la \nb\ra^{\fr{s}{2}}A(v^\al f^w)\right\|_{L^2}\right) \\
&\times \| {\bf B}M_\theta\|_{L^2}^s \left(e^{-m\nu t}\left\|{\bf 1}_{k\ne0}\underline{A}_k^{\sig_0+1,\fr25}\left(t, kt^{\rm ap}\right)(\widehat{M_\theta})_k(t)\right\|_{L^2_k}\right)^{1-s}.
\end{align}
Note that
\begin{align}\label{CLHL'}
\nu \fr{\la t\ra^s}{\sqrt{-\dot{\lm}(t)}} e^{-s(\dl-\dl_1)\nu^\fr13t} e^{(1-s)m\nu t}\les \nu\la t\ra^{s+\fr12+}e^{-\fr{s}{2}(\dl-\dl_1)\nu^\fr13t}\les \nu^\fr23.
\end{align}
Then it follows from the above two inequalities and \eqref{AM} that
\begin{align}\label{CLHL''}
\sum_{|\al|\le m}\int_0^t\fr{\mathsf{CL}_{M_\theta;\ne}^{\rm HL}(t')}{K_{|\al|}}dt'
\nn\les&\nu^{\fr12+\fr{s}{6}}\left(\int_0^t\sum_{|\al|\le m'} \fr{e^{-2|\al|\nu t'}}{K_{|\al|}}(-\dot{\lm}(t'))\left\|\la \nb\ra^{\fr{s}{2}}A(v^\al f^w)\right\|_{L^2}^2dt'\right)^\fr12 \\
&\times \| {\bf B}     \rho\|_{L^2_{t,x}}^{s} \left(\nu^\fr16\|f^w_{\ne}(t)\|_{L^2_t\mathcal{E}_m^{\sig_0+1,\fr25}}\right)^{1-s}.
\end{align}
Both \eqref{CLLH} and \eqref{CLHL''} are consistent with Proposition \ref{prop: btsp} for sufficiently small $\nu$.  When $(\sig,\frak{e})= (\sig_1, \fr13)  $ or $ (\sig_1, \fr25)$, $\mathsf{CL}_{M_\theta;\ne}^{\rm HL}$ can be bounded in the same way as $\mathsf{CL}_{M_\theta;\ne}^{\rm LH}$, see \eqref{CLLH}.

Compared with ${\sf CL}$, ${\sf L}$ is easier to treat, since $\fr{1}{|k|}$ hidden in $\hat{E}_k(t)$ will eliminate the extra half derivation on $\rho$ when $\rho$ is at high frequency and $\sig=\sig_0+1$. Indeed, for $\sig\in\{\sig_0+1,  \sig_1 \}$ and $k\ne0$, there holds
\be\label{elimi}
\fr{1}{|k|}\left\la k, kt^{\rm ap}\right\ra^{\sig}\les |k|^\fr12\left\la k, kt^{\rm ap}\right\ra^{\sig_0} \left\la t^{\rm ap}\right\ra.
\ee
Then \eqref{A-HL} reduces to
\be
 A_k(t,\eta){\bf 1}_{\left|\eta-kt^{\rm ap} \right|\le 2\left|k, kt^{\rm ap}\right|}
\les e^{c\lm(t)\la \eta-kt^{\rm ap}\ra^s} e^{-(\dl-\dl_1)\nu^\fr13t}\la t\ra{\bf B}_k(t).
\ee
Combining this with \eqref{ALH0}, similar to \eqref{CLLH}, we are led to 
\be
{\sf L}\les e^{-(\dl-\dl_1)\nu^\fr13 t}\fr{1}{\la t\ra^{\fr12+}} \left(\la t\ra^{\fr32+}\left\| {\bf B}\rho\right\|_{L^2_x}\right)\left( e^{-|\al|\nu t}\left\|A(v^\al f^w_{\ne})\right\|_{L^2}\right),
\ee
which is consistent with Proposition \ref{prop: btsp} provided we set ${\rm C}_f\gg {\rm C}_0$.

To bound $\mathsf{LE}_1$ and $\mathsf{LE}_2$, note that  for $\sig\in\{\sig_0+1,  \sig_1 \}$, there holds
\beno
\nb_\eta   A_k(t, \eta)=\left(s\lm(t)\fr{|k,\eta|\nb_\eta|\eta|}{\la k,\eta\ra^{2-s}}  +\sig\fr{|k,\eta|\nb_\eta|\eta|}{\la k,\eta\ra^2}+\fr{\nb_\eta\frak{m}_k(t,\eta)}{\frak{m}_k(t,\eta)}\right)A_k(t,\eta).
\eeno
If $k\ne0$, combining this with \eqref{m1'} yields
\be\label{A_eta1}
|\nb_\eta   A_k(t, \eta)|\les \left(\fr{1}{\la k,\eta\ra^{1-s}}+e^{\nu t}\nu^\fr13\right)A_k(t, \eta)\les e^{\nu t} A_k(t, \eta).
\ee
Whereas  for $k=0$, there holds
\begin{align}\label{A_eta2}
|\nb_\eta   A_0(t, \eta)|\les \fr{|\eta|}{\la \eta\ra^{2-s}}A_0(t, \eta)\les|\eta|A_0(t, \eta).
\end{align}
Now we treat $\mathsf{LE}_1$, which comes from the time-dependent Gaussian weight $e^{\lambda_1|v|^2}$. By choosing $|\dot\lambda_1(t)|$ small enough, we gain smallness even for this linear term. More precisely, by the definition of $\lm_1(t)$ in \eqref{def-lm1}, choosing $0<a_0\ll1$ but independent of $\nu$, we have
\beno
-\dot{\lm}_1(t)=a_0\nu \fr{t}{\la t\ra^{a_0+2}}\ll\nu.
\eeno
It follows from the above three inequalities that
\begin{align}\label{e-LE1}
\nn\mathsf{LE}_1
\nn\les&-\dot{\lm}_1(t)e^{-(2|\al|+1)\nu t}\left\|A(vv^\al f^w)\right\|_{L^2}\left(\|A(v^\al f^w_{\ne})\|_{L^2}+\left\|\pr_v^t A_0(v^\al f^w_0)\right\|_{L^2} \right)\\
\nn\le&\fr{1}{16}(-\dot{\lm}_1(t))e^{-2(|\al|+1)\nu t}\left\|A(vv^\al f^w)\right\|_{L^2}^2\\
\nn&+C(-\dot{\lm}_1(t))e^{-2|\al|\nu t}\left(\|A(v^\al f^w_{\ne})\|^2_{L^2}+\left\|\pr_v^t A_0(v^\al f^w_0)\right\|^2_{L^2} \right)\\
\nn\le&\fr{1}{16}(-\dot{\lm}_1(t))e^{-2(|\al|+1)\nu t}\left\|A(vv^\al f^w)\right\|_{L^2}^2\\
&+\fr{1}{16}e^{-2|\al|\nu t}\left(\nu^\fr13\|A(v^\al f^w_{\ne})\|^2_{L^2}+\nu\left\|\pr_v^t A_0(v^\al f^w_0)\right\|^2_{L^2} \right).
\end{align}
For $\mathsf{LE}_2$, which comes from the commutator $[v, A(t, \nabla)]$, the definition of $\lm_1(t)$ ensures that $\lm_1(t)\approx\nu\ll1$. Then similar to \eqref{e-LE1}, we obtain
\begin{align}\label{e-LE2}
\nn\mathsf{LE}_2\le&\fr{1}{16}\nu \lm_1(1-2\lm_1)e^{-2(|\al|+1)\nu t}\left\|A(vv^\al f^w)\right\|_{L^2}^2\\
&+\fr{1}{16}e^{-2|\al|\nu t}\left(\nu^\fr13\|A(v^\al f^w_{\ne})\|^2_{L^2}+\nu\left\|\pr_v^t A_0(v^\al f^w_0)\right\|^2_{L^2} \right).
\end{align}

Now we turn to $\mathsf{LE}_3$ and $\mathsf{LE}_4$. Here we gain the smallness from a well-chosen sequence of $K_{|\alpha|}$ in the spirit of \cite{bedrossian2017suppression},  which appears in the definition of our energy functional.

If $|\al|>1$, integrating by parts, $\mathsf{LE}_3$ can be rewritten as
\begin{align}
\nn \mathsf{LE}_3=&-4\nu\lm_1 e^{-\nu t}e^{-2|\al|\nu t}\frak{Re}\sum_{k\in\Z^n}\int_\eta \bar{\eta}(t;k,\eta)A \pr_\eta^\al( \widehat{f^w})_k\left(t, \eta \right)\cdot A\nb_\eta\pr^\al_\eta(\overline{\widehat{f^w}})_k(t,\eta) d\eta\\
\nn&-8\nu\lm_1 e^{-\nu t}e^{-2|\al|\nu t}\frak{Re}\sum_{k\in\Z^n}\int_\eta \bar{\eta}(t;k,\eta)A \pr_\eta^\al( \widehat{f^w})_k\left(t, \eta \right)\cdot\nb_\eta A\pr^\al_\eta(\overline{\widehat{f^w}})_k(t,\eta) d\eta\\
\nn&-4\nu\lm_1 e^{-2|\al|\nu t}\sum_{|\beta|=1}\begin{pmatrix}\al\\ \beta\end{pmatrix}\frak{Re}\sum_{k\in\Z^n}\int_\eta A \pr_\eta^{\al-\beta}( \widehat{f^w})_k\left(t, \eta \right)\pr_\eta^\beta\eta\cdot A\nb_\eta\pr^\al_\eta(\overline{\widehat{f^w}})_k(t,\eta) d\eta\\
\nn&-8\nu\lm_1 e^{-2|\al|\nu t}\sum_{|\beta|=1}\begin{pmatrix}\al\\ \beta\end{pmatrix}\frak{Re}\sum_{k\in\Z^n}\int_\eta A \pr_\eta^{\al-\beta}( \widehat{f^w})_k\left(t, \eta \right)\pr_\eta^\beta\eta\cdot\nb_\eta A\pr^\al_\eta(\overline{\widehat{f^w}})_k(t,\eta) d\eta\\
\nn=&\sum_{1\le j\le4}\mathsf{LE}_{3;(j)}.
\end{align}
It is easy to see that
\begin{align}\label{e-LE312'}
\mathsf{LE}_{3;(1)}
\nn\le&\fr{1}{16}\nu \lm_1(1-2\lm_1)e^{-2(|\al|+1)\nu t}\left\|A(vv^\al f^w)\right\|_{L^2}^2\\
&+\fr{1}{16}e^{-2|\al|\nu t}\left(\nu\left\|\pr_v^t A(v^\al f^w)\right\|^2_{L^2} \right),
\end{align}
and
\begin{align}\label{e-LE334'}
\mathsf{LE}_{3;(3)}
\nn\le&\fr{1}{16}\nu \lm_1(1-2\lm_1)e^{-2(|\al|+1)\nu t}\left\|A(vv^\al f^w)\right\|_{L^2}^2\\
&+CK_{|\al|-1}\nu \lm_1\sum_{|\beta|=1}\fr{e^{-2(|\al|-1)\nu t}}{K_{|\al|-1}}\left\|A(v^{\al-\beta}f^w)\right\|_{L^2}^2.
\end{align}
Using \eqref{A_eta1} and \eqref{A_eta2}, we have
\begin{align}\label{e-LE312}
\mathsf{LE}_{3;(2)}\nn\le&8\nu\lm_1 e^{-2|\al|\nu t}\|\pr_v^tA(v^\al f)\|_{L^2}\left(\|A(v^\al f^w_{\ne})\|_{L^2}+\left\|\pr_v^t A_0(v^\al f^w_0)\right\|_{L^2} \right) \\
\le&\fr{1}{16}e^{-2|\al|\nu t}\left(\nu^\fr13\|A(v^\al f^w_{\ne})\|^2_{L^2}
+\nu\left\|\pr_v^t A(v^\al f^w)\right\|^2_{L^2} \right),
\end{align}
and
\begin{align}\label{e-LE334}
\mathsf{LE}_{3;(4)}
\nn\les&\nu\lm_1 e^{-(2|\al|-1)\nu t}\sum_{|\beta|=1}\|A(v^{\al-\beta}f^w)\|_{L^2}\left(\|A(v^\al f^w_{\ne})\|_{L^2}+\left\|\pr_v^t A_0(v^\al f^w_0)\right\|_{L^2} \right)\\
\nn\le&\fr{1}{16}e^{-2|\al|\nu t}\left(\nu^\fr13\|A(v^\al f^w_{\ne})\|^2_{L^2}
+\nu\left\|\pr_v^t A(v^\al f^w)\right\|^2_{L^2} \right)\\
&+CK_{|\al|-1}\nu \lm_1\sum_{|\beta|=1}\fr{e^{-2(|\al|-1)\nu t}}{K_{|\al|-1}}\left\|A(v^{\al-\beta}f^w)\right\|_{L^2}^2.
\end{align}
Thanks to the fact $|\al|>|\beta|=1$, dividing \eqref{en-f-H}  by $K_{|\al|}$ and  summing with respect to $\al$, then the resulting summation arising from the last term on the right-hand side of \eqref{e-LE334'} and \eqref{e-LE334} can be absorbed by   
\be\label{damp-nu}
\nu\sum_{1\le|\al|\le m'}\fr{e^{-2|\al|\nu t}}{K_{|\al|}}\|A(v^\al f^w)\|_{L^2}^2,\quad {\rm with }\quad m'\in\{m, m+2\},
\ee
by choosing $K_{|\al|}\gg K_{|\al|-1}$ (even though ignoring the smallness in $\lm_1$).

If $|\al|=0$, then ${\sf LE}_{3;(3)}$ and ${\sf LE}_{3,(4)}$ will not appear, we just need to estimate $\sf{LE}_3$ by \eqref{e-LE312'} and \eqref{e-LE312}. If $|\al|=1$, we have
\begin{align}
\nn \mathsf{LE}_3=&{\sf LE}_{3;(1)}+{\sf LE}_{3;(2)}+4\nu\lm_1e^{-2|\al|\nu t}\left\|A(v^\al f^w)\right\|_{L^2}^2.
\end{align}
The last term  above can be absorbed directly by $\nu|\al| e^{-2|\al|\nu t}\| A(v^\al f^w)\|_{L^2}^2$ on the left hand side of \eqref{en-f-H}.

For $\mathsf{LE}_4$, if $|\al|>2$, we write
\begin{align}\label{LE4}
\mathsf{LE}_4
\nn=&-2\nu e^{\nu t} e^{-2|\al|\nu t} \sum_{ |\beta|=1}\begin{pmatrix}\al\\ \beta\end{pmatrix}\frak{Re}\sum_{k\in\mathbb{Z}^n}\int_\eta \bar{\eta}(t;k,\eta)^\beta A\pr_\eta^{\al-\beta}(\widehat{f^w})_k(t,\eta)\\
\nn&\times A\pr_\eta^\al(\overline{\widehat{f^w}})_k(t,\eta)d\eta\\
\nn&-2\nu e^{2\nu t}e^{-2|\al|\nu t} \sum_{ |\beta|=2}\begin{pmatrix}\al\\ \beta\end{pmatrix}c_\beta \frak{Re}\sum_{k\in\mathbb{Z}^n}\int_\eta A\pr_\eta^{\al-\beta}(\widehat{f^w})_k(t,\eta)A\pr_\eta^\al(\overline{\widehat{f^w}})_k(t,\eta)d\eta\\
=&\mathsf{LE}_{4;(1)}+\mathsf{LE}_{4;(2)},
\end{align}
where
\[
c_\beta=\begin{cases}
1, \quad \mathrm{if\ \ only\ \ one\ \ component\ \ of\ \ } \beta\ \ \mathrm{is\ \ nonzero},\\
0, \quad\mathrm{other wise}.
\end{cases}
\]
Then similar to \eqref{e-LE334}, one deduces that
\begin{align}\label{e-LE4}
\mathsf{LE}_4\nn\le&\fr{1}{16}\nu e^{-2|\al|\nu t}\left(\left\|\pr_v^tA(v^\al f^w)\right\|_{L^2}^2+{\bf 1}_{|\al|>2}\left\|A(v^\al f^w)\right\|_{L^2}^2\right)\\
\nn&+CK_{|\al|-1}\nu \sum_{ |\beta|=1}\fr{e^{-2(|\al|-1)\nu t}}{K_{|\al|-1}}\left\|A(v^{\al-\beta}f^w)\right\|_{L^2}^2\\
&+CK_{|\al|-2}\nu \sum_{ |\beta|=2}\fr{e^{-2(|\al|-2) \nu t} }{K_{|\al|-2}}c_\beta\left\|A(v^{\al-\beta}f^w)\right\|_{L^2}^2.
\end{align}
Since $|\al|>2$, $|\al-\beta|>0$ for $|\beta|=1, 2$. Then by choosing $K_{|\al|}\gg K_{|\al|-1}\gg K_{|\al|-2}$, the last two terms on the right hand side of \eqref{e-LE4} can be treated with the aid of \eqref{damp-nu}  analogously to \eqref{e-LE334}.

If $|\al|=1$, we have
\begin{align}\label{e-LE4'}
{\sf LE}_4\nn=&-2\nu e^{\nu t} e^{-2|\al|\nu t}\frak{Re}\sum_{k\in\Z^n}\int_\eta\bar{\eta}^\al(t;k,\eta)A(\widehat{f^w})_k(t,\eta)A\pr_\eta^\al(\overline{\widehat{f^w}})_k(t,\eta)d\eta\\
\nn\le&2\nu \left\|\pr_v^tA f^w\right\|_{L^2} \left(e^{-|\al|\nu t}\left\|A(v^\al f^w)\right\|_{L^2}\right)\\
\le& \fr{1}{16}\nu e^{-2|\al|\nu t}\left\|A(v^\al f^w)\right\|_{L^2}^2+16K_{0}\left(\nu \fr{1}{K_{0}}\left\|\pr_v^tA f^w\right\|_{L^2}^2\right).
\end{align}
The first term on the right hand side of \eqref{e-LE4'} can be absorbed  by $\nu|\al| e^{-2|\al|\nu t}\| A(v^\al f^w)\|_{L^2}^2$ on the left hand side of \eqref{en-f-H}. After dividing \eqref{en-f-H} by $K_{|\al|}$ and summing with respect to $\al$, the second term on the right hand side of \eqref{e-LE4'} can be absorbed by
\be\label{diffusion}
\nu\sum_{|\al|\le m'}\fr{e^{-2|\al|\nu t}}{K_{|\al|}}\left\|\pr_v^tA(v^\al f^w)\right\|_{L^2}^2,\quad {\rm with }\quad m'\in\{m, m+2\},
\ee
as long as $K_1\gg K_0$.

If $|\al|=2$, ${\sf LE}_{4;(1)}$ can be treated in the same way as  \eqref{e-LE4}. Note that ${\sf LE}_{4;(2)}$ will not appear unless $\al=(j,j), j= 1, 2, \cdots n$. Integrating by parts,  ${\sf LE}_{4;(2)}$ reduces to
\begin{align}
{\sf LE}_{4;(2)}=\nn&-2\nu e^{-2(|\al|-1)\nu t} c_\al \frak{Re}\sum_{k\in\mathbb{Z}^n}\int_\eta A(\widehat{f^w})_k(t,\eta)A\pr_\eta^\al(\overline{\widehat{f^w}})_k(t,\eta)d\eta\\
\nn=&2K_{|\al|-1}\nu \fr{e^{-2(|\al|-1)\nu t}}{K_{|\al|-1}} c_\al \left\|A (v_j f^w)\right\|_{L^2}^2\\
\nn&+4\nu e^{-2(|\al|-1)\nu t} c_\al \frak{Re}\sum_{k\in\mathbb{Z}^n}\int_\eta \pr_j^\eta A_k(t,\eta)(\widehat{f^w})_k(t,\eta)A\pr^\eta_j(\overline{\widehat{f^w}})_k(t,\eta)d\eta\\
=&{\sf LE}_{4;(2,1)}+{\sf LE}_{4;(2,2)},
\end{align}
for some $j\in\{ 1, 2, \cdots, n\}$. Using \eqref{A_eta1} and \eqref{A_eta2}  and the property of $\frak{m}_k(t,\eta)$ in \eqref{m2}, we find that
\begin{align}\label{LE422}
{\sf LE}_{4;(2,2)}\nn\les& \nu e^{-\nu t} c_\al\left\|A (v_j f^w)\right\|_{L^2}\left(\left\|Af^w_{\ne}\right\|_{L^2}+\left\|\pr_v^t Af^w_0\right\|_{L^2} \right)\\
\nn\le&CK_{1}c_\al\nu \fr{e^{-2\nu t}}{K_{1}} \left\|A (v_j f^w)\right\|_{L^2}^2
+CK_{0}c_\al \fr{1}{K_{0}} \left(\nu\left\|Af^w_{\ne}\right\|_{L^2}^2+\nu\left\|\pr_v^t Af^w_0\right\|_{L^2}^2 \right)\\
\nn\le&CK_{1}c_\al\nu \fr{e^{-2\nu t}}{K_{1}} \left\|A (v_j f^w)\right\|_{L^2}^2\\
&+CK_{0}c_\al \fr{1}{K_{0}} \left(\left\|\sqrt{-\fr{\pr_t\frak{m}}{\frak{m}}} A( f_{\neq}^w)\right\|_{L^2}^2+\nu\left\|\pr_v^tA(f^w)\right\|_{L^2}^2 \right).
\end{align}
${\sf LE}_{4;(2,1)}$ and the first term on the right hand side of \eqref{LE422} can be treated with the aid of \eqref{damp-nu}.  The last term on the right hand side of \eqref{LE422} can be treated with the aid of the good terms \eqref{diffusion} and 
\be
\sum_{|\al|\le m'}\fr{e^{-2|\al|\nu t}}{K_{|\al|}}\left\|\sqrt{-\fr{\pr_t\frak{m}}{\frak{m}}} A(v^\al f^w)\right\|_{L^2}^2\quad {\rm with }\quad m'\in\{m, m+2\}.
\ee

\subsection{Nonlinear contributions}\label{sec-non-1}
In this section, we consider the nonlinear contributions which contains the collisionless contributions and the collision contributions. For the collisionless contributions, we take advantage of the transport structure. For the collision contribution, the main difficulty is to treat the velocity localization loss. 
\subsubsection{Collisionless contributions}
To begin with, let us focus on the estimate of  $\mathsf{NE}_E$. In fact,
using \eqref{A-LH}, \eqref{A-HL}, \eqref{Young1} and \eqref{Young2}, and  \eqref{elimi} with $k$ replaced by $l$, we have
\begin{align}
\mathsf{NE}_E
=\nn&2\lm_1 e^{-\nu t}e^{-2|\al|\nu t}\sum_{k,l\in\Z^n}\int_\eta ( A\pr^\al_\eta \overline{\widehat{f^w}})_k(t,\eta)  A_k(t,\eta)\hat{\rho}_l(t)\fr{l}{|l|^2}\cdot  (\nb_\eta\pr^\al_\eta\widehat{f^w})_{k-l}\left(t,\eta-lt^{\rm ap}\right)d\eta\\
\nn\les&\lm_1 e^{-\nu t}e^{-2|\al|\nu t}\| A(v^\al f^w)\|_{L^2}\left\|  A(vv^\al f^w)\right\|_{L^2}
 e^{-(\dl-\dl_1)\nu^\fr13t}  \left\|  \la k\ra^{\fr{n}{2}+} e^{\dl \nu^\fr13t}e^{c\lm(t)\la k,kt^{\rm ap}\ra^s}\hat{\rho}_k(t)  \right\|_{L^2_k}\\
\nn&+\lm_1 e^{-\nu t}e^{-2|\al|\nu t}\| A(v^\al f^w)\|_{L^2}e^{-(\dl-\dl_1)\nu^\fr13t}\left\la t^{\rm ap} \right\ra\left\|{\bf B}     \rho\right\|_{L^2_x}\\
\nn&\times\left\|\la k\ra^{\fr{n}{2}+} e^{c\lm(t)\la k,\eta\ra^s}(\nb_\eta\pr^\al_\eta\widehat{f^w})_k(t,\eta)\right\|_{L^2}\\
\nn\le&\fr{1}{16}\nu\lm_1(1-2\lm_1) e^{-2(|\al|+1)\nu t}\left\|  A(vv^\al f^w)\right\|_{L^2}^2
+C\left(\la t \ra\left\|{\bf B}     \rho\right\|_{L^2_x}\right)^2\left(e^{-2|\al|\nu t}\| A(v^\al f^w)\|_{L^2}^2\right).
\end{align}

Next we turn to the estimates of $\mathsf{N}$, which can be split into two parts:
\begin{align}
\mathsf{N}=\nn&-e^{-2|\al|\nu t}\frak{Re}\sum_{k,l\in\mathbb{Z}^n}\int_\eta A_k(t,\eta)\hat{E}_l(t)\cdot i\left[\bar{\eta}(t; k,\eta)\pr_\eta^\al(\widehat{f^w})_{k-l}\left(t, \eta-lt^{\rm ap}\right)\right]\\
\nn&\times\left(A\pr_\eta^\al(\overline{\widehat{f^w}})\right)_k(t,\eta)d\eta\\
\nn&-ie^{\nu t}e^{-2|\al|\nu t}\sum_{\beta\le \al:|\beta|=1} \begin{pmatrix}\al\\ \beta\end{pmatrix}\frak{Re}\sum_{k, l\in\mathbb{Z}^n}\int_\eta A_k(t,\eta)\hat{E}_l^\beta(t)\pr_\eta^{\al-\beta}(\widehat{f^w})_{k-l}\left(t, \eta-lt^{\rm ap}\right)\\
\nn&\times A\pr_\eta^\al(\overline{\widehat{f^w}})_k(t,\eta)d\eta
=\mathsf{N}_M+\mathsf{N}_P.
\end{align}
We first treat ${\sf N}_M$. Two cases will be investigated according to ${\frak e}=\fr25$ or $\fr13$. It is worth pointing out that when $(\sig,\frak{e})=( \sig_1 , \fr25)$, we cannot use the ${\bf A}^{\sig_0+1,\fr25}$ weighted  high norm to control the ${\bf A}^{ \sig_1 ,\fr25}$ weighted low norm, since the latter is assigned more velocity localizations.\\
\noindent
{\bf Case 1: $A(t,\nb)\in \left\{{\bf A}^{\sig_0+1,\fr25}(t, \nb), {\bf A}^{ \sig_1 ,\fr25}(t,\nb)\right\}$.}
By using a classical commutator trick and frequency decomposition, we rewrite $\mathsf{N}_M$ as
\begin{align}
\mathsf{N}_M\nn=&-e^{-2|\al|\nu t}\frak{Re}\sum_{k\in\mathbb{Z}^n}\sum_{l\in\mathbb{Z}^n_*}\int_\eta \left(A_k(t,\eta)-A_{k-l}\left(t,\eta-lt^{\rm ap}\right)\right)\\
\nn&\times\hat{E}_l(t)\cdot \left[i\bar{\eta}(t; k,\eta)\pr_\eta^\al(\widehat{f^w})_{k-l}\left(t, \eta-lt^{\rm ap}\right)\right]A\pr_\eta^\al(\overline{\widehat{f^w}})_k(t,\eta)d\eta
=:\mathsf{N}_M^{\mathrm{LH}}+\mathsf{N}_M^{\mathrm{HL}},
\end{align}
where the splitting above is given according to the  partition of unity \eqref{pou}.

\noindent\uline{Treatment of $\mathsf{N}_M^{\mathrm{LH}}$.} To begin with, it is natural to divide the commutator into the following four parts:
\begin{align}
\nn &\fr{A_k(t,\eta)}{A_{k-l}\left(t,\eta-lt^{\rm ap}\right)}-1\\
\nn=&e^{\lm(t)\left(\la k,\eta\ra^s-\la k-l, \eta-lt^{\rm ap}\ra^s\right)}-1\\
\nn&+e^{\lm(t)\left(\la k,\eta\ra^s-\la k-l, \eta-lt^{\rm ap}\ra^s\right)}\left(\fr{\la k,\eta\ra^{\sigma}}{\la k-l,\eta-lt^{\rm ap}\ra^{\sigma}}-1\right)\\
\nn&+e^{\lm(t)\left(\la k,\eta\ra^s-\la k-l, \eta-lt^{\rm ap}\ra^s\right)}\fr{\la k,\eta\ra^{\sigma}}{\la k-l,\eta-lt^{\rm ap}\ra^{\sigma}}\left(\fr{\frak{m}_k(t,\eta)}{\frak{m}_{k-l}\left(t,\eta-lt^{\rm ap}\right)}-1\right)\\
\nn&+e^{\lm(t)\left(\la k,\eta\ra^s-\la k-l, \eta-lt^{\rm ap}\ra^s\right)}\fr{\la k,\eta\ra^{\sigma}}{\la k-l,\eta-lt^{\rm ap}\ra^{\sigma}}\fr{\frak{m}_k(t,\eta)}{\frak{m}_{k-l}\left(t,\eta-lt^{\rm ap}\right)}\left(\fr{e^{{\bf 1}_{k\ne0}\dl_1\nu^\fr25 t}}{e^{{\bf 1}_{k\ne l}\dl_1\nu^\fr25t}}-1\right)\\
\nn=&{\bf com}_1+{\bf com}_2+{\bf com}_3+{\bf com}_4,
\end{align}
 for $\sig\in\{\sig_0+1,  \sig_1 \}$.  Then we denote
\[
\mathsf{N}_M^{\mathrm{LH}}=\mathsf{N}_{M;1}^{\mathrm{LH}}+\mathsf{N}_{M;2}^{\mathrm{LH}}+\mathsf{N}_{M;3}^{\mathrm{LH}}+\mathsf{N}_{M;4}^{\mathrm{LH}}.
\]
By virtue of the elementary inequality $|e^x-1|\le |x|e^{|x|}$, the mean value theorem,  \eqref{app2}, \eqref{app3}, and the fact
\be\label{up-eta-bar}
|\bar{\eta}(t; k,\eta)|=e^{\nu t}\left|\left(\eta-lt^{\rm ap}\right)-(k-l)t^{\rm ap} \right|
\les e^{\nu t} \left\la t^{\rm ap}\right\ra \left| k-l, \eta-lt^{\rm ap}\right|,
\ee
we are led to
\begin{align}
\nn&|{\bf com}_1+{\bf com}_2|  |\bar{\eta}(t; k,\eta)|{\bf 1}_{\left|l, lt^{\rm ap}\right|<\fr12\left|k-l,\eta-lt^{\rm ap} \right|}\\
\nn\les&\left|l,lt^{\rm ap}\right|e^{c\lm(t)\la l,lt^{\rm ap}\ra^s}e^{\nu t} \left\la t^{\rm ap}\right\ra \left\la k-l, \eta-lt^{\rm ap}\right\ra^{\fr{s}{2}} \la k,\eta\ra^{\fr{s}{2}}.
\end{align}
Combining this with \eqref{tau-up},  \eqref{Young1} and Lemma \ref{lem-f-f} yields for $j=1, 2$
\begin{align}\label{e-com1}
\mathsf{N}_{M;j}^{\mathrm{LH}}
\nn\les&e^{-2|\al|\nu t}\sum_{k\in\mathbb{Z}^n,\, l\in\mathbb{Z}^n_*}\int_\eta \left(e^{\dl_1\nu^\fr13t}e^{c\lm(t)\la l,lt^{\rm ap}\ra^s}  \left\la t^{\rm ap}\right\ra^{3+a}| \hat{\rho}_l(t)|\right) \left(   \la \nu t\ra^{a+1}  e^{\nu t} e^{-\dl_1\nu^\fr13t}\right)\\
\nn&\times \fr{1}{\la  t\ra^{a+1}}\left|A\mathcal{F}\left[\la \nb\ra^{\fr{s}{2}}(v^\al f^w)\right]_{k-l}\left(t, \eta-lt^{\rm ap}\right)\right|\left|A\mathcal{F}\left[\la \nb\ra^{\fr{s}{2}}(v^\al f^w)\right]_k(t,\eta)\right|d\eta\\
\nn\les&-e^{-2|\al|\nu t}\dot{\lm}(t)\left\|\la \nb\ra^\fr{s}{2}A(v^\al f^w)\right\|_{L^2}^2\left(e^{-m\nu t}\left\| e^{\dl_1\nu^\fr13 t} \mathcal{F}\left[e^{\lm(t)\la\nb\ra^s}f_{\ne}\right](t)\right\|_{L^2_kL^\infty_\eta}\right)\\
\les&-e^{-2|\al|\nu t}\dot{\lm}(t)\left\|\la \nb\ra^\fr{s}{2}A(v^\al f^w)\right\|_{L^2}^2 \|f^w(t)\|_{\mathcal{E}^{ \sig_1 ,\fr13}_m},
\end{align}
where we have used the relation $\hat{\rho}_l(t)={\bf 1}_{l\ne0}\hat{f}_l\left(t,lt^{\rm ap}\right)$.

Next we turn to  $\mathsf{N}_{M;3}^{\mathrm{LH}}$. By \eqref{app3} and \eqref{up-eta-bar}, we have
\begin{align}
\mathsf{N}_{M;3}^{\mathrm{LH}}\nn\les&e^{-2|\al|\nu t}\sum_{k\in\mathbb{Z}^n,\, l\in\mathbb{Z}^n_*}\int_\eta e^{c\lm(t)\la l,lt^{\rm ap}\ra^s}\fr{|\hat{\rho}_l(t)|}{|l|}e^{\nu t} \left\la t^{\rm ap}\right\ra \left| k-l, \eta-lt^{\rm ap}\right|\\
\nn&\times \left|\fr{\frak{m}_k(t,\eta)}{\frak{m}_{k-l}\left(t,\eta-lt^{\rm ap}\right)}-1\right|\left|A\pr_\eta^\al(\widehat{f^w})_{k-l}\left(t, \eta-lt^{\rm ap}\right)\right|\left|A\pr_\eta^\al(\widehat{f^w})_k(t,\eta)\right|d\eta\\
=&\Big(\sum_{\substack{k\in\mathbb{Z}^n_*,\, l\in\mathbb{Z}^n_*, \\ k\ne l}}+\sum_{k=0,\, l\in\mathbb{Z}^n_*}+\sum_{\substack{k\in\mathbb{Z}^n_*,\, l\in\mathbb{Z}^n_*, \\ k= l}}\Big)\int_\eta\cdots d\eta=\mathsf{N}_{M;3,(1)}^{\mathrm{LH}}+\mathsf{N}_{M;3,(2)}^{\mathrm{LH}}+\mathsf{N}_{M;3,(3)}^{\mathrm{LH}}.
\end{align}
To bound $\mathsf{N}_{M;3,(1)}^{\mathrm{LH}}$, we further split it into two parts:
\begin{align}
\mathsf{N}_{M;3,(1)}^{\mathrm{LH}}\nn=&e^{-2|\al|\nu t}\sum_{\substack{k\in\mathbb{Z}^n_*,\, l\in\mathbb{Z}^n_*, \\ k\ne l}}\int_\eta{\bf 1}_{\left|\eta-lt^{\rm ap}\right|\le2|k-l|\left\la  t^{\rm ap}\right\ra}\cdots d\eta\\
\nn&+e^{-2|\al|\nu t}\sum_{\substack{k\in\mathbb{Z}^n_*,\, l\in\mathbb{Z}^n_*, \\ k\ne l}}\int_\eta{\bf 1}_{\left|\eta-lt^{\rm ap}\right|>2|k-l|\left\la t^{\rm ap}\right\ra}\cdots d\eta=\mathsf{N}_{M;3,(1)}^{\mathrm{LH};x}+\mathsf{N}_{M;3,(1)}^{\mathrm{LH};v}.
\end{align}
In view of  \eqref{nenene1} and \eqref{Young1}, similar to \eqref{e-com1},  we find that
\begin{align}
\mathsf{N}_{M;3,(1)}^{\mathrm{LH};x}\nn\les&e^{-2|\al|\nu t}\sum_{k\in\mathbb{Z}^n,\, l\in\mathbb{Z}^n_*}\int_\eta e^{c\lm(t)\la l,lt^{\rm ap}\ra^s}|\hat{\rho}_l(t)|e^{\nu t} \left\la t^{\rm ap}\right\ra^2\\
\nn&\times   \left|A\pr_\eta^\al(\widehat{f^w})_{k-l}\left(t, \eta-lt^{\rm ap}\right)\right|\left|A\pr_\eta^\al(\widehat{f^w})_k(t,\eta)\right|d\eta\\
\nn\les&\fr{1}{\la t\ra^{1+}}\left(e^{-m\nu t}\left\| e^{\dl_1\nu^\fr13 t}\mathcal{F}\left[e^{\lm(t)\la \nb\ra^s}f_{\ne}\right]\right\|_{L^2_kL^\infty_\eta}\right)\left(e^{-|\al|\nu t}\|A(v^\al f^w)\|_{L^2}\right)^2\\
\nn\les&\fr{1}{\la t\ra^{1+}} \|f^w(t)\|_{\mathcal{E}^{ \sig_1 ,\fr13}_m}\left(e^{-|\al|\nu t}\|A(v^\al f^w)\|_{L^2}\right)^2.
\end{align}
Similarly, for $\mathsf{N}_{M;3,(1)}^{\mathrm{LH};v}$, we use \eqref{nenene2}, \eqref{tau-up} and \eqref{Young1} to get
\begin{align}
\mathsf{N}_{M;3,(1)}^{\mathrm{LH};v}\nn\les&e^{-2|\al|\nu t}\sum_{k\in\mathbb{Z}^n,\, l\in\mathbb{Z}^n_*}\int_\eta e^{c\lm(t)\la l,lt^{\rm ap}\ra^s}\fr{|\hat{\rho}_l(t)|}{|l|}e^{2\nu t} \left\la t^{\rm ap}\right\ra t \left\la l,lt^{\rm ap}\right\ra\\
\nn&\times  \left|A\pr_\eta^\al(\widehat{f^w})_{k-l}\left(t, \eta-lt^{\rm ap}\right)\right|\left|A\pr_\eta^\al(\widehat{f^w})_k(t,\eta)\right|d\eta\\
\nn\les&\fr{1}{\la t\ra^{1+}} \|f^w(t)\|_{\mathcal{E}^{ \sig_1 ,\fr13}_m}\left(e^{-|\al|\nu t}\|A(v^\al f^w)\|_{L^2}\right)^2.
\end{align}
Next consider  $\mathsf{N}_{M;3,(2)}$. In fact,  from \eqref{0ne}, we obtain
\begin{align}
\mathsf{N}_{M;3,(2)}^{\mathrm{LH}}\nn\les&e^{-2|\al|\nu t}\sum_{ l\in\mathbb{Z}^n_*}\int_\eta e^{c\lm(t)\la l,lt^{\rm ap}\ra^s}\fr{|\hat{\rho}_l(t)|}{|l|}e^{\nu t} \left\la t^{\rm ap}\right\ra \left( |l|+\left |\eta-lt^{\rm ap}\right|\right)\\
\nn&\times  \min\left\{\fr{1}{|l|}, \fr{\la t\ra}{\left|\eta-lt^{\rm ap} \right|}\right\} \left|A\pr_\eta^\al(\widehat{f^w})_{-l}\left(t, \eta-lt^{\rm ap}\right)\right|\left|A\pr_\eta^\al(\widehat{f^w})_0(t,\eta)\right|d\eta\\
\nn\les&\fr{1}{\la t\ra^{1+}} \|f^w(t)\|_{\mathcal{E}^{ \sig_1 ,\fr13}_m}\left(e^{-|\al|\nu t}\|A(v^\al f^w_{\ne})\|_{L^2}\right)\left(e^{-|\al|\nu t}\|A(v^\al f^w_0)\|_{L^2}\right).
\end{align}
Similarly, we infer from \eqref{ne0} that
\begin{align}
\mathsf{N}_{M;3,(3)}^{\mathrm{LH}}\nn\les&e^{-2|\al|\nu t}\sum_{ k\in\mathbb{Z}^n_*}\int_\eta e^{c\lm(t)\la k,kt^{\rm ap}\ra^s}\fr{|\hat{\rho}_k(t)|}{|k|}e^{\nu t} \left\la t^{\rm ap}\right\ra \left(  |\eta|+\left|kt^{\rm ap}\right|\right)\\
\nn&\times\min\left\{\fr{1}{|k|}, \fr{\la t \ra}{|\eta|}\right\}\left|A\pr_\eta^\al(\widehat{f^w})_{0}\left(t, \eta-kt^{\rm ap}\right)\right|\left|A\pr_\eta^\al(\widehat{f^w})_k(t,\eta)\right|d\eta\\
\nn\les&\fr{1}{\la t\ra^{1+}} \|f^w(t)\|_{\mathcal{E}^{ \sig_1 ,\fr13}_m}\left(e^{-|\al|\nu t}\|A(v^\al f^w_{\ne})\|_{L^2}\right)\left(e^{-|\al|\nu t}\|A(v^\al f^w_0)\|_{L^2}\right).
\end{align}

Now we turn to $\mathsf{N}_{M;4}^{\mathrm{LH}}$, in which case the commutator ${\bf com}_4$ does not gain regularity and we use a interpolation argument. In fact, it is easy to see that
\begin{align*}
|{\bf com}_4|  {\bf 1}_{\left|l, lt^{\rm ap}\right|<\fr12\left|k-l,\eta-l\fr{1-e^{\nu t}}{\nu} \right|}
\les e^{c\lm(t)\la l,lt^{\rm ap}\ra^s}\left({\bf 1}_{k\ne0, k=l}\dl_1\nu^\fr25 te^{\dl_1\nu^\fr25t}+{\bf 1}_{ k=0, k\ne l}\dl_1\nu^\fr25t \right).
\end{align*}
Then
\begin{align}
\mathsf{N}_{M;4}^{\mathrm{LH}}\nn\les&e^{-2|\al|\nu t}\sum_{k\in\mathbb{Z}^n_*}\int_\eta e^{c\lm(t)\la k,kt^{\rm ap}\ra^s}\nu^\fr{2}{5} te^{\dl_1\nu^\fr25t}\fr{|\hat{\rho}_k(t)|}{|k|}\\
\nn&\times |\bar{\eta}(t;k,\eta)| \left|A\widehat{(v^\al f^w)}_{0}\left(t, \eta-kt^{\rm ap}\right)\right|\left|A\pr_\eta^\al(\widehat{f^w})_k(t,\eta)\right|d\eta\\
\nn&+e^{-2|\al|\nu t}\sum_{k\in\mathbb{Z}^n_*}\int_\eta e^{c\lm(t)\la k,kt^{\rm ap}\ra^s}\nu^\fr25 t\fr{|\hat{\rho}_k(t)|}{|k|}\\
\nn&\times |\bar{\eta}(t;0,\eta)| \left|A\widehat{(v^\al f^w)}_{-k}\left(t, \eta-kt^{\rm ap}\right)\right|\left|A\pr_\eta^\al(\widehat{f^w})_0(t,\eta)\right|d\eta
=:\mathsf{N}_{M;4,(1)}^{\mathrm{LH}}+\mathsf{N}_{M;4,(2)}^{\mathrm{LH}}.
\end{align}
If $|\eta|\le\left|kt^{\rm ap}\right|$, then
$
|\bar{\eta}(t; k,\eta)|=e^{\nu t}\left|\eta-kt^{\rm ap}\right|\les e^{\nu t}\left|kt^{\rm ap}\right|.
$
Combining this with \eqref{tau-up} yields
\begin{align}\label{NMLH41-1}
\mathsf{N}_{M;4,(1)}^{\mathrm{LH}}\nn\les&\fr{\nu^\fr{2}{5}}{\la t\ra}e^{-2|\al|\nu t}\sum_{k\in\mathbb{Z}^n_*}\int_\eta \left(e^{\dl \nu^\fr13 t}e^{c\lm(t)\la k,kt^{\rm ap}\ra^s} \left|kt^{\rm ap}\right|^3|\hat{\rho}_k(t)|\right) \\
\nn&\times \left(\la\nu t\ra^2e^{\nu t+\dl_1\nu^\fr25t-\dl\nu^\fr13 t}\right) \left|A\widehat{(v^\al f^w)}_{0}\left(t, \eta-kt^{\rm ap}\right)\right|\left|A\pr_\eta^\al(\widehat{f^w})_k(t,\eta)\right|d\eta\\
\les&\fr{\nu^\fr{2}{5}}{\la t\ra}\left\|{\bf B}\rho(t)\right\|_{L^2_x}\left(e^{-|\al|\nu t}\left\| A(v^\al f^w)\right\|_{L^2}\right)^2.
\end{align}
If $|\eta|>\left|kt^{\rm ap}\right|$, then
\begin{align}\label{interpolation1}
\nn&|\bar{\eta}(t;k,\eta)| \left|A\widehat{(v^\al f^w)}_{0}\left(t, \eta-kt^{\rm ap}\right)\right|\\
\nn\les&e^{\fr{s}{2}\nu t}|\eta|^{\fr{s}{2}}    |\bar{\eta}(t; k,\eta)|^{1-\fr{s}{2}} \left|A\widehat{(v^\al f^w)}_{0}\left(t, \eta-kt^{\rm ap}\right)\right|\\
\nn=&e^{\fr{s}{2-s}\nu t} \nu^{-\fr{1-s}{2-s}} \la t\ra^{\fr{a+1}{2-s}} \fr{|\eta|^{\fr{s}{2}}}{\la t\ra^{\fr{a+1}{2}}}   \left(\fr{1}{\la t\ra^{\fr{a+1}{2}}} \left|\mathcal{F}\left[|\pr_v|^\fr{s}{2}A(v^\al f^w)\right]_{0}\left(t, \eta-kt^{\rm ap}\right)\right|\right)^{\fr{s}{2-s}}\\
\nn&\times \left(\nu^\fr12\left|\bar{\eta}\left(t; 0,\eta-kt^{\rm ap}\right)\right| \left|A\widehat{(v^\al f^w)}_{0}\left(t, \eta-kt^{\rm ap}\right)\right|\right)^{\fr{2-2s}{2-s}}\\
=&e^{\fr{s}{2-s}\nu t} \nu^{-\fr{1-s}{2-s}} \la t\ra^{\fr{a+1}{2-s}} \fr{|\eta|^{\fr{s}{2}}}{\la t\ra^{\fr{a+1}{2}}}\left(*\right),
\end{align}
where we have used the fact
\be\label{eta-bar-1}
\bar{\eta}\left(t; 0,\eta-kt^{\rm ap}\right)=\bar{\eta}(t; k,\eta).
\ee
Thus,
\beq\label{NMLH41-2}
\mathsf{N}_{M;4,(1)}^{\mathrm{LH}}\nn&\les&\nu^{\gamma+\fr25-\fr{1-s}{2-s}}e^{-2|\al|\nu t}\sum_{k\in\mathbb{Z}^n_*}\int_\eta \nu^{-\gamma}\left(e^{\dl_1 \nu^\fr13 t}e^{c\lm(t)\la k,kt^{\rm ap}\ra^s} \left|t^{\rm ap}\right|^{\fr{a+1}{2-s}+1}|\hat{\rho}_k(t)|\right) \\
\nn&&\times  \left(\la\nu t\ra^{\fr{a+1}{2-s}+1}e^{\fr{s}{2-s}\nu t+\dl_1\nu^\fr25t-\dl_1\nu^\fr13 t}\right)\left(*\right)\left|\fr{|\eta|^{\fr{s}{2}}}{\la t\ra^{\fr{a+1}{2}}}A\pr_\eta^\al(\widehat{f^w})_k(t,\eta)\right|d\eta\\
\nn&\les&\left(\nu^{-\gamma}\|f^w(t)\|_{\mathcal{E}^{ \sig_1 ,\fr13}_m}\right)  \left(\sqrt{-\dot{\lm}(t)}e^{-|\al|\nu t}\left\|\la \nb\ra^{\fr{s}{2}}A(v^\al f^w)\right\|_{L^2}\right)\\
&&\times  e^{-|\al|\nu t}\left[\sqrt{-\dot{\lm}(t)}\left\|\la \nb\ra^{\fr{s}{2}}A(v^\al f^w)_0\right\|_{L^2_{v}}+\nu^\fr12\left\|\pr_v^tA(v^\al f^w)_0 \right\|_{L^2_v}\right],
\eeq
where we have used the fact that (note that $-\fr{1-s}{2-s}$ is increasing in $s$, and $s>\fr{1-3\gamma}{3-3\gamma}$)
\[
\gamma+\fr25-\fr{1-s}{2-s}>\gamma+\fr25-\fr{1-\fr{1-3\gamma}{3-3\gamma}}{2-\fr{1-3\gamma}{3-3\gamma}}=\gamma+\fr25+\fr{2}{3\gamma-5}\ge0 \quad\mathrm{for}\quad 0\le\gamma\le\fr13.
\]
The other term $\mathsf{N}_{M;4,(2)}^{\mathrm{LH}}$ can be treated similarly. In fact, on the one hand,  if
$\left|\eta-kt^{\rm ap}\right|\le \left| kt^{\rm ap}\right|$,
we have
$|\bar{\eta}(t; 0,\eta)|=e^{\nu t}|\eta|\les e^{\nu t}\left| kt^{\rm ap}\right|$.
Then $\mathsf{N}_{M;4,(2)}^{\mathrm{LH}}$ can be treated in the same manner as \eqref{NMLH41-1}.  On the other hand, if
$\left|\eta-kt^{\rm ap}\right|\ge \left| kt^{\rm ap}\right|$,
then
$|\bar{\eta}(t; 0,\eta)|=e^{\nu t}|\eta|\les e^{\nu t}\left|\eta- kt^{\rm ap}\right|$.
Combining this with the fact
$\bar{\eta}(t; 0,\eta)=\bar{\eta}\left(t; -k,\eta-kt^{\rm ap}\right)$,
it is easy to  see that \eqref{interpolation1} still holds with $\bar{\eta}(t; k,\eta)$ and $A\widehat{(v^\al f^w)}_{0}\left(t, \eta-kt^{\rm ap}\right)$ replaced by $\bar{\eta}(t; 0,\eta)$ and $A\widehat{(v^\al f^w)}_{-k}\left(t, \eta-kt^{\rm ap}\right)$, respectively. Then following \eqref{NMLH41-2}, we can bound $\mathsf{N}_{M;4,(2)}^{\mathrm{LH}}$ immediately.

\noindent\uline{Treatment of $\mathsf{N}_M^{\mathrm{HL}}$.} Different from \eqref{A-HL} and \eqref{A-HL0}, thanks to \eqref{app5}, one deduces that
\begin{align}\label{up-Aketa}
A_k(t,\eta){\bf 1}_{\left|k-l,\eta-lt^{\rm ap} \right|\le 2\left|l, lt^{\rm ap}\right|}
\les \left(e^{-(\dl-\dl_1) \nu^\fr13t}|l|^\fr12 \left\la t^{\rm ap}\right\ra \right){\bf B}_l(t) e^{c\lm(t)\la k-l,\eta-lt^{\rm ap}\ra^s}.
\end{align}
and
\begin{align}
\nn&A_{k-l}\left(t,\eta-lt^{\rm ap}\right){\bf 1}_{\left|k-l,\eta-lt^{\rm ap} \right|\le 2\left|l, lt^{\rm ap}\right|}\\
\nn\les&\left(e^{-\dl\nu^\fr13t} |l|^{\fr12} \left\la t^{\rm ap}\right\ra\right)\left(e^{\dl\nu^\fr13t}|l|^{\fr12}\left\la l, lt^{\rm ap}\right\ra^{\sigma_0}\right) e^{{\bf 1}_{k\ne l}\dl_1\nu^\fr25 t}e^{\lm(t)\left\la k-l, \eta-lt^{\rm ap} \right\ra^s}.
\end{align}
It follows from the above two inequalities, \eqref{up-eta-bar}  and  \eqref{Young2} that
\begin{align}\label{e-NMHL}
\mathsf{N}_M^{\mathrm{HL}}\nn\les& e^{-(\dl-\dl_1)\nu^\fr13t+\nu t}\left\la t^{\rm ap}\right\ra^{2}e^{-2|\al|\nu t}\sum_{k\in\mathbb{Z}^n,\, l\in\mathbb{Z}^n_*}\int_\eta\fr{1}{|l|^\fr12}\left|{B}^{\sig_0}\hat{\rho}_l(t)\right|\\
\nn&\times\left| k-l, \eta-lt^{\rm ap}\right|e^{{\bf 1}_{k\ne l}\dl_1\nu^\fr25 t}e^{\lm(t)\left\la k-l, \eta-lt^{\rm ap} \right\ra^s}\\
\nn&\times\left|\pr_\eta^\al(\widehat{f^w})_{k-l}\left(t, \eta-lt^{\rm ap}\right)\right|\left| A\pr_\eta^\al(\widehat{f^w})_k(t,\eta)d\eta\right|d\eta\\
\nn\les&e^{-2|\al|\nu t}\la t\ra^2\|{B}^{\sig_0}\rho(t)\|_{L^2_x}\|A(v^\al f^w)(t)\|_{L^2}\\
\nn&\times\left\|e^{{\bf}_{ k\ne 0}\dl_1\nu^\fr25 t}\mathcal{F}\left[\la\pr_x\ra^{\fr{n}{2}+}|\nb|e^{\lm(t)\la \nb\ra^s}(v^\al f^w)\right]_k(t)\right\|_{L^2_{k,\eta}}\\
\les&\fr{1}{\la t\ra^{\fr12+}}\left(\la t\ra^{\fr52+}\|{B}^{\sig_0}\rho(t)\|_{L^2_x}\right)\left(e^{-|\al|\nu t}\|A(v^\al f^w)(t)\|_{L^2}\right)^2.
\end{align}

\noindent{\bf Case 2: $A(t,\nb)={\bf A}^{ \sig_1 ,\fr13}(t,\nb)$.} Now we write ${\sf N}_M$ as
\begin{align}
\nn\mathsf{N}_M=&-e^{-2|\al|\nu t}\frak{Re}\sum_{k,l\in\mathbb{Z}^n}\int_\eta {\bf A}_k^{ \sig_1 ,\fr13}(t,\eta)\hat{\rho}_l(t)\fr{l}{|l|^2}\cdot \bar{\eta}(t; k,\eta)\pr_\eta^\al(\widehat{f^w})_{k-l}\left(t, \eta-lt^{\rm ap}\right)\\
\nn&\times \left({\bf A}^{ \sig_1 ,\fr13}\pr_\eta^\al(\overline{\widehat{f^w}})\right)_k(t,\eta)d\eta.
\end{align}
There is no need to use  the commutator trick as in Section \ref{sec-non-1}, since on the one hand, the derivatives landing on $f^w$ are tolerable.  On the other hand, the time weight $e^{\dl_1\nu^\fr13t}$ will land on $\rho$ instead of $f^w$. In fact, in view of the partition of unity \eqref{pou},  by \eqref{app3}, \eqref{app5} and \eqref{up-eta-bar}, we have
\begin{align}
\nn&{\bf A}_k^{ \sig_1 ,\fr13}(t,\eta)|\bar{\eta}(t; k,\eta)|\\
\nn\les&e^{-(\dl-\dl_1)\nu^\fr13t}e^{\nu t}\left(e^{\dl \nu^\fr13t}e^{c\lm(t)\left\la l, lt^{\rm ap}\right\ra^s}\left\la t^{\rm ap}\right\ra \right){\bf A}_{k-l}^{\sig_0+1,\fr25}\left(t, \eta-lt^{\rm ap} \right)\\
\nn&+e^{-(\dl-\dl_1)\nu^\fr13t}e^{\nu t}\left(e^{c\lm(t)\left\la k-l, \eta-lt^{\rm ap}\right\ra^s} |k-l, \eta-lt^{\rm ap}|\right)\left(\la t^{\rm ap}\ra e^{\dl \nu^\fr13t}e^{\lm(t)\left\la l, lt^{\rm ap}\right\ra^s}\left\la l, lt^{\rm ap}\right\ra^{\sig_1} \right).
\end{align}
Consequently, using \eqref{Young1} and \eqref{Young2}, we arrive at
\begin{align}
\nn\mathsf{N}_M\les&\fr{1}{\la t \ra^{\fr12+}}\left(\la t\ra^{\fr12} \| {\bf B} \rho(t)\|_{L^2_x}\right)\left(e^{-|\al|\nu t}\left\|{\bf A}^{\sig_0+1,\fr25}(v^\al f^w)\right\|_{L^2}\right)\left(e^{-|\al|\nu t}\left\|{\bf A}^{ \sig_1 ,\fr13}(v^\al f^w)\right\|_{L^2}\right).
\end{align}

Next we  consider $\mathsf{N}_P$. Note first that \eqref{up-Aketa}  holds for  all 
\[
A(t,\nb) \in\left\{{\bf A}^{\sig_0+1, \fr25} (t,\nb), {\bf A}^{ \sig_1 , \fr13} (t,\nb), {\bf A}^{ \sig_1 , \fr25} (t,\nb)\right\}.
\] 
Combining this with \eqref{A-LH} implies that for $l\ne0$, there holds
\begin{align}
\fr{1}{|l|}A_k(t,\eta)\nn\les& e^{-(\dl-\dl_1)\nu^\fr13t}\la t\ra  {\bf B}  _l(t)e^{c\lm(t)\la k-l,\eta-lt^{\rm ap}\ra^s}\\
\nn&+e^{-(\dl-\dl_1)\nu^\fr13t} e^{\dl \nu^\fr13t}e^{c\lm(t)\la l,lt^{\rm ap}\ra^s} A_{k-l}\left(t, \eta-lt^{\rm ap}\right).
\end{align}
From this, \eqref{Young1} and  \eqref{Young2}, we infer that
\begin{align}\label{e-NP}
\mathsf{N}_P\nn\les&e^{\nu t}e^{-2|\al|\nu t}e^{-(\dl-\dl_1)\nu^\fr13t} \|A(v^\al f^w)\|_{L^2}   \la t\ra  \| {\bf B} \rho(t)\|_{L^2_x}\sum_{\beta\le \al:|\beta|=1} \left\|A(v^{\al-\beta})f^w\right\|_{L^2}\\
\nn\les&\fr{1}{\la t \ra^{\fr12+}}\left(\la t\ra^{\fr32+}  \| {\bf B} \rho(t)\|_{L^2_x}\right)\left(e^{-|\al|\nu t}\|A(v^\al f^w)\|_{L^2}\right) \\
&\times\left(  \sum_{\beta\le \al:|\beta|=1} e^{-(|\al|-1)\nu t} \left\|A(v^{\al-\beta})f^w\right\|_{L^2} \right).
\end{align}
We would like to remark that, different from the linear contributions ${\sf LE}_3$ and ${\sf LE}_4$, the last factor on the right hand side of \eqref{e-NP} can be treated without resorting to \eqref{damp-nu}, because we can take $L^\infty_t$ norm in $t$ on it.

\subsubsection{Collision contributions (I): $A(t,\nb)={\bf A}^{\sig_0+1,\fr25}(t,\nb)$.}\label{sec-collision-I}
We first consider  $\mathsf{CN}_0$. Recalling \eqref{M0}, it is natural to divide $\mathsf{CN}_0$ into three parts according to the order of $\bar{\eta}(t;k,\eta)$:
\be\label{CN0}
\mathsf{CN}_0=\mathsf{CN}_{0;{\bf2}}+\mathsf{CN}_{0;{\bf 1}}+\mathsf{CN}_{0; {\bf0}},
\ee
here the subscripts `2, 1, 0' after the first subscripts `0' refers to the order of $\bar{\eta}(t;k,\eta)$ in \eqref{M0}.  More precisely,  we write 
$\mathsf{CN}_{0;{\bf2}}=\mathsf{CN}_{0;{\bf2}}[\rho]+\mathsf{CN}_{0;{\bf2}}[M_\theta]$,
where
\begin{align}
\nn&\mathsf{CN}_{0;{\bf2}}[\rho]\\
\nn=&\nu e^{-2|\al|\nu t}\frak{Re}\sum_{k\in\mathbb{Z}^n,\, l\in\mathbb{Z}^n}\int_\eta  A\pr_\eta^\al \left(\left|\bar{\eta}(t;k,\eta) \right|^2\hat{\rho}_l(t)(\widehat{f^w})_{k-l}(t,\eta-l t^{\rm ap})  \right)A\pr_\eta^\al(\overline{\widehat{f^w}})_k(t,\eta)d\eta\\
\nn=&-\nu e^{-2|\al|\nu t}\frak{Re}\sum_{{k\in \mathbb{Z}^n,\, l\in\mathbb{Z}^n_*}}\int_\eta A\pr_\eta^\al(\overline{\widehat{f^w}})_k(t,\eta)A_k(t,\eta)\hat{\rho}_l(t)\\\nn&\times\left|\bar{\eta}(t; k,\eta) \right|^2 \pr_\eta^\al(\widehat{f^w})_{k-l}\left(t,\eta-lt^{\rm ap}\right)d\eta\\
\nn&-\nu e^{-2|\al|\nu t} e^{\nu t}\sum_{\beta\le\al: |\beta|=1}\frak{Re}\sum_{{k\in \mathbb{Z}^n,\, l\in\mathbb{Z}^n_*}}\int_\eta A\pr_\eta^\al(\overline{\widehat{f^w}})_k(t,\eta)A_k(t,\eta)\hat{\rho}_l(t)\\
\nn&\times\bar{\eta}(t; k,\eta)^\beta\pr_\eta^{\al-\beta}(\widehat{f^w})_{k-l}\left(t,\eta-lt^{\rm ap}\right)d\eta\\
\nn&-\nu e^{-2|\al|\nu t} e^{2\nu t}\sum_{\beta\le\al: |\beta|=2}\frak{Re}\sum_{{k\in \mathbb{Z}^n,\, l\in\mathbb{Z}^n_*}}\int_\eta A\pr_\eta^\al(\overline{\widehat{f^w}})_k(t,\eta)A_k(t,\eta)\hat{\rho}_l(t)\\
\nn&\times\pr_\eta^{\al-\beta}(\widehat{f^w})_{k-l}\left(t,\eta-lt^{\rm ap}\right)d\eta
=\mathsf{CN}_{0;{\bf2}}[\rho]_{(1)}+\mathsf{CN}_{0;{\bf2}}[\rho]_{(2)}+\mathsf{CN}_{0;{\bf2}}[\rho]_{(3)},
\end{align}
and $\mathsf{CN}_{0;{\bf2}}[M_\theta]$ is given with $\rho$ replaced by $M_\theta$ in the definition of $\mathsf{CN}_{0;{\bf2}}[\rho]$ above.
Begin with $\mathsf{CN}_{0;{\bf2}}[\rho]_{(1)}$ and  divide  into three parts 
\begin{align}\label{split1}
\mathsf{CN}_{0;{\bf2}}[\rho]_{(1)}=\mathsf{CN}_{0;{\bf2}}[\rho]_{(1);0}^{\rm HL}+\mathsf{CN}_{0;{\bf2}}[\rho]_{(1);\ne}^{\rm HL}+\mathsf{CN}_{0;{\bf2}}[\rho]_{(1)}^{\rm LH},
\end{align}
where
\begin{align}
\mathsf{CN}_{0;{\bf2}}[\rho]_{(1);0}^{\rm HL}\nn=&-\nu e^{-2|\al|\nu t}\frak{Re}\sum_{k\in \mathbb{Z}^n_*}\int_\eta A\pr_\eta^\al(\overline{\widehat{f^w}})_k(t,\eta)A_k(t,\eta)\hat{\rho}_k(t)\\
\nn&\times {\bf 1}_{\left|\eta-kt^{\rm ap} \right|\le 2\left|k, kt^{\rm ap}\right|}\left|\bar{\eta}(t;k,\eta) \right|^2 \pr_\eta^\al(\widehat{f^w})_{0}\left(t,\eta-kt^{\rm ap}\right)d\eta,\\
\mathsf{CN}_{0;{\bf2}}[\rho]_{(1);\ne}^{\rm HL}\nn=&-\nu e^{-2|\al|\nu t}\frak{Re}\sum_{\substack{k\in \mathbb{Z}^n,\, l\in\mathbb{Z}^n_*\\ k\ne l}}\int_\eta A\pr_\eta^\al(\overline{\widehat{f^w}})_k(t,\eta)A_k(t,\eta)\hat{\rho}_l(t)\\
\nn&\times {\bf 1}_{\left|k-l,\eta-lt^{\rm ap} \right|\le 2\left|l, lt^{\rm ap}\right|}\left|\bar{\eta}(t; k,\eta) \right|^2 \pr_\eta^\al(\widehat{f^w})_{k-l}\left(t,\eta-lt^{\rm ap}\right)d\eta,
\end{align}
and
\begin{align*}
\mathsf{CN}_{0;{\bf2}}[\rho]_{(1);\ne}^{\rm LH}\nn=&-\nu e^{-2|\al|\nu t}\frak{Re}\sum_{\substack{k\in \mathbb{Z}^n,\, l\in\mathbb{Z}^n_*\\ k\ne l}}\int_\eta A\pr_\eta^\al(\overline{\widehat{f^w}})_k(t,\eta)A_k(t,\eta)\hat{\rho}_l(t)\\
\nn&\times {\bf 1}_{\left|l, lt^{\rm ap}\right|<\fr12\left|k-l,\eta-lt^{\rm ap} \right|}\left|\bar{\eta}(t; k,\eta) \right|^2 \pr_\eta^\al(\widehat{f^w})_{k-l}\left(t,\eta-lt^{\rm ap}\right)d\eta.
\end{align*}
We first consider $\mathsf{CN}_{0;2}[\rho]_{(1);0}^{\rm HL}$. 
Recalling that $\hat{\rho}_k(t)=\hat{f}_k\left(t, kt^{\rm ap}\right)$, and noting that a variants of \eqref{CLHL'} 
\begin{align*}
\nu \fr{\left\la t\right\ra^s}{\sqrt{-\dot{\lm}(t)}} e^{-s(\dl\nu^\fr13-\dl_1\nu^\fr25)t} e^{(1-s)m\nu t}e^{2\nu t}\les \nu\la t\ra^{s+\fr12+}e^{-\fr{s}{2}(\dl-\dl_1)\nu^\fr13t}\les \nu^\fr23
\end{align*}
holds, similar to \eqref{CLHL},  we infer from \eqref{A-HL0} that
\begin{align}\label{CNHL101}
\mathsf{CN}_{0;{\bf2}}[\rho]_{(1);0}^{\rm HL}\nn\les&\nu \left\la t\right\ra^s e^{-s(\dl-\dl_1)\nu^\fr13t} e^{2\nu t}e^{-2|\al|\nu t}\sum_{k\in\mathbb{Z}^n_*}\int_\eta |k|^\fr{s}{2}\left|A\pr_\eta^\al(\widehat{f^w})_k(t,\eta) \right| \\
\nn&\times \left|(\underline{A}^{\sig_0+1,\fr25} \hat{f})_k\left(t, kt^{\rm ap}\right)\right|^{1-s}\left| {\bf B}  \hat{\rho}_k(t) \right|^s \\
\nn&\times e^{c\lm(t)\la \eta-kt^{\rm ap}\ra^s}\left|\eta-kt^{\rm ap}\right|^2\left|\pr_\eta^\al(\widehat{f^w})_{0}\left(t,\eta-kt^{\rm ap}\right)\right|d\eta\\
\nn\les&\nu \fr{\left\la t\right\ra^s}{\sqrt{-\dot{\lm}(t)}} e^{-s(\dl-\dl_1)\nu^\fr13t} e^{(1-s)m\nu t}e^{2\nu t} \left(e^{-|\al|\nu t}\sqrt{-\dot{\lm}(t)}\left\|\la \nb\ra^{\fr{s}{2}}A(v^\al f^w)\right\|_{L^2}\right) \\
\nn&\times \|  {\bf B}  \rho\|_{L^2}^s \left(e^{-m\nu t}\|\underline{A}^{\sig_0+1,\fr25}\hat{f}_{\ne}\|_{L^2_kL^\infty_\eta}\right)^{1-s}\left(e^{-|\al|\nu t}\left\|e^{c\lm(t)\la \nb\ra^s}\left|\nb\right|^2\left(v^\al({f^w})_{0}\right)\right\|_{L^2}\right)\\
\nn\les&\nu^\fr23\left(e^{-|\al|\nu t}\sqrt{-\dot{\lm}(t)}\left\|\la \nb\ra^{\fr{s}{2}}A(v^\al f^w)\right\|_{L^2}\right) \\
&\times \|  {\bf B} \rho\|_{L^2}^s \|f^w_{\ne}(t)\|_{\mathcal{E}_m^{\sig_0+1,\fr25}}^{1-s}\left(e^{-|\al|\nu t}\left\|A\left(v^\al({f^w})_{0}\right)\right\|_{L^2}\right).
\end{align}
Furthermore, we can treat \eqref{CNHL101} like \eqref{CLHL''}.

For $\mathsf{CN}_{0;{\bf2}}[\rho]_{(1);\ne}^{\rm HL}$,  from \eqref{A-HL}, \eqref{up-eta-bar}, and \eqref{Young2}, and using the fact $\nu\la t^{\rm ap}\ra^2\les\nu^\fr13e^{\fr14\dl_1\nu^\fr13t}$, we infer that 
\begin{align}\label{e-JCg1HL}
\mathsf{CN}_{0;{\bf2}}[\rho]_{(1);\ne}^{\rm HL}\nn\les&\nu e^{-2|\al|\nu t}\sum_{\substack{k\in \mathbb{Z}^n,\, l\in\mathbb{Z}^n_*\\ k\ne l}}\int_\eta \left|A\pr_\eta^\al(\widehat{f^w})_k(t,\eta)\right|  \left|(\underline{A}^{\sig_0+1,\fr25} \hat{f})_l\left(t, lt^{\rm ap}\right)\right|e^{c\lm(t)\la k-l,\eta-lt^{\rm ap}\ra^s}\\
\nn&\times e^{2\nu t}\left\la t^{\rm ap}\right\ra^2\left|k-l, \eta-lt^{\rm ap} \right|^2 \left|\pr_\eta^\al(\widehat{f^w})_{k-l}\left(t,\eta-lt^{\rm ap}\right)\right|d\eta\\
\nn\les&\nu^\fr13e^{-\fr12\dl_1\nu^\fr13t}\left(e^{-|\al|\nu t}\|A(v^\al f^w)\|_{L^2}\right)\left(e^{-m\nu t}\|\underline{A}^{\sig_0+1,\fr25}\hat{f}_{\ne} \|_{L^2_{k} L^\infty_\eta}\right)\\
\nn&\times\left( e^{-|\al|\nu t}\left\|e^{\dl_1\nu^\fr13t}e^{c\lm(t)\la\nb\ra^s}|\nb|^2\la \pr_x\ra^{\fr{n}{2}+}\left(v^\al f^w\right)_{\ne} \right\|_{L^2}\right)\\
\nn\les&\nu^\fr13e^{-\fr12\dl_1\nu^\fr13t}\left(e^{-|\al|\nu t}\|A(v^\al f^w)\|_{L^2}\right) \| f_{\ne}^w(t)\|_{\mathcal{E}^{\sig_0+1,\fr25}_m}\\
&\times\left( e^{-|\al|\nu t}\left\| {\bf A}^{ \sig_1 ,\fr13}\left(v^\al f^w\right)_{\ne} \right\|_{L^2}\right).
\end{align}

As for $\mathsf{CN}_{0;{\bf2}}[\rho]_{(1)}^{\mathrm{LH}}$, using \eqref{A-LH}, the fact $\bar{\eta}(t; k,\eta)=\bar{\eta}(t;k-l,\eta-lt^{\rm ap})$ and \eqref{Young2}  , we have
\begin{align}\label{e-CNLH101}
\mathsf{CN}_{0;{\bf2}}[\rho]_{(1)}^{\mathrm{LH}}\nn\les&\nu e^{-2|\al|\nu t} \sum_{{k\in \mathbb{Z}^n,\, l\in\mathbb{Z}^n_*}}\int_\eta\left|\mathcal{F}[\pr_v^tA(v^\al f^w)]_k(t,\eta) \right| e^{\dl\nu^\fr13t} e^{c\lm(t)\la l, lt^{\rm ap}\ra^s} \left|\hat{\rho}_l(t) \right|\\
\nn&\times e^{-(\dl-\dl_1) \nu^\fr13 t} \left|\mathcal{F}[\pr_v^tA(v^\al f^w)]_{k-l}\left(t,\eta-lt^{\rm ap}\right)\right|d\eta\\
\les&\left(\nu e^{-|\al|\nu t} \left\|\pr_v^t A(v^\al f^w)\right\|_{L^2}\right)^2 \left\|B^{ \sig_1 }\rho(t) \right\|_{L^2_x}.
\end{align}
 We would like to emphasize that now we do not treat $\rho$ in terms of $f$ like \eqref{CNHL101} because we can not gain extra $e^{-m\nu t}$ decay.

When all derivatives are landing on $\rho$, $\mathsf{CN}_{0; {\bf2}}[\rho]^{\mathrm{HL}}_{(2)}$ and $\mathsf{CN}_{0; {\bf2}}[\rho]^{\mathrm{HL}}_{(3)}$ can be treated in the similar manner as $\mathsf{CN}_{0;{\bf2}}[\rho]_{(1)}^{\mathrm{HL}}$. In fact, similar to \eqref{CNHL101} and \eqref{e-JCg1HL}, we have
\begin{align}
\nn&\mathsf{CN}_{0; {\bf2}}[\rho]^{\mathrm{HL}}_{(2);0}+\mathsf{CN}_{0; {\bf2}}[\rho]^{\mathrm{HL}}_{(3);0}\\
\nn\les&\nu e^{\nu t}e^{(1-s)m\nu t}\fr{\la t \ra^s}{\sqrt{-\dot{\lm}(t)}}e^{-s(\dl\nu^\fr13t-\dl_1\nu^\fr25)t}\left(e^{-|\al|\nu t}\sqrt{-\dot{\lm}(t)}\left\|\la\nb\ra^{\fr{s}{2}}A(v^\al f^w)\right\|_{L^2}\right)\\
\nn&\times\left(e^{-m\nu t}\left\|\underline{A}_k^{\sig_0+1,\fr25}(t,kt^{\rm ap})\hat{\rho}_k(t)\right\|_{L^2_k}\right)^{1-s}
\left\| {\bf B}  \rho(t)\right\|_{L^2_x}^s\\
\nn&\times\sum_{\beta\le\al:1\le|\beta|\le2} e^{-(|\al|-|\beta|)\nu t}\left\|e^{c\lm(t)\la \eta\ra^s}\la\eta\ra\pr_\eta^{\al-\beta}(\widehat{f^w})_0(t,\eta)\right\|_{L^2_\eta}\\
\nn\les&\nu^{\fr23} e^{-\fr{s}{2}(\dl-\dl_1)\nu^\fr13t}\left(e^{-|\al|\nu t}\sqrt{-\dot{\lm}(t)}\left\|\la\nb\ra^{\fr{s}{2}}A(v^\al f^w)\right\|_{L^2}\right)\\
\nn&\times\|f_{\ne}(t)\|_{\mathcal{E}_m^{\sig_0+1,\fr25}}^{1-s}
\left\| {\bf B}  \rho(t)\right\|_{L^2_x}^s \sum_{\beta\le\al:1\le|\beta|\le2} e^{-(|\al|-|\beta|)\nu t}\left\|A(v^{\al-\beta} f^w)_0\right\|_{L^2}.
\end{align}
and
\begin{align}
\nn&\mathsf{CN}_{0; 2}[\rho]^{\mathrm{HL}}_{(2);\ne}+\mathsf{CN}_{0; 2}[\rho]^{\mathrm{HL}}_{(3);\ne}\\
\nn\les&\nu e^{\nu t}\la t\ra e^{-\dl_1\nu^\fr13t}\left(e^{-|\al|\nu t}\left\|A(v^\al f^w)\right\|_{L^2}\right)\left\|\underline{A}_k(t,kt^{\rm ap})\hat{\rho}_k(t)\right\|_{L^2_k}\\
\nn&\times \sum_{\beta\le\al:1\le|\beta|\le2} e^{-(|\al|-|\beta|)\nu t}\left\|e^{\dl_1\nu^\fr13 t}e^{c\lm(t)\la k,\eta\ra^s}\la k,\eta\ra\la k\ra^{\fr{n}{2}+}\pr_\eta^{\al-\beta}(\widehat{f^w})_k(t,\eta)\right\|_{L^2_{k,\eta}}\\
\nn\les&\nu^{\fr23}  e^{-\fr12\dl_1\nu^\fr13t}\left(e^{-|\al|\nu t}\left\|A(v^\al f^w)\right\|_{L^2}\right) \| f^w_{\ne}(t)\|_{\mathcal{E}^{\sig_0+1,\fr25}_m}\\
\nn&\times \sum_{\beta\le\al:1\le|\beta|\le2} e^{-(|\al|-|\beta|)\nu t}\left\| {\bf A}^{ \sig_1 ,\fr13}\left(v^{\al-\beta}{f^w}\right)_{\ne}\right\|_{L^2}.
\end{align}
When all derivatives are landing on $f^w$, similar to \eqref{e-CNLH101}, we arrive at
\begin{align}\label{CN0-remain}
\nn&\mathsf{CN}_{0;2}[\rho]_{(2)}^{\rm LH}+\mathsf{CN}_{0;2}[\rho]_{(3)}^{\rm LH}\\
\nn\les&\nu e^{-2|\al|\nu t}e^{\nu t}\sum_{\beta\le\al:|\beta|=1}\sum_{k,l\in\Z^n}\int_\eta |\eta(t; k,\eta)|\left|A\pr_\eta^\al(\widehat{f^w})_k(t,\eta) \right|\left(e^{\dl_1\nu^\fr25t}e^{c\lm(t)\la l,lt^{\rm ap}}\ra^s |\hat{\rho}_l(t)|\right)\\
\nn&\times \left|A\pr_\eta^{\al-\beta}(\widehat{f^w})_{k-l}(t,\eta-lt^{\rm ap})\right|d\eta\\
\nn&+\nu e^{-2|\al|\nu t}e^{2\nu t}\sum_{\beta\le\al:|\beta|=2}\sum_{k,l\in\Z^n}\int_\eta \left|A\pr_\eta^\al(\widehat{f^w})_k(t,\eta) \right|\left(e^{\dl_1\nu^\fr25t}e^{c\lm(t)\la l,lt^{\rm ap}}\ra^s |\hat{\rho}_l(t)|\right)\\
\nn&\times \left|A\pr_\eta^{\al-\beta}(\widehat{f^w})_{k-l}(t,\eta-lt^{\rm ap})\right|d\eta\\
\nn\les&\nu^\fr12 e^{-(\dl-\dl_1)\nu^\fr13t}\left(\nu^\fr12 e^{-|\al|\nu t}\left\|\pr_v^tA(v^\al f^w)\right\|_{L^2}\right)\left\| {\bf B}  \rho\right\|_{L^2_x}\\
\nn&\times\left(\sum_{\beta\le\al:|\beta|=1}e^{-(|\al|-1)\nu t}\left\|A(v^{\al-\beta}f^w)\right\|_{L^2}\right)\\
\nn&+\nu e^{-(\dl-\dl_1)\nu^\fr13t}\left(e^{-|\al|\nu t}\left\|A(v^\al f^w)\right\|_{L^2}\right)\left\| {\bf B}  \rho\right\|_{L^2_x}\\
&\times\left(\sum_{\beta\le\al:|\beta|=2}e^{-(|\al|-2)\nu t}\left\|A(v^{\al-\beta}f^w)\right\|_{L^2}\right).
\end{align}

Next consider $\mathsf{CN}_{0;{\bf 1}}$. Recalling \eqref{M0}, we write 
\begin{align}
\mathsf{CN}_{0;{\bf 1}}
\nn=&-\nu e^{-2|\al|\nu t}\frak{Re} \sum_{\substack{k\in \mathbb{Z}^n,\, l\in\mathbb{Z}^n_*}}\int_\eta A\pr_\eta^\al(\overline{\widehat{f^w}})_k(t,\eta)A_k(t,\eta)(\widehat{M_1})_l(t)\cdot i\bar{\eta}(t;k,\eta)\\
\nn&\times\pr_\eta^\al  (\widehat{f^w})_{k-l}\left(t,\eta-lt^{\rm ap}\right)d\eta\\
\nn&-\nu e^{\nu t}e^{-2|\al|\nu t} \sum_{\beta\le\al:|\beta|=1}\frak{Re}\sum_{\substack{k\in \mathbb{Z}^n,\, l\in\mathbb{Z}^n_*}}\int_\eta A\pr_\eta^\al(\overline{\widehat{f^w}})_k(t,\eta)A_k(t,\eta)(\widehat{M_1})_l^\beta(t) \\
\nn&\times i\pr_\eta^{\al -\beta} (\widehat{f^w})_{k-l}\left(t,\eta-lt^{\rm ap}\right)d\eta
=\mathsf{CN}_{0;{\bf 1},(1)}+\mathsf{CN}_{0;{\bf 1},(2)}.
\end{align}
Then similar to \eqref{CNHL101}--\eqref{e-CNLH101}, we have
\begin{align}
\nn&\mathsf{CN}_{0;{\bf 1},(1);0}^{\rm HL}+\mathsf{CN}_{0;{\bf 1},(2);0}^{\rm HL}\\
\nn\les&\nu^\fr23e^{-\fr{s}{2}(\dl-\dl_1)\nu^\fr13t}\left(e^{-|\al|\nu t}\sqrt{-\dot{\lm}(t)}\|\la \nb\ra^{\fr{s}{2}}A(v^\al f^w)\|_{L^2} \right)\\
\nn&\times\left( e^{-(m+1)\nu t}\|{\bf 1}_{k\ne0}\underline{A}^{\sig_0+1,\fr25}(\nb_\eta \hat{f})_k(t,\eta)\|_{L^2_kL^\infty_\eta}\right)^{1-s}\| {\bf B} M_1(t)\|_{L^2_x}^s\\
\nn&\times  \sum_{\beta\le \al:|\beta|\le1}\left(e^{-(|\al|-|\beta|)\nu t}\|A(v^{\al-\beta}f^w)\|_{L^2} \right)\\
\nn\les&\nu^\fr23e^{-\fr{s}{2}(\dl-\dl_1)\nu^\fr13t}\left(e^{-|\al|\nu t}\sqrt{-\dot{\lm}(t)}\|\la \nb\ra^{\fr{s}{2}}A(v^\al f^w)\|_{L^2} \right)\\
&\times\|f_{\ne}\|_{\mathcal{E}^{\sig_0+1,\fr25}_{m}}^{1-s}\| {\bf B}  M_1(t)\|_{L^2_x}^s\sum_{\beta\le \al:|\beta|\le1}\left(e^{-(|\al|-|\beta|)\nu t}\|A(v^{\al-\beta}f^w)\|_{L^2} \right),
\end{align}

\begin{align}
\mathsf{CN}_{0;{\bf 1},(1);\ne}^{\rm HL}+\mathsf{CN}_{0;{\bf 1},(2);\ne}^{\rm HL}
\nn\les&\nu^\fr35e^{-\fr12\dl_1\nu^\fr25 t}\left(e^{-|\al|\nu t}\|A(v^\al f^w)\|_{L^2}\right)^2
\|f_{\ne}\|_{\mathcal{E}^{\sig_0+1,\fr25}_{m}}\\
\nn&+\nu e^{-\fr12\dl_1\nu^\fr25 t}\left(e^{-|\al|\nu t}\|A(v^\al f^w)\|_{L^2}\right)
\|f_{\ne}\|_{\mathcal{E}^{\sig_0+1,\fr25}_{m}}\\
&\times\sum_{\beta\le \al:|\beta|=1}\left(e^{-(|\al|-1)\nu t}\|A(v^{\al-\beta} f^w)\|_{L^2}\right),
\end{align}
and
\begin{align}
\nn&\mathsf{CN}_{0;{\bf 1},(1)}^{\rm LH}+\mathsf{CN}_{0;{\bf 1},(2)}^{\rm LH}\\\nn\les&\nu^\fr12e^{-(\dl-\dl_1)\nu^\fr13t}\left(\nu^\fr12e^{-|\al|\nu t}\|\pr_v^tA(v^\al f^w)\|_{L^2}\right)\| {\bf B}  M_1\|_{L^2_x}\left(e^{-|\al|\nu t}\|A(v^\al f^w)\|_{L^2}\right)\\
\nn&+\nu e^{-(\dl-\dl_1)\nu^\fr13t}\left(e^{-|\al|\nu t}\|A(v^\al f^w)\|_{L^2}\right)\| {\bf B}  M_1\|_{L^2_x}\left(\sum_{\beta\le\al:|\beta|=1}e^{-(|\al|-1)\nu t}\|A(v^{\al-\beta} f^w)\|_{L^2}\right).
\end{align}
The treatment for $\mathsf{CN}_{0; {\bf0}}$ is similar and easier, and thus is omitted.

Now turn to $\mathsf{CN}_1$. Recalling \eqref{Merr} and \eqref{en-f-H}, according to the order of $\bar{\eta}(t;k,\eta)$, like \eqref{CN0}, we write
$\mathsf{CN}_1=\mathsf{CN}_{1;{\bf 1}}+\mathsf{CN}_{1;{\bf 0}}$.
Furthermore, let us denote
$\mathsf{CN}_{1;{\bf 1}}=\mathsf{CN}_{1;{\bf 1}}[\rho]-4\lm_1\mathsf{CN}_{1;{\bf 1}}[\rho+M_\theta]$,
where
\begin{align}\label{e-CN2}
\mathsf{CN}_{1;{\bf 1}}[\rho]
\nn=&-\nu e^{-\nu t} e^{-2|\al|\nu t}\frak{Re}\sum_{k\in\mathbb{Z}^n,l\in\mathbb{Z}^n_*}\int_\eta \left({A}\pr_{\eta}^\al(\overline{\widehat{f^w}})_k(t, \eta)\right){A}_k(t,\eta)\hat{\rho}_l(t)\\
\nn&\times\pr_\eta^\al\left(\bar{\eta}(t;k,\eta)\cdot(\nb_\eta\widehat{f^w})_{k-l}\left(t,\eta-lt^{\rm ap}\right)\right)d\eta\\
\nn=&-\nu e^{-\nu t} e^{-2|\al|\nu t}\frak{Re}\sum_{k\in\mathbb{Z}^n,l\in\mathbb{Z}^n_*}\int_\eta \left({A}\pr_{\eta}^\al(\overline{\widehat{f^w}})_k(t, \eta)\right){A}_k(t,\eta)\hat{\rho}_l(t)\\
\nn&\times \bar{\eta}(t;k,\eta)\cdot\pr_\eta^\al (\nb_\eta\widehat{f^w})_{k-l}\left(t,\eta-lt^{\rm ap}\right)d\eta\\
\nn&-\nu e^{-2|\al|\nu t}\sum_{\beta\le\al,|\beta|=1}\begin{pmatrix}\al\\\beta\end{pmatrix}\frak{Re}\sum_{k\in\mathbb{Z}^n,l\in\mathbb{Z}^n_*}\int_\eta \left({A}\pr_{\eta}^\al(\overline{\widehat{f^w}})_k(t, \eta)\right){A}_k(t,\eta)\hat{\rho}_l(t)\\
\nn&\times\underbrace{\pr_\eta^\beta\eta\cdot\nb_{\eta}\pr_\eta^{\al-\beta}}_{=\pr_\eta^\al}(\widehat{f^w})_{k-l}\left(t,\eta-lt^{\rm ap}\right)d\eta\\
=&\mathsf{CN}_{1;{\bf 1}}[\rho]_{(1)}+\mathsf{CN}_{1;{\bf 1}}[\rho]_{(2)},
\end{align}
and $\mathsf{CN}_{1;{\bf 1}}[\rho+M_\theta]=\mathsf{CN}_{1;{\bf 1}}[\rho]+\mathsf{CN}_{1;{\bf 1}}[M_\theta]$ can be given in the same way.
Like \eqref{split1}, $\mathsf{CN}_{1;{\bf 1}}[\rho]_{(1)}$ can be split as follows:
\[
\mathsf{CN}_{1;{\bf 1}}[\rho]_{(1)}=\mathsf{CN}_{1;{\bf 1}}[\rho]_{(1)}^{\mathrm{LH}}+\mathsf{CN}_{1;{\bf 1}}[\rho]_{(1);0}^{\mathrm{HL}}+\mathsf{CN}_{1;{\bf 1}}[\rho]_{(1);\ne}^{\mathrm{HL}}.
\]
Similar to \eqref{CNHL101}, using \eqref{A-HL0} and \eqref{up-eta-bar}, we are led to
\begin{align}\label{CNrhoHL0}
\mathsf{CN}_{1;{\bf 1}}[\rho]_{(1);0}^{\mathrm{HL}}\nn\les&\nu e^{-s(\dl-\dl_1)\nu^\fr13t} \left\la t\right\ra^{s}  e^{-2|\al|\nu t}\sum_{k\in\mathbb{Z}^n_*}\int_\eta |k|^\fr{s}{2}\left|{A}\pr_{\eta}^\al(\widehat{f^w})_k(t, \eta)\right|\\
\nn&\times\left|\underline{A}^{\sig_0+1,\fr25}\hat{f}_k\left(t,kt^{\rm ap}\right)\right|^{1-s}\left| {\bf B}  \hat{\rho}_k(t) \right|^s 
\mathcal{F}\left[e^{c\lm(t)\la \nb\ra^s}|\nb|(vv^\al f^w)\right]_{0}\left(t,\eta-kt^{\rm ap}\right)\\
\nn\les&\nu \fr{\left\la t\right\ra^s}{\sqrt{-\dot{\lm}(t)}} e^{-s(\dl-\dl_1)\nu^\fr13t} e^{(1-s)m\nu t}e^{\nu t} \left(e^{-|\al|\nu t}\sqrt{-\dot{\lm}(t)}\left\|\la \nb\ra^{\fr{s}{2}}A(v^\al f^w)\right\|_{L^2}\right) \\
\nn&\times \|  {\bf B}  \rho\|_{L^2}^s \left(e^{-m\nu t}\|\underline{A}^{\sig_0+1,\fr25}\hat{f}_{\ne}\|_{L^2_kL^\infty_\eta}\right)^{1-s}\left(e^{-(|\al|+1)\nu t}\left\|e^{c\lm(t)\la \nb\ra^s}\left|\nb\right|\left(vv^\al({f^w})_{0}\right)\right\|_{L^2}\right)\\
\nn\les&\nu^\fr23  \left(e^{-|\al|\nu t}\sqrt{-\dot{\lm}(t)}\left\|\la \nb\ra^{\fr{s}{2}}A(v^\al f^w)\right\|_{L^2}\right) \\
&\times \|  {\bf B}  \rho\|_{L^2}^s \|f_{\ne}\|_{\mathcal{E}^{\sig_0+1,\fr25}_m}^{1-s}\left(e^{-(|\al|+1)\nu t}\left\|{\bf A}^{ \sig_1 ,\fr25}\left(vv^\al{f^w}\right)_0\right\|_{L^2}\right).
\end{align}
Moreover, similar to \eqref{e-JCg1HL},  from \eqref{A-HL}, \eqref{up-eta-bar},   and \eqref{Young2}, we arrive at
\begin{align}\label{e-CN21ne}
\mathsf{CN}_{1;{\bf 1}}[\rho]_{(1);\ne}^{\mathrm{HL}}\nn\les&\nu    e^{-2|\al|\nu t}e^{\nu t}\left\la t\right\ra\sum_{\substack{k\in\mathbb{Z}^n,l\in\mathbb{Z}^n_*\\ k\ne l}}\int_\eta \left|{A}\pr_{\eta}^\al(\widehat{f^w})_k(t, \eta)\right|\left|\underline{A}^{\sig_0+1,\fr25}\hat{f}_l\left(t,lt^{\rm ap}\right)\right|\\
\nn&\times \mathcal{F}\left[e^{c\lm(t)\la \nb\ra^s}|\nb|(vv^\al f^w)\right]_{k-l}\left(t,\eta-lt^{\rm ap}\right)\\
\nn\les&\nu e^{-\dl_1\nu^\fr25t}e^{\nu t} \la t\ra  e^{-2|\al|\nu t}\|A(v^\al f^w)\|_{L^2_{x,v}}\|\underline{A}^{\sig_0+1,\fr25}\hat{f}_{\ne}\|_{L^2_kL^\infty_\eta}\\
\nn&\times\left\|e^{\dl_1\nu^\fr25t}e^{c\lm(t)\la \nb\ra^s}\la \pr_x\ra^{\fr{n}{2}+}|\nb|(vv^\al f^w_{\ne})\right\|_{L^2}\\
\nn\les&\nu\la t\ra e^{(m+2)\nu t}e^{-\dl_1\nu^\fr25t}\left(e^{-m \nu t}\|\underline{A}^{\sig_0+1,\fr25}\hat{f}_{\ne}\|_{L^2_kL^\infty_\eta}\right)\left(e^{-|\al|\nu t}\|A(v^\al f^w)\|_{L^2}\right)\\
\nn&\times\left(e^{-|\al|\nu t}\left\|{\bf A}^{ \sig_1 ,\fr25}(vv^\al{f^w_{\neq}})\right\|_{L^2}\right)\\
\nn\les&\nu^\fr35 e^{-\fr{\dl_1}{2}\nu^\fr25t}\|f_{\ne}^w\|_{\mathcal{E}^{\sig_0+1,\fr25}_m}\left(e^{-|\al|\nu t}\|A(v^\al f^w)\|_{L^2}\right)\\
&\times\left(e^{-(|\al|+1)\nu t}\left\|{\bf A}^{ \sig_1 ,\fr25}(vv^\al{f^w_{\neq}})\right\|_{L^2}\right).
\end{align}

As for $\mathsf{CN}_{1;{\bf 1}}[\rho]_{(1)}^{\mathrm{LH}}$, using \eqref{A-LH}  and \eqref{Young1},
\begin{align}\label{e-CN2LH}
\mathsf{CN}_{1;{\bf 1}}[\rho]_{(1)}^{\mathrm{LH}}\nn\les&\nu e^{-\nu t} e^{-2|\al|\nu t}\sum_{k\in\mathbb{Z}^n,l\in\mathbb{Z}^n_*}\int_\eta \left|\hat{\pr_v^t}{A}\pr_{\eta}^\al(\widehat{f^w})_k(t, \eta)\right|\left(e^{c\lm(t)\la l,lt^{\rm ap}\ra^s}e^{\dl \nu^\fr13t}\left|\hat{\rho}_l(t)\right|\right) \\
\nn&\times e^{-(\dl-\dl_1)\nu^\fr13t} \left|A\pr_\eta^\al\nb_\eta(\widehat{f^w})_{k-l}\left(t,\eta-lt^{\rm ap}\right)\right|d\eta\\
\nn\les&\left(\nu^\fr12e^{-|\al|\nu t}\left\|\pr_v^tA(v^\al f^w)\right\|_{L^2_{x,v}}\right)\left(\sqrt{-\dot{\lm}_1(t)}e^{-(|\al|+1)\nu t}\left\|A(vv^\al{f^w})\right\|_{L^2}\right)\\
\nn&\times  \left\|e^{-(\dl-\dl_1)\nu^\fr13t}\fr{\nu^\fr12}{\sqrt{-\dot{\lm}_1(t)}}\la k\ra^{\fr{n}{2}+}e^{c\lm(t)\la k,kt^{\rm ap}\ra^s}e^{\dl \nu^\fr13t}\hat{\rho}_k(t)\right\|_{L^2_k}\\
\les&\left\|B^{ \sig_1 }\rho\right\|_{L^2_x}\left(\nu^\fr12 e^{-|\al|\nu t} \left\|\pr_v^tA(v^\al f^w)\right\|_{L^2}\right)\left(\sqrt{-\dot{\lm}_1(t)}e^{-(|\al|+1)\nu t}\left\|A(vv^\al{f^w})\right\|_{L^2}\right),
\end{align}
where we have used
\[
e^{-(\dl-\dl_1)\nu^\fr13t}\fr{\nu^\fr12}{\sqrt{-\dot{\lm}_1(t)}}\les\left\la t^{\rm ap}\right\ra^{\fr{1+a_0}{2}}.
\]
Similarly, we can bound $\mathsf{CN}_{1;{\bf 1}}[\rho]_{(2)}$ as follows
\begin{align}
\mathsf{CN}_{1;{\bf 1}}[\rho]_{(2);0}^{\rm HL}\nn\les&\nu^\fr23  \left(e^{-|\al|\nu t}\sqrt{-\dot{\lm}(t)}\left\|\la \nb\ra^{\fr{s}{2}}A(v^\al f^w)\right\|_{L^2}\right) \\
&\times \|  {\bf B}  \rho\|_{L^2}^s \|f_{\ne}^w\|_{\mathcal{E}^{\sig_0+1,\fr25}_m}^{1-s}\left(e^{-|\al|\nu t}\left\|A\left(v^\al{f^w}\right)_0\right\|_{L^2}\right),\\
\mathsf{CN}_{1;{\bf 1}}[\rho]_{(2);\ne}^{\rm HL}\les&\nu e^{-\fr{\dl_1}{2}\nu^\fr25t}\|f^w_{\ne}\|_{\mathcal{E}^{\sig_0+1,\fr25}_m}\left(e^{-|\al|\nu t}\|A(v^\al f^w)\|_{L^2}\right)^2,
\end{align}
and
\begin{align}
\mathsf{CN}_{1;{\bf 1}}[\rho]_{(2)}^{\rm LH}
\les \| {\bf B}  \rho\|_{L^2_x}\left(e^{-|\al|\nu t}\|A(v^\al f^w)\|_{L^2}\right)^2.
\end{align}

To bound $\mathsf{CN}_{1;{\bf 0}}$, we write
\begin{align}
\mathsf{CN}_{1;{\bf 0}}\nn=&2\nu\lm_1e^{-\nu t}e^{-2|\al|\nu t}\frak{Re}\sum_{k\in\Z^n,l\in\Z^n_*}\int_\eta A_k(t,\eta)(\widehat{M_1})_l(t)\cdot i(\nb_\eta \pr^\al_\eta \widehat{f^w})_{k-l}(t,\eta-lt^{\rm ap})\\
\nn&\times A\pr^\al_\eta(\overline{\widehat{f^w}})_k(t,\eta)d\eta
=\mathsf{CN}_{1;{\bf 0}}^{\rm LH}+\mathsf{CN}_{1;{\bf 0};0}^{\rm HL}+\mathsf{CN}_{1;{\bf 0};\ne}^{\rm HL}.
\end{align}
Then recalling \eqref{M1f}, similar to \eqref{CNrhoHL0}--\eqref{e-CN2LH},  we have
\begin{align}\label{CN100HL}
\mathsf{CN}_{1;{\bf 0};0}^{\rm HL}\nn\les&\nu\lm_1e^{-\nu t}e^{-2|\al|\nu t}\la t\ra^s e^{s(\dl_1\nu^\fr25-\dl\nu^\fr13) t}\sum_{k\in\Z^n_*}\left(\underline{A}_k^{\sig_0+1,\fr25}(t,kt^{\rm ap})\left| (\widehat{M_1})_k(t)\right|\right)^{1-s}\\
\nn\times&\left| {\bf B}  (\widehat{M_1})_k(t) \right|^s\int_\eta\left(e^{c\lm(t)\la \eta-kt^{\rm ap}\ra^s}\left|(\nb_\eta\pr^\al_\eta\widehat{f^w})_0(t,\eta-kt^{\rm ap}) \right| \right)\\
\nn&\times \left(|k|^\fr{s}{2}\left|A\pr_\eta^\al(\widehat{f^w})_k(t,\eta) \right|\right)d\eta\\
\nn\les&\nu\lm_1e^{-\nu t}e^{-2|\al|\nu t}\la t\ra^s e^{s(\dl_1\nu^\fr25-\dl\nu^\fr13) t}\left\|\underline{A}_k^{\sig_0+1,\fr25}(t,kt^{\rm ap})(\widehat{M_1})_k(t)\right\|_{L^2_k}^{1-s}\left\| {\bf B}  M_1(t)\right\|_{L^2_x}^s\\
\nn&\times \left\|  {\bf A}^{ \sig_1 ,\fr25}(vv^\al f^w_0)\right\|_{L^2}\left\|\la\nb\ra^{\fr{s}{2}}A(v^\al f^w)\right\|_{L^2}\\
\nn\les&\nu \lm_1\fr{\la t\ra^s}{\sqrt{-\dot{\lm}(t)}} e^{s(\dl_1\nu^\fr25-\dl\nu^\fr13)t}e^{(1-s)m\nu t}\left(e^{-(m+1)\nu t}\left\|\underline{A}^{\sig_0+1,\fr25}(\nb_\eta\hat{f})_{\ne}\right\|_{L^2_kL^\infty_\eta}\right)^{1-s}\\
\nn&\times \left\| {\bf B}  M_1(t)\right\|_{L^2_x}^s\left(e^{-(|\al|+1)\nu t} \left\|  {\bf A}^{ \sig_1 ,\fr25}(vv^\al f^w_0)\right\|_{L^2}\right)\\
\nn&\times\left( e^{-|\al|\nu t} \sqrt{-\dot{\lm}(t)}\left\|\la\nb\ra^{\fr{s}{2}}A(v^\al f^w)\right\|_{L^2}\right)\\
\nn\les&\nu^\fr23 \lm_1\ e^{-\fr{s}{2}(\dl-\dl_1)\nu^\fr13t}\left\|f_{\ne}^w\right\|_{\mathcal{E}_m^{\sig_0+1,\fr25}}^{1-s}\left\| {\bf B}  M_1(t)\right\|_{L^2_x}^s\\
&\times \left(e^{-(|\al|+1)\nu t} \left\|  {\bf A}^{ \sig_1 ,\fr25}(vv^\al f^w)_0\right\|_{L^2}\right)\left( e^{-|\al|\nu t} \sqrt{-\dot{\lm}(t)}\left\|\la\nb\ra^{\fr{s}{2}}A(v^\al f^w)\right\|_{L^2}\right),
\end{align}

\begin{align}
\mathsf{CN}_{1;{\bf 0};\ne}^{\rm HL}\nn\les&\nu\lm_1e^{-\nu t}e^{-2|\al|\nu t}\sum_{\substack{k\in\Z^n,l\in\Z^n_*\\ l\ne k}}\int_\eta e^{\dl_1\nu^\fr25t}e^{\lm(t)\la l,lt^{\rm ap}\ra^s}\left\la l,lt^{\rm ap}\right\ra^{\sig+1}\left|(\widehat{M_1})_l(t) \right|\\
\nn&\times e^{c\lm(t)\la k-l,\eta-lt^{\rm ap}\ra^s}\left| \nb_\eta\pr_\eta^\al(\widehat{f^w})_{k-l}(t,\eta-lt^{\rm ap})\right| \left|A\pr_\eta^\al(\widehat{f^w})_k(t,\eta)\right|d\eta\\
\nn\les&\nu e^{-\dl_1\nu^\fr25t}\left\|\underline{A}_k^{\sig_0+1,\fr25}(t,kt^{\rm ap})(\widehat{M}_1)_k(t)\right\|_{L^2_k}\left(\lm_1e^{-(|\al|+1)\nu t}\|A(vv^\al f^w_{\ne})\|_{L^2}\right)\\\nn&\times\left(e^{-|\al|\nu t}\left\|A(v^\al  f^w)\right\|_{L^2}\right)\\
\nn\les&\nu e^{-\fr{\dl_1}{2}\nu^{\fr25}t}\left(e^{-(m+1)\nu t}\left\|\underline{A}^{\sig_0+1,\fr25}(\nb_\eta\hat{f})_k(t,\eta)\right\|_{L^2_kL^\infty_\eta}\right)\left(\lm_1e^{-(|\al|+1)\nu t}\|A(vv^\al f^w_{\ne})\|_{L^2}\right)\\
\nn&\times\left(e^{-|\al|\nu t}\left\|A(v^\al  f^w)\right\|_{L^2}\right)\\
\les&\nu e^{-\fr{\dl_1}{2}\nu^{\fr25}t}\|f_{\ne}\|_{\mathcal{E}^{\sig_0+1,\fr25}_m}\big(\lm_1e^{-(|\al|+1)\nu t}\|A(vv^\al f^w_{\ne})\|_{L^2}\big)
\big(e^{-|\al|\nu t}\left\|A(v^\al  f^w)\right\|_{L^2}\big),
\end{align}
and
\begin{align}
\mathsf{CN}_{1;{\bf 0}}^{\rm LH}\nn\les&\nu\left\| e^{\dl_1\nu^\fr25t}\la k\ra^{\fr{n}{2}+}e^{c\lm(t)\la k, kt^{\rm ap}\ra^s}(\widehat{M_1})_k(t)\right\|_{L^2_k}\\
\nn&\times \left(\lm_1e^{-(|\al|+1)\nu t}\left\|A(vv^\al f^w)\right\|_{L^2}\right)\left(e^{-|\al|\nu t}\left\|A(v^\al f^w)\right\|_{L^2}\right)\\
\les&\nu \left\| {\bf B}  M_1\right\|_{L^2_x}\left(\lm_1e^{-(|\al|+1)\nu t}\left\|A(vv^\al f^w)\right\|_{L^2}\right)\left(e^{-|\al|\nu t}\left\|A(v^\al f^w)\right\|_{L^2}\right).
\end{align}

The remaining term is ${\sf CN}_2$. Recalling \eqref{M2}, it is natural to write ${\sf CN}_2$ as
\beno
{\sf CN}_2=2\lm_1(1-2\lm_1){\sf CN}_2[\rho]-4\lm_1^2{\sf CN}_2[M_\theta]={\sf CN}_{2;\rho}+{\sf CN}_{2;M_\theta},
\eeno
where
\begin{align}
{\sf CN}_2[\rho]\nn=\nu e^{-2(|\al|+1)\nu t}\frak{Re}\sum_{k\in\mathbb{Z}^n,\, l\in\mathbb{Z}^n}\int_\eta  A_k(t,\eta)\hat{\rho}_l(t)\pr_\eta^\al(\Dl_\eta\widehat{f^w})_{k-l}(t, \eta-lt^{\rm ap})A\pr_\eta^\al(\overline{\widehat{f^w}})_k(t,\eta)d\eta,
\end{align}
and ${\sf CN}_2[M_\theta]$ can be given with $\rho$ replaced by $M_\theta$ as above. ${\sf CN}_{2;\rho}$ and ${\sf CN}_{2;M_\theta}$ can be treated in the similar manner.
Let us focus on the estimates of ${\sf CN}_{2;M_\theta}$, which can be split into three parts
${\sf CN}_{2;M_\theta}={\sf CN}_{2;M_\theta;0}+{\sf CN}_{2;M_\theta;\ne\ne} +{\sf CN}_{2;M_\theta;\ne0}$,
where 
\begin{align}
{\sf CN}_{2;M_\theta;0}\nn=&4\nu\lm_1^2 e^{-2(|\al|+1)\nu t}(\widehat{M_\theta})_0(t)\left\|A(vv^\al f^w)\right\|_{L^2}^2\\
\nn&+8\nu\lm_1^2 e^{-2(|\al|+1)\nu t}(\widehat{M_\theta})_0(t)\sum_{k\in\mathbb{Z}^n}\int_\eta  A_k(t,\eta)  (\nb_\eta \pr_\eta^\al\widehat{f^w})_{k}\left(t,\eta\right)\\
\nn&\cdot \nb_\eta A_k(t,\eta)\pr_\eta^\al(\overline{\widehat{f^w}})_k(t,\eta)d\eta={\sf CN}_{2;M_\theta;0,(1)}+{\sf CN}_{2;M_\theta;0,(2)},\\
{\sf CN}_{2;M_\theta;\ne\ne}\nn=&4\nu\lm_1^2e^{-2(|\al|+1)\nu t}\sum_{\substack{k\in\mathbb{Z}^n, l\in\mathbb{Z}^n_*\\ l\ne k}}\int_\eta  A_k(t,\eta) (\widehat{M_\theta})_l(t) (\nb_\eta \pr_\eta^\al\widehat{f^w})_{k-l}\left(t,\eta-lt^{\rm ap}\right)\\
\nn&\cdot A\nb_\eta\pr_\eta^\al(\overline{\widehat{f^w}})_k(t,\eta)d\eta\\
\nn&+8\nu\lm_1^2e^{-2(|\al|+1)\nu t}\sum_{\substack{k\in\mathbb{Z}^n, l\in\mathbb{Z}^n_*\\ l\ne k}}\int_\eta  A_k(t,\eta) (\widehat{M_\theta})_l(t) (\nb_\eta \pr_\eta^\al\widehat{f^w})_{k-l}\left(t,\eta-lt^{\rm ap}\right)\\
\nn&\cdot \nb_\eta A_k(t,\eta)\pr_\eta^\al(\overline{\widehat{f^w}})_k(t,\eta)d\eta={\sf CN}_{2;M_\theta;\ne\ne,(1)}+{\sf CN}_{2;M_\theta;\ne\ne,(2)},
\end{align}
and
\begin{align}
{\sf CN}_{2;M_\theta;\ne0}\nn=&-4\nu\lm_1^2e^{-2(|\al|+1)\nu t}\sum_{k\in\mathbb{Z}^n_*}\int_\eta  A_k(t,\eta) (\widehat{M_\theta})_k(t) (\Dl_\eta \pr_\eta^\al\widehat{f^w})_{0}\left(t,\eta-kt^{\rm ap}\right)\\
\nn&\times A\pr_\eta^\al(\overline{\widehat{f^w}})_k(t,\eta)d\eta.
\end{align}
Clearly, ${\sf CN}_{2;M_\theta;0,(1)}$ can be absorbed by the left hand side of \eqref{en-f-H}. For ${\sf CN}_{2;M_\theta;0,(2)}$, we infer from \eqref{A_eta1} and \eqref{A_eta2} that
\be\label{A_eta3}
|\nb_\eta A_k(t,\eta)|\les e^{\nu t}A_k(t,\eta).
\ee
Then
\begin{align}
{\sf CN}_{2;M_\theta;0,(2)}\nn\les&\lm_1\nu \left| (\widehat{M_\theta})_0(t)\right|\left(\lm_1 e^{-(|\al|+1)\nu t}\left\| A(vv^\al f^w)\right\|_{L^2}\right)\left(e^{-|\al|\nu t}\|A(v^\al f^w)\|_{L^2}\right).
\end{align}
We will estimate ${\sf CN}_{2;M_\theta;\ne\ne,(1)}$ by using the partition of unity \eqref{pou}. In fact, using \eqref{A-LH} and \eqref{Young1}, one deduces that
\begin{align}\label{CN2MLH1}
{\sf CN}_{2;M_\theta;\ne\ne,(1)}^{\rm LH}\nn\les&\nu e^{-(\dl-\dl_1)\nu^\fr13t}\left\|{\bf 1}_{k\ne0}e^{\dl\nu^\fr13 t}\la k\ra^{\fr{n}{2}+}e^{c\lm(t)\la k, kt^{\rm ap}\ra^s}(\widehat{M_\theta})_k(t)\right\|_{L^2_k}\\
\nn&\times \left(\lm_1 e^{-(|\al|+1)\nu t}\left\|A(vv^\al f^w)\right\|_{L^2}\right)^2\\
\les&\nu e^{-(\dl-\dl_1)\nu^\fr13t}\left\|B^{ \sig_1 }({M_\theta})_{\ne}(t)\right\|_{L^2_x} \left(\lm_1 e^{-(|\al|+1)\nu t}\left\|A(vv^\al f^w)\right\|_{L^2}\right)^2.
\end{align}
For  ${\sf CN}_{2;M_\theta;\ne\ne,(1)}^{\rm HL}$, thanks to \eqref{A-HL} and \eqref{Young2}, similar to \eqref{e-CN21ne}, using \eqref{AM}, we have
\begin{align}
{\sf CN}_{2;M_\theta;\ne\ne,(1)}^{\rm HL}\nn\les&\nu\lm_1^2 e^{-2(|\al|+1)\nu t}\left\| e^{\dl_1\nu^\fr25 t} \la k\ra^{\fr{n}{2}+} e^{c\lm(t)\la k,\eta\ra^s}(\nb_\eta\pr_\eta^\al\widehat{f^w_{\ne}})_k(t,\eta)\right\|_{L^2_{k,\eta}}\\
\nn&\times e^{-\dl_1\nu^\fr25 t}\left\|{\bf 1}_{k\ne0}\underline{A}_k^{\sig_0+1,\fr25}\left(t,kt^{\rm ap}\right)(\widehat{M_\theta})_k(t) \right\|_{L^2_k}\left\|A(vv^\al f^w)\right\|_{L^2}\\
\les&\lm_1\left(\nu\lm_1 e^{-2(|\al|+1)\nu t}\left\|A(vv^\al f^w)\right\|_{L^2}^2\right)\|f_{\ne}^w(t)\|_{\mathcal{E}_m^{\sig_0+1,\fr25}}.
\end{align}
In view of  \eqref{A_eta3}, the  other term  ${\sf CN}_{2;M_\theta;\ne\ne,(2)}$ can be bounded in a similar manner as  ${\sf CN}_{2;M_\theta;\ne\ne,(1)}$:
\begin{align}
{\sf CN}_{2;M_\theta;\ne\ne,(2)}\nn\les&\nu \lm_1  e^{-(\dl-\dl_1)\nu^\fr13t}\left\|B^{ \sig_1 }({M_\theta})_{\ne}(t)\right\|_{L^2_x} \left(\lm_1 e^{-(|\al|+1)\nu t}\left\|A(vv^\al f^w)\right\|_{L^2}\right)\\
\nn&\times\left( e^{-|\al|\nu t}\left\|A(v^\al f^w) \right\|_{L^2}\right)\\
\nn&+\nu\lm_1\left(\lm_1 e^{-(|\al|+1)\nu t}\left\|A(vv^\al f^w)\right\|_{L^2}\right)\left(e^{-|\al|\nu t}\left\|A(v^\al f^w)\right\|_{L^2}\right)\|f_{\ne}^w(t)\|_{\mathcal{E}_m^{\sig_0+1,\fr25}}.
\end{align}

Instead of using the partition of \eqref{pou}, we use  Bony's decomposition (see Section \ref{subsec-LP}) to bound ${\sf CN}_{2;M_\theta;\ne0}$. 
 More precisely, we write
\begin{align}
{\sf CN}_{2;M_\theta;\ne0}
\nn=&-4\nu\lm_1^2 e^{-2(|\al|+1)\nu t}\frak{Re}\sum_{N\ge8}\sum_{k\in\mathbb{Z}^n_*}\int_\eta  A_k(t,\eta) (\widehat{M_\theta})_k(t)_{<\fr{N}{8}} (\Dl_\eta \pr_\eta^\al\widehat{f^w})_{0}\left(t,\eta-kt^{\rm ap}\right)_N\\
\nn&\times A\pr_\eta^\al(\overline{\widehat{f^w}})_k(t,\eta) d\eta\\
\nn&-4\nu\lm_1^2 e^{-2(|\al|+1)\nu t}\frak{Re}\sum_{N\in \mathbb{D}}\sum_{k\in\mathbb{Z}^n_*}\int_\eta  A_k(t,\eta) (\widehat{M_\theta})_k(t)_{N} (\Dl_\eta \pr_\eta^\al\widehat{f^w})_{0}\left(t,\eta-kt^{\rm ap}\right)_{<16N}\\
\nn&\times A\pr_\eta^\al(\overline{\widehat{f^w}})_k(t,\eta)d\eta={\sf CN}_{2;M_\theta;\ne0}^{\rm LH}+{\sf CN}_{2;M_\theta;\ne0}^{\rm HL},
\end{align}
where 
\beno
(\widehat{M_\theta})_k(t)_{N}=\varphi\left(\fr{k, kt^{\rm ap}}{N}\right)(\widehat{M_\theta})_k(t),\quad (\widehat{M_\theta})_k(t)_{<\fr{N}{8}}=\chi\left(\fr{k, kt^{\rm ap}}{N/8}\right)(\widehat{M_\theta})_k(t),
\eeno
\beno
(\Dl_\eta \pr_\eta^\al\widehat{f^w})_{0}\left(t,\eta-kt^{\rm ap}\right)_N=\varphi\left(\fr{\eta-kt^{\rm ap}}{N}\right)(\Dl_\eta \pr_\eta^\al\widehat{f^w})_{0}\left(t,\eta-kt^{\rm ap}\right),
\eeno
and
\beno
(\Dl_\eta \pr_\eta^\al\widehat{f^w})_{0}\left(t,\eta-kt^{\rm ap}\right)_{<16N}=\chi\left(\fr{\eta-kt^{\rm ap}}{16N}\right)(\Dl_\eta \pr_\eta^\al\widehat{f^w})_{0}\left(t,\eta-kt^{\rm ap}\right).
\eeno
For any fixed $N\in\mathbb{D}$, the restriction on the integrand of ${\sf CN}_{2;M_\theta;\ne0}^{\rm HL}$ implies that
\[
\fr{N}{2}\le\left|k, kt^{\rm ap}\right|\le\fr{3}{2}N,\quad \left|\eta-kt^{\rm ap}\right|\le\fr34\cdot 16N=12N\le24\left|k, kt^{\rm ap}\right|.
\]
Then by virtue of \eqref{app5}, an analogue of \eqref{A-HL0} holds for each $N\in\mathbb{D}$. Consequently, similar to \eqref{CNHL101}, and using \eqref{AM}, we arrive at
\begin{align}\label{loss2m}
{\sf CN}_{2;M_\theta;\ne0}^{\rm HL}\nn\les&\nu \lm_1^2 e^{-2\nu t}e^{-2|\al|\nu t}\la t\ra^s e^{-s(\dl\nu^\fr13-\dl_1\nu^\fr25)t}\sum_{k\in\mathbb{Z}^n_*}\int_\eta  |k|^\fr{s}{2} \left|A\pr_\eta^\al(\widehat{f^w})_k(t,\eta) \right| \\
\nn&\times \left|\underline{A} _k^{\sig_0+1,\fr25}\left(t, kt^{\rm ap}\right)(\widehat{M_\theta})_k(t)\right|^{1-s}\left| {\bf B}  (\widehat{M_\theta})_k(t) \right|^s \\
\nn&\times  e^{c\lm(t)\la \eta-kt^{\rm ap}\ra^s}\left|\Dl_\eta\pr_\eta^\al(\widehat{f^w})_{0}\left(t,\eta-kt^{\rm ap}\right)\right|d\eta\\
\nn\les&\nu \lm_1^2 \fr{\la t\ra^s}{\sqrt{-\dot{\lm}(t)}}e^{-s(\dl\nu^\fr13-\dl_1\nu^\fr25)t}e^{(1-s)m\nu t} \left(e^{-|\al|\nu t}\sqrt{-\dot{\lm}(t)}\left\|\la\nb\ra^{\fr{s}{2}} A(v^\al f^w)\right\|_{L^2}\right)\\
\nn&\times \left(e^{-m\nu t}\left\|{\bf1}_{k\ne0}\underline{A}_k^{\sig_0+1,\fr25}\left(t, kt^{\rm ap}\right)(\widehat{M_\theta})_k(t)\right\|_{L^2_k}\right)^{1-s}\left\| {\bf B}  M_\theta\right\|_{L^2_x}^s\\
\nn&\times  e^{-(|\al|+2)\nu t}\left\|e^{c\lm(t)\la\nb\ra^s}(|v|^2v^\al f^w)_0 \right\|_{L^2}\\
\nn\les&\nu^\fr23\lm_1^2e^{-\fr{s}{2}(\dl-\dl_1)\nu^\fr13t}\left(e^{-|\al|\nu t}\sqrt{-\dot{\lm}(t)}\left\|\la\nb\ra^{\fr{s}{2}} A(v^\al f^w)\right\|_{L^2}\right)\\
&\times \left\|f_{\ne}^w\right\|_{\mathcal{E}_m^{\sig_0+1,\fr25}}^{1-s}\left\| {\bf B}  M_\theta\right\|_{L^2_x}^s  \left(e^{-(|\al|+2)\nu t}\left\|{\bf A}^{ \sig_1 ,\fr25}(|v|^2v^\al f^w)_0 \right\|_{L^2}\right).
\end{align}
Next, we will use integrating by parts to bound ${\sf CN}_{2;M_\theta;\ne0}^{\rm LH}$:
\begin{align}
{\sf CN}_{2;M_\theta;\ne0}^{\rm LH}\nn=&4\nu\lm_1^2 e^{-2(|\al|+1)\nu t}\frak{Re}\sum_{N\ge8}\sum_{k\in\mathbb{Z}^n_*}\int_\eta  A_k(t,\eta) (\widehat{M_\theta})_k(t)_{<\fr{N}{8}} (\nb_\eta \pr_\eta^\al\widehat{f^w})_{0}\left(t,\eta-kt^{\rm ap}\right)_N\\
\nn&\times A\nb_\eta\pr_\eta^\al(\overline{\widehat{f^w}})_k(t,\eta) d\eta\\
\nn&+8\nu\lm_1^2e^{-2(|\al|+1)\nu t}\frak{Re}\sum_{N\ge8}\sum_{k\in\mathbb{Z}^n_*}\int_\eta  A_k(t,\eta) (\widehat{M_\theta})_k(t)_{<\fr{N}{8}} (\nb_\eta \pr_\eta^\al\widehat{f^w})_{0}\left(t,\eta-kt^{\rm ap}\right)_N\\
\nn&\cdot \nb_\eta A_k(t,\eta)\pr_\eta^\al(\overline{\widehat{f^w}})_k(t,\eta) d\eta\\
\nn&+4\nu\lm_1^2e^{-2(|\al|+1)\nu t}\frak{Re}\sum_{N\ge8}\sum_{k\in\mathbb{Z}^n_*}\int_\eta  A_k(t,\eta) (\widehat{M_\theta})_k(t)_{<\fr{N}{8}} \fr{1}{N}(\nb_\eta\varphi)\left(\fr{\eta-kt^{\rm ap}}{N} \right)\\
\nn&\cdot (\nb_\eta \pr_\eta^\al\widehat{f^w})_{0}\left(t,\eta-kt^{\rm ap}\right) A\pr_\eta^\al(\overline{\widehat{f^w}})_k(t,\eta) d\eta\\
\nn=&{\sf CN}_{2;M_\theta;\ne0,(1)}^{\rm LH}+{\sf CN}_{2;M_\theta;\ne0,(2)}^{\rm LH}+{\sf CN}_{2;M_\theta;\ne0,(3)}^{\rm LH}.
\end{align}
For any fixed $N\in\mathbb{D}$, the restriction on the integrand of ${\sf CN}_{2;M_\theta;\ne0}^{\rm LH}$ implies that
\beno
\fr{N}{2}\le\left| \eta-kt^{\rm ap}\right|\le\fr{3N}{2}, \quad \left| k, kt^{\rm ap}\right| \le\fr34\cdot\fr{N}{8}\le \fr{3}{16}\left| \eta-kt^{\rm ap}\right|.
\eeno
Combining this with \eqref{app3}, we find that the analogue of \eqref{A-LH} with $l=k$ holds. Therefore, we have
\begin{align*}
{\sf CN}_{2;M_\theta;\ne0,(1)}^{\rm LH}\les&\nu e^{-(\dl-\dl_1)\nu^\fr13t}\left\|B^{\sig_1}({M_\theta})_{\ne}(t)\right\|_{L^2_x} \left(\lm_1 e^{-(|\al|+1)\nu t}\left\|A(vv^\al f^w)\right\|_{L^2}\right)^2,\\
{\sf CN}_{2;M_\theta;\ne0,(2)}^{\rm LH}\nn\les&\nu \lm_1 e^{-(\dl-\dl_1)\nu^\fr13t}\left\| {\bf B}  {M_\theta}(t)\right\|_{L^2_x} \left(\lm_1 e^{-(|\al|+1)\nu t}\left\|A(vv^\al f^w)\right\|_{L^2}\right)\\
&\times\left( e^{-|\al|\nu t}\left\|A(v^\al f^w)_0\right\|_{L^2}\right),
\end{align*}
and
\begin{align}\label{d-cutoff}
{\sf CN}_{2;M_\theta;\ne0,(3)}^{\rm LH}\nn\les&\left(\sum_{N\ge8, N\in\mathbb{D}}\fr{1}{N}\right)\nu\lm_1^2 e^{-2(|\al|+1)\nu t}\left\|{\bf 1}_{k\ne0}e^{\dl\nu^\fr13 t}\la k\ra^{\fr{n}{2}+}e^{c\lm(t)\la k, kt^{\rm ap}\ra^s}(\widehat{M_\theta})_k(t)\right\|_{L^2_k}\\
\nn&\times \left\|A(v^\al f^w)\right\|_{L^2}\left\|A(vv^\al f^w)_0\right\|_{L^2}\\
\les&\nu\lm_1e^{-\nu t} \left\| {\bf B}  {M_\theta}(t)\right\|_{L^2_x} \left(e^{-|\al|\nu t}\left\|A(v^\al f^w)\right\|_{L^2}\right)\left(\lm_1 e^{-(|\al|+1)\nu t}\left\|A(vv^\al f^w)_0\right\|_{L^2}\right).
\end{align}
We would like to remark that there is a redundant  $\lm_1$ or $\nu$ in all the estimates for ${\sf CN}_{2;M_\theta}$. Thus, the estimates for ${\sf CN}_{2;M_\theta;\ne\ne}$ and ${\sf CN}_{2;M_\theta;\ne0}$ from \eqref{CN2MLH1} to \eqref{d-cutoff} apply to ${\sf CN}_{2;\rho}$. We omit the details to avoid repetition.

\subsubsection{Collision contributions (II): $A(t,\nb)\in\left\{{\bf A}^{ \sig_1 ,\fr13}(t,\nb), {\bf A}^{ \sig_1 ,\fr25}(t,\nb)\right\}$.}
Before proceeding any further, it is worth pointing out that in the case under consideration, when all the derivatives land on $f^w$, i.e., $f^w$ is at high frequency, the collision nonlinearities can be treated in the same way as in Section \ref{sec-collision-I}. On the other hand, when $f^w$ is at low frequency, there is no need to use \eqref{rhof}--\eqref{M2f} to bound the macroscopic quantities $\rho, M_1, M_\theta$ in terms of $f$ due to the low regularity. In particular, the key point is to avoid the loss of velocity localizations in \eqref{CNrhoHL0}, \eqref{e-CN21ne}, \eqref{CN100HL} and \eqref{loss2m}. 

Let us first estimate \eqref{CNrhoHL0}, \eqref{e-CN21ne}, \eqref{CN100HL} and \eqref{loss2m} in different way from that in Section \ref{sec-collision-I}.
 Indeed, there is no need to distinguish $f^w$ is at zero frequency or not when $f^w$ is at low frequency.
From \eqref{A-HL},  \eqref{up-eta-bar},  and \eqref{Young2}, one deduces that
\begin{align}\label{e-CNL21ne}
\mathsf{CN}_{1;{\bf 1}}[\rho]_{(1)}^{\mathrm{HL}}\nn\les&\nu e^{-(\dl-\dl_1)\nu^\fr13t} e^{-2|\al|\nu t}\sum_{\substack{k\in\mathbb{Z}^n,l\in\mathbb{Z}^n_*}}\int_\eta \left|{A}\pr_{\eta}^\al(\widehat{f^w})_k(t, \eta)\right|\\
\nn&\times e^{\dl\nu^\fr13t}e^{\lm(t)\la l, lt^{\rm ap}\ra^s}\left\la l,lt^{\rm ap}\right\ra^{ \sig_1 } \left\la t^{\rm ap}\right\ra |\hat{\rho}_l(t)|\\
\nn&\times e^{{\bf 1}_{k\ne l}\dl_1\nu^\frak{e}t}\mathcal{F}\left[e^{c\lm(t)\la \nb\ra^s}|\nb|(vv^\al f^w)\right]_{k-l}\left(t,\eta-lt^{\rm ap}\right)\\
\nn\les&\nu e^{-(\dl-\dl_1)\nu^\fr13t}  e^{-2|\al|\nu t}\left\|A(v^\al f^w)\right\|_{L^2}\| {\bf B}  \rho\|_{L^2_x}\\
\nn&\times\left\|e^{{\bf 1}_{k\ne l}\dl_1\nu^\frak{e}t}e^{c\lm(t)\la \nb\ra^s}\la \pr_x\ra^{\fr{n}{2}+}|\nb|(vv^\al f^w)\right\|_{L^2}\\
\les& e^{-\fr12(\dl-\dl_1)\nu^\fr13t}\left(e^{-|\al|\nu t}\|A(v^\al f^w)\|_{L^2}\right)\| {\bf B}  \rho\|_{L^2_x}
\nu e^{-(|\al|+1)\nu t}\left\|A(vv^\al{f^w})\right\|_{L^2}.
\end{align}
Similarly, we have
\begin{align}
\mathsf{CN}_{1;{\bf 0}}^{\rm HL}\nn\les&\nu\lm_1e^{-(\dl-\dl_1)\nu^\fr13t}e^{-\nu t}e^{-2|\al|\nu t}\sum_{\substack{k\in\mathbb{Z}^n,l\in\mathbb{Z}^n_*}}\int_\eta \left|{A}\pr_{\eta}^\al(\widehat{f^w})_k(t, \eta)\right|\\
\nn&\times e^{\dl\nu^\fr13t}e^{\lm(t)\la l, lt^{\rm ap}\ra^s}\left\la l,lt^{\rm ap}\right\ra^{ \sig_1 } |(\widehat{M_1})_l(t)|\\
\nn&\times e^{{\bf 1}_{k\ne l}\dl_1\nu^\frak{e}t}\mathcal{F}\left[e^{c\lm(t)\la \nb\ra^s}|\nb|(vv^\al f^w)\right]_{k-l}\left(t,\eta-lt^{\rm ap}\right)\\
\les& \nu e^{-(\dl-\dl_1)\nu^\fr13t}\left(e^{-|\al|\nu t}\|A(v^\al f^w)\|_{L^2}\right)\| {\bf B}  \rho\|_{L^2_x}
\lm_1 e^{-(|\al|+1)\nu t}\left\|A(vv^\al{f^w})\right\|_{L^2}.
\end{align}
Compared with  \eqref{CN2MLH1} to \eqref{d-cutoff}, the treatment of ${\sf CN}_{2;M_\theta}$ when $M_\theta$ is at non-zero mode, denoted by ${\sf CN}_{2;M_\theta;\ne}$ below,  is more straightforward:
\begin{align}
{\sf CN}_{2;M_\theta;\ne}\nn=&-4\nu\lm_1^2 e^{-2(|\al|+1)\nu t}\frak{Re}\sum_{k\in\mathbb{Z}^n, l\in\mathbb{Z}^n_*}\int_\eta   A_k(t,\eta) (\widehat{M_\theta})_l(t) (\Dl_\eta \pr_\eta^\al\widehat{f^w})_{k-l}\left(t,\eta-lt^{\rm ap}\right)\\
\nn&\times  A\pr_\eta^\al(\overline{\widehat{f^w}})_k(t,\eta)d\eta\\
\nn=&4\nu\lm_1^2 e^{-2(|\al|+1)\nu t}\frak{Re}\sum_{k\in\mathbb{Z}^n, l\in\mathbb{Z}^n_*}\int_\eta   A_k(t,\eta) (\widehat{M_\theta})_l(t) (\nb_\eta \pr_\eta^\al\widehat{f^w})_{k-l}\left(t,\eta-lt^{\rm ap}\right)\\
\nn&\cdot  A\nb_\eta\pr_\eta^\al(\overline{\widehat{f^w}})_k(t,\eta)d\eta\\
\nn&+8\nu\lm_1^2 e^{-2(|\al|+1)\nu t}\frak{Re}\sum_{k\in\mathbb{Z}^n, l\in\mathbb{Z}^n_*}\int_\eta   A_k(t,\eta) (\widehat{M_\theta})_l(t) (\nb_\eta \pr_\eta^\al\widehat{f^w})_{k-l}\left(t,\eta-lt^{\rm ap}\right)\\
\nn&\cdot \nb_\eta  A_k(t,\eta)\pr_\eta^\al(\overline{\widehat{f^w}})_k(t,\eta)d\eta={\sf CN}_{2;M_\theta;\ne,(1)}+{\sf CN}_{2;M_\theta;\ne,(2)}.
\end{align}
Using again \eqref{A-HL}  and \eqref{Young2}, we have
\begin{align}
{\sf CN}_{2;M_\theta;\ne,(1)}^{\rm  HL}\nn\les&\nu\lm_1^2 e^{-2(|\al|+1)\nu t}\sum_{k\in\mathbb{Z}^n, l\in\mathbb{Z}^n_*}\int_\eta   e^{\dl_1\nu^\frak{e}t} e^{\lm(t)\la l,lt^{\rm ap}\ra^s}\la l,lt^{\rm ap}\ra^{ \sig_1 } |(\widehat{M_\theta})_l(t)| \\
\nn&\times e^{c\lm(t)\la k-l,\eta-lt^{\rm ap}\ra^s}\left|(\nb_\eta \pr_\eta^\al\widehat{f^w})_{k-l}\left(t,\eta-lt^{\rm ap}\right)\right| \left| A\nb_\eta\pr_\eta^\al(\overline{\widehat{f^w}})_k(t,\eta)\right|d\eta\\
\nn\les&\lm_1e^{-(\dl-\dl_1)\nu^\fr13t}\left\|B^{ \sig_1 }(M_\theta)_{\ne}(t)\right\|_{L^2_x}\left(\nu\lm_1e^{-2(|\al|+1)\nu t}\left\| A(vv^\al f^w)\right\|_{L^2}^2\right).
\end{align}
In view of \eqref{A-LH} and \eqref{Young1}, one deduces that
\begin{align}
{\sf CN}_{2;M_\theta;\ne,(1)}^{\rm LH}\nn\les&\nu\lm_1^2 e^{-2(|\al|+1)\nu t}\sum_{k\in\mathbb{Z}^n,\, l\in\mathbb{Z}^n_*}\int_\eta   e^{-(\dl-\dl_1)\nu^\fr13 t} e^{\dl\nu^\fr13t} e^{c\lm(t)\la l,lt^{\rm ap}\ra^s} |(\widehat{M_\theta})_l(t)| \\
\nn&\times \left|A(\nb_\eta \pr_\eta^\al\widehat{f^w})_{k-l}\left(t,\eta-lt^{\rm ap}\right)\right| \left| A\nb_\eta\pr_\eta^\al(\overline{\widehat{f^w}})_k(t,\eta)\right|d\eta\\
\nn\les&\lm_1e^{-(\dl-\dl_1)\nu^\fr13t}\left\|B^{ \sig_1 }(M_\theta)_{\ne}(t)\right\|_{L^2_x}\left(\nu\lm_1e^{-2(|\al|+1)\nu t}\left\| A(vv^\al f^w)\right\|_{L^2}^2\right).
\end{align}
Similarly, thanks to \eqref{A_eta3}, we obtain
\begin{align}
{\sf CN}_{2;M_\theta;\ne,(2)}^{\rm  HL}\nn\les&\nu\lm_1^2 e^{-(2|\al|+1)\nu t}\sum_{k\in\mathbb{Z}^n,\, l\in\mathbb{Z}^n_*}\int_\eta   e^{\dl_1\nu^\frak{e}t} e^{\lm(t)\la l,lt^{\rm ap}\ra^s}\la l,lt^{\rm ap}\ra^{ \sig_1 } |(\widehat{M_\theta})_l(t)| \\
\nn&\times e^{c\lm(t)\la k-l,\eta-lt^{\rm ap}\ra^s}\left|(\nb_\eta \pr_\eta^\al\widehat{f^w})_{k-l}\left(t,\eta-lt^{\rm ap}\right)\right| \left| A\pr_\eta^\al(\overline{\widehat{f^w}})_k(t,\eta)\right|d\eta\\
\nn\les&\lm_1\sqrt{\nu\lm_1}e^{-(\dl-\dl_1)\nu^\fr13t}\left\| {\bf B}  M_\theta(t)\right\|_{L^2_x}\left(\sqrt{\nu\lm_1}e^{-(|\al|+1)\nu t}\left\| A(vv^\al f^w)\right\|_{L^2}\right)\\
&\times\left(e^{-|\al|\nu t}\left\|A(v^\al f^w)\right\|_{L^2}\right),
\end{align}
and
\begin{align}
{\sf CN}_{2;M_\theta;\ne,(2)}^{\rm LH}
\nn\les&\lm_1\sqrt{\nu\lm_1}e^{-(\dl-\dl_1)\nu^\fr13t}\left\| {\bf B}  M_\theta(t)\right\|_{L^2_x}\left(\sqrt{\nu\lm_1}e^{-(|\al|+1)\nu t}\left\|A(vv^\al f^w)\right\|_{L^2}\right)\\
&\times\left(e^{-|\al|\nu t}\left\| A(v^\al f^w)\right\|_{L^2}\right).
\end{align}

The other `{\rm HL}' terms that does not involve the loss of velocity localizations can be treaded in a  similar manner as the corresponding  `{\rm LH}' terms. 
To avoid unnecessary duplication, we only show that the representative nonlinear term $\mathsf{CN}_{0;{\bf2}}[\rho]_{(1)}^{\rm HL}$ now can be treated analogously to \eqref{e-CNLH101}. In fact, using \eqref{A-HL} and \eqref{Young2} yields
\begin{align}
\mathsf{CN}_{0;{\bf2}}[\rho]_{(1)}^{\rm HL}\nn&\les\nu e^{-2|\al|\nu t} \sum_{{k\in \mathbb{Z}^n,\, l\in\mathbb{Z}^n_*}}\int_\eta\left|\mathcal{F}[\pr_v^tA(v^\al f^w)]_k(t,\eta) \right| e^{\dl\nu^\fr13t} e^{\lm(t)\la l, lt^{\rm ap}\ra^s}\la l,lt^{\rm ap}\ra^{ \sig_1 } \left|\hat{\rho}_l(t) \right|\\
\nn&\times e^{-(\dl-\dl_1) \nu^\fr13 t} e^{c\lm(t)\la k-l,\eta-lt^{\rm ap}\ra^s}\left|\mathcal{F}[\pr_v^t(v^\al f^w)]_{k-l}\left(t,\eta-lt^{\rm ap}\right)\right|d\eta\\
\les&\left(\nu e^{-|\al|\nu t} \left\|\pr_v^t A(v^\al f^w)\right\|_{L^2}\right)^2 \left\|B^{ \sig_1 }\rho(t) \right\|_{L^2_x}.
\end{align}

\section{Zero mode estimates}
In this section, we improve \eqref{H-f0}. To this end,  we first give the equations of $\hat{f}_0(t,\eta)$,   $(\nb_\eta\hat{f})_0(t,\eta)$, $(\nb_\eta^2\hat{f})_0(t,\eta)$ and $(\nb_\eta^3\hat{f})_0(t,\eta)$. In fact, direct calculations  from  \eqref{eq-f} show that
\be\label{eq-f0}
\pr_t\hat{f}_0(t,\eta)+\nu|e^{\nu t}\eta|^2\hat{f}_0(t,\eta)=\mathcal{N}_0(t,e^{\nu t}\eta),
\ee
\begin{align}\label{eq-1f0}
\pr_t(\nb_\eta\hat{f})_0(t,\eta)+\nu |e^{\nu t}\eta|^2(\nb_\eta\hat{f})_0(t,\eta)=-2\nu e^{2\nu t}\eta\hat{f}_0(t,\eta)+ \nb_\eta\left(\mathcal{N}_0(t, e^{\nu t}\eta)\right),
\end{align}
\begin{align}\label{2f0}
\nn&\pr_t(\pr^\eta_{jj'} \hat{f})_0(t,\eta)+\nu |e^{\nu t}\eta|^2(\pr_{jj'}^\eta\hat{f})_0(t,\eta)\\
\nn=&-2\nu e^{2\nu t}{\eta}_{j'}(\pr^\eta_j\hat{f})_0(t,\eta)-2\nu e^{2\nu t}{\eta}_j(\pr^\eta_{j'}\hat{f})_0(t,\eta)-2\nu e^{2\nu t}\dl_{jj'}\hat{f}_0(t,\eta)\\
&+\pr^\eta_{jj'}\left(\mathcal{N}_0(t, e^{\nu t}\eta)\right),
\end{align}
and
\begin{align}\label{eq-3f0}
\nn&\pr_t (\pr^\eta_{jj'j''}\hat{f})_0(t,\eta)+\nu |e^{\nu t}\eta|^2(\pr^\eta_{jj'j''}\hat{f})_0(t,\eta)\\
\nn=&-2\nu e^{2\nu t}\eta_{j''}(\pr^\eta_{jj'}\hat{f})_0(t,\eta)-2\nu e^{2\nu t}\eta_{j'}( \pr^\eta_{j''j}\hat{f})_0(t,\eta)-2\nu e^{2\nu t}\eta_{j}( \pr^\eta_{j'j''}\hat{f})_0(t,\eta)\\
\nn&-2\nu \dl_{j'j''} e^{2\nu t}(\pr^\eta_j \hat{f})_0(t,\eta)-2\nu \dl_{j''j} e^{2\nu t}(\pr^\eta_{j'} \hat{f})_0(t,\eta)-2\nu \dl_{jj'} e^{2\nu t}(\pr^\eta_{j''} \hat{f})_0(t,\eta)\\
&+\pr^\eta_{jj'j''}\left(\mathcal{N}_0(t,e^{\nu t}\eta)\right),
\end{align}
where
\begin{align}\label{e-N0}
\mathcal{N}_0(t,e^{\nu t}\eta)=-\sum_{k\in\mathbb{Z}^n_*}\hat{E}_k(t)\cdot i e^{\nu t}\eta \hat{f}_{-k}\left(t, \eta-kt^{\rm ap}\right)+\nu (\widehat{\mathcal{C}_\mu})_0(t,e^{\nu t}\eta)+\nu(\widehat{\mathcal{C}[g]})_0(t,e^{\nu t}\eta).
\end{align}
More precisely,
$(\widehat{\mathcal{C}_\mu})_0(t,e^{\nu t}\eta)=-(\widehat{M_\theta})_0(t)|e^{\nu t}\eta|^2\hat{\mu }(e^{\nu t}\eta)$ 
and
\begin{align}\label{e-g0}
\nn(\widehat{\mathcal{C}[g]})_0(t,e^{\nu t}\eta)=&-\sum_{k\in\mathbb{Z}^n_*}\hat{\rho}_k(t)|e^{\nu t}\eta|^2\hat{f}_{-k}\left(t, \eta-kt^{\rm ap}\right)-\sum_{k\in\mathbb{Z}^n_*}\hat{\rho}_k(t)\eta\cdot(\nb_\eta\hat{f})_{-k}\left(t, \eta-kt^{\rm ap}\right)\\
\nn&-\sum_{k\in\mathbb{Z}^n}(\widehat{M_\theta})_k(t)|e^{\nu t}\eta|^2\hat{f}_{-k}\left(t, \eta-kt^{\rm ap}\right)\\
&-\sum_{k\in\mathbb{Z}^n_*}(\widehat{M_1})_k(t)\cdot ie^{\nu t}\eta\hat{f}_{-k}\left(t, \eta-kt^{\rm ap}\right).
\end{align}

We now introduce the solution operator 
\be\label{S0eta}
S_0(t,0;\eta)=e^{-\fr12(e^{2\nu t}-1)|\eta|^2},\quad{\rm and}\quad S_0(t,\tau;\eta)=e^{-\fr12(e^{2\nu t}-e^{2\nu\tau})|\eta|^2}.
\ee
Applying  Duhamel's principle to \eqref{eq-f0}--\eqref{eq-3f0},  we get the   following expressions one by one:
\begin{align}\label{f0'}
\hat{f}_0(t,\eta)=&S_0(t,0;\eta)(\widehat{f_{\mathrm{in}}})_0(\eta)+\int_0^tS_0(t,\tau;\eta)\mathcal{N}_0(\tau, e^{\nu\tau}\eta)d\tau,\\
\label{1f0}
(\nb_\eta\hat{f})_0(t,\eta)\nn=&S_0(t,0;\eta)(\nb_\eta\widehat{f_{\mathrm{in}}})_0(\eta)+\int_0^tS_0(t,\tau;\eta)\nb_\eta\left(\mathcal{N}_0(\tau, e^{\nu\tau}\eta)\right)d\tau\\
&-2\nu\int_0^tS_0(t,\tau;\eta)e^{2\nu\tau}\eta \hat{f}_0(\tau, \eta) d\tau,\\
\label{2f0'}
(\pr^\eta_{jj'} \hat{f})_0(t,\eta)\nn=&S_0(t,0;\eta)(\pr^\eta_{jj'}\widehat{f_{\mathrm{in}}})_0(\eta)+\int_0^tS_0(t,\tau;\eta)\pr^\eta_{jj'}\left(\mathcal{N}_0(\tau, e^{\nu\tau}\eta)\right)d\tau\\
\nn&-2\nu\int_0^tS_0(t,\tau;\eta)e^{2\nu\tau}\eta_{j'}(\pr^\eta_j\hat{f})_0(\tau,\eta)d\tau\\
\nn&-2\nu\int_0^tS_0(t,\tau;\eta)e^{2\nu\tau}\eta_{j}(\pr^\eta_{j'}\hat{f})_0(\tau,\eta)d\tau\\
&-2\nu \dl_{jj'}\int_0^tS_0(t,\tau;\eta)e^{2\nu\tau}\hat{f}_0(\tau,\eta)d\tau,
\end{align}
and
\begin{align}\label{3f0}
&(\pr^\eta_{jj'j''} \hat{f})_0(t,\eta)\nn=S_0(t,0;\eta)(\pr^\eta_{jj'j''}\widehat{f_{\mathrm{in}}})_0(\eta)+\int_0^tS_0(t,\tau;\eta)\pr^\eta_{jj'j''}\left(\mathcal{N}_0(\tau, e^{\nu\tau}\eta)\right)d\tau\\
\nn& -2\nu\int_0^tS_0(t,\tau;\eta)e^{2\nu\tau}\eta_{j''}  (\pr^\eta_{jj'}\hat{f})_0(\tau,\eta)d\tau
-2\nu\int_0^tS_0(t,\tau;\eta)e^{2\nu\tau}\eta_{j'}  (\pr^\eta_{j''j}\hat{f})_0(\tau,\eta)d\tau\\
\nn&-2\nu\int_0^tS_0(t,\tau;\eta)e^{2\nu\tau}\eta_{j}  (\pr^\eta_{j'j''}\hat{f})_0(\tau,\eta)d\tau
-2\nu\dl_{j'j''}\int_0^tS_0(t,\tau;\eta)e^{2\nu\tau}  (\pr^\eta_{j}\hat{f})_0(\tau,\eta)d\tau\\
&-2\nu\dl_{j''j}\int_0^tS_0(t,\tau;\eta)e^{2\nu\tau}  (\pr^\eta_{j'}\hat{f})_0(\tau,\eta)d\tau
-2\nu\dl_{jj'}\int_0^tS_0(t,\tau;\eta)e^{2\nu\tau}  (\pr^\eta_{j''}\hat{f})_0(\tau,\eta)d\tau.
\end{align}
By \eqref{f0'}--\eqref{2f0'} and the semigroup property of $S_0(\cdot,\cdot;\eta)$, we have
\begin{align}\label{int-f0}
\nn& \nu\int_0^tS_0(t,\tau;\eta)e^{2\nu\tau} \left|\hat{f}_0(\tau, \eta)\right| d\tau \\
\nn\le&\nu\int_0^te^{2\nu\tau}d\tau  S_0(t,0;\eta)\left|(\widehat{f_{\mathrm{in}}})_0(\eta)\right|\\
\nn&+\nu\int_0^t\int_{\tau'}^t e^{2\nu\tau}d\tau  S_0(t,\tau';\eta)\left|\mathcal{N}_0(\tau', e^{\nu\tau'}\eta)\right|d\tau'\\
\les& e^{2\nu t}|\eta| \left(S_0(t,0;\eta) \fr{\left|(\widehat{f_{\mathrm{in}}})_0(\eta)\right|}{|\eta|}+\int_0^t S_0(t,\tau';\eta)\fr{\left|\mathcal{N}_0(\tau', e^{\nu\tau'}\eta)\right|}{|\eta|}d\tau'\right),
\end{align}

\begin{align}\label{int-1f0}
\nn&\nu \int_0^tS_0(t,\tau;\eta)e^{2\nu\tau} \left|(\nb_\eta\hat{f})_0(\tau,\eta)\right|d\tau\\
\nn\les&e^{2\nu t} \left(S_0(t,0;\eta) \left|(\nb_\eta\widehat{f_{\mathrm{in}}})_0(\eta)\right|+\int_0^t  S_0(t,\tau';\eta) \left|\nb_\eta\left(\mathcal{N}_0(\tau', e^{\nu\tau'}\eta)\right)\right|d\tau' \right)\\
&+ e^{2\nu t}|\eta|  \left(\nu\int_0^t    S_0(t,\tau';\eta)e^{2\nu\tau'}\left|\hat{f}_0(\tau', \eta)\right| d\tau'\right),
\end{align}
and
\begin{align}\label{int-2f0}
\nn&\nu\int_0^tS_0(t,\tau;\eta)e^{2\nu\tau}|\eta|\left|(\nb^2_\eta\hat{f})_0(\tau,\eta)\right|d\tau\\
\nn\les&e^{2\nu t}|\eta| \left(S_0(t,0;\eta) \left|(\nb_\eta^2\widehat{f_{\mathrm{in}}})_0(\eta)\right|+\int_0^t  S_0(t,\tau';\eta) \left|\nb_\eta^2\left(\mathcal{N}_0(\tau', e^{\nu \tau'\eta})\right)\right|d\tau'\right)\\
\nn&+|e^{\nu t}\eta|^2\left(\nu\int_0^t S_0(t,\tau';\eta)e^{2\nu\tau'}\left|(\nb_\eta\hat{f})_0(\tau',\eta)\right|d\tau'\right)\\
&+e^{2\nu t}|\eta| \left( \nu \int_0^t S_0(t,\tau';\eta) e^{2\nu\tau'}\left|\hat{f}_0(\tau',\eta)\right| d\tau'\right).
\end{align}
It follows from \eqref{f0'}--\eqref{int-2f0} that, for any $j\in\N$, we have
\begin{align*}
e^{\nu t}|e^{\nu t}\eta|^j\left|\hat{f}_0(t,\eta)\right|\les |e^{\nu t}\eta|^{j+1}\left(S_0(t,0;\eta)\fr{\left|(\widehat{f}_{\rm in})_0(\eta) \right|}{|\eta|} +\int_0^tS_0(t,\tau;\eta)\fr{\left|\mathcal{N}_0(\tau,e^{\nu\tau}\eta) \right|}{|\eta|}d\tau\right),
\end{align*}
\begin{align*}
\nn&|e^{\nu t}\eta|^j\left|(\nb_\eta\hat{f})_0(t,\eta) \right|\\
\nn\les&|e^{\nu t}\eta|^j\left(S_0(t,0;\eta)\left|(\nb_\eta\widehat{f}_{\rm in})_0(\eta)\right|+\int_0^tS_0(t,\tau;\eta)\left|\nb_\eta\left(\mathcal{N}_0(\tau, e^{\nu\tau}\eta) \right)\right|d\tau\right)\\
&+|e^{\nu t}\eta|^{j+2}\left(S_0(t,0;\eta)\fr{|(\widehat{f_{\rm in}})_0(\eta)|}{|\eta|}+\int_0^tS_0(t,\tau;\eta)\fr{\left| \mathcal{N}_0(\tau, e^{\nu\tau}\eta)\right|}{|\eta|}d\tau\right),
\end{align*}

\begin{align*}
\nn&e^{-\nu t} |e^{\nu t}\eta|^j \left|(\nb_\eta^2\hat{f})_0(t,\eta) \right|\\
\nn\les&e^{-\nu t} |e^{\nu t}\eta|^j \left(S_0(t,0;\eta)\left|(\nb_\eta^2\widehat{f_{\mathrm{in}}})_0(\eta)\right|+\int_0^tS_0(t,\tau;\eta)\left|\nb_\eta^2\left(\mathcal{N}_0(\tau,e^{\nu\tau}\eta)\right)\right|d\tau\right)\\
\nn&+|e^{\nu t}\eta|^{j+1} \left(S_0(t,0;\eta)\left|(\nb_\eta\widehat{f_{\mathrm{in}}})_0(\eta)\right|+\int_0^t  S_0(t,\tau';\eta)\left|\nb_\eta\left(\mathcal{N}_0(\tau',e^{\nu\tau'}\eta)\right)\right|d\tau' \right)\\
&+(|e^{\nu t}\eta|^{j+3}+|e^{\nu t}\eta|^{j+1}) \left(S_0(t,0;\eta)\fr{|(\widehat{f_{\rm in}})_0( \eta)|}{|\eta|}+\int_0^tS_0(t,\tau;\eta)\fr{|\mathcal{N}_0(\tau, e^{\nu\tau}\eta)|}{|\eta|} d\tau\right),
\end{align*}
and
\begin{align*}
\nn&e^{-2\nu t}|e^{\nu t} \eta|^j\left| \nb_\eta^3\hat{f}_0(t,\eta)\right|\\
\nn\les&e^{-2\nu t}|e^{\nu t}\eta|^j\left(S_0(t,0;\eta)\left|(\nb_\eta^3\widehat{f_{\mathrm{in}}})_0(\eta)\right|+\int_0^tS_0(t,\tau;\eta)\left|\nb_\eta^3\left(\mathcal{N}_0(\tau, e^{\nu\tau}\eta)\right)\right|d\tau\right)\\
\nn&+e^{-\nu t}|e^{\nu t}\eta|^{j+1} \left(S_0(t,0;\eta) \left|(\nb_\eta^2\widehat{f_{\mathrm{in}}})_0(\eta)\right|+\int_0^t  S_0(t,\tau';\eta) \left|\nb_\eta^2\left(\mathcal{N}_0(\tau',e^{\nu\tau'}\eta)\right)\right|d\tau'\right)\\
\nn&+ \left(|e^{\nu t}\eta|^{j+2}+|e^{\nu t}\eta|^{j}\right)\Bigg(S_0(t, 0;\eta)\left|(\nb_\eta\widehat{f_{\rm in}})_0(\eta) \right|\\
\nn&+\int_0^t S_0(t,\tau';\eta)\left|\nb_\eta\left(\mathcal{N}_0(\tau', e^{\nu\tau'}\eta)\right)\right|d\tau'\Bigg)\\
&+\left( |e^{\nu t}\eta|^{j+4}+|e^{\nu t}\eta|^{j+2}\right)\left(S_0(t,0;\eta)\fr{|(\widehat{f_{\rm in}})_0(\eta)|}{|\eta|}+
\int_0^tS_0(t,\tau;\eta)\fr{|\mathcal{N}_0(\tau,e^{\nu\tau}\eta)|}{|\eta|}d\tau \right).
\end{align*}
Collecting the above four estimates,  we arrive at
\begin{align}\label{e-al-f0-1}
\nn&\sum_{|\al|\le3}e^{-(|\al|-1)\nu t} |e^{\nu t}\eta|^j\left|(\pr^\al_\eta\hat{f})_0(t,\eta)\right|\\
\nn\les&\sum_{1\le\ell\le4}|e^{\nu t}\eta|^{j+\ell}\left(S_0(t,0;\eta)\fr{\left|(\widehat{f}_{\rm in})_0(\eta) \right|}{|\eta|} +\int_0^tS_0(t,\tau;\eta)\fr{\left|\mathcal{N}_0(\tau,e^{\nu\tau}\eta) \right|}{|\eta|}d\tau\right)\\
\nn&+\sum_{0\le\ell\le2}|e^{\nu t}\eta|^{j+\ell}\left(S_0(t,0;\eta)\left|(\nb_\eta\widehat{f}_{\rm in})_0(\eta)\right|+\int_0^tS_0(t,\tau;\eta)\left|\nb_\eta\left(\mathcal{N}_0(\tau, e^{\nu\tau}\eta)\right) \right|d\tau\right)\\
\nn&+\sum_{0\le\ell\le1}e^{-\nu t} |e^{\nu t}\eta|^{j+\ell} \Bigg(S_0(t,0;\eta)\left|(\nb_\eta^2\widehat{f_{\mathrm{in}}})_0(\eta)\right|\\
\nn&+\int_0^tS_0(t,\tau;\eta)\left|\nb_\eta^2\left(\mathcal{N}_0(\tau,e^{\nu\tau}\eta)\right)\right|d\tau\Bigg)\\
&+e^{-2\nu t}|e^{\nu t}\eta|^j\left(S_0(t,0;\eta)\left|(\nb_\eta^3\widehat{f_{\mathrm{in}}})_0(\eta)\right|+\int_0^tS_0(t,\tau;\eta)\left|\nb_\eta^3\left(\mathcal{N}_0(\tau, e^{\nu\tau}\eta)\right)\right|d\tau\right),
\end{align}
for any $j\in\N$. By \eqref{S0eta}, for any fixed $q\in\N$, there hold
\begin{align}\label{e-S0eta1}
|e^{\nu t}\eta|^qS_0(t,0;\eta)
&=|e^{\nu t}\eta|^qS_0(t,0;\eta)\left({\bf 1}_{\nu t\le1}+{\bf 1}_{\nu t>1}\right)\les|\eta|^q{\bf 1}_{\nu t\le1}+e^{-\fr14|\eta|^2}{\bf 1}_{\nu t>1},\\
\label{e-S0eta2}
|e^{\nu t}\eta|^qS_0(t,\tau;\eta)
&\les_q\left(\left(|e^{\nu t}\eta|^2-|e^{\nu\tau}\eta|^2 \right)^{\fr{q}{2}}+|e^{\nu\tau}\eta|^q\right)e^{-\fr12(e^{2\nu t}-e^{2\nu\tau})|\eta|^2}
\les \la e^{\nu\tau}\eta\ra^{q}.
\end{align}
Substituting \eqref{S0eta}, \eqref{e-S0eta1} and \eqref{e-S0eta2} into \eqref{e-al-f0-1}, and take sum with respect to $j$ over $\{0, 1, 2,\cdots, \sig_0-6\}$, we are led to
\begin{align}\label{e-al-f0-2}
\nn&\sum_{|\al|\le3}e^{-(|\al|-1)\nu t} e^{\lm(t)\la\eta\ra^s} \la e^{\nu t}\eta\ra^{\sig_0-6}\left|(\pr^\al_\eta\hat{f})_0(t,\eta)\right|\\
\nn\les&e^{\lm(0)\la\eta\ra^s}\sum_{|\al|\le3}\la\eta\ra^{\sig_0-3-|\al|}\left|(\nb_\eta^{|\al|}\widehat{f_{\rm in}})_0(\eta)\right|+\Big|\fr{(\widehat{f}_{\rm in})_0(\eta) }{|\eta|}\Big|\\
\nn& +\int_0^te^{\lm(\tau)\la\eta\ra^s}\la e^{\nu\tau}\eta \ra^{\sig_0-2}\fr{\left|\mathcal{N}_0(\tau,e^{\nu\tau}\eta) \right|}{|\eta|}d\tau\\
&+\sum_{1\le|\al|\le3}\int_0^te^{\lm(\tau)\la\eta\ra^s} \la e^{\nu\tau}\eta\ra^{\sig_0-3-|\al|}\left|\nb_\eta^{|\al|}\left(\mathcal{N}_0(\tau, e^{\nu\tau}\eta)\right)\right|d\tau.
\end{align}
Since $(\hat{f})_0(t,0)=\hat{\rho}_0(t)=0$, one easily deduces that 
\[
\Big|\fr{(\widehat{f}_{\rm in})_0(\eta) }{|\eta|}\Big|\les \left\|\nb_\eta(\widehat{f}_{\rm in})_0(\eta)\right\|_{L^\infty_\eta}\les \|\la v\ra^m f_{\rm in}\|_{L^2_v}.
\]
The other initial term can be bounded by $\left\| e^{\lm(0)\la \pr_v \ra^s}\la\pr_v\ra^{\sig_0-3}(\la v\ra^m f_{\rm in}) \right\|_{L^2_v}$ provided we choose $m>\fr{n}{2}+3.$

One can see from \eqref{e-N0} that the two nonlinear terms stemming from $\fr{\left|\mathcal{N}_0(\tau,e^{\nu\tau}\eta) \right|}{|\eta|}$ and $\left|(\nb_\eta^{|\al|}\left(\mathcal{N}_0(\tau, e^{\nu\tau}\eta)\right)\right|$ on the right hand side of \eqref{e-al-f0-2} are of the same type. Furthermore, there is no need to distinguish the collisionless contributions and collision contributions since we are working in a space with low norm. Accordingly, we only show the details of the  treatment of the first nonlinear term.  The second nonlinear term in \eqref{e-al-f0-2} can be treated similarly without additional complications. In fact, from \eqref{e-N0}, it is natural to write:
\begin{align*}
\int_0^te^{\lm(\tau)\la \eta\ra^s}\la e^{\nu \tau}\eta\ra^{\sig_0-2}\fr{\left|\mathcal{N}_0(\tau, e^{\nu\tau}\eta)\right|}{|\eta|}d\tau=I_{E}+I_{\mu}+I_{g}.
\end{align*}
Noting that $|\eta|\le |k\tau^{\rm ap}|+|\eta-k\tau^{\rm ap}|$, by using the partition of unity $1={\bf 1}_{|\eta-k\tau^{\rm ap}|\le 2|k\tau^{\rm ap}|}+{\bf 1}_{|k\tau^{\rm ap}|\le\fr12|\eta-k\tau^{\rm ap}|}$, \eqref{emb} and \eqref{tau-up},  we have
\begin{align}\label{I1}
I_E\nn\les&\sum_{k\in\Z^n_*}\int_0^t e^{\lm(\tau)\la \eta\ra^s}\la\eta\ra^{\sig_0-2} |\hat{\rho}_k(\tau)| e^{(\sig_0-1)\nu \tau}\left| \hat{f}_{-k}\left(\tau, \eta-k\tau^{\mathrm{ap}}\right)\right| d\tau\\
\nn\les&\sum_{k\in\Z^n_*}\int_0^t e^{\lm(\tau)\left\la k\tau^{\mathrm{ap}}\right\ra^s}\left\la k\tau^{\mathrm{ap}}\right\ra^{\sig_0-2}\la \tau\ra |\hat{\rho}_k(\tau)| \\
\nn&\times e^{(\sig_0-1)\nu \tau} e^{c\lm(\tau)\left\la \eta- k\tau^{\mathrm{ap}}\right\ra^s}\left| \hat{f}_{-k}\left(\tau, \eta-k\tau^{\mathrm{ap}}\right)\right| \fr{1}{\la\tau\ra}d\tau\\
\nn&+\sum_{k\in\Z^n_*}\int_0^t  e^{c\lm(\tau)\left\la k\tau^{\mathrm{ap}}\right\ra^s}|\hat{\rho}_k(\tau)|\la \tau\ra e^{(\sig_0-1)\nu \tau} e^{\lm(\tau)\left\la \eta- k\tau^{\mathrm{ap}}\right\ra^s}\\
\nn&\times\left\la\eta-k\tau^{\mathrm{ap}}\right\ra^{\sig_0-2}\left| \hat{f}_{-k}\left(\tau, \eta-k\tau^{\mathrm{ap}}\right)\right|\fr{1}{\la \tau\ra} d\tau\\
\nn\les&\sup_t\left(e^{-m\nu t}\left\|e^{\lm(t)\la \eta\ra^s}\la\eta\ra^{\sig_0-2}\hat{f}_k(t,\eta){\bf 1}_{k\ne0}\right\|_{L^2_k L^\infty_\eta}\right)\int_0^t\left\|{\bf B}\hat{\rho}_k(\tau)\right\|_{L^2_k}\fr{1}{\la\tau\ra}d\tau\\
\les&\eps\nu^{\gamma}\|{\bf B}\rho\|_{L^2_{t,x}}.
\end{align}
By \eqref{e-Mtheta0}, direct calculations give
\begin{align}\label{I2}
I_\mu\nn\les& \nu\int_0^t e^{\lm(\tau)\la \eta\ra^s}\la e^{\nu\tau}\eta\ra^{\sig_0-2}\left|(\widehat{M_\theta})_0(\tau)\right| e^{\nu\tau}|e^{\nu \tau}\eta| \left|\hat{\mu }(e^{\nu \tau}\eta)\right|d\tau\\
\nn\les&\sup_{\tau}\left(e^{2\nu \tau}\left|(\widehat{M_\theta})_0(\tau)\right| \right)\int_0^t \nu e^{-\nu \tau}\left(e^{\lm(0)\la \eta\ra^s}\la e^{\nu\tau}\eta\ra^{\sig_0-1}e^{-\fr{|e^{\nu\tau}\eta|^2}{2}} \right)d\tau\\
\les&\sup_{t}\left(e^{2\nu t}\left|(\widehat{M_\theta})_0(t)\right| \right)\les \eps^2\nu^{2\gamma}.
\end{align}
For $I_g$, according to \eqref{e-g0}, we write $I_g=I_{g;(1)}+I_{g;(2)}+I_{g;(3)}+I_{g;(4)}$.

We only sketch the treatment of $I_{g;(2)}$ and $I_{g;(3)}$ since the other two terms can be treated similarly. Indeed, similar to \eqref{I1} and \eqref{I2}, we have
\begin{align}\label{Ig2}
I_{g;(2)}\nn\les&\nu\sum_{k\in\Z^n_*}\int_0^t e^{\lm(\tau)\la \eta\ra^s}\la\eta\ra^{\sig_0-2} e^{(\sig_0-1)\nu \tau} |\hat{\rho}_k(t)| \left| e^{-\nu\tau}(\nb_\eta\hat{f})_{-k}\left(t, \eta-k\tau^{\mathrm{ap}}\right)\right| d\tau\\
\nn\les&\nu\sup_t\left(e^{-(m+1)\nu t}\left\|e^{\lm(t)\la \eta\ra^s}\la\eta\ra^{\sig_0-2}(\nb_\eta\hat{f})_k(t,\eta){\bf 1}_{k\ne0}\right\|_{L^2_k L^\infty_\eta}\right)\|{\bf B}\rho\|_{L^2_{t,x}}\\
\les& \eps\nu^{\gamma+1}\|{\bf B}\rho\|_{L^2_{t,x}},
\end{align}
and
\begin{align}\label{I4}
I_{g;(3)}\nn\les&\nu\sum_{k\in\Z^n_*}\int_0^t e^{\lm(\tau)\la \eta\ra^s}\la\eta\ra^{\sig_0-1} e^{\sig_0\nu \tau}|(\widehat{M_\theta})_k(\tau)| \left| \hat{f}_{-k}\left(\tau, \eta-k\tau^{\mathrm{ap}}\right)\right| d\tau\\
\nn&+\nu\int_0^t e^{\lm(\tau)\la \eta\ra^s}\la\eta\ra^{\sig_0-1} e^{\sig_0\nu \tau} |(\widehat{M_\theta})_0(t)| \left| \hat{f}_{0}\left(t, \eta\right)\right| d\tau\\
\les&\sup_t \left( e^{-m\nu t}\left\|e^{\lm(t)\la \eta\ra^s}\la\eta\ra^{\sig_0-1}\hat{f}_k(t,\eta)\right\|_{L^2_k L^\infty_\eta}\right)\\
\nn&\times\left(\|{\bf B}M_\theta\|_{L^2_{t,x}}+\sup_t\left(e^{(\sig_0+1+m)\nu t} |(\widehat{M_\theta})_0(t)| \right)\right)
\les\eps^2\nu^{2\gamma}.
\end{align}

We would like to remark that, for the nonlinear contributions involving $(\nb_\eta^{|\al|}\left(\mathcal{N}_0(\tau, e^{\nu\tau}\eta)\right)$,  when three derivatives lands on $(\nb_\eta \hat{f})_{-k}(t,\eta-kt^{\rm ap})$ that comes from the second term on the right hand side of \eqref{e-g0}, we have to deal with 
\begin{align}
\nn \tl{I}_{g;(2)}=\nu\sum_{k\in\Z^n_*}\int_0^t e^{\lm(\tau)\la \eta\ra^s}\la\eta\ra^{\sig_0-6} e^{(\sig_0-6)\nu \tau} |\hat{\rho}_k(t)| |\eta| \left| e^{-\nu\tau}(\nb_\eta^4\hat{f})_{-k}\left(t, \eta-k\tau^{\mathrm{ap}}\right)\right| d\tau.
\end{align}
Like \eqref{Ig2}, using \eqref{emb}, we have
\begin{align}
\tl{I}_{g;(2)}\nn\les&\nu\sup_t\left(e^{-( m+4)\nu t}\left\|e^{\lm(t)\la \eta\ra^s}\la\eta\ra^{\sig_0-6}(\nb_\eta^4\hat{f})_k(t,\eta){\bf 1}_{k\ne0}\right\|_{L^2_k H^{\fr{n}{2}+}_\eta}\right)\|{\bf B}\rho\|_{L^2_{t,x}}\\
\les&\nu \|f^w_{\ne}\|_{\mathcal{E}^{\sig_0+1,\fr25}_{m}}\|{\bf B}\rho\|_{L^2_{t,x}}\les \eps\nu^{\gamma+1}\|{\bf B}\rho\|_{L^2_{t,x}},
\end{align}
provided we choose $m>\fr{n}{2}+4$.

\begin{appendix}

\section{Elementary inequalities}\label{sec-app-ineq}
\begin{lem}[\cite{BM2015}]\label{lem-tri-in}
Let $0<s<1$ and $x, y\ge0$.
\begin{enumerate}
\item We have the triangle inequality in the opposite direction
\be\label{app1}
c_s\left(\la x\ra^s+\la y\ra^s\right)\le \la x+y\ra^s,
\ee
for some $c_s\in(0,1)$ depending only on $s$.
\item In general,
\be\label{app2}
\left|\la x\ra^s-\la y\ra^s \right|\les_s\fr{|x-y|}{\la x\ra^{1-s}+\la y\ra^{1-s}}.
\ee
\item If $x\le\fr{y}{K}$ for some $K>1$, then we have the improved triangle inequality
\be\label{app3}
 \la x+y\ra^s\le \la y\ra^s +\fr{s}{(K-1)^{1-s}}\la x\ra^s.
\ee
\item For $x\ge y$, we have the following improved triangle inequality 
\be\label{app4}
\la x+y\ra^s\le\left(\fr{\la x\ra}{\la x+y\ra}\right)^{1-s}\left(\la x\ra^s+\la y\ra^s\right).
\ee
\end{enumerate}
\end{lem}

\begin{coro} 
Let $0<s<1$ and $0\le x\le Ky$ for some $K\ge2$, then there exist  constants $c\in(0,1)$ and $C>1$ depending on $s$ and $K$, such that
\be\label{app5}
\la x+y\ra^s\le C+ \la y\ra^s+c\la x\ra^s.
\ee
\end{coro}
\begin{proof}
If $x\le \fr{y}{K}$, then \eqref{app5} follows from \eqref{app3} immediately. It suffices to consider the case $\fr{y}{K}<x\le Ky$. Let us focus on the sub-case $y\le x\le Ky$. Clearly, now the inequality \eqref{app4} holds, and $\fr{\la x\ra}{\la x+y\ra}=1$ when $x=y=0$. If $x\le1$, then 
\be
\la x+y\ra^s\le5^{\fr{s}{2}}.
\ee
If $x\ge1$, then
\be
\fr{\la x\ra}{\la x+y\ra}\le\sqrt{\fr{1+x^2}{1+x^2+y^2}}\le\sqrt{\fr{1+x^2}{1+(1+\fr{1}{K^2})x^2}}\le \sqrt{\fr{2}{2+\fr{1}{K^2}}},
\ee
since the function $\phi(z)=\fr{1+z}{1+az}$ is strictly decreasing in $z$ for each $a>1$. The treatment of rest sub-case $\fr{y}{K}\le x\le y$ is the same by swapping the positions of $x$ and $y$. Taking $C=5^{\fr{s}{2}}$ and $c=\min\left\{\fr{s}{(K-1)^{1-s}}, \sqrt{\fr{2}{2+\fr{1}{K^2}}} \right\}$, we complete the proof of this corollary.
\end{proof}

\begin{lem} The following hold for suitable integrable $f(k,\eta), g(k,\eta)$ and $h(k)$:
\be\label{Young1}
\Big|\sum_{k,l\in\mathbb{Z}^n}\int_{\mathbb{R}^n} f(k,\eta)h(l)g\left(k-l, \eta-lt^{\rm ap}\right)d\eta\Big|\les\|f\|_{L^2_{k,\eta}}\|g\|_{L^2_{k,\eta}}\left\|\la k\ra^{\fr{n}{2}+}h\right\|_{L^2_k},
\ee
and
\be\label{Young2}
\Big|\sum_{k,l\in\mathbb{Z}^n}\int_{\mathbb{R}^n} f(k,\eta)h(l)g\left(k-l, \eta-lt^{\rm ap}\right)d\eta\Big|\les\|f\|_{L^2_{k,\eta}}\left\|\la k\ra^{\fr{n}{2}+}g\right\|_{L^2_{k,\eta}}\left\|h\right\|_{L^2_k}.
\ee
\end{lem}

\begin{lem}\label{lem-alg}
Let $\lm>0, s\ge0$, and $\sig>\fr{n}{2}$. Denote
\be\label{underlineB}
\underline{B}_k^\sig(t)=e^{\lm\la k, kt^{\rm ap}\ra^s}\left\la k, kt^{\rm ap}\right\ra^\sig. 
\ee
Then there holds
\be\label{alg1}
\left\|\underline{B}_k^\sig(t) (\widehat{\phi \psi})_k(t)\right\|_{L^2_k}\les \left\|\underline{B}_k^\sig(t) \hat{\phi}_k(t)\right\|_{L^2_k}\left\|\underline{B}_k^\sig(t) \hat{\psi}_k(t)\right\|_{L^2_k},
\ee
and
\be\label{alg2}
\left\|\la k\ra^\fr12\underline{B}_k^\sig(t) (\widehat{\phi \psi})_k(t)\right\|_{L^2_k}\les \left\|\la k\ra^\fr12\underline{B}_k^\sig(t) \hat{\phi}_k(t)\right\|_{L^2_k}\left\|\la k\ra^\fr12\underline{B}_k^\sig(t) \hat{\psi}_k(t)\right\|_{L^2_k}.
\ee
\end{lem}
\begin{proof}
We use the following partition of unity 
\[
1={\bf 1}_{\left|k-l,(k-l)t^{\rm ap}\right|\le2\left|l,lt^{\rm ap}\right|}+{\bf 1}_{ \left|l,lt^{\rm ap}\right|<\fr12\left|k-l,(k-l)t^{\rm ap}\right|}.
\]
Both cases can be treated in the same way, so we only consider the case $\left|k-l,(k-l)t^{\rm ap}\right|\le2 \left|l,lt^{\rm ap}\right|$. In fact, thanks to \eqref{app5}, we have
\be\label{BHL}
\underline{B}_k^\sig(t){\bf 1}_{\left|k-l,(k-l)t^{\rm ap}\right|\le2\left|l,lt^{\rm ap}\right|}\les {\bf 1}_{l\ne0}\underline{B}_l^\sig(t) e^{c\lm\la k-l, (k-l)t^{\rm ap}\ra^s},
\ee
 for  some  $c\in(0,1)$. Accordingly, by using Young's convolution inequality, we are let to
\begin{align}\label{e-alg1}
\nn&\left\|\underline{B}_k^\sig(t){\bf 1}_{\left|k-l,(k-l)t^{\rm ap}\right|\le2\left|l,lt^{\rm ap}\right|}\left|(\widehat{\phi\psi})_k(t)\right|\right\|_{L^2_k}\\
\nn\les&\Big\|\sum_{l\in\Z^n_*}\underline{B}_l^\sig(t)|\hat{\phi}_l(t)| e^{c\lm\la k-l, (k-l)t^{\rm ap}\ra^s}|\hat{\psi}_{k-l}(t)|\Big\|_{L^2_k}\\
\les&\left\|\underline{B}_k^\sig(t) \hat{\phi}_k(t)\right\|_{L^2_k}\left\|e^{c\lm\la k, kt^{\rm ap}\ra^s}\left\la k\right\ra^{\fr{n}{2}+} \hat{\psi}_k(t)\right\|_{L^2_k}.
\end{align}
Then \eqref{alg1} follows immediately. The inequality \eqref{alg2} can be achieved by slightly modifying the above proof. In fact, we have
$
\la k\ra^\fr12\les \la l\ra^\fr12\la k-l\ra^\fr12.
$
Combining this with \eqref{BHL}, similar to \eqref{e-alg1}, we get \eqref{alg2}. This completes the proof of Lemma \ref{lem-alg}.
\end{proof}

\section{Properties of $S$}
 \begin{lem}[\cite{bedrossian2017suppression}]
 For all $t\ge0$ and $0<\nu\le1$, there hold
 \begin{enumerate}
 \item $S_k(t)$ is strictly decreasing in $t$ and there exists a constant $\dl_0$ such that
 \be\label{S-prop1}
 0<S_k(t)<\exp\left\{-\dl_0\fr{|k|^2}{\nu^2}\min\left\{(\nu t)^3, \nu t\right\}\right\}.
 \ee
 \item For all $p\in(0, 1]$ and all chosen $\dl\in(0,p\dl_0]$, we have
 \be\label{S-prop2}
S^p_k(t)\les_{p,\dl}e^{-\dl \nu^{\fr13}|k|^\fr23t}.
 \ee
  \end{enumerate}
 \end{lem}
 
\begin{lem}[\cite{bedrossian2017suppression}]
There holds for any fixed $p\in(0,1)$,
\be\label{S-prop3}
S^p_k(t-\tau)\fr{|k|}{2\nu^2}\left(1-e^{-\nu (t-\tau)} \right)^2\les \nu^{-\fr23}.
\ee
\end{lem}

\begin{lem}[\cite{bedrossian2017suppression}]
There holds for any fixed $p\in(0,1)$,
\be\label{S-prop4}
S^p_k(t-\tau)\fr{|k|^2}{8\nu^4}\left(1-e^{-\nu(t-\tau)} \right)^4\les\nu^{-\fr43}.
\ee
\end{lem}
\section{Properties of the multiplier $\mathfrak{m}(t,\nabla)$}\label{sec:multiplier m}

\begin{lem}\label{lem-com-m}
The following hold for the multiplier $\mathfrak{m}(t,\nabla)$:
\begin{enumerate}
\item There is a universal constant $c_{\frak{m}}>0$ independent of $t, k, \eta$ and $\nu$, such that
\be\label{m1}
c_{\frak{m}}\le \frak{m}_k(t,\eta)\le1.
\ee
\item For $k\ne0$, there holds
\be\label{m1'}
\fr{\left|\nb_\eta \mathfrak{m}_k(t,\eta) \right|}{\mathfrak{m}_k(t,\eta)}\les e^{\nu t}\nu^{\fr13}.
\ee
\item For $k\ne0$, there holds
\be\label{m2}
\nu^\fr13\les-\fr{\pr_t\frak{m}_k(t, \eta)}{\frak{m}_k(t,  \eta)}+\nu\left|e^{\nu t}\eta-k\fr{e^{\nu t}-1}{\nu} \right|^2.
\ee
\item For $k\ne0$,  $l\ne0$ and $k\ne l$, there holds
\be\label{nenene1}
\left|1-\fr{\mathfrak{m}_k(t,\eta)}{\mathfrak{m}_l(t,\xi)} \right|\les  \fr{ \la k-l\ra}{\max\{|k|, |l|\}}.
\ee
In addition, if $|\xi|\ge2|l|\left\la  t^{\rm ap} \right\ra$, we have
\be\label{nenene2}
\left|1-\fr{\mathfrak{m}_k(t,\eta)}{\mathfrak{m}_l(t,\xi)} \right|\les  \fr{t\la t^{\rm ap}\ra e^{\nu t}  \la k-l, \eta-\xi\ra}{\la\xi\ra}.
\ee
\item For $k=0, l\ne0$, there holds
\be\label{0ne}
\left|1-\fr{\mathfrak{m}_k(t,\eta)}{\mathfrak{m}_l(t,\xi)} \right|\les\min\left\{\fr{1}{|l|}, \fr{\la t \ra}{|\xi|}\right\}.
\ee
\item For $k\ne0, l=0$, there holds
\be\label{ne0}
\left|1-\fr{\mathfrak{m}_k(t,\eta)}{\mathfrak{m}_l(t,\xi)} \right|\les\min\left\{\fr{1}{|k|}, \fr{\la t \ra}{|\eta|}\right\}.
\ee
\end{enumerate}
\end{lem}
\begin{proof}
Recalling \eqref{def-m}, it is easy to check that
\be\label{def-m'}
\fr{\pr_t\mathfrak{m}_k(t,\eta)}{\mathfrak{m}_k(t,  \eta)}=\begin{cases}0, \quad\mathrm{if}\quad k=0,\\
-\fr{\nu^\fr13}{1+\nu^\fr23e^{2\nu t}\left| \eta-kt^{\mathrm{ap}}\right|^2}\fr{e^{2\nu t}|\nu\eta-k|^2}{\la e^{\nu t}(\nu\eta-k)\ra^2}, \quad\mathrm{if}\quad k\ne0, 
\end{cases}
\quad \text{and}\quad\mathfrak{m}_k(0,\eta)=1.
\ee
From \eqref{def-m} and \eqref{def-m'}, one can get \eqref{m1}--\eqref{m2} directly, see \cite{bedrossian2017suppression} and the associated updated arXiv version for more details. We only focus on the commutator estimates \eqref{nenene1}--\eqref{ne0} below. Before proceeding any further, we claim that for $\nu\eta\ne k$, there holds
\begin{align}\label{claim}
\int_0^t\fr{\nu^\fr13}{1+\left|\nu^\fr13 e^{\nu t'}\eta-\nu^\fr13k\fr{e^{\nu t'}-1}{\nu}\right|^2}\fr{e^{2\nu t'}|\nu\eta-k|^2}{\la e^{\nu t'}(\nu\eta-k)\ra^2}dt'\les \fr{1}{|k|}.
\end{align}
In fact, by using the change of variable $\tau=\left|\eta-\fr{k}{\nu} \right|e^{\nu t'}$, $d\tau=|\nu\eta-k|e^{\nu t'}dt'$, we have
\begin{align}
\nn&\int_0^t\fr{\nu^\fr13}{1+\left|\nu^\fr13 e^{\nu t'}\eta-\nu^\fr13k\fr{e^{\nu t'}-1}{\nu}\right|^2}\fr{e^{2\nu t'}|\nu\eta-k|^2}{\la e^{\nu t'}(\nu\eta-k)\ra^2}dt'\\
\nn=&\int_{|\eta-\fr{k}{\nu}|}^{|\eta-\fr{k}{\nu}|e^{\nu t}} \fr{\left({\bf 1}_{|\nu\tau|\ge\fr{|k|}{4}}+{\bf 1}_{|\nu\tau|<\fr{|k|}{4}}\right)\nu^\fr13}{1+\nu^\fr23\left| \fr{\eta-\fr{k}{\nu}}{|\eta-\fr{k}{\nu}|}\tau+\fr{k}{\nu}\right|^2}\fr{\nu\tau}{\la\nu\tau\ra^2}d\tau={\rm I}_1+{\rm I}_2.
\end{align}
It is easy to check, if $|\nu\tau|>\fr{|k|}{4}$, it holds that
\[
{\rm I}_1\les\fr{1}{|k|}\int_{|\eta-\fr{k}{\nu}|}^{|\eta-\fr{k}{\nu}|e^{\nu t}} \fr{\nu^\fr13}{1+\nu^\fr23\left| \fr{\eta-\fr{k}{\nu}}{|\eta-\fr{k}{\nu}|}\tau+\fr{k}{\nu}\right|^2}d\tau\les\fr{1}{|k|}.
\]
For ${\rm I_2}$, since $\nu\eta-k\neq 0$, we let $\alpha_{k,\eta}=\fr{\eta-\fr{k}{\nu}}{|\eta-\fr{k}{\nu}|}$ and further split it into two parts:
\begin{align}
{\rm I}_2
\nn\les&\int_{|\eta-\fr{k}{\nu}|}^{|\eta-\fr{k}{\nu}|e^{\nu t}} \fr{{\bf 1}_{|\nu\tau|<\fr{|k|}{4}}\nu^\fr13\Big({\bf 1}_{|\alpha_{k,\eta}\cdot k|\le\fr{|k|}{2}}+{\bf 1}_{|\alpha_{k,\eta}\cdot k|>\fr{|k|}{2}}\Big)}{1+\nu^\fr23\left[\left( \tau+\fr{k}{\nu}\cdot \alpha_{k,\eta}\right)^2+\fr{|k|^2}{\nu^2}-\left(\fr{k}{\nu}\cdot\alpha_{k,\eta}\right)^2\right]}d\tau={\rm I}_2^{(1)}+{\rm I}_2^{(2)}.
\end{align}
If $|\alpha_{k,\eta}\cdot k|\le\fr{|k|}{2}$, then $\fr{|k|^2}{\nu^2}-\left(\fr{k}{\nu}\cdot\alpha_{k,\eta}\right)^2\ge\fr34\fr{|k|^2}{\nu^2}$. Thus,
\begin{align}
{\rm I}_2^{(1)}\nn\les&\int_{|\nu\tau|<\fr{|k|}{4}}\fr{1}{\nu^\fr13\fr{|k|^2}{\nu^2}}d\tau
\les \fr{\nu^\fr23}{|k|}.
\end{align}
If $|\alpha_{k,\eta}\cdot k|>\fr{|k|}{2}$, using the fact $|\tau|<\fr{|k|}{4\nu}$, we have  $\left|\tau+\fr{k}{\nu}\cdot\alpha_{k,\eta} \right|\ge \fr{|k|}{4\nu}$. Thus,
\begin{align}
{\rm I}_2^{(2)}\nn\les&\int_{|\eta-\fr{k}{\nu}|}^{|\eta-\fr{k}{\nu}|e^{\nu t}} \fr{{\bf 1}_{|\nu\tau|<\fr{|k|}{4}}\nu^\fr13{\bf 1}_{|\alpha_{k,\eta}\cdot k|>\fr{|k|}{2}}}{1+\nu^\fr23\left( \tau+\fr{k}{\nu}\cdot \alpha_{k,\eta}\right)^2}d\tau\les
\int_{\left|\tau+\fr{k}{\nu}\cdot\alpha_{k,\eta} \right|\ge \fr{|k|}{4\nu}}\fr{ d\nu^\fr13\left(\tau+\fr{k}{\nu}\cdot\alpha_{k,\eta}\right)}{\nu^\fr23\left( \tau+\fr{k}{\nu}\cdot \alpha_{k,\eta}\right)^2}\les \fr{\nu^\fr23}{|k|}.
\end{align}
So the claim \eqref{claim} is true.

{\bf Case 1: $k\ne0$,  $l\ne0$ and $k\ne l$.} We will consider the following three sub-cases.

 {\it Case 1.1: $\nu\eta\ne k$ and $\nu\xi\ne l$. }  By virtue of the elementary inequality $|e^x-1|\le |x|e^{|x|}$ and \eqref{claim}, we have
\begin{align}
\left|1-\fr{\mathfrak{m}_k(t,\eta)}{\mathfrak{m}_l(t,\xi)} \right|
\nn\les&\int_0^t\fr{\nu^\fr13}{1+\left|\nu^\fr13 e^{\nu t'}\eta-\nu^\fr13k\fr{e^{\nu t'}-1}{\nu}\right|^2}\fr{e^{2\nu t'}|\nu\eta-k|^2}{\la e^{\nu t'}(\nu\eta-k)\ra^2}dt'\\
\nn&+\int_0^t\fr{\nu^\fr13}{1+\left|\nu^\fr13 e^{\nu t'}\xi-\nu^\fr13l\fr{e^{\nu t'}-1}{\nu}\right|^2}\fr{e^{2\nu t'}|\nu\xi-l|^2}{\la e^{\nu t'}(\nu\xi-l)\ra^2}dt'\\
\nn\les&\fr{1}{|k|}+\fr{1}{|l|}\les\fr{\la k-l\ra}{\max \{|k|, |l| \}}
\end{align}

In addition, if  $|\xi|\ge2|l|\left\la  t^{\rm ap} \right\ra\ge2\left| l t'^{\mathrm{ap}} \right|$, then
\be\label{gain-xi}
\fr{\nu^\fr13}{\big(1+\nu^\fr23 e^{2\nu t'}\left|\xi-lt'^{\rm ap}\right|^2\big)^\fr12}
\les\fr{1}{1+e^{\nu t'}|\xi|}.
\ee
Combining this with the elementary inequality $|e^x-1|\le |x|e^{|x|}$, we are led to
\begin{align}\label{com-nene}
\left|1-\fr{\mathfrak{m}_k(t,\eta)}{\mathfrak{m}_l(t,\xi)} \right|
\nn\les&\int_0^t\fr{\nu^\fr13}{1+\nu^\fr23e^{2\nu t'}\left| \xi-l t'^{\mathrm{ap}}\right|^2} \fr{1+\nu^\fr23e^{2\nu t'}\left| \xi-l t'^{\mathrm{ap}}\right|^2}{1+\nu^\fr23e^{2\nu t'}\left| \eta-k t'^{\mathrm{ap}}\right|^2}dt'\\
\nn&+\int_0^t\fr{\nu^\fr13}{1+\nu^\fr23 e^{2\nu t'}\left|\xi-lt'^{\rm ap}\right|^2}dt'\\
\nn\les&\int_0^t\fr{\nu^\fr13}{1+\nu^\fr23e^{2\nu t'}\left| \xi-l t'^{\mathrm{ap}}\right|^2}\left(1+ \fr{\nu^\fr23e^{2\nu t'}\big|\left| \xi-l t'^{\mathrm{ap}}\right|^2-\left| \eta-kt'^{\mathrm{ap}}\right|^2\big|}{1+\nu^\fr23e^{2\nu t'}\left| \eta-k t'^{\mathrm{ap}}\right|^2}\right)dt'\\
\nn\les&\fr{t}{\la\xi\ra}+\nu^\fr13\la t^{\rm ap}\ra e^{\nu t}|k-l,\eta-\xi|\int_0^t\fr{\nu^\fr13}{\big(1+\nu^\fr23e^{2\nu t'}\left| \xi-l t'^{\mathrm{ap}}\right|^2\big)^\fr12}dt'\\
\les&\fr{t\la t^{\rm ap}\ra e^{\nu t}  \la k-l, \eta-\xi\ra}{\la\xi\ra}.
\end{align}

{\it Case 1.2: $\nu\eta=k$.} Now we have $m_k(t,\eta)=1$. It suffices to consider the case $\nu\xi\ne l$, otherwise $1-\fr{\mathfrak{m}_k(t,\eta)}{\mathfrak{m}_l(t,\xi)}=0$, there is nothing to prove. Then thanks to \eqref{claim}, 
\begin{align}
\left|1-\fr{\mathfrak{m}_k(t,\eta)}{\mathfrak{m}_l(t,\xi)} \right|\nn=&\left|1-\fr{1}{\mathfrak{m}_l(t,\xi)} \right|\les |1-\frak{m}_l(t,\xi)|\\
\nn\les& \int_0^t\fr{\nu^\fr13}{1+\nu^\fr23e^{2\nu t'}\left| \xi-l t'^{\mathrm{ap}}\right|^2}\fr{e^{2\nu t'}|\nu\xi-l|^2}{\la e^{\nu t'}(\nu\xi-l)\ra^2}dt'\les\fr{1}{|l|}.
\end{align}

If $|\xi|\ge2|l|\left\la  t^{\rm ap} \right\ra$, using \eqref{gain-xi}, we obtain
\begin{align}
\nn\left|1-\fr{\mathfrak{m}_k(t,\eta)}{\mathfrak{m}_l(t,\xi)} \right|\les \int_0^t\fr{\nu^\fr13}{1+\nu^\fr23e^{2\nu t'}\left| \xi-l t'^{\mathrm{ap}}\right|^2}dt'\les\fr{t}{\la\xi\ra}.
\end{align}

{\it Case 1.3: $\nu\xi=l$.} Now we have $m_l(t,\xi)=1$. It suffices to consider the case $\nu\eta\ne k$. In view of \eqref{claim}, we arrive at
\begin{align}
\left|1-\fr{\mathfrak{m}_k(t,\eta)}{\mathfrak{m}_l(t,\xi)} \right|\nn=|1-\frak{m}_k(t,\eta)|\les \int_0^t\fr{\nu^\fr13}{1+\nu^\fr23e^{2\nu t'}\left| \eta-k t'^{\mathrm{ap}}\right|^2}\fr{e^{2\nu t'}|\nu\eta-k|^2}{\la e^{\nu t'}(\nu\eta-k)\ra^2}dt'\les\fr{1}{|k|}.
\end{align}

If $|\xi|\ge2|l|\left\la  t^{\rm ap} \right\ra$, note that when $\nu\xi=l$, \eqref{gain-xi} reduces to
\be\label{gain-xi'}
\fr{\nu^\fr13}{\left(1+\nu^\fr23 e^{2\nu t'}\left|\xi-lt'^{\rm ap}\right|^2\right)^\fr12}
\les\fr{1}{1+|\fr{l}{\nu}|}=\fr{1}{1+|\xi|}.
\ee
Therefore, similar to \eqref{com-nene}, we have
\begin{align}
\nn\left|1-\fr{\mathfrak{m}_k(t,\eta)}{\mathfrak{m}_l(t,\xi)} \right|\les& \int_0^t\fr{\nu^\fr13}{1+\nu^\fr23e^{2\nu t'}\left| \xi-l t'^{\mathrm{ap}}\right|^2} \fr{1+\nu^\fr23e^{2\nu t'}\left| \xi-l t'^{\mathrm{ap}}\right|^2}{1+\nu^\fr23e^{2\nu t'}\left| \eta-k t'^{\mathrm{ap}}\right|^2}dt'\\
\nn\les&\fr{t\la t^{\rm ap}\ra e^{\nu t}  \la k-l, \eta-\xi\ra}{\la\xi\ra}.
\end{align}

{\bf Case 2: $k=0, l\ne0$.} Now it suffices to consider the case $\nu\xi\ne l$.  Using again the elementary inequality $|e^x-1|\le |x|e^{|x|}$ and \eqref{claim}, one deduces that
\beno
\left|1-\fr{\mathfrak{m}_k(t,\eta)}{\mathfrak{m}_l(t,\xi)} \right|\les\int_0^t\fr{\nu^\fr13}{1+\nu^\fr23e^{2\nu t'}\left| \xi-l t'^{\mathrm{ap}}\right|^2}\fr{e^{2\nu t}|\nu\xi-l|^2}{\la e^{\nu t}(\nu\xi-l)\ra^2}dt'\les\fr{1}{|l|}.
\eeno
If $|\xi|\le2|l|\left\la t^{\rm ap} \right\ra$, then
\beno
\left|1-\fr{\mathfrak{m}_k(t,\eta)}{\mathfrak{m}_l(t,\xi)} \right|\les\fr{\left\la t^{\rm ap} \right\ra}{|\xi|}.
\eeno
If $|\xi|>2|l|\left\la t^{\rm ap} \right\ra$, using \eqref{gain-xi} again yields
\beno
\left|1-\fr{\mathfrak{m}_k(t,\eta)}{\mathfrak{m}_l(t,\xi)} \right|\les\int_0^t\fr{\nu^\fr13}{1+\nu^\fr23 e^{2\nu t'}\left|\xi-lt'^{\rm ap}\right|^2}dt\les\fr{ t}{\la \xi\ra}.
\eeno
Then \eqref{0ne} follows from the above three inequalities  immediately, and \eqref{ne0} can be obtained in the same way as \eqref{0ne}. The proof of Lemma \ref{lem-com-m} is completed.
\end{proof}

\section{Littlewood-Paley decomposition and paraproducts}\label{subsec-LP}
In the following, we define the Littlewood-Paley decomposition on $\R^n$. Let $\chi \in  C_0^\infty(\mathbb{R}^n)$ be  a radial function such that
$\chi(\xi)=1$ for $|\xi|\le\fr12$ and $\chi(\xi)=0$ for $|\xi|\ge\fr34$ and define $\varphi(\xi)=\chi\left(\fr{\xi}{2}\right)-\chi(\xi)$, supported in the range $\left\{\xi\in\R^n: \fr12\le|\xi|\le\fr32\right\}$. Then we have the partition of unity
\be\label{dec1}
1=\chi(\xi)+\sum_{N\in 2^{\mathbb{N}}}\varphi_N(\xi),
\ee
where $\varphi_N(\xi)=\varphi\left(\fr{\xi}{N}\right)$. For $f\in L^2(\mathbb{R}^n)$, let us define
\begin{align*}
\hat{f}(\xi)_N&=\varphi_N(\xi)\hat{f}(\xi),\quad
\hat{f}(\xi)_{\fr12}=\chi(\xi)\hat{f}(\xi),\\
\hat{f}(\xi)_{<N}&=\hat{f}(\xi)_{\fr12}+\sum_{K\in2^{\mathbb{N}}: K<M}\hat{f}(\xi)_{K}=\chi\left(\fr{\xi}{N}\right)\hat{f}(\xi).
\end{align*}
Then the Littlewood-Paley decomposition  in the Fourier variable is defined by
\[
\hat{f}(\xi)=\hat{f}(\xi)_{\fr12}+\sum_{N\in2^{\mathbb{N}}}\hat{f}(\xi)_N.
\]
The Littlewood-Paley decomposition can be given analogously when the partition of unity \eqref{dec1} takes the value of a given function in $\xi$.

Now we are in a position to define the paraproduct decomposition, introduced by Bony \cite{Bony1981}. Given suitable functions $f,g$, the product of $fg$ can be decomposed as follows
\beqno
fg&=&T_fg+T_gf+\mathcal{R}(f,g)\\
&=&\sum_{N\ge8}f_{<\fr{N}{8}}g_N+\sum_{N\ge8}g_{<\fr{N}{8}}f_N+\sum_{N\in\mathbb{D}}\sum_{\fr{N}{8}\le N'\le8N}g_{N'}f_N\\
&=&\sum_{N\ge8}f_{<\fr{N}{8}}g_N+\sum_{N\in\mathbb{D}}g_{<16N}f_N,
\eeqno
where all the sums are understood to run over $\mathbb{D}=\left\{\fr12, 1, 2, \cdots, 2^j, \cdots \right\}$.

\end{appendix}

\bigbreak
\noindent{\bf Acknowledgments}
JB was supported by the US NSF Award DMS-2108633. RZ is partially supported by NSF of China under  Grants 12222105. This work was done when RZ was visiting NYU Abu Dhabi. He appreciates the hospitality of NYU.

\noindent{\bf Conflict of interest:}  We confirm that we do not have any conflict of interest. 

\noindent{\bf Data availibility:} The manuscript has no associated data. 

\bibliographystyle{siam.bst} 
\bibliography{references.bib}

\end{document}